\documentclass{amsart}
\usepackage[fontsize=11pt]{scrextend}
\usepackage{amsmath,graphicx,amsfonts,amssymb,amsthm,hyperref,amscd,esint,bbm,verbatim,mathabx,abstract}
\usepackage{nccmath}
\usepackage[margin=1in]{geometry}
\usepackage{mathtools}
\usepackage[style=numeric,backend=biber]{biblatex}
\mathtoolsset{showonlyrefs}
\newtheorem{theorem}{Theorem}[section]
\newtheorem{proposition}{Proposition}[section]

\newtheorem{lemma}[theorem]{Lemma}
\title{Global, Non-scattering solutions to the energy critical wave maps equation}
\date{}
\author{Mohandas Pillai} 
\numberwithin{equation}{section}
\begin{document}
\maketitle
\begin{abstract} \noindent We consider the 1-equivariant energy critical wave maps problem with two-sphere target. Using a method based on matched asymptotic expansions, we construct infinite time relaxation, blow-up, and intermediate types of solutions that have topological degree one. More precisely, for a symbol class of admissible, time-dependent length scales, we construct solutions which can be decomposed as a ground state harmonic map (soliton) re-scaled by an admissible length scale, plus radiation, and small corrections which vanish (in a suitable sense) as time approaches infinity. Our class of admissible length scales includes positive and negative powers of t, with exponents sufficiently small in absolute value. In addition, we obtain solutions with soliton length scale undergoing damped or undamped oscillations in a bounded set, or undergoing unbounded oscillations, for all sufficiently large t.
 \end{abstract}
\tableofcontents
\section{Introduction}
We consider the wave maps equation for maps $\Phi: \mathbb{R}^{1+2} \rightarrow \mathbb{S}^{2}$. This wave maps equation is the Euler-Lagrange equation associated to 
$$\mathcal{S}(\Phi) = \int_{\mathbb{R}^{2+1}} \langle \partial^{\alpha} \Phi(t,x),\partial_{\alpha}\Phi(t,x) \rangle_{g(\Phi(t,x))} dtdx$$
where $g$ denotes the round metric on $\mathbb{S}^{2}$, and the $\alpha$ indices are contracted using the Minkowski metric. We consider the 1-equivariant symmetry reduction of this wave maps equation, which corresponds to writing
$$\Phi_{u}(t,x)=(\sin(u(t,r))\cos(\phi),\sin(u(t,r))\sin(\phi),\cos(u(t,r)))$$
where $(r,\phi)$ are polar coordinates on $\mathbb{R}^{2}$. The resulting equation for $u$ is the following.
\begin{equation}\label{wm} -\partial_{tt}u+\partial_{rr}u+\frac{1}{r}\partial_{r}u-\frac{\sin(2u)}{2r^{2}}=0, \quad r>0\end{equation}
Sufficiently regular solutions to \eqref{wm} satisfy the condition that the energy $E_{WM}(u,\partial_{t}u)$ is independent of time, where
$$E_{WM}(u,v) = \pi \int_{0}^{\infty} \left(v^{2}+\left(\partial_{r}u\right)^{2}+\frac{\sin^{2}(u)}{r^{2}}\right) rdr$$
We note that the family of solitons, $Q_{\lambda}(r)=2\arctan(r\lambda)$, for $\lambda>0$, are solutions to \eqref{wm}, which minimize $E_{WM}(u,0)$ within a class of functions $u$ such that $\Phi_{u}$ has topological degree one.

The work of Shatah and Tahvildar-Zadeh, \cite{stz}, studied the Cauchy problem associated to \eqref{wm}, with data $(u_{0},u_{1})$ such that 
$$(x_{1},x_{2}) \mapsto (\frac{x_{1}u_{0}(r)}{r},\frac{x_{2} u_{0}(r)}{r})\in H^{1}_{\text{loc}}(\mathbb{R}^{2}), \quad (x_{1},x_{2}) \mapsto (\frac{x_{1}u_{1}(r)}{r},\frac{x_{2} u_{1}(r)}{r})\in L^{2}_{\text{loc}}(\mathbb{R}^{2})$$ 
As in the previous work of the author, \cite{wm}, we will say that $u$ is a finite energy solution to \eqref{wm} if $u$ is a solution to \eqref{wm} in the sense of distributions, with $\Phi_{u}\in C^{0}_{t}\dot{H}^{1}(\mathbb{R}^{2})$ and $\partial_{t}\Phi_{u}\in C^{0}_{t}L^{2}(\mathbb{R}^{2})$.
 Throughout this work, we will consider the following wave equation with various right-hand sides.
\begin{equation}\label{linrpot}-\partial_{t}^{2}u+\partial_{r}^{2}u+\frac{1}{r}\partial_{r}u-\frac{u}{r^{2}}=0\end{equation}
The quantity $E(u,\partial_{t}u)$ is formally conserved for solutions to \eqref{linrpot}, where
\begin{equation}\label{ecoercive} E(u,v)=\pi \int_{0}^{\infty} \left(v^{2}+\left(\partial_{r}u\right)^{2}+\frac{u^{2}}{r^{2}}\right) rdr\end{equation}
The work \cite{ckls} of Cote, Kenig, Lawrie, and Schlag classified all solutions, $u$, to \eqref{wm}, which satisfy the condition that $\Phi_{u}$ has topological degree one, and  $E_{WM}(Q_{1},0)<E_{WM}(u,\partial_{t}u) < 3 E_{WM}(Q_{1},0)$. In particular, the result of \cite{ckls} implies that any such solution which also exists globally in time can be decomposed as
\begin{equation}\label{cklsdecomp}Q_{\frac{1}{\lambda(t)}}(r)+\varphi_{L}(t,r)+\epsilon(t,r)\end{equation}
where $\varphi_{L}$ solves \eqref{linrpot}, $\lambda(t)=o(t), \quad t \rightarrow \infty$, and $\epsilon \rightarrow 0$ in an appropriate sense, as $t \rightarrow \infty$. According to \cite{ckls}, at the time of its writing, there were no known constructions of solutions to \eqref{wm} which can be decomposed as above, for $\lambda(t) \rightarrow 0$ or $\lambda(t) \rightarrow \infty$ as $t \rightarrow \infty$. To the knowledge of the author, the only currently known examples of such solutions with $\lambda(t) \rightarrow 0$, are those constructed in the previous work of the author, \cite{wm}. More precisely, for all $b>0$, and all functions $\lambda_{0}\in C^{\infty}([100,\infty))$ which satisfy the following conditions for some constants $C_{l},C_{m},C_{m,k}>0$,
$$\frac{C_{l}}{\log^{b}(t)} \leq \lambda_{0}(t) \leq \frac{C_{m}}{\log^{b}(t)}, \quad |\lambda_{0}^{(k)}(t)| \leq \frac{C_{m,k}}{t^{k} \log^{b+1}(t)}, k\geq 1, \quad t \geq 100$$
the work \cite{wm} constructs finite energy solutions to \eqref{wm}, for $t$ sufficiently large, of the form 
$$u(t,r)=Q_{\frac{1}{\lambda(t)}}(r) + v_{rad}(t,r)+u_{e}(t,r)$$
where $v_{rad}$ solves \eqref{linrpot}, with $E(v_{rad},\partial_{t}v_{rad})<\infty$, 
$$E(u_{e},\partial_{t}\left(Q_{\frac{1}{\lambda(t)}}(r)+u_{e}\right)) \leq \frac{C}{t^{2}\log^{2b}(t)}$$
and $$\lambda(t)=\lambda_{0}(t)+e(t), \quad |e(t)| \leq \frac{C}{\log^{b}(t) \sqrt{\log(\log(t))}}$$
The main result of this work can be summarized as follows. For each positive $\lambda \in C^{\infty}((50,\infty))$ satisfying the following for all $t$ sufficiently large and  $C_{l},C_{u},C_{2}>0$ sufficiently small (see \eqref{lambdasetdef} for the precise conditions)
\begin{equation}\frac{-C_{l}}{t} \leq \frac{\lambda'(t)}{\lambda(t)} \leq \frac{C_{u}}{t}, \frac{|\lambda^{(k)}(t)|}{\lambda(t)} \leq \frac{C_{k}}{t^{k}},\text{ for }k \geq 2\end{equation}
this work constructs a solution to \eqref{wm} which can be decomposed as in \eqref{cklsdecomp} (see Theorem \ref{mainthm} for the precise sense in which $\epsilon \rightarrow 0$). This class of $\lambda$ includes positive and negative powers of $t$, as well as oscillatory functions which satisfy any combination of the following (see the Remarks after Theorem \ref{mainthm})
\begin{equation}\begin{split}&\liminf_{t \rightarrow \infty} \lambda(t) = 0 \text{ or } \liminf_{t \rightarrow \infty} \lambda(t)=\lambda_{0}>0 \quad \text{   and  }
\limsup_{t \rightarrow \infty} \lambda(t) = \lambda_{1}>0 \text{ or } \limsup_{t \rightarrow \infty} \lambda(t)=\infty \\
&\text{ where } \lambda_{0} \leq \lambda_{1}\end{split}\end{equation}
(in addition to oscillatory $\lambda$ such that $\liminf_{t \rightarrow \infty} \lambda(t) =\limsup_{t \rightarrow \infty} \lambda(t)= 0 \text{  or  } \infty$). The method of construction of the ansatz of this work is quite different than that used in \cite{wm}, see Remark 7 after Theorem \ref{mainthm} (as well as Section 3) for a comparison.

To the knowledge of the author, the solutions constructed in this work are the first examples of solutions to \eqref{wm} of the form \eqref{cklsdecomp}, with $\lambda(t) \rightarrow \infty$. Our class of solutions also enlarges the symbol class of known infinite time blow-up rates for \eqref{wm}, and also includes solutions for which $\lambda(t)$ is a power of $t$, or oscillates as described above, see also Remarks 4, 5, 6 after Theorem \ref{mainthm}. 

Before we state our main theorem, we will have to precisely describe the set of admissible $\lambda(t)$, and this will require a short discussion of the work \cite{kst} of Krieger, Schlag, and Tataru. The work \cite{kst} constructed a continuum of finite time blow-up solutions to \eqref{wm}, with blow-up rates given by $\lambda(t)=t^{1+\nu}$, for $\nu>\frac{1}{2}$. (Here, $\lambda(t)$ is the length scale of the soliton). 

When completing our approximate solution to \eqref{wm} to an exact one, we use two important notions from \cite{kst} in this present work. First, we will use the distorted Fourier transform, $\mathcal{F}$, associated to (a conjugation of) the elliptic part of the wave equation obtained by linearizing \eqref{wm} around $Q_{1}$. This distorted Fourier transform is defined in Section 5 of \cite{kst}. Second, we will use the ``transference operator'',$\mathcal{K}$, defined in Section 6 of \cite{kst} by
$$\mathcal{F}(r \partial_{r}u)=-2\xi\partial_{\xi}\mathcal{F}(u)+\mathcal{K}(\mathcal{F}(u))$$
In order to precisely describe the set of $\lambda(t)$ to which our main theorem applies, we will have to define a few absolute constants. First, we let $\rho$ denote the density of the spectral measure of $\mathcal{F}$, as defined in Theorem 5.3 of \cite{kst}. From Proposition 5.7b of \cite{kst}, there exists $C_{\rho}>0$, such that, for all $y,z>0$,
$$\frac{\rho(y)}{\rho(z)} \leq C_{\rho}\left(\frac{y}{z}+\frac{z}{y}\right)$$ 
Second, by Theorem 6.1, and Proposition 6.2 of \cite{kst}, the operators $\mathcal{K}$ and $[\xi \partial_{\xi},\mathcal{K}]$ are bounded on $L^{2,\alpha}_{\rho}$, for example, for $\alpha=0,\frac{1}{2}$, where
\begin{equation}\label{l2alphadef}||f||_{L^{2,\alpha}_{\rho}} = ||f(\xi) \langle\xi\rangle^{\alpha}\sqrt{\rho(\xi)}||_{L^{2}(d\xi)}\end{equation}
Now, we define $\Lambda$ to be the set of positive functions $\lambda\in C^{\infty}((50,\infty))$ such that there exists $T_{\lambda}>100$, constants $C_{l},C_{u},C_{2}\geq 0$ satisfying \eqref{lambdaonlyconstr}, \eqref{cofconstr1p1}, and \eqref{cofconstr2}, and constants $C_{k} \geq0$ for $k \geq 3$, such that the following hold for $t \geq T_{\lambda}$
\begin{equation}\label{lambdasetdef}\frac{-C_{l}}{t} \leq \frac{\lambda'(t)}{\lambda(t)} \leq \frac{C_{u}}{t},\quad \frac{|\lambda^{(k)}(t)|}{\lambda(t)} \leq \frac{C_{k}}{t^{k}},\text{ for }k \geq 2\end{equation}
where, for $M:=\text{max}\{C_{l},C_{u}\}$,
\begin{equation}\label{lambdaonlyconstr} 0\leq C_{u}<\frac{1}{30}-\frac{C_{l}}{5}\end{equation}
\begin{equation}\label{cofconstr1p1}\begin{split} \sqrt{C_{\rho}} \cdot \frac{179}{267} \cdot \begin{aligned}[t]&\left(M(1+2||\mathcal{K}||_{\mathcal{L}(L^{2}_{\rho})})+M^{2}\left(\frac{1}{4}+2||[\xi\partial_{\xi},\mathcal{K}]||_{\mathcal{L}(L^{2}_{\rho})} + ||\mathcal{K}||_{\mathcal{L}(L^{2}_{\rho})}^{2}\right)\right.\\
&\left.+C_{2}\left(\frac{1}{2}+||\mathcal{K}||_{\mathcal{L}(L^{2}_{\rho})}\right)+4\left(M^{2}+\frac{C_{2}}{3}(3+2\pi^{2})\right)\right)< \frac{1}{3}\end{aligned} \end{split}\end{equation}
\begin{equation}\label{cofconstr2}\begin{split}M &< \frac{1}{3 \sqrt{C_{\rho}}} \cdot \frac{1513}{1044} \cdot \frac{1}{\left(1+2||\mathcal{K}||_{\mathcal{L}(L^{2,\frac{1}{2}}_{\rho})}\right)}\end{split}\end{equation}
\emph{Remark}: Given any positive function $f \in C^{\infty}((0,\infty))$ such that there exists $T_{f}>100$ so that
\begin{equation}\label{fsymbol}\frac{|f^{(k)}(t)|}{f(t)} \leq \frac{C_{f,k}}{t^{k}}, \quad k \geq 1, \quad t \geq T_{f}\end{equation}
if, for $d>0$, we define $\lambda$ by
$$\lambda(t) = f(t^{\frac{1}{d}}), t \geq 50$$
then, $\lambda \in \Lambda$ for $d$ chosen sufficiently large so that the smallness constraints \eqref{lambdaonlyconstr} through \eqref{cofconstr2} are satisfied. The main theorem of this paper is the following.
\begin{theorem}\label{mainthm} For all $\lambda \in \Lambda$, there exists $T_{0}>0$ such that there exists a finite energy solution, $u$, to \eqref{wm} for $t \geq T_{0}$, satisfying the following properties.
\begin{equation}\label{solndecomp}u(t,r)=Q_{\frac{1}{\lambda(t)}}(r) + u_{e}(t,r)+v_{rad}(t,r)\end{equation}
where
$$-\partial_{t}^{2}v_{rad}+\partial_{r}^{2}v_{rad}+\frac{1}{r}\partial_{r}v_{rad}-\frac{v_{rad}}{r^{2}}=0, \quad E(v_{rad},\partial_{t}v_{rad})<\infty$$
and
$$E(u_{e},\partial_{t}\left(Q_{\frac{1}{\lambda(t)}}+u_{e}\right)) \leq \frac{C \log^{2}(t)}{t^{2-2C_{u}}}$$
where $C_{u}$ is as in \eqref{lambdasetdef}
\end{theorem}
\textit{Remark 1}. The function $u_{e}$ appearing in \eqref{solndecomp} satisfies
$$u_{e}(t,r) = u_{e,0}(t,r)+v_{6,0}(t,r)$$
where $u_{e,0}$ is fairly explicit, and $v_{6,0}$ is constructed with a fixed point argument, and satisfies
$$(t,r,\theta) \mapsto e^{i \theta} v_{6,0}(t,r) \in C^{0}_{t}([T_{0},\infty),H^{2}(\mathbb{R}^{2}))$$
$$(t,r,\theta) \mapsto e^{i \theta} \partial_{t}v_{6,0}(t,r) \in C^{0}_{t}([T_{0},\infty),H^{1}(\mathbb{R}^{2}))$$
where $(r,\theta)$ are polar coordinates on $\mathbb{R}^{2}$. This follows from the continuity of dilation on $L^{2}$, and Lemma 10.1 of \cite{kst}.\\
\\
\textit{Remark 2}. Our class of solutions includes infinite time relaxation solutions, in other words, solutions of the form \eqref{solndecomp}, for $\lambda(t) \rightarrow \infty$ as $t$ approaches infinity. For example, we can apply the remark before Theorem \ref{mainthm} to $f(t) =t$, to get that
$$\lambda_{1}(t)=t^{C_{u}} \quad t \geq 50$$ 
is in $\Lambda$, for $C_{u}>0$ sufficiently small.\\
\\
\textit{Remark 3}. We also have infinite time blow-up solutions, obtained by applying the remark before Theorem \ref{mainthm} to $f(t)=t^{-1}$, which shows that
$$\lambda_{2}(t)=t^{-C_{l}}, \quad t \geq 50$$
is in $\Lambda$, for $0<C_{l}$ sufficiently small.\\
\\
\textit{Remark 4}. (Oscillatory $\lambda$, part 1). If $0<\lambda_{0} \leq \lambda_{1}$ are two real numbers, and $0 \leq \alpha_{0},\alpha_{1}$ with $\alpha_{1}+\alpha_{0}<1$, define $\lambda$ by
$$\lambda(t) = \frac{\lambda_{0} \log^{-\alpha_{0}}(t) + \lambda_{1} \log^{\alpha_{1}}(t)}{2}+\frac{\left(\lambda_{1}\log^{\alpha_{1}}(t) - \lambda_{0}\log^{-\alpha_{0}}(t)\right)}{2} \sin(\log(\log(t)))$$
Then, $\lambda(t) \geq \lambda_{0} \log^{-\alpha_{0}}(t)$. Therefore,
$$\frac{|\lambda^{(k)}(t)|}{\lambda(t)} \leq \frac{C_{k}}{t^{k}\log^{-\alpha_{0}-\alpha_{1}+1}(t)}, \quad t \geq 50, \quad k \geq 1$$
Since $\alpha_{0}+\alpha_{1}<1$, there exists $T_{\lambda}>100$ so that \eqref{cofconstr1p1},\eqref{cofconstr2}, and \eqref{lambdaonlyconstr} are satisfied for $t \geq T_{\lambda}$ (for appropriate choices of $C_{u},C_{l},C_{2}$). So, $\lambda \in \Lambda$ and, by considering any combination of 
$$\alpha_{0} =0 \text{  or  } \alpha_{0}>0, \quad \text{ and } \alpha_{1}=0 \text{  or  }\alpha_{1}>0$$
we can satisfy any combination of the following
$$\liminf_{t \rightarrow \infty} \lambda(t) = 0 \text{ or } \liminf_{t \rightarrow \infty} \lambda(t)=\lambda_{0} \quad \text{   and  }
\limsup_{t \rightarrow \infty} \lambda(t) = \lambda_{1} \text{ or } \limsup_{t \rightarrow \infty} \lambda(t)=\infty$$ 
\textit{Remark 5}. (Oscillatory $\lambda$, part 2). We have solutions of the form \eqref{solndecomp}, for $\lambda(t)$ with bounded (damped or undamped), or unbounded oscillations for all $t$ sufficiently large. To ease notation, let
\begin{equation}\label{fdef}f(x,y,z):=\begin{aligned}[t]&x(1+2||\mathcal{K}||_{\mathcal{L}(L^{2}_{\rho})})+y\left(\frac{1}{4}+2||[\xi\partial_{\xi},\mathcal{K}]||_{\mathcal{L}(L^{2}_{\rho})} + ||\mathcal{K}||_{\mathcal{L}(L^{2}_{\rho})}^{2}\right)+z\left(\frac{1}{2}+||\mathcal{K}||_{\mathcal{L}(L^{2}_{\rho})}\right)\\
&+4\left(y+\frac{z}{3}(3+2\pi^{2})\right)\end{aligned}\end{equation}
so that the constraint \eqref{cofconstr1p1} is
$$\sqrt{C_{\rho}} \cdot \frac{179}{267} f(M,M^{2},C_{2}) < \frac{1}{3}$$
We can, for example, let 
$$0<c_{0}<\text{min}\{\frac{2}{73},\frac{4 m}{2 m+9}\}$$
and $$0\leq |a| < \frac{c_{0}}{2-c_{0}}$$
where
$$m=\text{min}\{\frac{1513}{3132 \sqrt{C_{\rho}}  (1+2||\mathcal{K}||_{\mathcal{L}(L^{2,\frac{1}{2}}_{\rho})})},\frac{89}{179 \sqrt{C_{\rho}} f(1,1,1)}\}. $$
Then, define
$$C_{u}=|a|+\frac{c_{0}}{2-c_{0}}= C_{l}, \quad C_{2}=\left(|a|+\frac{2c_{0}}{2-c_{0}}\right)(1+|a|)$$
and
$$\lambda_{3}(t):=t^{a}\left(2+c_{0}\sin(\log(t))\right), \quad t \geq 50$$
For instance $C_{2}\geq M \geq M^{2}$.   
We thus have $f(M,M^{2},C_{2}) \leq C_{2} f(1,1,1)$, and this shows that \eqref{cofconstr1p1},\eqref{cofconstr2}, and \eqref{lambdaonlyconstr} are satisfied, and $\lambda_{3}\in\Lambda$. When $a=0$, this is an example of $\lambda(t)$ undergoing undamped oscillations while staying in a bounded set for all $t$ sufficiently large, if $a<0$, $\lambda$ undergoes damped oscillations, and if $a>0$, $\lambda(t)$ oscillates, while staying positive, but is unbounded.\\
\\
\textit{Remark 6}. If $f_{j}$ satisfy \eqref{fsymbol} for $j=1,2$, then, since $f_{j}(t)>0$, we have, for $k \geq 1$, and $t \geq T_{f_{1}}+T_{f_{2}}$,
$$\frac{|(f_{1}+f_{2})^{(k)}(t)|}{f_{1}(t)+f_{2}(t)} \leq \frac{|f_{1}^{(k)}(t)|}{f_{1}(t)}+\frac{|f_{2}^{(k)}(t)|}{f_{2}(t)} \leq \frac{D_{k}}{t^{k}}$$
Therefore, by the remark before Theorem \ref{mainthm}, we get
$$\lambda_{5}(t) = \left(f_{1}+f_{2}\right)(t^{\frac{1}{d}}) \in \Lambda, \quad d>0,\text{ sufficiently large}$$
In particular, for any $c_{0}>0$, some sufficiently small $c_{1},a>0$, and a sufficiently large $d>0$,
$$c_{0}+t^{-|a|}\left(2+c_{1}\sin(\frac{\log(t)}{d})\right) \in \Lambda$$
\textit{Remark 7}. We provide a quick comparison and contrast of the method used in this work and the previous work of the author \cite{wm}. In \cite{wm}, the leading order part of $\lambda(t)$, say $\lambda_{0}(t)$ was chosen from an appropriate class of functions. Then, the radiation $v_{2}$ was inserted into the ansatz by hand, and its data was chosen so as to allow  $\lambda(t)=\lambda_{0}(t)$ to be the leading order solution to an equation resulting from enforcing that the principal part of an error term is orthogonal to $\phi_{0}(\frac{\cdot}{\lambda(t)})$. \\
\\
Here, the exact $\lambda(t)$ is prescribed from the beginning of the argument. An accurate approximate solution is obtained by constructing approximate solutions for small and large $r$, matching them in an intermediate region, and then, writing a function which interpolates between these two approximate solutions (see Section 3 for more information). This method was inspired by the discussions (in the textbooks of Nayfeh, \cite{n}, and Bender-Orszag, \cite{bo}) of matched asymptotic expansions for one-dimensional boundary value problems with a singular perturbation.\\
\\
In order to be able to match our approximate solutions in the intermediate region, we need to use general solutions to various inhomogeneous wave equations, for example a particular solution which has zero Cauchy data at infinity, plus a free wave (radiation). The radiation component of our solution in this work therefore naturally arises from the fact that we need to use general solutions to inhomogeneous wave equations when doing matching, rather than being inserted into the ansatz by hand. The relation between the radiation component of our solution and $\lambda(t)$ is thus determined by a matching condition, rather than by enforcing an orthogonality condition. Despite the very different approaches to the construction of the ansatz, and different methods of determining the relation between $\lambda(t)$ and the radiation, we still have the same leading order relation in this work between $\lambda(t)$ and the radiation (for example, when we restrict attention to $\lambda(t)=\frac{1}{\log^{b}(t)}$ for $b>0$ so that $\lambda$ is in the admissible class of rates from \cite{wm}) as we had in \cite{wm}. (See the discussion following \eqref{v20eqnsummary}).\\
\\
Finally, the method of completion of our ansatz to an exact solution to \eqref{wm} uses a simpler version of the same argument used in the previous work of the author \cite{wm}. (As we will describe in the summary of the proof, there is no orthogonality condition on the error term of our ansatz here, as opposed to \cite{wm}, and this is what makes the iteration here simpler, at the expense of requiring a more accurate ansatz). We remark again that this final step of completing the ansatz to an exact solution uses the distorted Fourier transform, and the transference identity from \cite{kst}.\\
\\
Now, we briefly mention previous results which are related to this work.  For the energy critical wave maps problem with $\mathbb{S}^{m}$ target, the work \cite{tao} of Tao proved global regularity at small energies. The works \cite{stthreshold1} and \cite{stthreshold2}, of Sterbenz and Tataru established a threshold theorem for the energy critical wave maps equation with a general compact Riemannian manifold target. The work \cite{ckls1}, of Cote, Kenig, Lawrie, and Schlag, proved a threshold theorem for the 1-equivariant, energy critical wave maps problem with degree zero data, and a refined threshold, which accounts for the topological degree. The work \cite{lo}, of Lawrie and Oh proved an analogous result, but without the equivariance restriction. 

The works of Jendrej, and Jendrej, Lawrie \cite{j}, \cite{jl}, constructed and classified, respectively, topological degree zero, threshold energy solutions to the $k$-equivariant wave maps equation for $k \geq 2$. The subsequent works of Jendrej and Lawrie, \cite{jendrej2020asymptotic} and \cite{jendrej2020uniqueness} give a more precise classification of the topological degree 0, threshold energy solutions to the $k$-equivariant critical wave map problem with $\mathbb{S}^{2}$ target, for $k \geq 4$. The work, \cite{r}, of Rodriguez, classifies threshold energy degree 0 solutions to the 1-equivariant wave maps problem, in particular obtaining a finite-time blow-up solution in this setting.

 As previously mentioned, the work \cite{kst} of Krieger, Schlag, and Tataru constructed finite time blow-up solutions to \eqref{wm} with a continuum of possible rates. The subsequent work of Gao and Krieger, \cite{gk}, extended the set of solutions constructed in \cite{kst} to include ones for which $\lambda(t) = t^{1+\nu}$, for $\nu>0$. The stability of these solutions under equivariant perturbations was studied in the work of Krieger and Miao \cite{km}, and under non-equivariant perturbations, in the recent work of Krieger, Miao, and Schlag, \cite{kms}. The works \cite{kstym} and \cite{kstslw} of Krieger, Schlag, and Tataru are analogs of \cite{kst} for an equivariant reduction of the $4+1$-dimensional Yang-Mills equation with gauge group $SO(4)$, and the $3+1$-dimensional quintic, focusing semilinear wave equation, respectively. We also remark that the work of Perelman \cite{psm} constructs solutions of a similar form to those of \cite{kst}, for the 1-equivariant Schrodinger map problem with domain $\mathbb{R}^{1+2}$ and target $\mathbb{S}^{2}$. In addition, the work of Rodnianski and Sterbenz, \cite{rs}, constructed finite time blow-up solutions to the $k$ equivariant wave maps problem, for $k \geq 4$. Finite time blow-up solutions to the critical wave maps equation in all equivariance classes, as well as for the $4+1$ dimensional Yang-Mills equation with gauge group $SO(4)$ were constructed in the work \cite{rr} of Raphael and Rodnianski. The work \cite{jlr} of Jendrej, Lawrie, and Rodriguez also constructed new finite time blow-up solutions to \eqref{wm}. The work of Bejenaru, Krieger, and Tataru, \cite{bkt} constructed solutions to \eqref{wm} with energy close to that of $Q_{1}$, and whose modulated soliton component has length scale bounded away from $0$ and infinity, for all time.

The work \cite{dk}, of Donninger and Krieger, constructed infinite time blow-up and infinite time relaxation solutions to the quintic, focusing, energy critical semilinear wave equation on $\mathbb{R}^{1+3}$, with rates $\lambda(t) = t^{\mu}$, where $|\mu|$ is sufficiently small, but $\mu$ can be positive or negative. The procedure used in this work is quite different than that used in \cite{dk}. We also note that the work of Gustafson, Nakanishi, and Tsai, \cite{gnt}, constructs solutions to the 2-equivariant harmonic map heat flow with soliton length scale having several possible asymptotic behaviors, including approaching zero, a positive constant, infinity, or having various combinations of finite or infinite $\limsup$ with positive or zero $\liminf$ as $t$ approaches infinity. 
 \subsection{Acknowledgments}Part of this work was completed while the author was a graduate student of Daniel Tataru, whom the author thanks for useful discussions. This material is based upon work partially supported by the National Science Foundation under Grant No. DMS-1800294.
\section{Notation}
\subsection{Index of terms in the ansatz}\label{indexsubsection}
Our final ansatz will involve several functions, which are listed here, along with the references to the equations they solve. 
\begin{align}&u_{ell}, \eqref{uelleqn} \quad &&u_{ell,2}, \eqref{uell2eqn}\quad &&&u_{w}, \eqref{uweqn} \quad &&&&u_{w,2},\eqref{uw2eqndef}\\
&v_{2,2}, \eqref{v22def} \quad &&f_{5,0}, \eqref{f50def} \quad &&&f_{ex,sub}, \eqref{fexsub}\quad &&&& f_{ell,2}, \text{Lemma } \ref{fell2lemma}\\
&f_{2,2}, \eqref{f22def}\quad &&u_{N_{0}}, \eqref{un0eqn}\quad &&&u_{N_{0},corr}, \eqref{un0corrdef}, \eqref{v24eqn}, \eqref{un0elleqn} \quad &&&& u_{N_{0},corr,2} \eqref{un0corr2eqn}\\
&u_{N_{2}}, \eqref{un2eqn} \end{align}
The first five corrections in the list above are combined into $u_{e}$ and $u_{wave}$, which are defined in \eqref{uewdef}, and are further combined into $u_{c}$, which is defined in \eqref{ucdef}. $u_{c}$ is combined with other corrections, to form $u_{a}$, which is defined in \eqref{uadef}. The sum of the rest of the terms in the listing above are combined into the sum of $u_{n}$ (defined in \eqref{undef}), and $u_{N_{2}}$).\\
If $f \in C^{k}(D)$, and $a \in D$, we let $P_{k,a}(f)(x) = \sum_{j=0}^{k} \frac{f^{(j)}(a)}{j!} (x-a)^{j}$. \\
We let $m_{\leq 1}$ denote the following cutoff function.
\begin{equation}\label{mleq1def}m_{\leq 1}(x) \in C^{\infty}([0,\infty)), \quad 0 \leq m_{\leq 1}(x) \leq 1, \quad m_{\leq 1}(x) =\begin{cases} 1, \quad x \leq \frac{1}{2}\\
0, \quad x \geq 1\end{cases}\end{equation}
We will use the notation
\begin{equation}\label{Inotation}\begin{split}I_{f}(s,y,\rho,\theta) &= \partial_{2}f(s,\sqrt{\rho^{2}+y^{2}+2\rho y \cos(\theta)}) \frac{(y+\rho\cos(\theta))^{2}}{(\rho^{2}+y^{2}+2 \rho y \cos(\theta))} \\
&+ \frac{f(s,\sqrt{\rho^{2}+y^{2}+2 \rho y \cos(\theta)})}{\sqrt{\rho^{2}+y^{2}+2 \rho y \cos(\theta)}} \left(1-\frac{(y+\rho \cos(\theta))^{2}}{\rho^{2}+y^{2}+2 \rho y \cos(\theta)}\right), \quad s > M, y,\rho \geq 0,  y\neq \rho , \theta \in[0, 2\pi]\end{split}\end{equation}
for various choices of $M \gg 1$, and functions $f:(M,\infty)\times [0,\infty) \rightarrow \mathbb{R}$ throughout the paper. \\
We denote, by $I_{n}$ and $K_{n}$, the modified Bessel functions of the first and second kind, respectively. We use the standard definition of $\langle \cdot \rangle$, namely
$$\langle x \rangle = \sqrt{1+x^{2}}$$
The wave maps equation, \eqref{wm}, linearized around $Q_{1}(r)$ takes the form
$$-\partial_{tt}u-L^{*}Lu=0, \quad L(f) = f'(r)-\frac{\cos(Q_{1}(r))}{r} f(r)$$
where $\left(\cdot\right)^{*}$ denotes the $L^{2}(r dr)$ adjoint. We will also use the notation
$$L_{\frac{1}{\lambda(t)}}(f) = f'(r) - \frac{\cos(Q_{1}(\frac{r}{\lambda(t)}))}{r} f(r)$$
 We denote by $e_{2}$ and $\phi_{0}$ the following two linearly independent solutions to $L^{*}L(u)=0$:
\begin{equation}\label{e2def}e_{2}(R) = \frac{R^4+4 R^2 \log (R)-1}{2 R \left(R^2+1\right)}, \quad \phi_{0}(R) = \frac{2 R}{1+R^{2}}\end{equation}
We will use $\widehat{\cdot}$ to denote the Hankel transform of order 1:
$$\widehat{f}(\xi) = \int_{0}^{\infty} J_{1}(r \xi) f(r) r dr$$
Recalling the definition of $||f||_{L^{2,\alpha}_{\rho}}$ given in the introduction, \eqref{l2alphadef}, we will use the notation $||A||_{\mathcal{L}(L^{2,\alpha}_{\rho})}$ to denote the operator norm of a bounded operator $A: L^{2,\alpha}_{\rho} \rightarrow L^{2,\alpha}_{\rho}$. We use the standard notation for the dilogarithm function.
$$\text{Li}_{2}(x) = \int_{0}^{x} -\frac{\log (1-y)}{y} dy$$
\section{Summary of the Proof}\label{summarysection}
We start by letting $\lambda \in \Lambda$, and considering, for $t$ sufficiently large, the modulated soliton $Q_{\frac{1}{\lambda(t)}}(r)$. The error term of the modulated soliton is $\partial_{t}^{2}Q_{\frac{1}{\lambda(t)}}(r)$. We fix $g(t)=t^{\alpha}$, where $\alpha$ satisfies \eqref{alphaconstr}. (A priori, one might not know what constraints $\alpha$ must satisfy, and one could leave $\alpha$ general, until the very end of the argument, at which point the constraints would be clear. This is what was originally done, but, this results in many long expressions that can be greatly simplified once the constraints \eqref{alphaconstr} are imposed. This is why we impose the constraints from the beginning of the argument. We will provide some intuition about why \eqref{alphaconstr} has its particular form throughout this section).

 Our first step is to obtain an approximate solution whose error term is small for all $r$, to the following linear equation
\begin{equation}-\partial_{tt}u_{1}+\partial_{rr}u_{1}+\frac{1}{r}\partial_{r}u_{1} - \frac{\cos(2Q_{1}(\frac{r}{\lambda(t)}))}{r^{2}}u_{1} = \partial_{t}^{2}Q_{1}(\frac{r}{\lambda(t)})\end{equation}
Our plan is to start with a function of the form
\begin{equation}\label{ucsummary}u_{c}(t,r)=\chi_{\leq 1}(\frac{r}{h(t)})u_{small r} + (1-\chi_{\leq 1}(\frac{r}{h(t)}))u_{large r}, \quad h(t):=g(t)\lambda(t)\end{equation}
where  $u_{small r}$ and $u_{large r}$ are good approximate solutions for $r \lesssim h(t)$ and $h(t) \lesssim r$, respectively, and whose asymptotic expansions for large $t$ in the region $r \sim h(t)$ ``match''. Here, by ``match'', we mean that the terms at various orders in the aforementioned expansions of our corrections take the form
$$ \sum_{k,l=-|j|}^{j} q_{k,l}(t) r^{k} \log(r)^{l}$$
for various choices of $j\geq 0$, and the difference between these sums associated to $u_{small r}$ and $u_{large r}$ is equal to zero (see sections \ref{first_order_matching_section}, \ref{second_order_matching_pt1_section}, \ref{second_order_matching_pt2_section}, \ref{third_order_matching_section}). 

We start with a correction to the soliton error term that we will use for $r \lesssim h(t)$. (In other words, we start with a part of $u_{small r}$.)  More precisely, we  consider the general solution (in general, in order to do the matching process described above, one must keep sufficiently many degrees of freedom in the various corrections, which will be later fixed when we impose the matching) to the ODE
$$\partial_{r}^{2}u_{ell}+\frac{1}{r}\partial_{r}u_{ell}-\frac{\cos(2Q_{\frac{1}{\lambda(t)}}(r))}{r^{2}}u_{ell}=\partial_{t}^{2}Q_{\frac{1}{\lambda(t)}}(r)$$
which does not have a singularity near the origin. (This ODE was also considered in the previous works \cite{kst}, \cite{rr}, but we will end up choosing a different solution to this ODE than what was considered in the aforementioned works). The general solution that does not have a singularity at the origin depends on one function of $t$, say, $c_{1}(t)$, see \eqref{velldef}. The linear error term associated to $u_{ell}$ is $\partial_{t}^{2}u_{ell}(t,r)$. (Even though $u_{ell}$ will appear in our ansatz only after being multiplied by $\chi_{\leq 1}(\frac{r}{h(t)})$, it is still useful for us to consider the linear error term associated to $u_{ell}$ alone, when trying to improve the error term in the region $r \lesssim h(t)$). Strictly speaking, this error term depends on $c_{1}(t)$, which is not yet chosen, but we can still give the reader an idea of the size of the error term of $u_{ell}$ by noting that 
$$|\partial_{t}^{2}u_{ell}(t,r)| \leq \begin{cases} \frac{C r \lambda(t)}{t^{4}} + C\frac{r}{\lambda(t)}\sum_{j=0}^{2}\frac{|c_{1}^{(j)}(t)|}{t^{2-j}}, \quad r \leq \lambda(t)\\
\frac{C \lambda(t)}{r} \sum_{j=0}^{2} \frac{|c_{1}^{(j)}(t)|}{t^{2-j}} + \frac{C  r \lambda(t)}{t^{4}} (1+\log(\frac{r}{\lambda(t)})), \quad r \geq \lambda(t)\end{cases}$$
Even for $r \leq \lambda(t)$, the error term of $u_{ell}$ is not quite good enough for our purposes. (The error term, $F_{5}$, of our final ansatz, satisfies the following estimates, for some $\delta>0$ (see Lemma \ref{uansatzerrorest})).
\begin{equation}\label{f5estsummary}\frac{||F_{5}(t,r)||_{L^{2}(r dr)}}{\lambda(t)^{2}} \leq \frac{C_{5} \log^{30}(t)}{t^{4+2\delta}}, \quad ||L_{\frac{1}{\lambda(t)}}F_{5}(t,r)||_{L^{2}(r dr)} \leq \frac{C \log^{2}(t)}{t^{2+5\alpha}}+ \frac{C \log^{6}(t)}{t^{4+2\alpha-3C_{u}}}+\frac{C \log^{30}(t)}{t^{9/2-C_{l}-15C_{u}}}\end{equation}
Therefore, we correct the error term $\partial_{t}^{2}u_{ell}(t,r)$, with a second term in $u_{small r}$, by considering a particular solution to the following ODE (see \eqref{vell2formula0}).
\begin{equation}\partial_{rr}u_{ell,2}(t,r)+\frac{1}{r}\partial_{r}u_{ell,2}(t,r)-\frac{\cos(2Q_{1}(\frac{r}{\lambda(t)}))}{r^{2}} u_{ell,2}(t,r) = \partial_{t}^{2}u_{ell}(t,r)\end{equation}
Defining $u_{small,r}=u_{ell}+u_{ell,2}$ will turn out to be sufficient for us. (Strictly speaking, the error terms of $u_{c}$ depend not only on the quality of the matching between $u_{small r}$ and $u_{large r}$ (for the terms which involve derivatives of $\chi_{\leq 1}$) but also on $c_{1}(t)$, since $u_{ell,2}$ depends on $c_{1}(t)$. Also, $c_{1}(t)$ is chosen during a second order matching process, and is not specifically chosen to make the error of $u_{ell,2}$ small, see section \ref{second_order_matching_pt2_section}).

After defining $u_{ell,2}$, we start defining the various components of $u_{large r}$. In particular, we consider the solution to 
\begin{equation}-\partial_{tt}u_{w}+\partial_{rr}u_{w}+\frac{1}{r}\partial_{r}u_{w}-\frac{u_{w}}{r^{2}} = \partial_{t}^{2} Q_{1}(\frac{r}{\lambda(t)})\end{equation}
given by
$$u_{w}(t,r)=v_{1}(t,r)+v_{2}(t,r)$$
where
$v_{1}$ solves 
$$-\partial_{tt}v_{1}+\partial_{rr}v_{1}+\frac{1}{r}\partial_{r}v_{1}-\frac{v_{1}}{r^{2}} = \partial_{t}^{2} Q_{1}(\frac{r}{\lambda(t)})$$
with 0 Cauchy data at infinity, and $v_{2}$ solves the following Cauchy problem
$$\begin{cases}-\partial_{tt}v_{2}+\partial_{rr}v_{2}+\frac{1}{r}\partial_{r}v_{2}-\frac{v_{2}}{r^{2}} = 0\\
v_{2}(0)=0\\
\partial_{t}v_{2}(0,r)=v_{2,0}(r)\end{cases}$$
where $v_{2,0}\in L^{2}(r dr)$ will be chosen later. We remark that, in general, one needs to use general solutions, rather than particular solutions, to the inhomogeneous equations defining corrections, if one wishes to use the matching procedure described above.  We also remark that, if $\lambda'(t) \neq 0$, then, $(1-\chi_{\leq 1}(\frac{r}{h(t)}))u_{w}$ has infinite kinetic energy, as does $Q_{\frac{1}{\lambda(t)}}(r)$ (because $\partial_{t}Q_{\frac{1}{\lambda(t)}}(r) \not \in L^{2}(r dr)$, for $\lambda'(t) \neq 0$). However, there is sufficient cancellation in the sum $v_{1}(t,r)+Q_{\frac{1}{\lambda(t)}}(r)$ for large $r$, so that $\partial_{t}\left((1-\chi_{\leq 1}(\frac{r}{h(t)}))v_{1}(t,r)+Q_{\frac{1}{\lambda(t)}}(r)\right) \in L^{2}(r dr)$, see \eqref{w1fordelicateenergy}. (The function $w_{1}$ appearing in \eqref{w1fordelicateenergy} is such that $\partial_{t}\left((1-\chi_{\leq 1}(\frac{r}{h(t)}))\left(v_{1}(t,r)-w_{1}(t,r)\right)\right) \in L^{2}(r dr)$).

Next, we do the ``first order matching'' as follows. As long as $\widehat{v_{2,0}}(\xi)$ satisfies, for example, that $|\widehat{v_{2,0}}(\xi)| \leq \frac{C}{\xi^{3}}, \quad \xi \geq \frac{1}{100}$ (so that the integral below in \eqref{v2fo} converges), the leading contributions of $u_{ell}(t,r)$, $v_{1}(t,r)$, and $v_{2}(t,r)$ in the region $r \sim h(t)$ are, respectively:
\begin{equation}u_{ell,first order}=\frac{r}{\lambda(t)} \left(\frac{1}{2}\lambda'(t)^{2} + \lambda(t) \lambda''(t) - \lambda(t) \lambda''(t) \log(\frac{r}{\lambda(t)})\right)\end{equation}

\begin{equation}v_{1,first order}(t,r) = r\left(\log(2)+\frac{1}{2}\right) \lambda''(t) + r \int_{t}^{2t} \frac{(\lambda''(s) - \lambda''(t))}{s-t} ds + r \lambda''(t) \log(\frac{t}{r}) + r \int_{2t}^{\infty} \frac{\lambda''(s) ds}{(s-t)}\end{equation}
 
\begin{equation}\label{v2fo}v_{2,first order}=\frac{-r}{4} \left( -2 \int_{0}^{\infty} \xi \sin(t\xi) \widehat{v_{2,0}}(\xi) d\xi\right)\end{equation}
Note that the $r \log(r)$ terms from $u_{ell,first order}$ and $v_{1,first order}$ already match. Therefore, it is possible to achieve 
$$u_{ell,first order}=v_{1,first order}+v_{2,first order}$$
by choosing $v_{2,0}$ appropriately. In particular, we choose $v_{2,0}$ to satisfy the following equation, for all sufficiently large $t$:
\begin{equation}\label{v20eqnsummary}\begin{split}&F(t):=-2 \int_{0}^{\infty} \xi \sin(t\xi) \widehat{v_{2,0}}(\xi) d\xi \\
&= 4\left(\left(\log(2)-\frac{1}{2}\right)\lambda''(t) + \int_{t}^{2t} \left(\frac{\lambda''(s)-\lambda''(t)}{s-t}\right) ds+\lambda''(t) \log(\frac{t}{\lambda(t)}) + \int_{2t}^{\infty} \frac{\lambda''(s) ds}{s-t} - \frac{\lambda'(t)^{2}}{2 \lambda(t)}\right)\end{split}\end{equation}
In particular, we have
\begin{equation}\widehat{v_{2,0}}(\xi) = \frac{-1}{\pi\xi} \int_{0}^{\infty} H(t) \sin(t\xi) dt\end{equation}
with
\begin{equation}H(t) = 4\left(\left(\log(2)-\frac{1}{2}\right)\lambda''(t) + \int_{t}^{2t} \left(\frac{\lambda''(s)-\lambda''(t)}{s-t}\right) ds+\lambda''(t) \log(\frac{t}{\lambda(t)}) + \int_{2t}^{\infty} \frac{\lambda''(s) ds}{s-t} - \frac{\lambda'(t)^{2}}{2 \lambda(t)}\right) \psi(t)\end{equation}
where $\psi$ is a relatively unimportant cutoff defined in \eqref{psidef}.

The function $v_{2}$ is the leading part of our radiation in the matching region. (The other free waves $v_{2,2}$ and $v_{2,4}$ which are part of our ansatz and added later on in the argument, have much more decay in $t$ in the entire region, for instance, $r \leq \frac{t}{2}$, and contribute to higher-order matching).  Therefore, the relation between the leading part of the radiation in the matching region, and $\lambda(t)$ is determined by the matching of leading terms from parts of $u_{small r}$ with corresponding terms from $u_{large r}$.

Note that the integral on the left-hand side of \eqref{v20eqnsummary} is the precise integral which determined the relation between the radiation and the leading order dynamics of $\lambda(t)$ in the previous work of the author, \cite{wm} (see Section 3, pg. 6 of \cite{wm}, keeping in mind that $K_{1}(x) = \frac{1}{x}+O\left(x |\log(x)|\right), \quad x \rightarrow 0$), though it arose from very different considerations, namely the inner product of the linear error term of $v_{2}$ with $\phi_{0}(\frac{\cdot}{\lambda(t)})$ rather than the near origin behavior of $v_{2}$. The fact that there is a connection between these two quantities also appears in (4.166), pg. 34 of \cite{wm}.

Despite the very different approaches between the two works, if we consider the example, for $b>0$, $\lambda(t)=\lambda_{b}(t)=\frac{1}{\log^{b}(t)}$, then, $\lambda_{b}$ is in the admissible class of $\lambda$ in the work \cite{wm} (as described in the discussion following \eqref{cklsdecomp}), and the leading order relation between $\lambda_{b}(t)$ and the radiation of this work is the same as that of \cite{wm} (see the third equation of pg. 6 of \cite{wm}). 

We provide some intuition on why the $\alpha$ appearing in $g(t)=t^{\alpha}$ has the constraints \eqref{alphaconstr}. Note that, if $u_{ell}(t,r)$ and $u_{w}(t,r)$ are roughly of the same size in a region $r \sim q(t)$ for some $q(t) \gg \lambda(t)$, then, firstly, the linear error term of $u_{ell}(t,r)$ is $\partial_{t}^{2}u_{ell}(t,r)$, and is roughly on the order of $\frac{u_{ell}(t,r)}{t^{2}}$. The linear error term of $u_{w}(t,r)$ is $\left(\frac{\cos(2Q_{1}(\frac{r}{\lambda(t)}))-1}{r^{2}}\right) u_{w}(t,r)$, is roughly $\frac{\lambda(t)^{2}}{q(t)^{4}} u_{w}(t,r)$, for $r \sim q(t)$. Since we assume that $u_{ell}(t,r)$ and $u_{w}(t,r)$ are roughly of the same size for $r \sim q(t)$, the linear error terms would be of comparable size if 
$$q(t) \sim \sqrt{t\lambda(t)}$$
This is not quite exactly what we have, but provides some intuition on why it is natural to expect that $\frac{1}{2}+\frac{C_{l}}{2}$ appears as part of \eqref{alphaconstr}. In addition, we choose to eliminate certain error terms of borderline size which become large if $g(t)$ is taken too large. In this way, certain other error terms of borderline size can be made perturbative by having $\alpha>\frac{1}{2}+\frac{C_{l}}{2}$. On the other hand, some other, less delicate errors become too large if $g(t)$ is taken to be too large, which is why \eqref{alphaconstr} involves an upper bound on $\alpha$ as well.

We then add a second term in $u_{large r}$, namely, $u_{w,2}+v_{2,2}$ (see \eqref{uw2eqndef}), which corrects the linear error term associated to $w_{1}+v_{2}$, where $w_{1}$ is part of $v_{1}$, and is defined in \eqref{w1firsteqn}. We also remark that $u_{w,2}$ is a particular solution to an inhomogeneous wave equation (and it has zero Cauchy data at infinity), while $v_{2,2}$ is a free wave, which we will choose later on, as part of a higher-order matching process.

Next, we consider the terms in an expansion of $u_{large r}$ and $u_{small r}$ in the matching region, which are roughly of size $r^{3} \log^{k}(r)$, multiplied by expressions that involve roughly four derivatives of $\lambda(t)$. In particular, we note that $u_{ell,2}$, $w_{1}$, and $v_{2}$, respectively, contribute the following terms when expanded in the matching region.
\begin{equation}v_{ell,2,0,main}(t,\frac{r}{\lambda(t)}) = \frac{r^{3}}{8} \partial_{t}^{2}\left(\frac{\lambda'(t)^{2}}{2\lambda(t)} + \lambda''(t) - \lambda''(t) \log(\frac{r}{\lambda(t)})\right) + \frac{3}{32} r^{3} \lambda''''(t)\end{equation}

\begin{equation}\begin{split}w_{1,cubic,main}(t,r) &=\frac{3}{32} r^{3} \lambda''''(t)+\frac{r^{3}}{8}\left( \lambda''''(t) \left(\log(2)+\frac{1}{2}\right)-\log(r) \lambda''''(t) + \log(t) \lambda''''(t)\right. \\
&\left.+ \int_{t}^{2t} \frac{\lambda''''(s)-\lambda''''(t)}{s-t} ds + \int_{2t}^{\infty} \frac{\lambda''''(s) ds}{s-t}\right)\end{split}\end{equation}

$$v_{2,cubic,main}(t,r) = \frac{-r^{3}}{32} F''(t)$$

Note that these quantities do not involve any degrees of freedom (like, for instance, integration constants, or free waves) which can be tuned so as to guarantee matching. In fact, the coefficient of $r^{3}$ in  $v_{2,cubic,main}$  is precisely one-eighth of two time derivatives of the $r$ coefficient in $v_{2,first order}$ (recall \eqref{v2fo}). On the other hand, the $r^{3}$ coefficient of $v_{ell,2,0,main}(t,\frac{r}{\lambda(t)})$ and that of $w_{1,cubic,main}(t,r)$ are not precisely one-eighth of two time derivatives of the $r$ coefficients of $u_{ell,first order}(t,r)$ and $v_{1,first order}(t,r)$, respectively. However, the $r^{3}$ coefficient of the \emph{difference} of these two functions is precisely one-eighth of two time derivatives of the $r$ coefficient of $u_{ell,first order}(t,r)-v_{1,first order}(t,r)$. Note the important cancellation between the $\frac{3}{32} r^{3} \lambda''''(t)$ terms when $v_{ell,2,0,main}(t,\frac{r}{\lambda(t)})$ and $w_{1,cubic,main}(t,r)$ are subtracted. Therefore, the matching of the $r^{3}$ terms is already accomplished with our above choice of $v_{2,0}$. In other words, by the choice of $v_{2,0}$, we have
\begin{equation} v_{ell,2,0,main}(t,\frac{r}{\lambda(t)})-\left(w_{1,cubic,main}(t,r)+v_{2,cubic,main}(t,r)\right) =0\end{equation}
Next, we consider the terms arising from expansions of $u_{small r}$ and $u_{large r}$ in the matching region which are roughly of size $\frac{\log^{k}(r)}{r}$ multiplied by roughly two derivatives of $\lambda(t)$.
These terms coming from $u_{ell}$, $u_{w}$, and $u_{w,2}$, respectively, are
\begin{equation}\begin{split}&v_{ell,sub,cont}(t,r)\\
 &= \frac{c_{1}(t) \lambda(t)-\frac{3}{2} \lambda(t) \lambda'(t)^2}{r} \\
& -\frac{\lambda(t)^2 \lambda''(t) \left(-6 \log (r) (4 \log (\lambda(t))+1)+6 \log (\lambda(t)) (2 \log (\lambda(t))+1)+12 \log ^2(r)+\pi ^2+12\right)}{6 r}\end{split}\end{equation}

\begin{equation}v_{ex,cont}(t,r) = \frac{1}{2r}\begin{aligned}[t]&\left(\log(r)\left(-4 \lambda(t)\lambda'(t)^{2}-2 \lambda(t)^{2}\lambda''(t)\right)\right.\\
&+\left.4 \lambda(t)\log(\lambda(t))\lambda'(t)^{2}-\lambda(t)^{2}\lambda''(t) + 2 \lambda(t)^{2}\lambda''(t)\log(\lambda(t))\right)\end{aligned}\end{equation}
and
\begin{equation}\begin{split}u_{w,2,ell,0,cont}(t,r) = \frac{-2}{r}&\left(\log(r)\left(-\lambda(t)\lambda'(t)^{2}-\lambda(t)^{2}\lambda''(t)-2 \lambda(t)^{2} \log(\lambda(t))\lambda''(t)\right)\right.\\
&\left.+\lambda(t) \log(\lambda(t))\lambda'(t)^{2}+\lambda(t)^{2}\log(\lambda(t))\lambda''(t) + \lambda(t)^{2}\log^{2}(\lambda(t))\lambda''(t)\right.\\
&+\left.\lambda(t)^{2}\log^{2}(r)\lambda''(t)+\frac{1}{2} \lambda(t)^2 \lambda''(t)-\frac{1}{12} \pi ^2 \lambda(t)^2 \lambda''(t)\right)\end{split}\end{equation}
Note that the only free parameter we can choose in this expression is $c_{1}$, which multiplies $\frac{\lambda(t)}{r}$. On the other hand, both $u_{w,2,ell,0,cont}$ and $v_{ex,cont}$ involve $\frac{\log^{k}(r)}{r}$ terms, for $k> 0$. It therefore appears that it is not possible to enforce
$$v_{ell,sub,cont}(t,r)=v_{ex,cont}(t,r)+u_{w,2,ell,0,cont}(t,r)$$
by choosing $c_{1}(t)$ appropriately. However, it turns out that the $\frac{\log(r)}{r}$ and $\frac{\log^{2}(r)}{r}$ terms from $v_{ell,sub,cont}$ happen to exactly match those from $v_{ex,cont}+u_{w,2,ell,0,cont}$, which therefore allows us to choose $c_{1}(t)$ as just discussed. (This ``automatic'' matching of the logarithmically higher order terms for large $r$ is reminiscent of what we saw at the first order as well, recall the remark after \eqref{v2fo}). In other words, we have
\begin{equation}v_{ell,sub,cont}(t,r)-\left(u_{w,2,ell,0,cont}(t,r)+v_{ex,cont}(t,r)\right) = -\frac{\lambda(t) \left(-6 c_{1}(t)+\left(3+2 \pi ^2\right) \lambda(t) \lambda''(t)+9 \lambda'(t)^2\right)}{6 r}\end{equation}
which vanishes if 
\begin{equation}c_{1}(t) := \frac{3}{2}\lambda'(t)^{2}+\frac{\lambda(t) \lambda''(t)}{2}+\frac{\pi^{2}}{3} \lambda(t)\lambda''(t)\end{equation}
From the point of view of constructing an approximate solution of the form \eqref{ucsummary}, the matching conditions are needed to reduce the size of error terms involving derivatives of $\chi_{\leq 1}(\frac{r}{h(t)})$. Our three matching conditions thus far reduce these error terms, but not quite enough, so that we need to do a third order matching. We remark that the previous two matching conditions are called second order matching, part 1 and 2 in the paper, since, if $\overline{g}(t)=\sqrt{\lambda(t) t}$ (again, we do not quite exactly choose $g(t)=\overline{g}(t)$, but $\overline{g}$ is the scale at which the elliptic and wave corrections are expected to have comparably sized error terms) then, the terms from the previous two matching conditions are of comparable size:
$$\frac{\overline{g}(t)^{3} \lambda(t)}{t^{4}} =\frac{\lambda(t)^{\frac{5}{2}}}{t^{5/2}}, \quad \text{ and } \frac{\lambda(t)^{3}}{t^{2} \cdot \overline{g}(t)} = \frac{\lambda(t)^{5/2}}{t^{5/2}}$$ 

The third order matching involves comparing terms from $u_{ell,2}$, $v_{1}$, $u_{w,2}$, and $v_{2,2}$, which are on the order of $r \log^{k}(r)$, for $0 \leq k \leq 3$, multiplied by terms involving roughly four derivatives of $\lambda$. The computations of these terms are lengthy, and the terms themselves result in  very long expressions, which we will not reproduce here. We remark that the $u_{ell,2}$ contributions are given in \eqref{ue3def1}, and the $v_{1}$ and $u_{w,2}$ contributions are individually computed in Lemmas \ref{uw2sub0leading}, \ref{vexsub0leading}, \ref{ellminusell0},  \eqref{q50defthirdorder} combined with Lemma \ref{q5lemma}, and \eqref{q410def}, combined with Lemma \ref{q41lemma}. A very careful inspection reveals that, for $j=1,2,3$ the $r \log^{j}(r)$ terms involved in \eqref{ue3def1} exactly match those terms from the $v_{1}$ and $u_{w,2}$ contributions (which we denote, in \eqref{uw3finalexpr}, by $u_{w,3}$) see \eqref{f3def}. We therefore again see the ``automatic'' matching of logarithmically higher order in $r$ (for large $r$) terms involved in the expansions. Precisely because of this ``automatic'' matching, we can choose the data for $v_{2,2}$ in a similar way for $v_{2}$ in order to do the third order matching, which is precisely stated in Proposition \ref{thirdorderprop}. 

There are a few more corrections needed to be added to our ansatz in order to improve the linear terms, namely $f_{5,0}$, $f_{ex,sub}, f_{ell,2}$, and $f_{2,2}$. These corrections are not as delicate as those mentioned up to this point, so we refer the reader to Section \ref{indexsubsection} and the equation references therein. We also remark that the cutoff, $\chi_{\leq 1}(x)=1-\chi_{\geq 1}(x)$, where $\chi_{\geq 1}(x)$ is chosen in Lemma \ref{chilemma}, so as to satisfy certain orthogonality conditions, for technical reasons.

At this stage, our ansatz is 
\begin{equation} u_{a}(t,r) = u_{c}(t,r)+f_{5,0}(t,r)+f_{ex,sub}(t,r)+f_{ell,2}(t,r)+f_{2,2}(t,r)\end{equation}
with
\begin{equation}u_{c}(t,r) = \chi_{\leq 1}(\frac{r}{\lambda(t) g(t)}) u_{e}(t,r) + \left(1-\chi_{\leq 1}(\frac{r}{\lambda(t) g(t)})\right)u_{wave}(t,r)\end{equation}
and
\begin{equation}u_{e}(t,r) = u_{ell}(t,r)+u_{ell,2}(t,r), \quad u_{wave}(t,r) = u_{w}(t,r)+u_{w,2}(t,r)+v_{2,2}(t,r)\end{equation}

The ansatz $u_{a}$ has a linear error term, $e_{a}$, which is small enough for our purposes, see Lemma \ref{eaestlemma}. We also remark that the radiation component of our ansatz thus far, is $v_{2}+v_{2,2}$. On the other hand, the nonlinear error terms of $u_{a}$ are not small enough. For example, the $v_{2}$ cubic self-interactions near the cone are roughly of the size
$$|\frac{\cos(2Q_{1}(\frac{r}{\lambda(t)}))}{2r^{2}} \left(\sin(2v_{2})-2v_{2}\right)| \leq \frac{C}{t^{2}} \cdot \frac{1}{t^{3/2}}$$
More precisely, the set of nonlinear interactions between the terms of $u_{a}$ is decomposed into $N_{0}+N_{1}$, where $N_{1}$ is perturbative, and 
\begin{equation}N_{0}(t,r) = \frac{\cos(2Q_{1}(\frac{r}{\lambda(t)}))}{2r^{2}} \left(\sin(2u_{a,0})-2u_{a,0}\right) + \frac{\sin(2Q_{1}(\frac{r}{\lambda(t)}))}{2r^{2}} \left(\cos(2u_{a,0})-1\right)\end{equation}
where
$$u_{a,0}(t,r) = \chi_{\leq 1}(\frac{r}{h(t)}) u_{ell}(t,r) + \left(1-\chi_{\leq 1}(\frac{r}{h(t)})\right)\left(w_{1}(t,r)+v_{2}(t,r)+v_{2,2}(t,r)\right)$$
First, we define $u_{N_{0}}$ to be the solution to the following equation with zero Cauchy data at infinity.
\begin{equation}\label{un0summary}-\partial_{t}^{2}u_{N_{0}}+\partial_{r}^{2}u_{N_{0}}+\frac{1}{r}\partial_{r}u_{N_{0}}-\frac{u_{N_{0}}}{r^{2}} = N_{0}(t,r)\end{equation}
The linear error term of $u_{N_{0}}$ is $e_{N_{0}}=\left(\frac{\cos(2Q_{1}(\frac{r}{\lambda(t)}))-1}{r^{2}}\right) u_{N_{0}}$, which is small for $r$ such that $h(t) \lesssim r$ (see Lemma \ref{en0estlemma} for the precise details). Therefore, the addition of $u_{N_{0}}$ reduces us to the task of eliminating an error term which is localized to the region $r \lesssim h(t)$. We carry out this task as follows. We start with solving the equation
\begin{equation}\partial_{rr}u_{N_{0},ell}+\frac{1}{r}\partial_{r}u_{N_{0},ell}-\frac{\cos(2Q_{\frac{1}{\lambda(t)}}(r))}{r^{2}}u_{N_{0},ell} = m_{\leq 1}(\frac{r}{2h(t)}) e_{N_{0}}(t,r)\end{equation}
and then inserting the following function into the ansatz
\begin{equation}u_{N_{0},corr}(t,r):=\left(u_{N_{0},ell}(t,r) - \frac{r \lambda(t)}{4} \langle m_{\leq 1}(\frac{R \lambda(t)}{2h(t)}) e_{N_{0}}(t,R \lambda(t)),\phi_{0}(R)\rangle_{L^{2}(R dR)}\right) m_{\leq 1}(\frac{2 r}{t}) + v_{2,4}(t,r)\end{equation}
where $v_{2,4}$ solves
\begin{equation}-\partial_{tt}v_{2,4}+\partial_{rr}v_{2,4}+\frac{1}{r}\partial_{r}v_{2,4}-\frac{v_{2,4}}{r^{2}}=0\end{equation}
and $v_{2,4}(t,r)$ matches $\frac{r \lambda(t)}{4} \langle m_{\leq 1}(\frac{R \lambda(t)}{2h(t)}) e_{N_{0}}(t,R \lambda(t)),\phi_{0}(R)\rangle_{L^{2}(R dR)}$ for small $r$. (where $m_{\leq 1}$ is a cutoff, and is otherwise unimportant). In particular, we choose the initial velocity, say $v_{2,5}$ of $v_{2,4}$ by requiring, for all $t$ sufficiently large,
$$\frac{-r}{4}\cdot -2 \int_{0}^{\infty} \xi \sin(t\xi) \widehat{v_{2,5}}(\xi) d\xi =\frac{r \lambda(t)}{4} \langle m_{\leq 1}(\frac{R \lambda(t)}{2h(t)}) e_{N_{0}}(t,R \lambda(t)),\phi_{0}(R)\rangle_{L^{2}(R dR)}$$ 
This is analogous to how we chose the data for $v_{2}$. The point is that the error term of $v_{2,4}(t,r)$ is worst for small $r$, while that of $u_{N_{0},ell}(t,r)$ is worst for large $r$. On the other hand, the error term of $v_{2,4}(t,r) - \frac{r \lambda(t)}{4} \langle m_{\leq 1}(\frac{R \lambda(t)}{2h(t)}) e_{N_{0}}(t,R \lambda(t)),\phi_{0}(R)\rangle_{L^{2}(R dR)}$ is much smaller than the error term of $v_{2,4}(t,r)$ alone, for small $r$, and the error term of \\
$u_{N_{0},ell}(t,r)- \frac{r \lambda(t)}{4} \langle m_{\leq 1}(\frac{R \lambda(t)}{2h(t)}) e_{N_{0}}(t,R \lambda(t)),\phi_{0}(R)\rangle_{L^{2}(R dR)}$ is much smaller than the error term of $u_{N_{0},ell}(t,r)$ alone for large $r$. So, the term $ \frac{r \lambda(t)}{4} \langle m_{\leq 1}(\frac{R \lambda(t)}{2h(t)}) e_{N_{0}}(t,R \lambda(t)),\phi_{0}(R)\rangle_{L^{2}(R dR)}$ cancels the worst behavior of both $v_{2,4}$ and $u_{N_{0},ell}$ in the regions where the associated error terms are largest. This is again reminiscent of arguments related to matched asymptotic expansions for one-dimensional boundary value problems with a singular perturbation, see \cite{n}. After adding one more term to our ansatz, namely $m_{\geq 1}(\frac{r}{t}) u_{N_{0},corr,2}(t,r)$, in order to eliminate the linear error term associated to $v_{2,4}$ for large $r$, we have the improved ansatz $u_{a}+u_{n}$, where
\begin{equation}u_{n}(t,r):=u_{N_{0}}(t,r)+u_{N_{0},corr}(t,r)+m_{\geq 1}(\frac{r}{t}) u_{N_{0},corr,2}(t,r)\end{equation}
At this stage, the linear error terms of $u_{a}+u_{n}$ are perturbative, but, some nonlinear interactions are not. In particular, the correction $u_{n}$ involves a free wave, $v_{2,4}(t,r)$, which has no more pointwise decay in $t$ near the cone than $v_{2}$ did.
The nonlinear error terms of $u_{a}+u_{n}$ are decomposed into $N_{2}+N_{3}$, where
\begin{equation}\begin{split}N_{2}(t,r):&=\frac{\cos (2 Q_{1}(\frac{r}{\lambda(t)}))}{2r^{2}} (-2 (u_{a0}+v_{2,4})+\sin (2 (u_{a,0}+v_{2,4}))-(\sin (2 u_{a,0})-2 u_{a,0}))\\
&+\frac{\sin (2 Q_{1}(\frac{r}{\lambda(t)}))}{2r^{2}} (\cos (2 u_{a,0}+2 v_{2,4})-\cos (2 u_{a,0}))\\
\end{split}\end{equation}
and the precise formula for $N_{3}$ is not important for the purposes of summarizing the main points of the argument. $N_{3}$ is defined in \eqref{n3exp}, and is perturbative.

We recall that, after adding $u_{a}$ into our ansatz, we had perturbative linear error terms, and some non-perturbative nonlinear error terms, collected together in $N_{0}$. After adding the correction to $N_{0}$, namely $u_{n}$, we again have perturbative linear error terms, and some non-perturbative nonlinear error terms, and the pointwise in $t$ decay of the nonlinear self interactions of $v_{2,4}(t,r)$ near the cone are no better than that of the nonlinear self interactions of $v_{2}(t,r)$ near the cone. It thus might not be clear why $u_{a}+u_{n}$ is an improvement over $u_{a}$. However, the point is that $v_{2,4}(t,r)$ has much more decay in the variable $\langle t-r \rangle$, in the region $\frac{t}{2} \leq r \leq t$, than does $v_{2}$ (compare Lemma \ref{v24estlemma}, and Lemma \ref{v2estlemma}). This allows us to eliminate $N_{2}$ with a correction, $u_{N_{2}}$, which solves the following equation with 0 Cauchy data at infinity.
\begin{equation}-\partial_{t}^{2}u_{N_{2}}+\partial_{r}^{2}u_{N_{2}}+\frac{1}{r}\partial_{r}u_{N_{2}}-\frac{u_{N_{2}}}{r^{2}}=N_{2}(t,r)\end{equation}

It is possible to eliminate $N_{2}$ in this manner because the decay of $v_{2,4}(t,r)$ in the variable $\langle t-r \rangle$ inside the cone leads to $u_{N_{2}}$ having decay in the variable $\langle t-r \rangle$ inside the light-cone, see Lemma \ref{un2lemma}. This decay of $u_{N_{2}}$ is important because the linear error term of $u_{N_{2}}$ is $\left(\frac{\cos(2Q_{1}(\frac{r}{\lambda(t)}))-1}{r^{2}}\right)u_{N_{2}}(t,r)$, which is thus most delicate in the region, for example, $r \leq \frac{t}{2}$, because
$$|\left(\frac{\cos(2Q_{1}(\frac{r}{\lambda(t)}))-1}{r^{2}}\right)u_{N_{2}}(t,r)| \leq \frac{C \lambda(t)^{2}}{t^{4}} |u_{N_{2}}(t,r)|, \quad r \geq \frac{t}{2}$$ 
On the other hand, when $r \leq \frac{t}{2}$, then, $\frac{1}{\langle t-r \rangle} \leq \frac{C}{t}$, which means that the aforementioned decay of $u_{N_{2}}(t,r)$ in the variable $\langle t-r\rangle$, for $r<t$, makes the linear error term associated to $u_{N_{2}}$ much smaller than otherwise.

Without this extra decay in the variable $\langle t-r \rangle$, we might have had to eliminate $N_{2}$ by using a matching process involving a free wave, as in $u_{N_{0},corr}$, but then, the nonlinear self interactions of this free wave would yet again produce nonlinear error terms which are not perturbative from the point of view of our procedure. The key point here is that $u_{N_{2}}$ has zero Cauchy data at infinity, and much more than $\frac{1}{\sqrt{t}}$ pointwise decay near the cone.

The linear error term of $u_{N_{2}}$ is perturbative, see Lemma \ref{en231estlemma}. Our final ansatz is
$$u_{ansatz}(t,r)=u_{a}(t,r)+u_{n}(t,r)+u_{N_{2}}(t,r)$$
and its nonlinear error terms are perturbative, see Lemma \ref{n4estlemma}, and recall the importance of $u_{N_{2}}$ eliminating $N_{2}$, while also not containing a free wave (which leads to improved decay  of $u_{N_{2}}$ near the cone).\\
\\
At this stage, we are ready to complete the ansatz to an exact solution of \eqref{wm}. This step is done by following a simpler version of the analogous step in \cite{wm}, with one extra detail. For completeness, we provide a summary of this step here. Substituting 
$$u(t,r)=Q_{1}(\frac{r}{\lambda(t)})+u_{ansatz}(t,r)+v_{6}(t,r)$$
into \eqref{wm}, we get
\begin{equation}\label{v6eqnsummary}\begin{split}&-\partial_{t}^{2}v_{6}+\partial_{r}^{2}v_{6}+\frac{1}{r}\partial_{r}v_{6}-\frac{\cos(2Q_{1}(\frac{r}{\lambda(t)}))}{r^{2}}v_{6}=F_{5}+F_{3}(v_{6})\end{split}\end{equation}
where $F_{3}$ is defined in \eqref{finalf3def}, and contains error terms involving $v_{6}$ both linearly and nonlinearly, and $F_{5}$ is the error term of $u_{ansatz}$, which we recall satisfies \eqref{f5estsummary}. We remark that \eqref{f5estsummary} is satisfied for $\delta>0$ partly due to \eqref{lambdaonlyconstr} and \eqref{alphaconstr}. We will solve \eqref{v6eqnsummary} by first formally deriving the equation for $y$ (namely \eqref{yeqnsummary}) given by
\begin{equation}y(t,\xi) = \mathcal{F}(\sqrt{\cdot} v_{6}(t,\cdot \lambda(t)))(\xi \lambda(t)^{2})\end{equation}
where $\mathcal{F}$ denotes the distorted Fourier transform of \cite{kst} (which is defined in section 5 of \cite{kst}). Then, we will prove that \eqref{yeqnsummary} admits a solution, say $y_{0}$ (with 0 Cauchy data at infinity) which has enough regularity to rigorously justify the statement that if $v_{6}$ given by the following expression, with $y=y_{0}$
\begin{equation}\label{v6intermsofysummary}v_{6}(t,r) = \sqrt{\frac{\lambda(t)}{r}} \mathcal{F}^{-1}\left(y(t,\frac{\cdot}{\lambda(t)^{2}})\right)\left(\frac{r}{\lambda(t)}\right),\end{equation}
then, $v_{6}$ is a solution to \eqref{v6eqnsummary}. We have (see also (5.4), (5.5), pg. 145 of \cite{wm})
\begin{equation}\label{yeqnsummary}\begin{split} \partial_{tt}y+\omega y = -\mathcal{F}(\sqrt{\cdot}F_{5}(t,\cdot \lambda(t)))(\omega \lambda(t)^{2})+F_{2}(y)(t,\omega) -\mathcal{F}(\sqrt{\cdot}F_{3}(v_{6}(y))(t,\cdot \lambda(t)))(\omega \lambda(t)^{2})\end{split}\end{equation}
where $v_{6}(y)$, which appears in the argument of $F_{3}$, is the expression given in \eqref{v6intermsofysummary}, and
\begin{equation}\begin{split}F_{2}(y)(t,\omega) &=-\frac{\lambda'(t)}{\lambda(t)} \partial_{t}y(t,\omega) + \frac{2 \lambda'(t)}{\lambda(t)} \mathcal{K}\left(\partial_{1}y(t,\frac{\cdot}{\lambda(t)^{2}})\right)(\omega \lambda(t)^{2}) +\left(\frac{-\lambda''(t)}{2\lambda(t)}+\frac{\lambda'(t)^{2}}{4 \lambda(t)^{2}}\right)y(t,\omega) \\
&+ \frac{\lambda''(t)}{\lambda(t)} \mathcal{K}\left(y(t,\frac{\cdot}{\lambda(t)^{2}})\right)(\omega \lambda(t)^{2}) +2\frac{\lambda'(t)^{2}}{\lambda(t)^{2}}\left([\xi \partial_{\xi},\mathcal{K}](y(t,\frac{\cdot}{\lambda(t)^{2}}))\right)(\omega \lambda(t)^{2}) \\
&-\frac{\lambda'(t)^{2}}{\lambda(t)^{2}} \mathcal{K}\left(\mathcal{K}(y(t,\frac{\cdot}{\lambda(t)^{2}}))\right)(\omega \lambda(t)^{2})\end{split}\end{equation}
where $\mathcal{K}$ is the transference operator of \cite{kst} (which is defined in section 6 of \cite{kst}). We solve \eqref{yeqnsummary} in the space $Z$ , defined in \eqref{Zdef}, by showing that the operator
\begin{equation}\label{tdefsummary}T(y)(t,\omega):=\int_{t}^{\infty}\frac{\sin((x-t)\sqrt{\omega})}{\sqrt{\omega}}\left(-\mathcal{F}(\sqrt{\cdot}\left(F_{5}+F_{3}(v_{6}(y))\right)(x,\cdot \lambda(x)))(\omega \lambda(x)^{2})+F_{2}(y)(x,\omega)\right) dx\end{equation}
is a strict contraction on $\overline{B}_{1}(0) \subset Z$. This process essentially only uses Minkowski's inequality, as well as the following simple property of the density of the spectral measure of $\mathcal{F}$: \eqref{rhoscaling}. The reason why this iteration is simpler than that of \cite{wm} is that here, there is no orthogonality condition on the error which we need to exploit, unlike in \cite{wm}. On the other hand, there is one extra detail which is present here, that is not present in \cite{wm}. One of the estimates we have on $F_{2}$ is the following (see Lemma \ref{f2lemma} for the complete set of estimates).
\begin{equation}\label{f2summaryest}\begin{split}&||F_{2}(t,\omega)||_{L^{2}(\rho(\omega \lambda(t)^{2})d\omega)}\\
&\leq ||\partial_{1}y(t,\omega)||_{L^{2}(\rho(\omega \lambda(t)^{2})d\omega)} \frac{|\lambda'(t)|}{\lambda(t)} \left(1+2 ||\mathcal{K}||_{\mathcal{L}(L^{2}_{\rho})}\right)  \\
&+ ||y(t,\omega)||_{L^{2}(\rho(\omega \lambda(t)^{2})d\omega)} \begin{aligned}[t]&\left(\left(\frac{\lambda'(t)}{\lambda(t)}\right)^{2} \left(\frac{1}{4}+2||[\xi\partial_{\xi},\mathcal{K}]||_{\mathcal{L}(L^{2}_{\rho})} + ||\mathcal{K}||^{2}_{\mathcal{L}(L^{2}_{\rho})}\right) + \frac{|\lambda''(t)|}{\lambda(t)} \left(\frac{1}{2}+||\mathcal{K}||_{\mathcal{L}(L^{2}_{\rho})}\right)\right)\end{aligned}\end{split}\end{equation}
Using the symbol-type estimates on $\lambda(t)$, we see that the terms of \eqref{f2summaryest} are of critical size, noting that \eqref{tdefsummary} roughly loses two powers of $t$ decay relative to its integrand. Therefore, we need the constants appearing in the symbol type estimates on $\lambda$ to be sufficiently small in order to guarantee that \eqref{tdefsummary} is a contraction. All of the terms estimated  in Lemmas \ref{l1nlemma} and \ref{f2lemma} with quantitative constants are where \eqref{lambdaonlyconstr}, \eqref{alphaconstr}, \eqref{cofconstr1p1} and \eqref{cofconstr2} are used. This detail does not appear in \cite{wm} because, there, $\frac{|\lambda^{(k)}(t)|}{\lambda(t)} \leq \frac{C_{k}}{t^{k} \log(t)}$.

\section{Construction of the Ansatz}
\noindent Let $\lambda \in \Lambda$. By definition of $\Lambda$, there exists $T_{\lambda}>100$ such that \eqref{lambdasetdef} is true. Let $T_{2}\geq e^{900}(1+T_{\lambda})$ be such that
\begin{equation}\label{t2constraint}t^{\frac{2}{3}} \lambda(t) < \frac{t}{100}, \quad t \geq T_{2}\end{equation}
Let $T_{0}\geq T_{2}$ be otherwise arbitrary. For the whole paper, we work in the region $t \geq T_{0}$, and $C$ denotes a constant, \emph{independent} of $T_{0}$, unless otherwise specified. We define $g$ by
\begin{equation}\label{gdef}g(t) := t^{\alpha}\end{equation}
where $\alpha$ is any number satisfying
\begin{equation}\label{alphaconstr}\alpha > \frac{1}{2}+\frac{C_{l}}{2}, \quad \alpha < \frac{2}{3}-\frac{5}{6}C_{u}-\frac{C_{l}}{6}\end{equation}
(Note that \eqref{lambdaonlyconstr} implies that $\frac{2}{3} - \frac{5}{6} C_{u}-\frac{C_{l}}{6} > \frac{1}{2}+\frac{C_{l}}{2}$, so such $\alpha$ as above exists. We also remind the reader that some intuition behind this choice of $g$ is given in Section \ref{summarysection}.) We note that the definition of $\Lambda$ implies that, for $x \geq t \geq T_{\lambda}$,
\begin{equation}\label{lambdacomparg}\left(\frac{x}{t}\right)^{-C_{l}} \leq \frac{\lambda(x)}{\lambda(t)} \leq \left(\frac{x}{t}\right)^{C_{u}}, \quad t \mapsto \frac{\lambda(t)}{t^{\frac{1}{30}}} \text{ is decreasing,} \quad \frac{\sup_{y \in [100,x]}(\lambda(y)\log(y))}{\sup_{y \in [100,t]}(\lambda(y)\log(y))} \leq C \left(\frac{x}{t}\right)^{C_{u}} \frac{\log(x)}{\log(t)} \end{equation}
We will construct an approximate solution to 
\begin{equation}\label{u1eqn}-\partial_{tt}u_{1}+\partial_{rr}u_{1}+\frac{1}{r}\partial_{r}u_{1} - \frac{\cos(2Q_{1}(\frac{r}{\lambda(t)}))}{r^{2}}u_{1} = \partial_{t}^{2}Q_{1}(\frac{r}{\lambda(t)})\end{equation}
starting by matching explicit solutions to certain approximations of the operator on the left-hand side.  
\subsection{Small $r$ corrections}
\subsubsection{First Iteration}
We first consider the ODE
\begin{equation}\label{uelleqn}\partial_{rr}u_{ell}+\frac{1}{r}\partial_{r}u_{ell}-\frac{\cos(2Q_{1}(\frac{r}{\lambda(t)}))}{r^{2}}u_{ell} = \partial_{t}^{2}Q_{1}(\frac{r}{\lambda(t)}), \quad u_{ell}(t,0)=0\end{equation}
We get
$$u_{ell}(t,r) = v_{ell}(t,\frac{r}{\lambda(t)})$$
where
\begin{equation}\label{velldef}v_{ell}(t,R) = f_{1}(R) \lambda'(t)^{2} + \lambda(t) \lambda''(t) f_{2}(R) + \frac{c_{1}(t) R}{1+R^{2}}\end{equation}
and
$$f_{1}(R) = \frac{R (-2+R^{2})}{2(1+R^{2})}, \quad f_{2}(R) = \frac{R^{2}(-1+2R^{2}) - (-1+R^{4})\log(1+R^{2}) + 2 R^{2} Li_{2}(-R^{2})}{2R(1+R^{2})}$$
and $c_{1}$ will be chosen later.
\subsubsection{Second iteration}
We define a second order correction, $u_{ell,2}(t,r)$, which is useful in the region of small $r$, by the following solution to 
\begin{equation}\label{uell2eqn}\partial_{rr}u_{ell,2}(t,r)+\frac{1}{r}\partial_{r}u_{ell,2}(t,r)-\frac{\cos(2Q_{1}(\frac{r}{\lambda(t)}))}{r^{2}} u_{ell,2}(t,r) = \partial_{t}^{2}\left(v_{ell}(t,\frac{r}{\lambda(t)})\right)\end{equation}
\begin{equation}\label{uell2def}u_{ell,2}(t,r) = v_{ell,2}(t,\frac{r}{\lambda(t)})\end{equation}
where $v_{ell,2}$ is defined by
\begin{equation}\label{vell2formula0}\begin{split}v_{ell,2}(t,R)&= - \frac{\phi_{0}(R)}{2} \int_{0}^{R} \lambda(t)^{2} \partial_{t}^{2}\left(v_{ell}(t,\frac{r}{\lambda(t)})\right)\Bigr|_{r=s\lambda(t)}s e_{2}(s) ds \\
&+ e_{2}(R) \int_{0}^{R} \lambda(t)^{2} \partial_{t}^{2}\left(v_{ell}(t,\frac{r}{\lambda(t)})\right)\Bigr|_{r=s\lambda(t)} \frac{s \phi_{0}(s)}{2} ds\end{split}\end{equation}
In order to compute explicit terms in an expansion of $v_{ell,2}(t,R)$ for large $R$, we let 
\begin{equation}\label{velllgR}v_{ell,lgR}(t,R) = \frac{1}{2} R \lambda'(t)^2+R \lambda(t) (1-\log (R)) \lambda''(t)\end{equation}
be the leading behavior of $v_{ell}(t,R)$ for large $R$ (recall \eqref{velldef}), and define
$$err_{0}(t,r) = \lambda(t)^{2} \partial_{t}^{2}\left(v_{ell,lgR}(t,\frac{r}{\lambda(t)})\right)= r \lambda(t)^{2} \partial_{t}^{2}\left(\frac{\lambda'(t)^{2}}{2 \lambda(t)}+\lambda''(t)-\lambda''(t) \log(\frac{r}{\lambda(t)})\right)$$
To understand $v_{ell,2}(t,R)$ for large $R$, we let
\begin{equation}\begin{split}v_{ell,2,0}(t,R) &= - \frac{\phi_{0}(R)}{2} \int_{0}^{R} err_{0}(t,s\lambda(t))s e_{2}(s) ds + e_{2}(R) \int_{0}^{R} err_{0}(t,s\lambda(t)) \frac{s \phi_{0}(s)}{2} ds\end{split}\end{equation}
We get
\begin{equation}\label{vell20def}\begin{split} v_{ell,2,0}(t,R)&= f_{3}(R) \left(2 \lambda'(t)^{4} - 7 \lambda(t) \lambda'(t)^{2} \lambda''(t) + 4 \lambda(t)^{2} \lambda''(t)^{2} + 6 \lambda(t)^{2} \lambda'(t)\lambda'''(t)\right)+f_{4}(R) \lambda(t)^{3} \lambda''''(t)\end{split}\end{equation}
where
\begin{equation}\label{f3elldef}f_{3}(R) = \frac{2\left(R^{2}\left(-2+6R^{2}+R^{4}\right)-2(-1+R^{4})\log(1+R^{2})+4R^{2}Li_{2}(-R^{2})\right)}{32(R+R^{3})}\end{equation}
\begin{equation}\label{f4def}\begin{split} f_{4}(R)= \frac{1}{32(R+R^{3})}&\left(R^{2}\left(-12+44R^{2}+7R^{4}-4(-2+6R^{2}+R^{4})\log(R)\right)\right.\\
&\left.+8\left(-1+R^{4}\right)(-1+\log(R))\log(1+R^{2})\right.\\
&\left.+4\left(-1+4R^{2}+R^{4}-4R^{2}\log(R)\right)Li_{2}(-R^{2})+16R^{2}Li_{3}(-R^{2})\right)\end{split}\end{equation}
For later convenience, we split 
$$v_{ell,2,0}(t,\frac{r}{\lambda(t)})=v_{ell,2,0,main}(t,\frac{r}{\lambda(t)})+soln_{1}(t,\frac{r}{\lambda(t)})$$
and write
\begin{equation}\label{soln2def}v_{ell,2}(t,R) = v_{ell,2,0,main}(t,R)+soln_{1}(t,R)+soln_{2}(t,R)\end{equation}
where
\begin{equation}\label{vell20maindef}v_{ell,2,0,main}(t,s) = \frac{s^{3} \lambda(t)^{3}}{8}\left(\partial_{t}^{2}\left(\frac{\lambda'(t)^{2}}{2\lambda(t)}+\lambda''(t)+\lambda''(t)\log(\lambda(t))\right)-\lambda''''(t)\log(s\lambda(t)) + \frac{3}{4} \lambda''''(t)\right)\end{equation}
\begin{equation}\label{soln1def}\begin{split}soln_{1}(t,s) &= \left(f_{3}(s)-f_{3,0}(s)\right)\left(2\lambda'(t)^{4}-7 \lambda(t) \lambda'(t)^{2}\lambda''(t) + 4 \lambda(t)^{2}\lambda''(t)^{2}+6 \lambda(t)^{2}\lambda'(t)\lambda'''(t)\right) \\
&+ \left(f_{4}(s)-f_{4,0}(s)\right) \lambda(t)^{3} \lambda''''(t)\end{split}\end{equation}
and
$$f_{3,0}(x) = \frac{x^{3}}{16}, \quad f_{4,0}(x) = \frac{1}{32} x^3 (7-4 \log (x)) $$

Finally, 
\begin{equation}\label{ell2minusell20}\begin{split}&soln_{2}(t,\frac{s}{\lambda(t)}) = u_{ell,2}(t,s) - v_{ell,2,0,main}(t,\frac{s}{\lambda(t)})-soln_{1}(t,\frac{s}{\lambda(t)})\\
&=-\frac{\phi_{0}(\frac{s}{\lambda(t)})}{2}\int_{0}^{\frac{s}{\lambda(t)}} err_{1}(t,q\lambda(t)) q e_{2}(q)dq+\frac{e_{2}(\frac{s}{\lambda(t)})}{2}\int_{0}^{\frac{s}{\lambda(t)}} err_{1}(t,q\lambda(t)) q \phi_{0}(q)dq\end{split}\end{equation}
where
\begin{equation}\label{err1def} err_{1}(t,r)=\lambda(t)^{2}\partial_{t}^{2}\left(v_{ell}(t,\frac{r}{\lambda(t)})-v_{ell,lgR}(t,\frac{r}{\lambda(t)})\right)\end{equation}
Note in particular, that
\begin{equation}\label{vell20r3cof} v_{ell,2,0}(t,\frac{r}{\lambda(t)}) = \frac{r^{3}}{8} \partial_{t}^{2}\left(\frac{\lambda'(t)^{2}}{2\lambda(t)}+\lambda''(t)-\lambda''(t) \log(\frac{r}{\lambda(t)})\right) + \frac{3}{32} r^{3} \lambda''''(t)+O\left(r \log^{2}(r)\right), \quad r \rightarrow \infty\end{equation}
This observation will be important when we match the small $r$ corrections to the large $r$ corrections. Also, the leading behavior of $err_{1}(t,r)$ in the region $r \geq \lambda(t)$ is $err_{1,0}$ given by
\begin{equation}\label{err10def} err_{1,0}(t,r) = \lambda(t)^{2} \partial_{t}^{2}\left(\lambda(t)\left(\frac{-3}{2r}\lambda'(t)^{2}+\frac{\lambda(t)\lambda''(t)}{r}\left(-(2+\frac{\pi^{2}}{6}) +  \log(\frac{r}{\lambda(t)}) -2 \log^{2}(\frac{r}{\lambda(t)}) \right)+\frac{c_{1}(t)}{r}\right)\right)\end{equation}
\subsection{First Wave Iteration}
We consider the PDE
\begin{equation}\label{uweqn}-\partial_{tt}u_{w}+\partial_{rr}u_{w}+\frac{1}{r}\partial_{r}u_{w}-\frac{u_{w}}{r^{2}} = \partial_{t}^{2} Q_{1}(\frac{r}{\lambda(t)})\end{equation}
We consider the solution $u_{w}$, that can be written as 
\begin{equation}\label{uwdef}u_{w}(t,r) = v_{1}(t,r) + v_{2}(t,r)\end{equation}
where $v_{1}$ solves 
$$-\partial_{tt}v_{1}+\partial_{rr}v_{1}+\frac{1}{r}\partial_{r}v_{1}-\frac{v_{1}}{r^{2}} = \partial_{t}^{2} Q_{1}(\frac{r}{\lambda(t)})$$
with 0 Cauchy data at infinity, and $v_{2}$ solves
$$\begin{cases}-\partial_{tt}v_{2}+\partial_{rr}v_{2}+\frac{1}{r}\partial_{r}v_{2}-\frac{v_{2}}{r^{2}} = 0\\
v_{2}(0)=0\\
\partial_{t}v_{2}(0,r)=v_{2,0}(r)\end{cases}$$
where $v_{2,0}\in L^{2}(r dr)$ will be chosen later. To describe $v_{1}(t,r)$, we decompose it as
\begin{equation}\label{v1splitting}v_{1}(t,r) = w_{1}(t,r) + v_{ex}(t,r)\end{equation}
where $w_{1}$ solves
\begin{equation}\label{w1firsteqn}-\partial_{tt}w_{1}+\partial_{rr}w_{1}+\frac{1}{r}\partial_{r}w_{1}-\frac{w_{1}}{r^{2}} = \frac{-2 \lambda''(t)}{r}\end{equation}
with 0 Cauchy data at infinity. We use the same procedure used in Section 4.2 of \cite{wm} to obtain (4.25) of \cite{wm}, to derive the following expression which can be directly checked to be the solution to \eqref{w1firsteqn} with 0 Cauchy data at infinity.
\begin{equation}\label{firstw1formula}\begin{split}w_{1}(t,r)&= \frac{-1}{2 \pi}\int_{t}^{\infty} ds \int_{0}^{s-t} \frac{\rho d\rho}{\sqrt{(s-t)^{2}-\rho^{2}}} \int_{0}^{2\pi} \left(-2\lambda''(s)\right)\frac{\left(r+\rho \cos(\theta)\right)}{r^{2}+\rho^{2}+2 r \rho \cos(\theta)} d\theta \\
&= \int_{t}^{\infty} ds \frac{\lambda''(s)}{r} \int_{0}^{s-t} \frac{\rho d\rho}{\sqrt{(s-t)^{2}-\rho^{2}}}\left(1+\text{sgn}(r^{2}-\rho^{2})\right) \\
&= \frac{2}{r} \int_{t}^{t+r} ds \lambda''(s) (s-t) + \frac{2}{r} \int_{t+r}^{\infty} ds \lambda''(s) \left((s-t) - \sqrt{(s-t)^{2}-r^{2}}\right)\\
&=\frac{-2}{r}\left(\lambda(t+r)-\lambda(t)-r\lambda'(t+r)\right)+ \frac{2}{r} \int_{t+r}^{\infty} ds \lambda''(s) \left((s-t) - \sqrt{(s-t)^{2}-r^{2}}\right)\end{split}\end{equation}
(For $r \neq \rho$, we can evaluate the following integral using Cauchy's residue theorem).
\begin{equation}\label{rinverseint}\int_0^{2 \pi } \frac{\rho \cos (\theta)+r}{\rho^2+2 \rho r \cos (\theta)+r^2} \, d\theta = \frac{\pi  \left(\text{sgn}\left(r^2-\rho^2\right)+1\right)}{r}\end{equation}
First, we note that $w_{1}(t,\cdot) \in C^{\infty}((0,\infty))$. This follows from
\begin{equation}\label{w1simp}w_{1}(t,r) = \frac{-2}{r}\left(\lambda(t+r)-\lambda(t)-r\lambda'(t+r)\right)+ 2 r \int_{1}^{\infty} \lambda''(t+r y)\left(y-\sqrt{y^{2}-1}\right) dy\end{equation}
and the smoothness and symbol type estimates on $\lambda$. Next, we introduce some notation. Let
\begin{equation}\label{w1main}\begin{split} w_{1,main}(t,r) =  r\left(\log(2)+\frac{1}{2}\right) \lambda''(t) + r \int_{t}^{2t} \frac{(\lambda''(s) - \lambda''(t))}{s-t} ds + r \lambda''(t) \log(\frac{t}{r}) + r \int_{2t}^{\infty} \frac{\lambda''(s) ds}{(s-t)}\end{split}\end{equation}
\begin{equation}\label{w1subdef} w_{1,sub}(t,r) = w_{1}(t,r) -w_{1,main}(t,r)\end{equation}
\begin{equation}\begin{split}w_{1,cubic,main}(t,r) &=\frac{3}{32} r^{3} \lambda''''(t)+\frac{r^{3}}{8}\left( \lambda''''(t) \left(\log(2)+\frac{1}{2}\right)-\log(r) \lambda''''(t) + \log(t) \lambda''''(t)\right. \\
&\left.+ \int_{t}^{2t} \frac{\lambda''''(s)-\lambda''''(t)}{s-t} ds + \int_{2t}^{\infty} \frac{\lambda''''(s) ds}{s-t}\right)\end{split}\end{equation}
The following lemma will be useful later on when we match the small $r$ and large $r$ corrections.
\begin{lemma} \label{w1strlemma}
For $0 \leq j \leq 8$, $0 \leq k \leq 3$ and $r \leq t$,
\begin{equation}\begin{split}&\Bigr|\partial_{t}^{j}\partial_{r}^{k}\left(w_{1,sub}(t,r)-w_{1,cubic,main}(t,r)-\frac{r^{5} \lambda^{(6)}(t)}{576}\left(5+\log(8)\right)\right.\\
&\left.-\frac{r^{5}}{192} \left(\int_{t}^{2t} \frac{\left(\lambda^{(6)}(s)-\lambda^{(6)}(t)\right)ds}{s-t} + \lambda^{(6)}(t) \log(\frac{t}{r})+\int_{2t}^{\infty} \frac{\lambda^{(6)}(s) ds}{s-t}\right)\right)\Bigr| \\
&\leq \frac{C \lambda(t) r^{7-k}}{t^{8+j}} \left(\log(t)+|\log(r)|\right)\end{split}\end{equation}
\end{lemma}
\begin{proof}
As a first step, we start with \eqref{w1simp}, and subtract and add the Taylor polynomial centered at $0$, of degree 3, of $\lambda(t+r)-\lambda(t)-r \lambda'(t+r) + \frac{r^{2}}{2} \lambda''(t+r)$, regarded as a function of $r$ for each fixed $t$. For the integral term of \eqref{w1simp}, we first note that
$$y-\sqrt{y^{2}-1} = \frac{1}{2y} + O\left(\frac{1}{y^{3}}\right)$$
Therefore, we have
\begin{equation}\label{w1firstorderintstep}2 r \int_{1}^{\infty} \lambda''(t+ry) (y-\sqrt{y^{2}-1}) dy = 2 r \int_{1}^{\infty} \frac{\lambda''(t+r y)}{2y} dy + 2 r \int_{1}^{\infty} \lambda''(t+r y) \left(y-\sqrt{y^{2}-1}-\frac{1}{2y}\right) dy\end{equation}
After integrating by parts twice in the second integral on the right-hand side, we get the following expression for $w_{1}$, whose second line is equal to $w_{1,main}(t,r)$, and the other terms vanish faster than $O(r \log(r))$ as $r$ approaches zero.
\begin{equation}\label{w1expr}\begin{split} w_{1}(t,r) &= \frac{-2}{r} \left(\lambda(t+r) - \lambda(t) -r \lambda'(t+r) +\frac{r^{2}}{2} \lambda''(t+r) - \frac{r^{3}}{6} \lambda'''(t)\right)\\
&+r\left(\log(2)+\frac{1}{2}\right) \lambda''(t) + r \int_{t}^{2t} \frac{(\lambda''(s) - \lambda''(t))}{s-t} ds + r \lambda''(t) \log(\frac{t}{r}) + r \int_{2t}^{\infty} \frac{\lambda''(s) ds}{(s-t)}\\
&+r (\log(2) + \frac{1}{2}) \left(\lambda''(t+r) - \lambda''(t) - r \lambda'''(t)\right)\\
&-\left(\frac{r^{2}}{6} - r^{2} (1+\log(\frac{1}{2}))\right) \left(\lambda'''(t+r) - \lambda'''(t)\right)- r \int_{t}^{t+r} \frac{\lambda''(s) - \lambda''(t) - (s-t) \lambda'''(t)}{(s-t)} ds\\
&+\frac{1}{r} \int_{t+r}^{\infty} \lambda''''(s) \left(\frac{(s-t)^{3}}{3} \left(1-\left(1-\frac{r^{2}}{(s-t)^{2}}\right)^{3/2} - \frac{3 r^{2}}{2 (s-t)^{2}}\right) \right.\\
&\left.+ r^{2}(s-t) \left(1-\sqrt{1-\frac{r^{2}}{(s-t)^{2}}}+\log\left(1+\frac{\sqrt{1-\frac{r^{2}}{(s-t)^{2}}}-1}{2}\right)\right)\right) ds\end{split}\end{equation}
Repeating this process, we make manifest the terms composing $w_{1,cubic,main}(t,r)$, and recall the notation for $P_{k,a}(f)$ given just above \eqref{Inotation}.
\begin{equation}\label{w1cubeterms}\begin{split} w_{1}(t,r) &= r \left(\log(2)+\frac{1}{2}\right)\lambda''(t) + r \int_{t}^{2t} \frac{\left(\lambda''(s)-\lambda''(t)\right)}{s-t} ds + r \lambda''(t) \log(\frac{t}{r}) + r \int_{2t}^{\infty} \frac{\lambda''(s) ds}{s-t}\\
&-\frac{2}{r}\left(\lambda(t+r)-\lambda(t)-r \lambda'(t+r)+\frac{r^{2}}{2} \lambda''(t+r)-\frac{r^{3}}{6} \lambda'''(t) - \frac{r^{4}}{8} \lambda''''(t) - \frac{r^{5}}{20} \lambda'''''(t)\right)\\
&+\frac{r^{3}}{8} \lambda''''(t) \left(\log(2)+\frac{1}{2}\right)+\frac{3}{32} r^{3} \lambda''''(t)\\
&+r\left(\log(2)+\frac{1}{2}\right)\left(\lambda''(t+r)-P_{3,0}(r \mapsto \lambda''(t+r))\right)\\
&-r^{2}\left(\frac{-5}{6}+\log(2)\right)\left(\lambda'''(t+r)-P_{2,0}(r\mapsto \lambda'''(t+r))\right)\\
&-r \int_{t}^{t+r} \frac{\left(\lambda''(s)-P_{3,t}(\lambda'')(s)\right)}{s-t}ds\\
&-\frac{r^{3}}{8} \log(r) \lambda''''(t) + \frac{r^{3}}{8} \log(t)\lambda''''(t) + \frac{r^{3}}{8} \int_{t}^{2t} \frac{\lambda''''(s)-\lambda''''(t)}{s-t} ds\\
&+\frac{r^{3}}{8} \int_{2t}^{\infty} \frac{\lambda''''(s)}{s-t} ds - \frac{r^{3}}{8} \int_{t}^{t+r} \frac{\lambda''''(s)-\lambda''''(t)-(s-t)\lambda'''''(t)}{s-t} ds\\
&-\frac{r^{3}}{96} \left(41-60\log(2)\right)\left(\lambda''''(t+r)-\lambda''''(t)-r\lambda'''''(t)\right) \\
&+ \left(\lambda'''''(t+r)-\lambda'''''(t)\right)r^{4} \frac{(299-420\log(2))}{1440}+\frac{1}{r}\int_{t+r}^{\infty}\lambda''''''(s) K(r,s-t) ds\end{split}\end{equation}
\begin{equation}\label{firstkdef}\begin{split}K(r,w)&=\frac{1}{1440} \left(24 w^{5} \left(1-\left(\sqrt{1-\frac{r^{2}}{w^{2}}}+\frac{r^{2}}{2w^{2}}+\frac{r^{4}}{8w^{4}}\right)\right)+332 r^{2}w^{3} \left(1-\left(\sqrt{1-\frac{r^{2}}{w^{2}}}+\frac{r^{2}}{2w^{2}}\right)\right)\right.\\
&\left.-64 r^{4} w\left(\sqrt{1-\frac{r^{2}}{w^{2}}}-1\right) + 240 r^{2} w^{3} \left(\log\left(\frac{1}{2}\left(1+\sqrt{1-\frac{r^{2}}{w^{2}}}\right)\right)+\frac{r^{2}}{4w^{2}}\right)\right.\\
&\left.+180 r^{4} w \log\left(\frac{1}{2}\left(1+\sqrt{1-\frac{r^{2}}{w^{2}}}\right)\right)\right)\end{split}\end{equation}
Finally, to prove the lemma, we treat \eqref{w1cubeterms} line by line. We show in detail how we obtain the leading part of terms of different forms. In the following computations, we work in the region $r \leq t$. By Taylor's theorem, we have
\begin{equation}\begin{split}& |\frac{-2}{r}\left(\lambda(t+r)-\lambda(t)-r\lambda'(t+r)+\frac{r^{2}}{2}\lambda''(t+r)-\frac{r^{3}}{6} \lambda'''(t)-\frac{r^{4}}{8} \lambda''''(t)-\frac{r^{5}}{20} \lambda'''''(t)\right)\\
&-\left( \frac{-\lambda^{(6)}(t) r^{5}}{36} - \frac{\lambda^{(7)}(t) r^{6}}{168}\right)|\\
&\leq C \frac{r^{7}\lambda(t)}{t^{8}}\end{split}\end{equation}
where we used \eqref{lambdacomparg}, and the fact that $r \leq t$. Other terms in \eqref{w1cubeterms} of the same form are treated with the same argument. Next, we have
$$-r \int_{t}^{t+r} \frac{ds}{(s-t)} \left(\frac{(s-t)^{4}}{24}\lambda^{(6)}(t)+\frac{(s-t)^{5}}{120} \lambda^{(7)}(t)\right)=-r \left(\frac{r^{4}}{4 \cdot 24} \lambda^{(6)}(t) + \frac{r^{5}}{600} \lambda^{(7)}(t)\right)$$
and
\begin{equation}\begin{split}&\Bigl|-r \int_{t}^{t+r} \frac{\left(\lambda''(s)-P_{3,t}(\lambda'')(s)\right)}{s-t}ds-\left(-r \int_{t}^{t+r} \frac{ds}{(s-t)} \left(\frac{(s-t)^{4}}{24}\lambda^{(6)}(t)+\frac{(s-t)^{5}}{120} \lambda^{(7)}(t)\right)\right)\Bigr|\\
&= |-r \int_{t}^{t+r} \frac{\left(\lambda''(s)-P_{5,t}(\lambda'')(s)\right)ds}{s-t}|\leq C \frac{\lambda(t) r^{7}}{t^{8}}\end{split}\end{equation}
where we again used Taylor's theorem and \eqref{lambdacomparg}. Again, other terms in \eqref{w1cubeterms} of the same form are treated with the same argument. Finally, we recall \eqref{firstkdef}, and note that
$$K(r,w) - \frac{r^6}{192 w} = O\left(\frac{r^{8}}{w^{3}}\right), \quad w \rightarrow \infty$$
Therefore, with the same procedure used in \eqref{w1firstorderintstep},
\begin{equation}\begin{split}&\frac{1}{r}\int_{t+r}^{\infty} \lambda^{(6)}(s) K(r,s-t) ds\\
&=\frac{r^{5}}{192} \left(\int_{t}^{2t} \frac{\left(\lambda^{(6)}(s)-\lambda^{(6)}(t)\right)}{(s-t)} ds + \lambda^{(6)}(t) \log(\frac{t}{r}) + \int_{2t}^{\infty} \frac{\lambda^{(6)}(s)}{(s-t)} ds\right)\\
&-\frac{r^{5}}{192} \int_{t}^{t+r} \frac{\left(\lambda^{(6)}(s)-\lambda^{(6)}(t)-(s-t)\lambda^{(7)}(t)\right)}{(s-t)} ds - \frac{r^{6}}{192}\lambda^{(7)}(t) - \lambda^{(6)}(t) \frac{r^{5} (218-315 \log(2))}{2880} \\
&+\lambda^{(7)}(t) r^{6} \frac{(2413-3465 \log(2))}{100800} + \frac{\left(\lambda^{(7)}(t+r)-\lambda^{(7)}(t)\right)}{100800}\cdot r^{6}\left(2413-3465 \log(2)\right)\\
&-\left(\lambda^{(6)}(t+r)-\lambda^{(6)}(t)-r \lambda^{(7)}(t)\right) r^{5} \frac{(218-315 \log(2))}{2880}-r^{6} \lambda^{(7)}(t)\frac{(218-315 \log(2))}{2880}+\text{Err}\end{split}\end{equation}
where\begin{equation}|\text{Err}| \leq \frac{1}{r} \int_{t+r}^{\infty} |W(s) \lambda^{(8)}(s)| ds\end{equation}
and, for $y=\frac{s-t}{r}$
\begin{equation}\begin{split} W(s) = \frac{r^7}{2880} &\left(-\frac{128}{35} \left(\sqrt{1-\frac{1}{y^2}}-1\right) y+15 y \log \left(\frac{1}{2} \sqrt{1-\frac{1}{y^2}}+\frac{1}{2}\right)-\frac{1779}{35} \left(\sqrt{1-\frac{1}{y^2}}+\frac{1}{2 y^2}-1\right) y^3\right.\\
&\left.-\frac{1518}{35} \left(\sqrt{1-\frac{1}{y^2}}-\left(-\frac{1}{8 y^4}-\frac{1}{2 y^2}+1\right)\right) y^5\right.\left.-\frac{8}{7}  \left(\sqrt{1-\frac{1}{y^2}}+\left(\frac{1}{16 y^6}+\frac{1}{8 y^4}+\frac{1}{2 y^2}\right)-1\right) y^7\right.\\
&\left.+12 \left(2 y^5+5 y^3\right) \left(\frac{3}{32 y^4}+\frac{1}{4 y^2}+\log \left(\frac{1}{2} \sqrt{1-\frac{1}{y^2}}+\frac{1}{2}\right)\right)-\frac{45}{8 y}\right)\end{split}\end{equation}
We finish the proof by noting that $$W(s) = O\left(\frac{r^{8}}{(s-t)}\right), \quad s \rightarrow \infty$$
and estimating $\text{Err}$. The higher derivatives are treated similarly.
 \end{proof}
We remark that $w_{1,main}(t,r)$, along with the expression of $u_{ell}$ will end up determining the data for the free wave $v_{2}$, as $w_{1,main}(t,r)$ will turn out to be the leading contribution of $w_{1}$ in the matching region. \\
\\
We also record some pointwise estimates on $w_{1}$ and some of its derivatives.
\begin{lemma}\label{w1estlemma}
We have the following estimates. For $0 \leq k \leq 8$ and $0 \leq j \leq 3$,
\begin{equation}\label{w1estlemmaeqn}|\partial_{t}^{k}\partial_{r}^{j}w_{1}(t,r)| \leq \begin{cases} \frac{C r^{1-j} \lambda(t) \left(\log(t)+|\log(r)|\right)}{t^{2+k}}, \quad r \leq t\\
\frac{C}{r^{1+j} t^{k}} \sup_{x \in [t,t+r]}\left(\lambda(x)\right), \quad r \geq t\end{cases}\end{equation}
\end{lemma}
\begin{proof}
By Lemma \ref{w1strlemma}, it suffices to consider \eqref{w1simp} in the region $r \geq t$: 
\begin{equation}\begin{split}|\frac{2}{r}\left(\lambda'(t+r) r -(\lambda(t+r)-\lambda(t))\right)|& \leq \frac{C}{r}\left(\lambda(t+r)+\lambda(t)+ r \frac{\lambda(t+r)}{t+r}\right) \leq \frac{C}{r}\left(\lambda(t+r)+\lambda(t)\right) \end{split}\end{equation}
Then, we have
\begin{equation}\begin{split} |2r \int_{1}^{\infty} \lambda''(t+r y) \left(y-\sqrt{y^{2}-1}\right)dy| &\leq \frac{C r \lambda(t+r)}{t+r} \int_{1}^{\infty} \frac{dy}{y(t+r y)} \leq \frac{C r \lambda(t+r)}{r (t+r)} \frac{\log(1+\frac{t}{r})}{\frac{t}{r}}\\
&\leq \frac{C \lambda(t+r)}{t+r}, \quad r \geq t\end{split}\end{equation}
where we used \eqref{lambdacomparg}. Then, using the symbol type estimates on $\lambda$, we get \eqref{w1estlemmaeqn} for $r \geq t$.  \end{proof}
\noindent Next, we study $v_{ex}$ (\eqref{v1splitting}) which solves the following equation with 0 Cauchy data at infinity.
\begin{equation}\label{vexeqn}-\partial_{t}^{2}v_{ex}+\partial_{r}^{2}v_{ex}+\frac{1}{r}\partial_{r}v_{ex}-\frac{v_{ex}}{r^{2}}= \partial_{t}^{2} Q_{1}(\frac{r}{\lambda(t)}) + \frac{2 \lambda''(t)}{r}:=RHS(t,r)\end{equation}
From integral identities of \cite{gr}, we have
\begin{equation}\label{rhshat}\widehat{RHS}(t,\xi) = 2 \xi K_{0}(\xi \lambda(t))\lambda(t) \lambda'(t)^{2} + \frac{2\lambda''(t)}{\xi} \left(1-\xi \lambda(t) K_{1}(\xi \lambda(t))\right)\end{equation}
We have the following representation formula for $v_{ex}$
\begin{lemma}\label{vexreplemma} For $r >0$, $v_{ex}$ is given by
\begin{equation}\label{vexdef}\begin{split} v_{ex}(t,r)&= \int_{t}^{\infty} ds \int_{0}^{\infty} d\xi J_{1}(r\xi) \sin((t-s)\xi) \widehat{RHS}(s,\xi)\\
&= - \int_{0}^{\infty} \frac{J_{1}(r\xi)}{\xi} \widehat{RHS}(t,\xi) d\xi-\int_{0}^{\infty} d\xi J_{1}(r\xi) \int_{t}^{\infty} ds \frac{\sin((t-s)\xi)}{\xi^{2}} \partial_{s}^{2}\widehat{RHS}(s,\xi)\end{split}\end{equation}
\end{lemma}
\begin{proof}
From asymptotics of $K_{1}, K_{0}$ (see \cite{gr}), we get, for $k=0,1,2$,
\begin{equation}\label{rhssymb0}|\partial_{s}^{k}\widehat{RHS}(s,\xi)| \leq C\frac{\lambda(s)^{2}}{s^{2+k}} \frac{\xi \lambda(s)}{1+\xi^{2}\lambda(s)^{2}} \langle \log(\xi\lambda(s))\rangle\end{equation}
Using
$$|J_{1}(r\xi)| \leq \frac{C}{\sqrt{r \xi}}, \quad r >0$$
the first integral on the right-hand side of \eqref{vexdef} converges absolutely, for $r >0$. Therefore, by the Fubini theorem, and \eqref{rhssymb0}, we get
\begin{equation}\begin{split}&\int_{t}^{\infty} ds \int_{0}^{\infty} d\xi J_{1}(r\xi) \sin((t-s)\xi) \widehat{RHS}(s,\xi) = \int_{0}^{\infty} d\xi \int_{t}^{\infty}ds J_{1}(r\xi) \partial_{s}\left(\frac{\cos((t-s)\xi)}{\xi}\right) \widehat{RHS}(s,\xi)\\
&=- \int_{0}^{\infty} \frac{J_{1}(r\xi)}{\xi} \widehat{RHS}(t,\xi) d\xi-\int_{0}^{\infty} d\xi J_{1}(r\xi) \int_{t}^{\infty} ds \frac{\sin((t-s)\xi)}{\xi^{2}} \partial_{s}^{2}\widehat{RHS}(s,\xi)\end{split}\end{equation}
Moreover, using $$|\sin((t-s)\xi)|\leq 1, \quad |J_{1}'(x)| \leq C, \quad |J_{1}''(x)| \leq \frac{C}{\sqrt{x}}, \quad x >0, \text{ and }\eqref{rhssymb0}$$
the dominated convergence theorem shows that, for all $r_{0}>0$, and $k=1,2$,
\begin{equation}\begin{split}&\partial_{r}^{k}\left(\int_{0}^{\infty} d\xi J_{1}(r\xi) \int_{t}^{\infty} ds \frac{\sin((t-s)\xi)}{\xi^{2}} \partial_{s}^{2}\widehat{RHS}(s,\xi)\right)\Bigr|_{r=r_{0}} \\
&= \int_{0}^{\infty} d\xi \partial_{r}^{k}(J_{1}(r\xi))|_{r=r_{0}} \cdot  \int_{t}^{\infty} ds \frac{\sin((t-s)\xi)}{\xi^{2}} \partial_{s}^{2}\widehat{RHS}(s,\xi)\end{split}\end{equation}
A similar argument shows that we can differentiate (up to) two times under the integral sign in $t$. Then, by the fact that the first term on the second line of \eqref{vexdef} is a solution to the following ODE (which follows, for instance, by direct computation of the integral, see \eqref{vexell})
$$\partial_{r}^{2}f + \frac{1}{r} \partial_{r}f - \frac{f}{r^{2}} = RHS(t,r)$$
we see that $v_{ex}$, the solution to \eqref{vexeqn} with zero Cauchy data at infinity, is given by \eqref{vexdef}.  \end{proof}
Note that we could not have done the integration by parts in the $s$ variable to get from the first to the second line of \eqref{vexdef} if $RHS$ was simply equal to $\partial_{t}^{2}Q_{1}(\frac{r}{\lambda(t)})$, because there would be too large of a singularity of the resulting integrand at low frequencies. Also, we have
\begin{equation}\label{vexell}\begin{split}v_{ex,ell}(t,r):=&-\int_{0}^{\infty} \frac{J_{1}(r\xi)}{\xi} \widehat{RHS}(t,\xi) d\xi \\
&= \frac{-\left(2 \lambda(t) \log(1+\frac{r^{2}}{\lambda(t)^{2}})\lambda'(t)^{2} + \left(\lambda(t)^{2} \log(1+\frac{r^{2}}{\lambda(t)^{2}})+r^{2} \log(1+\frac{\lambda(t)^{2}}{r^{2}})\right)\lambda''(t)\right)}{2 r}\end{split}\end{equation}
Finally, we consider $v_{2}$ solving
$$\begin{cases} -\partial_{tt}v_{2}+\partial_{rr}v_{2}+\frac{1}{r}\partial_{r}v_{2}-\frac{v_{2}}{r^{2}} = 0\\
v_{2}(0)=0\\
\partial_{t}v_{2}(0)=v_{2,0}\end{cases}$$
with $v_{2,0}$ not yet chosen. For $\widehat{v_{2,0}}\in C^{0}((0,\infty))$ such that $|\widehat{v_{2,0}}(\xi)| \leq C \left(\sqrt{\xi}(1+\xi)^{5/2}\right)^{-1}$, we have
\begin{equation}\label{v2}\begin{split}v_{2}(t,r) &= \int_{0}^{\infty} J_{1}(r\xi) \sin(t\xi) \widehat{v_{2,0}}(\xi) d\xi = -\frac{r}{4 \pi} \int_{0}^{\pi} \sin^{2}(\theta) \left(F(t+r\cos(\theta))+F(t-r\cos(\theta))\right)d\theta\\
&=\frac{-r}{4} F(t) -\frac{r^{3}}{32}F''(t) -\frac{r}{2\pi} \int_{0}^{\pi} \sin^{2}(\theta) \left(F(t+r \cos(\theta))-F(t)-\frac{r^{2}}{2} \cos^{2}(\theta) F''(t)\right)d\theta \end{split}\end{equation}
where
\begin{equation}\label{Fdef}F(t) = -2 \int_{0}^{\infty} \xi \sin(t\xi) \widehat{v_{2,0}}(\xi) d\xi\end{equation}
and we used
$$J_{1}(x) = \frac{x}{\pi} \int_{0}^{\pi} \cos(x\cos(\theta))\sin^{2}(\theta) d\theta$$
\subsection{First order matching}\label{first_order_matching_section}
The error of $u_{ell}$ in solving \eqref{u1eqn} is $\partial_{t}^{2} u_{ell}(t,r)$, which is large for large values of $r$, because of the growth of $u_{ell}(t,r)$ for large $r$. On the other hand, the error of $u_{w}$ in solving \eqref{u1eqn} is $\left(\frac{\cos(2Q_{1}(\frac{r}{\lambda(t)}))-1}{r^{2}}\right) u_{w}(t,r)$.
The largest contributions to this error term arise from substituting the second line of \eqref{w1expr} plus \eqref{v2}, expanded for $r \ll t$ into $u_{w}$ in the error term.  Therefore, we choose the data for $v_{2}$ in order to match the largest contributions to the error terms of $u_{ell}$ and $u_{w}$. In particular, $v_{ell,lgR}$, defined in \eqref{velllgR} and recalled here makes the largest contribution to the error term of $u_{ell}$: 
\begin{equation}\label{vellfirstorder}v_{ell,lgR}(t,R) = R \left(\frac{1}{2}\lambda'(t)^{2} + \lambda(t) \lambda''(t) - \lambda(t) \lambda''(t) \log(R)\right)\end{equation}
Also, as per Lemma \ref{w1strlemma}, the main contribution from $w_{1}(t,r)$ is 
\begin{equation}w_{1,main}(t,r) = r\left(\log(2)+\frac{1}{2}\right) \lambda''(t) + r \int_{t}^{2t} \frac{(\lambda''(s) - \lambda''(t))}{s-t} ds + r \lambda''(t) \log(\frac{t}{r}) + r \int_{2t}^{\infty} \frac{\lambda''(s) ds}{(s-t)}\end{equation}
and, by \eqref{v2}, the main contribution from $v_{2}$ is 
\begin{equation}\label{v2main}v_{2,main}(t,r) = \frac{-r}{4} \left( -2 \int_{0}^{\infty} \xi \sin(t\xi) \widehat{v_{2,0}}(\xi) d\xi\right)\end{equation}
Note that the $r \log(r)$ terms from $w_{1,main}$ and \eqref{vellfirstorder} (where $R = \frac{r}{\lambda(t)}$) are the same. Therefore, it is possible to choose $v_{2,0}$ so that (for example, for all $t \geq 2 T_{\lambda}$)
\begin{equation} \frac{r}{\lambda(t)} \left(\frac{1}{2}\lambda'(t)^{2} + \lambda(t) \lambda''(t) - \lambda(t) \lambda''(t) \log(\frac{r}{\lambda(t)})\right)=v_{2,main}(t,r)+w_{1,main}(t,r)\end{equation} 
In other words, we choose $v_{2,0}$ to satisfy, for all $t \geq 2 T_{\lambda}$,
\begin{equation}\label{v20eqn}\begin{split}&-2 \int_{0}^{\infty} \xi \sin(t\xi) \widehat{v_{2,0}}(\xi) d\xi \\
&= 4\left(\left(\log(2)-\frac{1}{2}\right)\lambda''(t) + \int_{t}^{2t} \left(\frac{\lambda''(s)-\lambda''(t)}{s-t}\right) ds+\lambda''(t) \log(\frac{t}{\lambda(t)}) + \int_{2t}^{\infty} \frac{\lambda''(s) ds}{s-t} - \frac{\lambda'(t)^{2}}{2 \lambda(t)}\right)\end{split}\end{equation}
We remind the reader of the discussion following \eqref{v20eqnsummary}, which compares \eqref{v20eqn} to the relation between the radiation and $\lambda(t)$ from \cite{wm}.

\noindent Since we will only need \eqref{v20eqn} to be true for all $t$ sufficiently large, and since we only assume $\lambda(t)$ to be defined for $t > 50$, we use a (relatively unimportant) cutoff, $\psi \in C^{\infty}([0,\infty))$, such that
\begin{equation}\label{psidef}\psi(x) = \begin{cases} 0, \quad x \leq T_{\lambda}\\
1, \quad x \geq 2 T_{\lambda}\end{cases}, \quad 0 \leq \psi(x) \leq 1\end{equation}
where we recall that $T_{\lambda}>100$ is part of the definition of $\Lambda$, see \eqref{lambdasetdef}, and then get
\begin{equation}\label{v20hatform}\widehat{v_{2,0}}(\xi) = \frac{-1}{\pi\xi} \int_{0}^{\infty} H(t) \sin(t\xi) dt\end{equation}
with
\begin{equation}\label{hdef}H(t) = 4\left(\left(\log(2)-\frac{1}{2}\right)\lambda''(t) + \int_{t}^{2t} \left(\frac{\lambda''(s)-\lambda''(t)}{s-t}\right) ds+\lambda''(t) \log(\frac{t}{\lambda(t)}) + \int_{2t}^{\infty} \frac{\lambda''(s) ds}{s-t} - \frac{\lambda'(t)^{2}}{2 \lambda(t)}\right) \psi(t)\end{equation}
We can now record some basic estimates on $\widehat{v_{2,0}}$ and $v_{2}$. We remark that the estimates of Lemma \ref{v20ests} below imply that the conditions on $\widehat{v_{2,0}}$ stated just above \eqref{v2} are satisfied. 
\begin{lemma} \label{v20ests} For $0 \leq k$, and $N \geq 1$, there exist $C_{k},C_{k,N}$ such that 
$$\xi^{k}|\widehat{v_{2,0}}^{(k)}(\xi)|  \leq \begin{cases} C_{k} \int_{100}^{\frac{1}{\xi}} \frac{\lambda(\sigma) \log(\sigma) d\sigma}{\sigma} + C_{k} \lambda(\frac{1}{\xi}) \log(\frac{1}{\xi}), \quad \xi \leq \frac{1}{100}\\
\frac{C_{k,N}}{\xi^{N}},  \quad \xi > \frac{1}{100}\end{cases}$$
\end{lemma}
\begin{proof}
We start with \eqref{hdef}. By the mean value theorem, there exists $x \in [t,s]$ such that
\begin{equation}\begin{split} |H(t)| &\leq C\left(\frac{\lambda(t)}{t^{2} } + \int_{t}^{2t} |\lambda'''(x)|ds + \log(t) |\lambda''(t)| + \frac{1}{t} \int_{2t}^{\infty} \frac{\lambda(s) ds}{s^{2}} + \frac{ \lambda(t)}{t^{2}}\right)|\psi(t)|\leq \frac{C \lambda(t) \log(t)}{t^{2}} \mathbbm{1}_{\{t \geq T_{\lambda}\}}\end{split}\end{equation}
where we used \eqref{lambdacomparg}. Using the symbol-type nature of the estimates on $\lambda$, we get, for $k \geq 0$,
\begin{equation}\label{hsymbforv20}|H^{(k)}(t)| \leq C_{k} \mathbbm{1}_{\{t \geq T_{\lambda}\}} \frac{\lambda(t) \log(t)}{t^{2+k}}\end{equation}
Then, for $\xi \geq \frac{1}{100}$, we integrate by parts in the formula
$$\widehat{v_{2,0}}(\xi) = \frac{-1}{\pi \xi} \int_{0}^{\infty} H(\sigma) \sin(\sigma \xi) d\sigma$$
noting that no boundary terms arise, to get that, for each $k \in \mathbb{N}$, there exists $C_{k}>0$ such that
$$|\widehat{v_{2,0}}(\xi)| \leq \frac{C_{k}}{\xi^{k}}, \quad \xi \geq  \frac{1}{100}$$
For $\xi \leq  \frac{1}{100}$, we have
\begin{equation}|\widehat{v_{2,0}}(\xi)| \leq \begin{cases} C \int_{T_{\lambda}}^{\frac{1}{\xi}} \frac{\lambda(\sigma) \log(\sigma) d\sigma}{\sigma} + \frac{C}{\xi} \int_{\frac{1}{\xi}}^{\infty} \frac{\lambda(\sigma) \log(\sigma)}{\sigma^{2} } d\sigma \leq C \int_{100}^{\frac{1}{\xi}} \frac{\lambda(\sigma) \log(\sigma)}{\sigma} d\sigma + C \lambda(\frac{1}{\xi}) \log(\frac{1}{\xi}),\text{  } \xi \leq \text{min}\{\frac{1}{100},\frac{1}{T_{\lambda}}\}\\
\frac{C}{\xi} \int_{\frac{1}{\xi}}^{\infty} \frac{\lambda(\sigma) \log(\sigma)}{\sigma^{2}} d\sigma \leq C \int_{100}^{\infty} \frac{\log(\sigma)}{\sigma^{2-C_{u}}} d\sigma \leq C \lambda(\frac{1}{\xi}) \log(\frac{1}{\xi}), \quad \frac{1}{T_{\lambda}}\leq \xi \leq \frac{1}{100}\end{cases}\end{equation}
where we used $T_{\lambda}\geq100$ and \eqref{lambdacomparg} for the first case, and the fact that $\text{min}_{x\in[100,T_{\lambda}]}(\lambda(x)) \log(100) >0$ for the second. In particular, the constants $C$ depend on (the fixed function) $\lambda$, but \emph{not} on $\xi$. Finally, Lemma \ref{v20ests} for $k >0$ follows from \eqref{hsymbforv20} and
$$\widehat{v_{2,0}}(\xi) = \frac{-1}{\pi \xi^{2}} \int_{0}^{\infty} H(\frac{\omega}{\xi}) \sin(\omega) d\omega$$
 \end{proof}
\begin{lemma}\label{v2estlemma}
For $0 \leq k\leq 7$, $j=0,1$, we have
\begin{equation}\label{v2est}|\partial_{t}^{k}\partial_{r}^{j}v_{2}(t,r)| \leq  \frac{C r^{1-j} \lambda(t) \log(t)}{t^{2+k}}, \quad r \leq \frac{t}{2}\end{equation}
For all $j+k \leq 7$, we have the following two estimates
\begin{equation}|\partial_{t}^{k}\partial_{r}^{j}v_{2}(t,r)| \leq \frac{C \log(r)}{\sqrt{r}\langle|t-r|\rangle^{\frac{1}{2}+k+j}} \sup_{x \in [100,r]}\left(\lambda(x) \log(x)\right), \quad r \geq \frac{t}{2}\end{equation}
and
$$|\partial_{t}^{j}\partial_{r}^{k}v_{2}(t,r)| \leq \frac{C}{\sqrt{r}}, \quad r \geq \frac{t}{2}$$
\begin{equation}\label{dtdrv2est}|\left(\partial_{t}+\partial_{r}\right)v_{2}(t,r)| \leq \frac{C \log(r) \sup_{x \in [100,r]}\left(\lambda(x)\log(x)\right)}{r^{3/2}\sqrt{\langle t-r \rangle}}, \quad r >\frac{t}{2}\end{equation}
\begin{equation}\label{drrv2est} |\partial_{r}^{2} v_{2}(t,r)| \leq  \frac{C r \lambda(t) \log(t)}{t^{4}}, \quad r \leq \frac{t}{2}\end{equation}
\end{lemma}
\begin{proof}
To estimate $v_{2}$, we start with the region $r \leq \frac{t}{2}$. Here, we use \eqref{v2} and \eqref{hsymbforv20}, to get
\begin{equation}\begin{split}|v_{2}(t,r)| &\leq C r \int_{0}^{\pi} |F(t+r \cos(\theta))|d\theta \leq C r \int_{0}^{\pi} \frac{\lambda(t+r \cos(\theta)) \log(t+r \cos(\theta))}{(t+r \cos(\theta))^{2}} d\theta\leq \frac{C r \lambda(t) \log(t)}{t^{2}}\end{split}\end{equation}
where we used \eqref{lambdacomparg}. For the region $r \geq \frac{t}{2}$, we start with the first term on the right-hand side of \eqref{v2}, namely
$$v_{2}(t,r) = \int_{0}^{\infty} J_{1}(r\xi) \sin(t\xi) \widehat{v_{2,0}}(\xi) d\xi$$
One estimate which will be useful in the region $r \geq \frac{t}{2}$ (in particular, it is useful in the region $|t-r| \leq 100$) is the following. Using $|J_{1}(x)| \leq \frac{C}{\sqrt{x}}, x>0$ (though a much better estimate is true for small $x>0$)
we get
\begin{equation}\label{v2sqrtest}\begin{split}|v_{2}(t,r)|&\leq \frac{C}{\sqrt{r}}\int_{0}^{\frac{1}{100}} \frac{1}{\sqrt{\xi}}\left(\frac{\log(\frac{1}{\xi})}{\xi^{C_{u}}}+\int_{100}^{\frac{1}{\xi}} \frac{\log(\sigma) d\sigma}{\sigma^{1-C_{u}}}\right) d\xi + \frac{C}{\sqrt{r}}\int_{\frac{1}{100}}^{\infty} \frac{d\xi}{\xi^{5}}\leq \frac{C}{\sqrt{r}} \int_{0}^{\frac{1}{100}} \frac{\log(\frac{1}{\xi}) d\xi}{\xi^{C_{u}+1/2}}+\frac{C}{\sqrt{r}} \leq \frac{C}{\sqrt{r}}\end{split}\end{equation}
 where we also used \eqref{lambdacomparg}.
The higher derivatives of $v_{2}$ are treated similarly. To treat the region where $|t-r| > 100$ and $r \geq \frac{t}{2}$, we use
\begin{equation}\label{v2largerintstep12}v_{2}(t,r) = \int_{0}^{\frac{1}{r}} J_{1}(r\xi) \sin(t\xi) \widehat{v_{2,0}}(\xi) d\xi + \int_{\frac{1}{r}}^{\infty} J_{1}(r\xi) \sin(t\xi) \widehat{v_{2,0}}(\xi) d\xi \end{equation}
We start with
\begin{equation}\begin{split}|\int_{0}^{\frac{1}{r}} J_{1}(r\xi) \sin(t\xi) \widehat{v_{2,0}}(\xi) d\xi| &\leq C \int_{r}^{\infty} d\sigma \int_{0}^{\frac{1}{\sigma}} d\xi \frac{r \xi \lambda(\sigma) \log(\sigma)}{\sigma } + C \int_{100}^{r} d\sigma \int_{0}^{\frac{1}{r}} d\xi r \xi \frac{\lambda(\sigma)\log(\sigma)}{\sigma }\\
&+ C \int_{0}^{\frac{1}{r}} d\xi r \xi \lambda(\frac{1}{\xi}) \log(\frac{1}{\xi})\\
&\leq \frac{C r \lambda(r) \log(r)}{r} \int_{r}^{\infty} \frac{d\sigma}{\sigma^{2}} + \frac{C}{r} \int_{100}^{r} \frac{\lambda(\sigma) \log(\sigma) d\sigma}{\sigma}+ \frac{C r \lambda(r) \log(r)}{r^{2}} \\
&\leq \frac{C}{r} \sup_{\sigma \in [100,r]}\left(\lambda(\sigma) \log(\sigma)\right) \log(r)\end{split}\end{equation}
Here, we used \eqref{lambdacomparg}, and Lemma \ref{v20ests} in the region $\xi \leq \frac{1}{100}$, since $\frac{1}{r} \ll \frac{1}{100}$, along with Fubini's theorem to switch the order of the $\xi$ and $\sigma$ integrals. It remains to treat the following integral.
\begin{equation}\begin{split}\int_{\frac{1}{r}}^{\infty} J_{1}(r\xi) \sin(t\xi) \widehat{v_{2,0}}(\xi) d\xi &= -\int_{\frac{1}{r}}^{\infty} \frac{1}{\sqrt{\pi r \xi}} \left(\cos(r \xi)-\sin(r\xi)\right) \sin(t\xi) \widehat{v_{2,0}}(\xi) d\xi\\
&+\int_{\frac{1}{r}}^{\infty} \left(J_{1}(r\xi)+\left(\frac{\cos(r\xi)}{\sqrt{\pi r \xi}} -\frac{\sin(r\xi)}{\sqrt{\pi r \xi}}\right)\right) \sin(t\xi) \widehat{v_{2,0}}(\xi) d\xi\end{split}\end{equation}
We start with
\begin{equation}\label{v2largerintstep}\begin{split}&\int_{\frac{1}{r}}^{\infty} \frac{1}{\sqrt{\pi r \xi}} \left(\cos(r \xi)-\sin(r\xi)\right) \sin(t\xi) \widehat{v_{2,0}}(\xi) d\xi\\
&=\frac{1}{2\sqrt{\pi r}} \int_{\frac{1}{r}}^{\infty} \frac{1}{\sqrt{\xi}} \left(\sin((t+r)\xi)+\sin((t-r)\xi) + \cos((t+r)\xi)-\cos((t-r)\xi)\right) \widehat{v_{2,0}}(\xi) d\xi\\
&=\frac{1}{2\sqrt{\pi r}} \int_{\frac{1}{\sqrt{r}}}^{\infty} \left(\sin((t+r)\omega^{2})+\cos((t+r)\omega^{2})+\sin((t-r)\omega^{2})-\cos((t-r)\omega^{2})\right)\widehat{v_{2,0}}(\omega^{2}) \cdot 2 d\omega\end{split}\end{equation}

We will show in detail how to treat the term involving $\sin((t-r)\omega^{2})$. The other terms can be treated with a similar argument.
\begin{equation}\begin{split}\frac{1}{\sqrt{\pi r}} \int_{\frac{1}{\sqrt{r}}}^{\infty} \sin((t-r)\omega^{2}) \widehat{v_{2,0}}(\omega^{2}) d\omega &= \frac{1}{\sqrt{\pi r}} \int_{\frac{1}{\sqrt{r}}}^{\frac{1}{\sqrt{|t-r|}}} \sin((t-r)\omega^{2}) \widehat{v_{2,0}}(\omega^{2}) d\omega \\
&+ \frac{1}{\sqrt{\pi r}} \int_{\frac{1}{\sqrt{|t-r|}}}^{\infty} \sin((t-r)\omega^{2}) \widehat{v_{2,0}}(\omega^{2}) d\omega\end{split}\end{equation}
For the first term, we get
\begin{equation} \begin{split}|\frac{1}{\sqrt{\pi r}} \int_{\frac{1}{\sqrt{r}}}^{\frac{1}{\sqrt{|t-r|}}} \sin((t-r)\omega^{2}) \widehat{v_{2,0}}(\omega^{2}) d\omega| &\leq \frac{C}{\sqrt{r}} \int_{\frac{1}{\sqrt{r}}}^{\frac{1}{\sqrt{|t-r|}}} \sup_{x \in [100,\frac{1}{\omega^{2}}]}\left(\lambda(x) \log(x)\right) \log(\frac{1}{\omega^{2}}) d\omega\\
&\leq \frac{C \sup_{x \in [100,r]}\left(\lambda(x) \log(x)\right) \log(r)}{\sqrt{r}\sqrt{|t-r|}}\end{split}\end{equation}
On the other hand, we integrate by parts to get
\begin{equation}\begin{split} &|\frac{1}{\sqrt{\pi r}} \int_{\frac{1}{\sqrt{|t-r|}}}^{\infty} 2 \omega \sin((t-r)\omega^{2}) \frac{\widehat{v_{2,0}}(\omega^{2})}{2 \omega} d\omega|\leq \frac{C}{\sqrt{r}} \left(\frac{|\widehat{v_{2,0}}(\frac{1}{|t-r|})|}{\sqrt{|t-r|}} + \int_{\frac{1}{\sqrt{|t-r|}}}^{\infty} \frac{|\widehat{v_{2,0}}'(\omega^{2})| + \frac{|\widehat{v_{2,0}}(\omega^{2})|}{\omega^{2}}}{|t-r|} d\omega \right)\end{split}\end{equation}
This gives
\begin{equation}\begin{split} &|\frac{1}{\sqrt{\pi r}} \int_{\frac{1}{\sqrt{|t-r|}}}^{\infty} 2 \omega \sin((t-r)\omega^{2}) \frac{\widehat{v_{2,0}}(\omega^{2})}{2 \omega} d\omega|\\
&\leq \frac{C}{\sqrt{r}} \frac{|\widehat{v_{2,0}}(\frac{1}{|t-r|})|}{\sqrt{|t-r|}} + \frac{C}{\sqrt{r} |t-r|} \int_{\frac{1}{\sqrt{|t-r|}}}^{\frac{1}{10}} \frac{1}{\omega^{2}} \sup_{x \in [100,|t-r|]}\left(\lambda(x) \log(x)\right) \log(\frac{1}{\omega^{2}})d\omega+\frac{C}{\sqrt{r}|t-r|} \int_{\frac{1}{10}}^{\infty} \frac{d\omega}{\omega^{50}}\\
&\leq \frac{C \log(|t-r|)}{\sqrt{r}\sqrt{|t-r|}} \sup_{x\in [100,|t-r|]}\left(\lambda(x) \log(x)\right)\end{split}\end{equation}
Finally
\begin{equation}\begin{split}&|\int_{\frac{1}{r}}^{\infty} \left(J_{1}(r\xi)+\left(\frac{\cos(r\xi)}{\sqrt{\pi r \xi}} -\frac{\sin(r\xi)}{\sqrt{\pi r \xi}}\right)\right) \sin(t\xi) \widehat{v_{2,0}}(\xi) d\xi|\\
&\leq C \int_{\frac{1}{r}}^{\infty} \frac{1}{(r\xi)^{3/2}} |\widehat{v_{2,0}}(\xi)| d\xi\leq C \int_{\frac{1}{r}}^{\frac{1}{100}} \frac{d\xi}{(r\xi)^{3/2}} \sup_{x \in [100,\frac{1}{\xi}]}\left(\lambda(x) \log(x)\right) \log(\frac{1}{\xi})+C \int_{\frac{1}{100}}^{\infty} \frac{d\xi}{(r\xi)^{3/2}} \frac{1}{\xi^{101}}\\
&\leq \frac{C}{r} \sup_{x \in [100,r]}\left(\lambda(x) \log(x)\right) \log(r)\end{split}\end{equation}
Combining this with \eqref{v2sqrtest} finishes the estimation of $v_{2}(t,r)$ in the region $r \geq \frac{t}{2}$. The derivatives of $v_{2}$ are treated similarly. The only important difference in the procedure used to estimate $\partial_{t}v_{2}(t,r)+\partial_{r}v_{2}(t,r)$ is that we exploit the fact that $\left(\partial_{t}+\partial_{r}\right)\sin((t-r)\xi) =0$, and similarly with $\cos((t-r)\xi)$. \end{proof}
Finally, we define  $v_{2,sub}$ by 
\begin{equation}\label{v2subdef} \begin{split} v_{2,sub}(t,r) &= \frac{-r}{2\pi} \int_{0}^{\pi} \sin^{2}(\theta) \left(F(t+r\cos(\theta))-F(t)\right)d\theta=v_{2}(t,r)-v_{2,main}(t,r)\end{split}\end{equation}
(where we recall that $F$ is defined in \eqref{Fdef}). We also define
$$v_{2,cubic,main}(t,r) = \frac{-r^{3}}{32} F''(t)$$
Then, the $v_{2}$ analog of Lemma \ref{w1strlemma} is 
\begin{lemma}\label{v2strlemma} We have the following estimates. For $0 \leq j \leq 8$, $0 \leq k \leq 2$ and $r \leq \frac{t}{2}$,
 $$|\partial_{t}^{j}\partial_{r}^{k}\left(v_{2,sub}(t,r)-v_{2,cubic,main}(t,r)+\frac{r^{5}}{768} F^{(4)}(t)\right)| \leq \frac{C r^{7-k} \lambda(t)\log(t)}{t^{8+j}}$$
\end{lemma}
\begin{proof} We expand \eqref{v2} and directly estimate, as in Lemma \ref{v2estlemma}  \end{proof}
\subsection{Second Wave Iteration}
The second wave correction, $u_{w,2}$ is defined as the solution to
\begin{equation}\label{uw2eqndef}-\partial_{tt}u_{w,2}+\partial_{rr}u_{w,2}+\frac{1}{r}\partial_{r}u_{w,2}-\frac{u_{w,2}}{r^{2}}=\left(\frac{\cos(2Q_{1}(\frac{r}{\lambda(t)}))-1}{r^{2}}\right)\left(w_{1}+v_{2}\right):=RHS_{2}(t,r)\end{equation}
with $0$ Cauchy data at infinity. (We carried out the first order matching before defining $u_{w,2}$ so that we could choose $v_{2,0}$ before having to consider the equation defining $u_{w,2}$). Later on, we will add, to $u_{w,2}$, a free wave, $v_{2,2}$, solving
\begin{equation}\label{v22def}\begin{cases}-\partial_{tt}v_{2,2}+\partial_{rr}v_{2,2}+\frac{1}{r}\partial_{r}v_{2,2}-\frac{v_{2,2}}{r^{2}}=0\\
v_{2,2}(0,r)=0\\
\partial_{t}v_{2,2}(0,r) = v_{2,3}(r)\end{cases}\end{equation}
with $v_{2,3}$ chosen so as to satisfy a third order matching condition which we will describe later on. We start by proving estimates on $\widehat{RHS_{2}}$, which will allow us to justify a representation formula for $u_{w,2}$. For this, we start with the following definitions.
\begin{equation}\label{rhsdefs}RHS_{2,1}(t,r) = \left(\frac{\cos(2Q_{1}(\frac{r}{\lambda(t)}))-1}{r^{2}}\right)w_{1}, \quad RHS_{2,2}(t,r) = \left(\frac{\cos(2Q_{1}(\frac{r}{\lambda(t)}))-1}{r^{2}}\right)v_{2}\end{equation}   
Then, we have the following.
\begin{lemma}\label{rhs21lemma} Recalling that $w_{1}$ is defined in \eqref{firstw1formula}, we have, for $2 \leq k \leq 4$ and $0 \leq j \leq 1$,
\begin{equation}\begin{split}& \xi^{j}t^{k}|\partial_{\xi}^{j}\partial_{t}^{k} \widehat{RHS_{2,1}}(t,\xi)|+ \xi^{j} |\partial_{\xi}^{j}\widehat{RHS_{2,1}}(t,\xi)| \leq \begin{cases} \frac{C \xi \lambda(t)^{3} \log^{2}(t)}{t^{2} }, \quad \xi \leq \frac{1}{\lambda(t)}\\
\frac{C (\log(t)+|\log(\xi)|)}{\sqrt{\lambda(t)}\xi^{5/2} t^{2} }, \quad \frac{1}{\lambda(t)} < \xi\end{cases}\end{split}\end{equation}
\end{lemma}
\begin{proof} In the regions where $\xi \leq \frac{1}{\lambda(t)}$, we use
$$\widehat{RHS_{2,1}}(t,\xi) = \int_{0}^{\infty} J_{1}(r\xi) \left(\frac{\cos(2Q_{1}(\frac{r}{\lambda(t)}))-1}{r^{2}}\right) w_{1}(t,r) r dr$$
and Lemma \ref{w1estlemma}. When $\xi \geq \frac{1}{\lambda(t)}$, we use
$$\widehat{RHS_{2,1}}(t,\xi) = \frac{1}{\xi^{2}}\int_{0}^{\infty} J_{1}(r\xi) H_{1}\left(\left(\frac{\cos(2Q_{1}(\frac{r}{\lambda(t)}))-1}{r^{2}}\right)w_{1}(t,r)\right) r dr$$
where
$$H_{1}(g) = -\partial_{r}^{2}g-\frac{1}{r}\partial_{r}g+\frac{g}{r^{2}}$$
and then use the estimates on $w_{1}$ from Lemma \ref{w1estlemma}. We use the same procedure to estimate $\partial_{t}^{k} \widehat{RHS_{2,1}}(t,\xi)$. To estimate $\partial_{\xi}\partial_{t}^{k}\widehat{RHS_{2,1}}(t,\xi)$, we start with
$$\widehat{RHS_{2,1}}(t,\xi) = \int_{0}^{\infty} RHS_{2,1}(t,\frac{x}{\xi}) J_{1}(x) \frac{x dx}{\xi^{2}}$$
differentiate under the integral, and use the symbol-type nature of the estimates in Lemma \ref{w1estlemma}.
 \end{proof}
Note that some of the estimates in the following lemma can be combined into a single estimate with multiple cases of numbers of derivatives, but is presented as is for convenience.
\begin{lemma} \label{kernelests} Let 
$$f(x) = \left(\frac{\cos(2Q_{1}(x))-1}{x^{2}}\right)=\frac{-8}{(1+x^{2})^{2}}$$
and
\begin{equation}\label{kdef}K(y,z) = \int_{0}^{\infty} J_{1}(xy) f(x) J_{1}(x z) x dx\end{equation}
Then, 
\begin{equation}\label{kexpform}K(y,z) = 4 \cdot \begin{cases} z I_{0}(z) K_{1}(y)-y I_{1}(z)K_{2}(y), \quad 0 < z < y\\
y I_{0}(y) K_{1}(z) - z I_{1}(y) K_{2}(z), \quad 0 < y < z\end{cases}:= 4 \cdot \begin{cases} K_{y>z}(y,z), \quad 0 < z < y\\
K_{z>y}(y,z), \quad 0 < y < z\end{cases}\end{equation}
and, we have the following estimates, for $z \neq y$, $0 \leq j \leq 1$, and  $0 \leq k \leq 4$:
\begin{equation}|\partial_{y}^{j}\partial_{z}^{k} K(y,z)| \leq \left(\text{max}\{1,\frac{1}{y}\}\right)^{j} \begin{cases}  C e^{-|y-z|} \left(\frac{|z-y|}{\sqrt{zy}} + \frac{\sqrt{\text{max}\{y,z\}}}{\text{min}\{y,z\}^{3/2}}\right), \quad y,z>1\\
C \sqrt{y} e^{-y} z^{\frac{1}{2}\left((-1)^{k}+1\right)}, \quad y>1>z\\
C \sqrt{z} e^{-z} y, \quad z>1>y\end{cases}\end{equation}
For $1>y,z$, $y \neq z$, and $0 \leq j, k \leq 1$,
\begin{equation}\begin{split}&|\partial_{y}^{j}\partial_{z}^{k}K(y,z)| \leq C y^{1-j} z^{1-k} (1+|\log(\text{max}\{y,z\})|), \quad |\partial_{y}^{j}\partial_{z}^{2+k}K(y,z)| \leq \frac{C y^{-j} \text{min}\{y,z\}}{z^{k}\text{max}\{y,z\}}\\
&|\partial_{y}^{j}\partial_{z}^{4+k} K(y,z)| \leq \begin{cases}
\frac{C z^{1-k}}{y^{1+j}}, \quad 1>y>z\\
\frac{C y^{1-j}}{z^{3+k}}, \quad 1>z>y
\end{cases}\end{split} \end{equation}
If $1 \leq j \leq 2$, $0 \leq n \leq 1$, and  $0 \leq k \leq 4-j$, then, for $0 < \omega, \xi$ and $\omega \neq \xi$, we have
\begin{equation}\begin{split}&|\partial_{\xi}^{n}\partial_{\omega}^{k}\partial_{t}^{j}\left(K(\xi\lambda(t),\omega\lambda(t))\right)| \\
&\leq \frac{C \lambda(t)^{k}}{t^{j}} \left(\text{max}\{\lambda(t),\frac{1}{\xi}\}\right)^{n}\begin{cases} e^{-\lambda(t)|\xi-\omega|} \left(\frac{\lambda(t)^{j} |\xi-\omega|^{j+1}}{\sqrt{\omega \xi}}+\frac{\sqrt{\text{max}\{\xi,\omega\}}}{\text{min}\{\xi,\omega\}^{3/2}\lambda(t)}\right), \quad \omega, \xi >  \frac{1}{\lambda(t)}\\
e^{-\lambda(t) \xi} \left(\xi \lambda(t)\right)^{1/2+j} \left(\omega \lambda(t)\right)^{\frac{1}{2}\left((-1)^{k}+1\right)}, \quad \xi > \frac{1}{\lambda(t)} > \omega\\
e^{-\omega \lambda(t)} \xi \lambda(t) \left(\omega \lambda(t)\right)^{\frac{1}{2}+j}, \quad \omega > \frac{1}{\lambda(t)} > \xi\end{cases}\end{split}\end{equation}
For $1 \leq j \leq 2$ , $0 \leq n,k \leq 1$, and $\frac{1}{\lambda(t)} \geq \omega , \xi$,  $\omega \neq \xi$, we have
$$|\partial_{\xi}^{n}\partial_{t}^{j} \partial_{\omega}^{k}\left(K(\xi\lambda(t),\omega\lambda(t))\right)| \leq \frac{C \lambda(t)^{2}}{t^{j}} \omega^{1-k}\xi^{1-n}\left(1+|\log(\text{max}\{\xi,\omega\}\lambda(t))|\right)$$
 For $0 \leq n,k \leq 1$, $1 \leq j \leq 2$, and $\frac{1}{\lambda(t)} > \omega, \xi$, $\omega \neq \xi$, we have
$$|\partial_{\xi}^{n}\partial_{\omega}^{2+k}\partial_{t}^{j}\left(K(\xi\lambda(t),\omega \lambda(t))\right)| \leq \frac{C \lambda(t)^{2}}{t^{j} } \frac{\text{min}\{\omega,\xi\}}{\text{max}\{\omega,\xi\} \omega^{k}\xi^{n}}$$
\end{lemma}

\begin{proof}
We start with the table of Gradshteyn and Ryzhik, \cite{gr}, entry 6.541, 1. A special case of this identity says that, for $c>0$, 
\begin{equation}\label{grint}\int_{0}^{\infty} J_{1}(x y) J_{1}(x z) \frac{x dx}{x^{2}+c^{2}} = \begin{cases} I_{1}(z c) K_{1}(y c), \quad 0 < z < y\\
I_{1}(y c) K_{1}(z c), \quad 0 < y < z\end{cases}\end{equation}
Upon differentiating in $c$ (it is possible to differentiate under the integral, by the dominated convergence theorem), and setting $c=1$, we get
$$K(y,z) = -8 \int_{0}^{\infty} \frac{J_{1}(x y) J_{1}(x z) x dx}{(1+x^{2})^{2}} = 4 \cdot \begin{cases} z I_{0}(z) K_{1}(y)-y I_{1}(z)K_{2}(y), \quad 0 < z < y\\
y I_{0}(y) K_{1}(z) - z I_{1}(y) K_{2}(z), \quad 0 < y < z\end{cases}$$
At this stage, the estimates on $K(y,z)$ and its derivatives follow from a straightforward application of asymptotics of Bessel functions, which can be found from numerous sources, for example, \cite{gr}. To estimate the time derivatives of  $K(\xi\lambda(t), \omega \lambda(t))$, we let $V=y\partial_{y}+z\partial_{z}$, and note that
$$\partial_{t}\left(K(\xi\lambda(t),\omega\lambda(t))\right) = \frac{\lambda'(t)}{\lambda(t)} V(K)\Bigr|_{\substack{y=\xi\lambda(t)\\ z=\omega\lambda(t)}}.$$
For $1 \leq j \leq 4$, we can iterate this to obtain expressions for $\partial_{t}^{j}\left(K(\xi\lambda(t),\omega\lambda(t))\right)$ in terms of $V^{n}(K)$ for $1 \leq n \leq j$. Then, we note that, for all $1 \leq j \leq 4$, there exist constants $c_{k,j,w}$ such that
$$V^{j}(K)(y,z) = \sum_{w=0}^{j} z^{w} \sum_{k=0}^{j} c_{k,j,w} (y-z)^{k} \partial_{y}^{k} \left(\partial_{y}+\partial_{z}\right)^{w} K(y,z).$$ 
This observation (and a decomposition of the above form, except with $y$ and $z$ switched (note that $V$ is symmetric in $y$ and $z$)), combined with the same procedure used above, allows us to estimate the time derivatives of  $K(\xi\lambda(t), \omega \lambda(t))$.
 \end{proof}
We can now estimate $\widehat{RHS_{2,2}}(t,\xi)$.
\begin{lemma}\label{rhs22lemma} With $RHS_{2,2}$ given in \eqref{rhsdefs}, and $m_{\geq 1}(x) = 1-m_{\leq 1}(x)$, where $m_{\leq 1}$ is defined in \eqref{mleq1def}, we have
\begin{equation}\label{rhs22hatformula}\widehat{RHS_{2,2}}(t,\xi) = \int_{0}^{\infty} \widehat{v_{2,0}}(\omega) \sin(t\omega) K(\xi \lambda(t), \omega \lambda(t)) d\omega\end{equation}
and, for any $N>10$, there exists $C_{N}>0$ such that we have the following estimates for $0 \leq k \leq 1$:
\begin{equation}|\partial_{\xi}^{k}\widehat{RHS_{2,2}}(t,\xi)| \leq \frac{C \lambda(t) \log^{2}(t)}{t^{2}} \sup_{x \in [100,\frac{t}{2}]}\left(\lambda(x) \log(x)\right)\begin{cases} \lambda(t)\xi^{1-k} \log(t), \quad \xi \leq \frac{1}{\lambda(t)}\\
\lambda(t)^{k}\sqrt{\xi \lambda} e^{-\xi \lambda(t)} + \frac{C_{N}\lambda(t)^{k}}{\langle \xi \rangle^{N}}, \quad \xi > \frac{1}{\lambda(t)}\end{cases}\end{equation}
\begin{equation}\label{dxidttrhs22est}\begin{split}&|\partial_{\xi}^{k}\left(\partial_{t}^{2} \widehat{RHS_{2,2}}(t,\xi)-\frac{8 \widehat{v_{2,0}}(\xi) \xi \lambda(t)^{2} m_{\geq 1}(\xi t) \sin(\xi t)}{t^{4}}\right)| \\
&\leq \frac{C \lambda(t) \log^{2}(t)}{t^{4}} \sup_{x \in [100,\frac{t}{2}]}\left(\lambda(x) \log(x)\right) \begin{cases} \xi^{1-k} \lambda(t)  \log(t), \quad \xi \leq \frac{1}{\lambda(t)}\\
\left(\xi \lambda(t)\right)^{5/2} \lambda(t)^{k}  e^{-\xi \lambda(t)} + \frac{C_{N} \lambda(t)^{k} (1+\lambda(t))}{\langle \xi \rangle^{N}}, \quad \xi > \frac{1}{\lambda(t)} \end{cases}\end{split}\end{equation}
\end{lemma}
\begin{proof} Equation \eqref{rhs22hatformula} follows from insertion of \eqref{v2} into \eqref{rhsdefs}, and Fubini's theorem. To prove the stated estimates, we start with the decomposition
\begin{equation}\label{dttrhs22hat}\begin{split} \partial_{t}^{2}\widehat{RHS_{2,2}}(t,\xi) &= - \int_{0}^{\infty} d\omega \widehat{v_{2,0}}(\omega) \omega^{2} \sin(t\omega) K(\xi\lambda(t),\omega \lambda(t))\\
&+2 \int_{0}^{\infty} d\omega \widehat{v_{2,0}}(\omega) \omega \cos(\omega t) \partial_{t}\left(K(\xi \lambda(t),\omega \lambda(t))\right)\\
&+\int_{0}^{\infty} d\omega \widehat{v_{2,0}}(\omega) \sin(\omega t) \partial_{t}^{2}\left(K(\xi \lambda(t),\omega\lambda(t))\right)\end{split}\end{equation}
(Strictly speaking, we first split the integral in \eqref{rhs22hatformula} over the regions $\omega \leq \xi$ and $\omega \geq \xi$, differentiate under the integral sign, and then combine the resulting integrals). Define $int_{i}(t,\xi)$ to be the $ith$ line on the right-hand side of the expression above. Starting with $int_{1}$, we make the following decomposition (where $m_{\leq 1}$ is defined in \eqref{mleq1def})
$$int_{1}(t,\xi) = int_{1,a}(t,\xi) +int_{1,b}(t,\xi), \quad int_{1,a}(t,\xi) = -\int_{0}^{\infty} d\omega \widehat{v_{2,0}}(\omega) \omega^{2} \sin(\omega t) m_{\leq 1}(\omega t) K(\xi \lambda(t), \omega \lambda(t))$$
For $int_{1,a}$, we separately treat the cases $\xi < \frac{1}{t}$, $\frac{1}{\lambda(t)} > \xi > \frac{1}{t}$, and $\xi > \frac{1}{\lambda(t)}$. In each case, we directly estimate $int_{1,a}$ using Lemmas \ref{v20ests} and \ref{kernelests}. To estimate $int_{1,b}$, we decompose the integral into the regions $\omega \leq \xi$ and $\omega \geq \xi$, and integrate by parts four times in the $\omega$ variable. We recall \eqref{kexpform} and note that
$$4\left(\partial_{z}^{3} K_{y>z}(y,z)-\partial_{z}^{3} K_{z>y}(y,z)\right)\Bigr|_{z=y}=\frac{8}{y}$$
which explains the form of \eqref{dxidttrhs22est}. Then, we consider the two cases $\frac{1}{\lambda(t)} < \frac{1}{100}$, and $\frac{1}{\lambda(t)} > \frac{1}{100}$. Within each of these two cases, we then consider various regions of $\xi$. For example, in the case $\frac{1}{\lambda(t)} < \frac{1}{100}$, we consider the regions
$$\xi < \frac{1}{2 t}, \quad \frac{1}{2t} < \xi < \frac{1}{\lambda(t)}, \quad \frac{1}{\lambda(t)} < \xi < \frac{1}{100}, \quad \frac{1}{100} < \xi$$
Then, we directly estimate the resulting integrals using Lemmas \ref{v20ests} and \ref{kernelests}. After this, we combine the estimates in the separate cases $\frac{1}{\lambda(t)} < \frac{1}{100}$, and $\frac{1}{\lambda(t)} > \frac{1}{100}$. We then use the same procedure to estimate $int_{2}, int_{3}$ (where we integrate by parts $4-i+1$ times to treat $int_{i,b}$) and combine everything to get \eqref{dxidttrhs22est} for $k=0$. \\
\\
To estimate $\partial_{\xi}\partial_{t}^{2}\widehat{RHS_{2,2}}(t,\xi)$, we start with
\begin{equation}\label{dttrhs22hatint}\begin{split}\partial_{t}^{2} \widehat{RHS}_{2,2}(t,\xi) &= \int_{0}^{\infty} \widehat{v_{2,0}}(\omega)\begin{aligned}[t]&\left(-\omega^{2}\sin(t\omega) K(\xi\lambda(t),\omega\lambda(t)) + 2 \omega \cos(\omega t) \partial_{t}\left(K(\xi\lambda(t),\omega\lambda(t))\right) \right.\\
&\left.+ \sin(\omega t) \partial_{t}^{2}\left(K(\xi\lambda(t),\omega\lambda(t))\right)\right) m_{\leq 1}(\omega t) d\omega\end{aligned}\\
&+\frac{8 \widehat{v_{2,0}}(\xi) \xi \lambda(t)^{2} m_{\geq 1}(\xi t) \sin(\xi t)}{t^{4}}\\
&-\int_{0}^{\infty} \frac{\sin(\omega t)}{t^{4}}\partial_{\omega}^{4}\left(\widehat{v_{2,0}}(\omega) \omega^{2} K(\xi\lambda(t),\omega\lambda(t)) m_{\geq 1}(\omega t)\right) d\omega\\
&+2 \int_{0}^{\infty} \frac{\sin(\omega t)}{t^{3}} \partial_{\omega}^{3}\left(\widehat{v_{2,0}}(\omega) \omega \partial_{t}\left(K(\xi\lambda(t),\omega\lambda(t))\right) m_{\geq 1}(\omega t)\right) d\omega\\
&-\int_{0}^{\infty} \frac{\sin(\omega t)}{t^{2}} \partial_{\omega}^{2}\left(\widehat{v_{2,0}}(\omega)m_{\geq 1}(\omega t) \partial_{t}^{2}\left(K(\xi\lambda(t),\omega\lambda(t))\right)\right) d\omega \end{split}\end{equation}
where $m_{\geq 1}(x) = 1-m_{\leq 1}(x)$. We split each integral of \eqref{dttrhs22hatint} into the regions $\omega \leq \xi$ and $\omega \geq \xi$, and differentiate with respect to $\xi$. Finally, we use the estimates of Lemmas \ref{kernelests}, \ref{v20ests} to finish the proof of \eqref{dxidttrhs22est}. The estimation of $\partial_{\xi}^{k}\widehat{RHS}_{2,2}(t,\xi)$ for $k=0,1$ is done similarly. 
 \end{proof}
\noindent With the same procedure used to establish Lemma \ref{vexreplemma}, we get that $u_{w,2}$, the solution to \eqref{uw2eqndef} with zero Cauchy data at infinity, is given by the following. 
\begin{equation}\label{uw2def}\begin{split}u_{w,2}(t,r) &= \int_{t}^{\infty} ds \int_{0}^{\infty} d\xi \sin((t-s)\xi) J_{1}(r\xi) \widehat{RHS_{2}}(s,\xi)\\
&= - \int_{0}^{\infty} J_{1}(r\xi) \frac{\widehat{RHS_{2}}(t,\xi)}{\xi} d\xi - \int_{0}^{\infty} d\xi \int_{t}^{\infty} ds \frac{\sin((t-s)\xi)}{\xi^{2}} \partial_{s}^{2}\widehat{RHS_{2}}(s,\xi) J_{1}(r\xi)\end{split}\end{equation} 
Using Lemmas \ref{w1estlemma} and \ref{v2estlemma}, we can use Fubini's theorem to get that the first integral in \eqref{uw2def} is
\begin{equation}\label{uw2elldef}u_{w,2,ell}(t,r):=- \int_{0}^{\infty} J_{1}(r\xi) \frac{\widehat{RHS_{2}}(t,\xi)}{\xi} d\xi = -\frac{1}{2}\left(\frac{1}{r}\int_{0}^{r} s^{2} RHS_{2}(t,s) ds + r \int_{r}^{\infty} RHS_{2}(t,s) ds\right)\end{equation}
It will be useful to compute the following function, which will turn out to contain the leading behavior of $u_{w,2,ell}$ in the matching region.
\begin{equation}\label{uw2ell0def}u_{w,2,ell,0}(t,r):=-\frac{1}{2}\left(\frac{1}{r}\int_{0}^{r} s^{2} RHS_{2,0}(t,s) ds + r \int_{r}^{\infty} RHS_{2,0}(t,s) ds\right)\end{equation}
with
\begin{equation}\label{rhs20def}\begin{split}RHS_{2,0}(t,s) &:= \left(\frac{\cos(2Q_{1}(\frac{s}{\lambda(t)}))-1}{s^{2}}\right) s \left(f_{1}(t) - \lambda''(t)\log(s)\right) \\
&= \left(\frac{\cos(2Q_{1}(\frac{s}{\lambda(t)}))-1}{s^{2}}\right)\left(w_{1,main}+v_{2,main}\right)(t,s) \end{split}\end{equation}
and
\begin{equation}\label{f1tdef}f_{1}(t) = \lambda''(t) + \frac{\lambda'(t)^{2}}{2 \lambda(t)} + \lambda''(t) \log(\lambda(t))\end{equation}
where we recall the definitions of $w_{1,main}$ and $v_{2,main}$ in  \eqref{w1main} and \eqref{v2main}, which are the main parts of $w_{1}$ and $v_{2}$, respectively, in the matching region. We have
\begin{equation}\label{uw2ell0exp}\begin{split} u_{w,2,ell,0}(t,r)&= \frac{2 f_{1}(t) \lambda(t)^2 \log \left(\frac{r^2}{\lambda(t)^2}+1\right)}{r}\\
&-\frac{\lambda(t)^2 \lambda''(t) \left(\text{Li}_2\left(-\frac{r^2}{\lambda(t)^2}\right)+2 \log (r) \left(\log \left(\frac{r^2}{\lambda(t)^2}+1\right)-\frac{r^2}{\lambda(t)^2+r^2}\right)+\log \left(\frac{r^2}{\lambda(t)^2}+1\right)\right)}{r}\\
& -r \lambda''(t) \left(\frac{2 \lambda(t)^2 \log (r)}{\lambda(t)^2+r^2}+\log \left(\frac{\lambda(t)^2}{r^2}+1\right)\right)\end{split}\end{equation}
The leading behavior of $u_{w,2,ell,0}$ in the matching region is $u_{w,2,ell,0,cont}$, which is defined by
\begin{equation}\label{uw2ell0contdef}\begin{split}u_{w,2,ell,0,cont}(t,r):= \frac{-2}{r}&\left(\log(r)\left(-\lambda(t)\lambda'(t)^{2}-\lambda(t)^{2}\lambda''(t)-2 \lambda(t)^{2} \log(\lambda(t))\lambda''(t)\right)\right.\\
&\left.+\lambda(t) \log(\lambda(t))\lambda'(t)^{2}+\lambda(t)^{2}\log(\lambda(t))\lambda''(t) + \lambda(t)^{2}\log^{2}(\lambda(t))\lambda''(t)\right.\\
&+\left.\lambda(t)^{2}\log^{2}(r)\lambda''(t)+\frac{1}{2} \lambda(t)^2 \lambda''(t)-\frac{1}{12} \pi ^2 \lambda(t)^2 \lambda''(t)\right)\end{split}\end{equation}
In the above computation, we have used asymptotics of $\text{Li}_{2}$, for example, from \cite{gr}. In the course of proving Proposition \ref{thirdorderprop}, which will occur when we do a third order matching, we will record refined estimates on the difference between $u_{w,2}$ and its leading parts in the matching region. On the other hand, we will also need estimates which are global in the spatial coordinate, of the difference between the full correction $u_{w,2}$, and the piece $u_{w,2,ell,0,cont}$. Now that we have established Lemmas \ref{rhs21lemma} and \ref{rhs22lemma}, we can estimate $u_{w,2}-u_{w,2,ell,0,cont}$.
\begin{lemma}\label{uw2minuselllemma} For $0 \leq k \leq 1$, the following two estimates are true:
\begin{equation}\label{uw2minusellest}\begin{split}&|\partial_{r}^{k}\left(u_{w,2}(t,r)-u_{w,2,ell}\right)(t,r)| \leq Cr \frac{\lambda(t)^{2-k}(1+\lambda(t))\log^{5}(t+r)}{t^{4}} \sup_{x \in [100,\frac{t}{2}]}\left(\lambda(x) \log(x)\right), \quad r \geq g(t) \lambda(t)\end{split}\end{equation}
and
\begin{equation}\label{uw2largerest}|\partial_{r}^{k}u_{w,2}(t,r)| \leq C \lambda(t)^{\frac{1}{2}-k}(1+\lambda(t)^{2+k}) \log^{2}(t)\log^{2}(r) \frac{\sup_{x \in [100,t]} \left(\lambda(x) \log(x)\right) }{\sqrt{r}t^{2}}, \quad r \geq \frac{t}{2}\end{equation}
For $0 \leq j \leq 8$ and $0 \leq k \leq 1$, the following two estimates are true.
\begin{equation}\label{ellminusell0inside}|\partial_{r}^{k}\partial_{t}^{j}\left(u_{w,2,ell}-u_{w,2,ell,0}\right)(t,r)| \leq \frac{C r^{1-k} \lambda(t)^{2}\log(t)}{t^{4+j}} \sup_{x \in [100,t]}\left(\lambda(x) \log(x)\right), \quad r \leq \frac{t}{2}\end{equation}
and
$$|\partial_{r}^{k}\left(u_{w,2,ell}- u_{w,2,ell,0}\right)(t,r)| \leq \frac{C \lambda(t)^{2} \log(r)}{r^{1+k} t^{2}} \sup_{x \in [100,t+r]} \left(\lambda(x) \log(x)\right), \quad r > \frac{t}{2}$$
For $0 \leq j,k \leq 5$,
\begin{equation}\label{ell0minusell0contprelim} |\partial_{t}^{j}\partial_{r}^{k}\left(u_{w,2,ell,0}(t,r)-u_{w,2,ell,0,cont}(t,r)\right)| \leq \begin{cases} \frac{C \lambda(t)^{3} \left(\log^{2}(r) + \log^{2}(t)\right)}{r^{1+k} t^{2+j}}, \quad r \leq \lambda(t)\\
\frac{C \lambda(t)^{5} \left(|\log(r)|+\log(t)\right)}{r^{3+k} t^{2+j} }, \quad r \geq \lambda(t)\end{cases}\end{equation}
For $0 \leq k \leq 1$,
\begin{equation}\label{uw2ellsymb} r|\partial_{r} u_{w,2,ell}(t,r)| + t^{k}|\partial_{t}^{k}u_{w,2,ell}(t,r)| \leq \begin{cases} \frac{C r \lambda(t)(\log(t)+|\log(r)|)}{t^{2}}, \quad r \leq \lambda(t)\\
\frac{C \lambda(t)^{2} \log(t)}{r t^{2}} \sup_{x \in [100,t]}\left(\lambda(x) \log(x)\right), \quad \lambda(t) < r < \frac{t}{2}\\
\frac{C \lambda(t)^{2} \log(r) \sup_{x \in [100,t+r]}\left(\lambda(x) \log(x)\right)}{\sqrt{r} t^{5/2}}, \quad r > \frac{t}{2}\end{cases}\end{equation}
For $0 \leq k \leq 2$,
\begin{equation}\label{dttuw2ellsymb} |\partial_{t}^{2+k}u_{w,2,ell}(t,r)| \leq \begin{cases} \frac{C r \lambda(t)(\log(t)+|\log(r)|)}{t^{4+k}}, \quad r \leq \lambda(t)\\
\frac{C \lambda(t)^{2} \log(t)}{r t^{4+k}} \sup_{x \in [100,t]}\left(\lambda(x) \log(x)\right), \quad \lambda(t) < r < \frac{t}{2}\\
\frac{C \lambda(t)^{2} \log(r) \sup_{x \in [100,t+r]}\left(\lambda(x) \log(x)\right)}{\sqrt{r}}\left(\frac{1}{r^{4}\langle t-r \rangle^{\frac{1}{2}+k}}+\frac{1}{t^{9/2+k}}\right), \quad r > \frac{t}{2}\end{cases}\end{equation}
Finally, we have the following estimate, for all $r >0$. (Recall the definition of $E$ in \eqref{ecoercive}).
\begin{equation}\label{uw2enest}\sqrt{E(u_{w,2}(t),\partial_{t}u_{w,2})} +|u_{w,2}(t,r)| \leq \frac{C \log(t)}{t^{1-C_{u}}}\end{equation}
\end{lemma}
\begin{proof} We have
\begin{equation}\label{uw2ellminusell0initialstep}\begin{split}&u_{w,2,ell}(t,r)-u_{w,2,ell,0}(t,r) \\
&= -\frac{1}{2}\left(\frac{1}{r}\int_{0}^{r} s^{2}\left(RHS_{2}(t,s)-RHS_{2,0}(t,s)\right) ds + r \int_{r}^{\infty} \left(RHS_{2}(t,s)-RHS_{2,0}(t,s)\right) ds\right)\end{split}\end{equation}
where we recall the definitions of $RHS_{2,0}$ and $RHS_{2}$ from \eqref{rhs20def}, and \eqref{uw2eqndef}, respectively. Inside the integrals, in the region $s \leq \frac{t}{2}$, we use
$$RHS_{2}(t,s) - RHS_{2,0}(t,s) = \left(\frac{\cos(2Q_{1}(\frac{s}{\lambda(t)}))-1}{s^{2}}\right) \left(v_{2,sub}(t,s) + w_{1,sub}(t,s)\right)$$
and the estimates on $v_{2,sub}$ and $w_{1,sub}$ from Lemmas \ref{w1strlemma} and \ref{v2strlemma}. (Recall the definitions of $w_{1,sub}$ and $v_{2,sub}$ in \eqref{w1subdef} and \eqref{v2subdef}, respectively). On the other hand, in the region $s \geq \frac{t}{2}$, we use 
$$\left(RHS_{2}- RHS_{2,0}\right)(t,s) = \left(\frac{\cos(2Q_{1}(\frac{s}{\lambda(t)}))-1}{s^{2}}\right) \left(v_{2}(t,s) + w_{1}(t,s) - s \left(f_{1}(t)-\lambda''(t) \log(s)\right)\right)$$
where $f_{1}$ is given by \eqref{f1tdef}, and the estimates from Lemmas \ref{w1estlemma} and \ref{v2estlemma}. The only detail to note is that, when establishing \eqref{ellminusell0inside} for $j \geq 1$, we will have terms of the following form, for $1 \leq k \leq j$.
$$-\frac{r}{2} \int_{\frac{t}{2}}^{\infty} g(s,t)\partial_{t}^{k} v_{2}(t,s) ds$$
For these, since the estimates on $\partial_{t}^{k}v_{2}(t,s)$ in the region $s \geq \frac{t}{2}$ from Lemma \ref{v2estlemma} are not a factor of $t^{-k}$ better than the corresponding estimates on $v_{2}(t,s)$, we write the following (iterating as needed if $j >2$)
\begin{equation}\label{dtkv2rewrite}\partial_{t}^{k}v_{2}(t,s) = \begin{cases} \left(\partial_{t}+\partial_{s}\right)v_{2}(t,s) -\partial_{s}v_{2}(t,s), \quad k=1\\
\partial_{s}^{2}v_{2}+\frac{1}{s}\partial_{s}v_{2}-\frac{v_{2}}{s^{2}}, \quad k=2\end{cases}\end{equation}
and integrate by parts for the terms involving $s$ derivatives, which is why \eqref{ellminusell0inside} has a factor of $\frac{1}{t^{j}}$.\\
\\
To estimate $u_{w,2}-u_{w,2,ell}$, we  first note that 
\begin{equation}\label{uw2minusuw2ell}u_{w,2}(t,r)-u_{w,2,ell}(t,r) = \int_{t}^{\infty} ds \int_{0}^{\infty} d\xi \frac{\sin((s-t)\xi)}{\xi^{2}} \partial_{s}^{2}\widehat{RHS_{2}}(s,\xi) J_{1}(r\xi)\end{equation}
Then, we recall \eqref{dxidttrhs22est}, and write
\begin{equation}\label{RHS23}\partial_{s}^{2}\widehat{RHS_{2}}(s,\xi) = \partial_{s}^{2}\widehat{RHS_{2,3}}(s,\xi)+ \frac{8 \widehat{v_{2,0}}(\xi) \xi \lambda(t)^{2} m_{\geq 1}(\xi t) \sin(\xi t)}{t^{4}}\end{equation}
We  split the $s$ integral in \eqref{uw2minusuw2ell} involving $\partial_{s}^{2}\widehat{RHS_{2,3}}(s,\xi)$ into the regions $s-t \leq r$ and $s-t \geq r$. Then, when $s-t \leq r$, we split the $\xi$ integral in \eqref{uw2minusuw2ell} over the three regions $\xi \leq \frac{1}{r}$, $\frac{1}{r} \leq \xi \leq \frac{1}{s-t}$ and $\frac{1}{s-t} < \xi$. We make a similar splitting in the region $s-t \geq r$, with the following detail. In the region $\xi \geq \frac{1}{s-t}$, and $s-t \geq r$, we integrate by parts in the $\xi$ variable, integrating $\sin((s-t)\xi)$, and differentiating the rest of the integrand. 
Then, we use
$$|J_{1}(x)| \leq \begin{cases} C x , \quad x \leq 1\\
\frac{C}{\sqrt{x}}, \quad x \geq 1\end{cases}$$
along with \eqref{lambdacomparg} and Lemmas \ref{rhs21lemma} and \ref{rhs22lemma} to estimate each resulting integral. Recalling \eqref{RHS23}, we now need to consider the $\frac{8 \widehat{v_{2,0}}(\xi) \xi \lambda(t)^{2} m_{\geq 1}(\xi t) \sin(\xi t)}{t^{4}}$ contribution to \eqref{uw2minusuw2ell}. For this, we first use
$$ \sin((s-t)\xi) \sin(\xi s) = \frac{1}{2}\left(\cos(t\xi)-\cos((2s-t)\xi)\right)$$
and then split the $\xi$ integral in \eqref{uw2minusuw2ell} over the regions $\xi \leq \frac{1}{t}$ and $\xi\geq \frac{1}{t}$, integrating by parts in $\xi$ in the second integral (where we integrate $\cos(t\xi)-\cos((2s-t)\xi)$ and differentiate the rest of the integrand). This establishes \eqref{uw2minusellest} for $k=0$. To establish \eqref{uw2minusellest} for $k=1$, we start by differentiating \eqref{uw2minusuw2ell} in $r$, and then use the same procedure as for $u_{w,2}-u_{w,2,ell}$.\\
\\
To estimate $u_{w,2}(t,r)$ for $r \geq \frac{t}{2}$, we start with
\begin{equation}\label{uw2formulaintermediate}u_{w,2}(t,r) = \int_{t}^{\infty} ds \int_{0}^{\infty} d\xi \sin((t-s)\xi)\widehat{RHS_{2}}(s,\xi) J_{1}(r\xi)\end{equation}
Then, we use the same procedure as above, except that, whenever $|s-t-r|\geq 1$, $\xi \geq \frac{1}{s-t}$ and $\xi \geq \frac{1}{r}$, we write
\begin{equation} J_{1}(x) = \frac{\sin (x)}{\sqrt{\pi } \sqrt{x}}-\frac{\cos (x)}{\sqrt{\pi } \sqrt{x}} + J_{1}(x) -\left(\frac{\sin (x)}{\sqrt{\pi } \sqrt{x}}-\frac{\cos (x)}{\sqrt{\pi } \sqrt{x}}\right)\end{equation}
and use
$$\sin (y) \left(\frac{\sin (x)}{\sqrt{\pi x }}-\frac{\cos (x)}{\sqrt{\pi x }}\right)=\frac{-\sin (y-x)-\sin (x+y)+\cos (y-x)-\cos (x+y)}{2\sqrt{ \pi x }}$$
with $y=(s-t)\xi$ and $x=r\xi$. Then, we integrate by parts in $\xi$, and estimate all of the resulting terms directly. We use the same procedure for $\partial_{r}u_{w,2}(t,r)$ in the region $r \geq \frac{t}{2}$. This yields an extra factor of $1+\frac{1}{\lambda(t)}$ relative to the estimates for $u_{w,2}(t,r)$ in the region $r \geq \frac{t}{2}$. (We can differentiate under the integral sign in \eqref{uw2formulaintermediate} by the dominated convergence theorem, and the fact that $\xi \widehat{RHS}_{2}(s,\xi) \in L^{1}_{s,\xi}((t,\infty) \times (0,\infty))$).\\
\\
Finally, to obtain \eqref{uw2ellsymb}, we recall the definition of $u_{w,2,ell}$, in \eqref{uw2elldef}.
$$u_{w,2,ell}(t,r) = -\frac{1}{2}\left(\frac{1}{r}\int_{0}^{r} s^{2} RHS_{2}(t,s) ds + r \int_{r}^{\infty} RHS_{2}(t,s) ds\right)$$
If $r \leq \frac{t}{2}$, we write
\begin{equation}\label{dtkuw2ellrleqtover2}\partial_{t}^{k} u_{w,2,ell}(t,r) = \frac{-1}{2r} \int_{0}^{r} s^{2}\partial_{t}^{k}RHS_{2}(t,s) ds - \frac{r}{2}\int_{\frac{t}{2}}^{\infty} \partial_{t}^{k}RHS_{2}(t,s) ds -\frac{r}{2}\int_{r}^{\frac{t}{2}} \partial_{t}^{k}RHS_{2}(t,s) ds\end{equation}
If $r \geq \frac{t}{2}$, we have
\begin{equation}\label{dtkuw2ellrgeqtover2}\partial_{t}^{k} u_{w,2,ell}(t,r) = \frac{-1}{2r} \int_{0}^{\frac{t}{2}} s^{2}\partial_{t}^{k}RHS_{2}(t,s) ds - \frac{1}{2r}\int_{\frac{t}{2}}^{r} s^{2}\partial_{t}^{k}RHS_{2}(t,s) ds -\frac{r}{2}\int_{r}^{\infty} \partial_{t}^{k}RHS_{2}(t,s) ds\end{equation}
The point of this decomposition is the following. The function $\frac{\cos(2Q_{1}(\frac{r}{\lambda(t)}))-1}{r^{2}}$ is a symbol in $t$ and $r$. The estimates on $w_{1}(t,r)$ and $v_{2}(t,r)$ from Lemmas \ref{w1estlemma} and \ref{v2estlemma} improve by a factor of $\frac{1}{t}$ with each time derivative taken (for example, for up to 4 derivatives), in the region $r \leq \frac{t}{2}$. On the other hand, $v_{2}(t,r)$ is a free wave, and our estimates on, for example, $\partial_{t}v_{2}(t,r)$ from Lemma \ref{v2estlemma} are not a power of $t$ better than the estimates on $v_{2}(t,r)$ in the region $r \geq \frac{t}{2}$. Therefore, when $s \geq \frac{t}{2}$, we rewrite $\partial_{t}^{k} v_{2}(t,s)$ as in \eqref{dtkv2rewrite} (iterating these as needed, to treat $k=3,4$). Then, we insert this re-writing into \eqref{dtkuw2ellrleqtover2} and \eqref{dtkuw2ellrgeqtover2}, and integrate by parts in $s$ to remove all $s$ derivatives from $v_{2}$ (or $\left(\partial_{t}+\partial_{s}\right)v_{2}$, as appropriate). A direct estimation of all of the integral and boundary terms then gives rise to \eqref{uw2ellsymb} and \eqref{dttuw2ellsymb}. The form of the estimates \eqref{uw2ellsymb} and \eqref{dttuw2ellsymb} are different in the region $r \geq \frac{t}{2}$ because, for sufficiently regular $f$,
\begin{equation}\begin{split} &-\frac{1}{2r} \int_{\frac{t}{2}}^{r} s^{2}f(t,s)\partial_{s}v_{2}(t,s) ds - \frac{r}{2}\int_{r}^{\infty} f(t,s)\partial_{s}v_{2}(t,s) ds\\
&=\frac{1}{2r} \left(\frac{t}{2}\right)^{2} f(t,\frac{t}{2}) v_{2}(t,\frac{t}{2}) + \frac{1}{2r} \int_{\frac{t}{2}}^{r} v_{2}(t,s) \partial_{s}\left(s^{2}f(t,s)\right) ds + \frac{r}{2}\int_{r}^{\infty} v_{2}(t,s) \partial_{s}f(t,s) ds\end{split}\end{equation}
In other words, the boundary terms at $r$ arising from one integration by parts exactly cancel. This does not happen if $\partial_{s}v_{2}(t,s)$ in the first line is replaced by $\partial_{s}^{k} v_{2}(t,s)$, and we integrate by parts $k$ times, for $k \geq 2$. The estimate \eqref{ell0minusell0contprelim} follows from the explicit formulae \eqref{uw2ell0exp} and \eqref{uw2ell0contdef}. Finally, we prove the energy estimate, \eqref{uw2enest}. By the dominated convergence theorem, and fact that $\xi \widehat{RHS}_{2}(s,\xi) \in L^{1}_{s,\xi}((t,\infty) \times (0,\infty))$, we can differentiate once in either $t$ or $r$ under the integral in \eqref{uw2formulaintermediate}. The estimate on $||\partial_{t}u_{w,2}||_{L^{2}(r dr)}$ implied by \eqref{uw2enest} then follows directly from Minkowski's inequality, the $L^{2}$ isometry property of the Hankel transform of order 1, and the estimates from Lemmas \ref{w1estlemma} and \ref{v2estlemma}. Next, we have
$$\partial_{r}u_{w,2}(t,r)+\frac{u_{w,2}(t,r)}{r} = \int_{t}^{\infty} ds \int_{0}^{\infty} d\xi \sin((t-s)\xi) \xi J_{0}(r\xi) \widehat{RHS}_{2}(s,\xi)$$
Using again Minkowski's inequality, and the $L^{2}$ isometry property of the Hankel transforms of orders 0 and 1, we get
$$||\partial_{r}u_{w,2}(t,r)+\frac{u_{w,2}(t,r)}{r}||_{L^{2}(r dr)} \leq C \int_{t}^{\infty} ds ||RHS_{2}(s,r)||_{L^{2}(r dr)}<\infty$$
where the last inequality is true by the estimates from Lemmas \ref{w1estlemma} and \ref{v2estlemma}. Therefore, by the dominated convergence theorem,
$$\int_{0}^{\infty} \left(\partial_{r}u_{w,2}(t,r)+\frac{u_{w,2}(t,r)}{r}\right)^{2} rdr = \lim_{M \rightarrow \infty} \int_{\frac{1}{M}}^{M} \left(\partial_{r}u_{w,2}(t,r)+\frac{u_{w,2}(t,r)}{r}\right)^{2} rdr$$
Then, we note that
$$\left(\partial_{r}u_{w,2}(t,r)+\frac{u_{w,2}(t,r)}{r}\right)^{2} = (\partial_{r}u_{w,2}(t,r))^{2}+\frac{u_{w,2}(t,r)^{2}}{r^{2}} + \frac{\partial_{r}(u_{w,2}^{2})}{r}$$
For each $M>0$,
$$\int_{\frac{1}{M}}^{M} \frac{\partial_{r}(u_{w,2}^{2})(t,r)}{r} rdr = u_{w,2}^{2}(t,M)-u_{w,2}^{2}(t,\frac{1}{M})$$
Again because of $\xi \widehat{RHS}_{2}(s,\xi) \in L^{1}_{s,\xi}((t,\infty) \times (0,\infty))$ and the dominated convergence theorem, 
$$\lim_{r \rightarrow 0} \left(\int_{t}^{\infty} ds \int_{0}^{\infty} d\xi \sin((s-t)\xi)\widehat{RHS_{2}}(s,\xi) J_{1}(r\xi)\right) =0$$
Then, by \eqref{uw2largerest}, we get 
$$\lim_{M \rightarrow \infty} \int_{\frac{1}{M}}^{M} \frac{\partial_{r}(u_{w,2}^{2})(t,r)}{r} rdr = 0$$
Therefore, 
$$\infty> \int_{0}^{\infty} \left(\partial_{r}u_{w,2}(t,r)+\frac{u_{w,2}(t,r)}{r}\right)^{2} rdr = \lim_{M \rightarrow \infty} \int_{\frac{1}{M}}^{M} \left((\partial_{r}u_{w,2}(t,r))^{2}+\frac{u_{w,2}(t,r)^{2}}{r^{2}}\right) rdr$$
By the monotone convergence theorem, we thus get
$$\int_{0}^{\infty} \left((\partial_{r}u_{w,2}(t,r))^{2}+\frac{u_{w,2}(t,r)^{2}}{r^{2}}\right) rdr \leq C \left(\int_{t}^{\infty} ds ||RHS_{2}(s,r)||_{L^{2}(r dr)}\right)^{2}$$
Finally, \eqref{uw2enest} follows from
\begin{equation}\label{ptwseenest}\begin{split}|u_{w,2}(t,r)| &\leq \int_{t}^{\infty}  \int_{0}^{\infty}  |\widehat{RHS_{2}}(s,\xi)| \sqrt{\xi} \frac{|J_{1}(r\xi)|}{\sqrt{\xi}} d\xi ds\\
&\leq \int_{t}^{\infty}  \left(\int_{0}^{\infty} |\widehat{RHS_{2}(s,\xi)}|^{2} \xi d\xi\right)^{1/2} \left(\int_{0}^{\infty} \frac{|J_{1}(x)|^{2}}{x} dx\right)^{1/2}ds\leq C \int_{t}^{\infty}||RHS_{2}(s,r)||_{L^{2}(r dr)}ds\end{split}\end{equation}
 \end{proof}
\subsection{Second order matching, part 1}\label{second_order_matching_pt1_section}
The first question is whether the choice of $v_{2,0}$ allows for matching between the next order ( roughly on the order of $\frac{r^{3} \lambda(t)}{t^{4}}$) terms in $w_{1}(t,r)$ and $v_{ell,2}(t,\frac{r}{\lambda(t)})$. For clarity, we recall the third, seventh, and eighth lines of \eqref{w1cubeterms}, and the expression \eqref{vell20r3cof}. The main contribution of $v_{ell,2,0}(t,\frac{r}{\lambda(t)})$ to the matching is
\begin{equation}v_{ell,2,0,main}(t,\frac{r}{\lambda(t)}) = \frac{r^{3}}{8} \partial_{t}^{2}\left(\frac{\lambda'(t)^{2}}{2\lambda(t)} + \lambda''(t) - \lambda''(t) \log(\frac{r}{\lambda(t)})\right) + \frac{3}{32} r^{3} \lambda''''(t)\end{equation}
The $r^{3}$ contribution of $w_{1}(t,r)$ is (recall \eqref{w1strlemma})
\begin{equation}\begin{split}w_{1,cubic,main}(t,r) &=\frac{3}{32} r^{3} \lambda''''(t)+\frac{r^{3}}{8}\left( \lambda''''(t) \left(\log(2)+\frac{1}{2}\right)-\log(r) \lambda''''(t) + \log(t) \lambda''''(t)\right. \\
&\left.+ \int_{t}^{2t} \frac{\lambda''''(s)-\lambda''''(t)}{s-t} ds + \int_{2t}^{\infty} \frac{\lambda''''(s) ds}{s-t}\right)\end{split}\end{equation}
Finally, the $v_{2}$ contribution at this order is
$$v_{2,cubic,main}(t,r) = \frac{-r^{3}}{32} F''(t)$$
Recall that we already chose $v_{2,0}$ to allow for matching of terms of the form $r f(t)$ coming from $v_{ell}(t,\frac{r}{\lambda(t)})$ and $w_{1}(t,r)+v_{2}(t,r)$. Note that the coefficient of $r^{3}$ given in the expression above for $v_{2,cubic,main}$  is precisely one-eighth of two time derivatives of the $r$ coefficient of $v_{2}$ (see \eqref{v2}). Although the $r^{3}$ coefficient of $v_{ell,2,0,main}(t,\frac{r}{\lambda(t)})$ and that of $w_{1,cubic,main}(t,r)$ are not precisely one-eighth of two time derivatives of the main $r$ coefficients in the matching region of $v_{ell}(t,\frac{r}{\lambda(t)})$ and $w_{1}(t,r)$, respectively, the $r^{3}$ coefficient of the \emph{difference} of these two functions is precisely one-eighth of two time derivatives of the main $r$ coefficient in the matching region of $v_{ell}(t,\frac{r}{\lambda(t)})-w_{1}(t,r)$. Note the important cancellation between the $\frac{3}{32} r^{3} \lambda''''(t)$ terms when $v_{ell,2,0,main}(t,\frac{r}{\lambda(t)})$ and $w_{1,cubic,main}(t,r)$ are subtracted. Therefore, the matching of the $r^{3}$ terms is already accomplished with our previous choice of $v_{2,0}$. In other words, by the choice of $v_{2,0}$, we have
\begin{equation} v_{ell,2,0,main}(t,\frac{r}{\lambda(t)})-\left(w_{1,cubic,main}(t,r)+v_{2,cubic,main}(t,r)\right) =0\end{equation}

\subsection{Second order matching, part 2}\label{second_order_matching_pt2_section}
Next, we choose the coefficient $c_{1}(t)$ in \eqref{velldef}. For convenience, we define $v_{ell,sub}$ by
\begin{equation}\label{vellsubdef}\begin{split}v_{ell,sub}(t,R) &= v_{ell}(t,R) -  R \left(\frac{1}{2}\lambda'(t)^{2} + \lambda(t) \lambda''(t) - \lambda(t) \lambda''(t) \log(R)\right) \\
&= \left(2c_{1}(t)- 3\lambda'(t)^{2}\right)f_{4}(R) + \lambda(t) \lambda''(t) f_{3}(R)\end{split}\end{equation}
with
$$f_{4}(R) = \frac{R}{2(1+R^{2})}$$ 
$$f_{3}(R) = \frac{R^{2} \left(-3+2\left(1+R^{2}\right)\log(R)\right)-\left(R^{4}-1\right)\log(1+R^{2})+2R^{2} Li_{2}(-R^{2})}{2R(1+R^{2})}$$
The function $c_{1}(t)$ appears as part of the coefficient of the $\frac{1}{r}$ term in a large $r$ expansion of $u_{ell,sub}(t,r)=v_{ell,sub}(t,\frac{r}{\lambda(t)})$. On the other hand, there are also higher order terms, of size $\frac{\log(r)}{r}$ and $\frac{\log^{2}(r)}{r}$ appearing in the expansion of $u_{ell,sub}(t,r)$ for large $r$, and $c_{1}(t)$ does not appear in the coefficients of these terms. However, it turns out that these terms happen to exactly match corresponding $\frac{\log(r)}{r}$ and $\frac{\log^{2}(r)}{r}$ terms coming from expansions of $v_{ex}$ and $u_{w,2,ell,0}$, which means that we only need to match terms of size $\frac{1}{r}$, which can be done by choosing an appropriate $c_{1}(t)$. To show this, we note the following.\\
\\
The main contribution of $v_{ell,sub}(t,\frac{r}{\lambda(t)})$ in the matching region is
\begin{equation}\label{vellsubcontdef}\begin{split}&v_{ell,sub,cont}(t,r)\\
 &= \frac{c_{1}(t) \lambda(t)-\frac{3}{2} \lambda(t) \lambda'(t)^2}{r} \\
& -\frac{\lambda(t)^2 \lambda''(t) \left(-6 \log (r) (4 \log (\lambda(t))+1)+6 \log (\lambda(t)) (2 \log (\lambda(t))+1)+12 \log ^2(r)+\pi ^2+12\right)}{6 r}\end{split}\end{equation}
On the other hand, the main contribution of $v_{ex}(t,r)$ in the matching region is (recall \eqref{vexell})
\begin{equation}\label{vexcontdef}v_{ex,cont}(t,r) = \frac{1}{2r}\begin{aligned}[t]&\left(\log(r)\left(-4 \lambda(t)\lambda'(t)^{2}-2 \lambda(t)^{2}\lambda''(t)\right)\right.\\
&+\left.4 \lambda(t)\log(\lambda(t))\lambda'(t)^{2}-\lambda(t)^{2}\lambda''(t) + 2 \lambda(t)^{2}\lambda''(t)\log(\lambda(t))\right)\end{aligned}\end{equation}
Finally, the main contribution of $u_{w,2,ell,0}(t,r)$ in the matching region is (recall \eqref{uw2ell0contdef})
\begin{equation}\begin{split}u_{w,2,ell,0,cont}(t,r) = \frac{-2}{r}&\left(\log(r)\left(-\lambda(t)\lambda'(t)^{2}-\lambda(t)^{2}\lambda''(t)-2 \lambda(t)^{2} \log(\lambda(t))\lambda''(t)\right)\right.\\
&\left.+\lambda(t) \log(\lambda(t))\lambda'(t)^{2}+\lambda(t)^{2}\log(\lambda(t))\lambda''(t) + \lambda(t)^{2}\log^{2}(\lambda(t))\lambda''(t)\right.\\
&+\left.\lambda(t)^{2}\log^{2}(r)\lambda''(t)+\frac{1}{2} \lambda(t)^2 \lambda''(t)-\frac{1}{12} \pi ^2 \lambda(t)^2 \lambda''(t)\right)\end{split}\end{equation}
These expressions give
\begin{equation}\label{secondordermatchb}v_{ell,sub,cont}(t,r)-\left(u_{w,2,ell,0,cont}(t,r)+v_{ex,cont}(t,r)\right) = -\frac{\lambda(t) \left(-6 c_{1}(t)+\left(3+2 \pi ^2\right) \lambda(t) \lambda''(t)+9 \lambda'(t)^2\right)}{6 r}\end{equation}
(As previously mentioned, the $\frac{\log(r)}{r}$ and $\frac{\log^{2}(r)}{r}$ terms from $v_{ell,sub,cont}$, $u_{w,2,ell,0,cont}$ and $v_{ex,cont}$ cancel in the above expression). We therefore choose $c_{1}(t)$ so as to make 
$$v_{ell,sub,cont}(t,r)-\left(u_{w,2,ell,0,cont}(t,r)+v_{ex,cont}(t,r)\right)=0$$ by defining
\begin{equation}\label{c1def}c_{1}(t) = \frac{3}{2}\lambda'(t)^{2}+\frac{\lambda(t) \lambda''(t)}{2}+\frac{\pi^{2}}{3} \lambda(t)\lambda''(t)\end{equation}
Note that the matching condition used to determine the initial velocity of $v_{2}$ is precisely
\begin{equation}\label{firstmatch}u_{ell}(t,r) - v_{ell,sub}(t,\frac{r}{\lambda(t)}) = v_{2}(t,r) - v_{2,sub}(t,r) + w_{1}(t,r)-w_{1,sub}(t,r)\end{equation}
where we recall that $w_{1,sub}$ and $v_{2,sub}$ were defined in \eqref{w1subdef} and \eqref{v2subdef}, respectively.
\subsection{Third order matching}\label{third_order_matching_section}
It will be convenient for us to define
\begin{equation}\label{vexsubdef}v_{ex,sub}(t,r) := v_{ex}(t,r)-v_{ex,ell}(t,r)= -\int_{0}^{\infty} d\xi J_{1}(r\xi) \int_{t}^{\infty} ds \frac{\sin((t-s)\xi)}{\xi^{2}} \partial_{s}^{2}\widehat{RHS}(s,\xi)\end{equation}
We will now choose $v_{2,2}$ (recall \eqref{v22def}) so as to match the principal terms (which are of the form $r \log^{k}(r) g_{k}(t)$) in $u_{ell,2}(t,r)-v_{ell,2,0,main}(t,\frac{r}{\lambda(t)})$ and $v_{ex,sub}(t,r)+u_{w,2}(t,r) - u_{w,2,ell,0,cont}(t,r)$. The main proposition of this section is the following. 
\begin{proposition}[Third order matching]\label{thirdorderprop} For $v_{2,2}$ defined in \eqref{v22def}, where $v_{2,3}$ is defined by \eqref{v23def}, and $F_{3}$ and $G_{3}$ are given in \eqref{f3def} and \eqref{g3def}, respectively, we have the following estimate. For $0 \leq j,k \leq 1$ or $k=2,j=0$, and $g(t)\lambda(t) \leq r \leq 2 g(t)\lambda(t)$,
\begin{equation}\begin{split} |&\partial_{t}^{j}\partial_{r}^{k}\left(u_{ell,2}(t,r)-v_{ell,2,0,main}(t,\frac{r}{\lambda(t)})-\left(u_{w,2}(t,r)-u_{w,2,ell,0,cont}(t,r)+v_{2,2}(t,r)+v_{ex}-v_{ex,cont}-q_{4,2}\right)\right)| \\
&\leq \frac{C \lambda(t)^{2-k} \log(t)}{t^{2+j}  g(t)^{3+k}} + \frac{C g(t)^{2-k} \lambda(t)^{4-k} \log^{4}(t) \sup_{x \in [100,t]}\left(\lambda(x) \log(x)\right)}{t^{5+j}}\end{split}\end{equation}
where $q_{4,2}$ is defined in \eqref{rhs42def} and satisfies the following inequality for $k=0,1$, and $r >0$.
\begin{equation}\begin{split}&|\partial_{t}^{k}q_{4,2}(t,r)| + \sqrt{E(\partial_{t}^{k}q_{4,2},\partial_{t}^{k+1}q_{4,2})} +||\partial_{r}^{2} q_{4,2}||_{L^{2}((\lambda(t)g(t), 2 \lambda(t)g(t)),r dr)}\leq \frac{C \lambda(t)^{5} \log^{3}(t)}{t^{5} }\end{split}\end{equation}
\end{proposition}
\begin{proof}
We start by computing the leading parts of $u_{ell,2}(t,r)-v_{ell,2,0,main}(t,\frac{r}{\lambda(t)})$ and $v_{ex,sub}(t,r)+u_{w,2}(t,r) - u_{w,2,ell,0,cont}(t,r)$ in the matching region, $r \sim g(t)\lambda(t)$. \\
\\
We recall our expressions for $v_{ex,sub}$ and $u_{w,2}-u_{w,2,ell}$:
$$v_{ex,sub}(t,r)= -\int_{0}^{\infty} d\xi J_{1}(r\xi) \int_{t}^{\infty} ds \frac{\sin((t-s)\xi)}{\xi^{2}} \partial_{s}^{2}\widehat{RHS}(s,\xi)$$
\begin{equation}\label{uw2minusuw2elleqn}u_{w,2}(t,r)-u_{w,2,ell}(t,r) = \int_{t}^{\infty} ds \int_{0}^{\infty} d\xi \frac{\sin((s-t)\xi)}{\xi^{2}} \partial_{s}^{2}\widehat{RHS_{2}}(s,\xi) J_{1}(r\xi)\end{equation}
In particular, these functions solve the following equations with 0 Cauchy data at infinity
$$\left(-\partial_{t}^{2}+\partial_{r}^{2}+\frac{1}{r}\partial_{r}-\frac{1}{r^{2}}\right)\left(u_{w,2}-u_{w,2,ell}\right) = \partial_{t}^{2} u_{w,2,ell}$$
$$\left(-\partial_{t}^{2}+\partial_{r}^{2}+\frac{1}{r}\partial_{r}-\frac{1}{r^{2}}\right)v_{ex,sub} = \partial_{t}^{2}v_{ex,ell}$$
In order to compute their leading behavior in the matching region, $r \sim g(t)\lambda(t)$, we define
$u_{w,2,sub,0}(t,r)$ and $v_{ex,sub,0}(t,r)$ to be the solutions to the following equations with 0 Cauchy data at infinity
\begin{equation}\label{uw2sub0eqn}\left(-\partial_{t}^{2}+\partial_{r}^{2}+\frac{1}{r}\partial_{r}-\frac{1}{r^{2}}\right)u_{w,2,sub,0} = \partial_{t}^{2} u_{w,2,ell,0,cont}\end{equation}
$$\left(-\partial_{t}^{2}+\partial_{r}^{2}+\frac{1}{r}\partial_{r}-\frac{1}{r^{2}}\right)v_{ex,sub,0}= \partial_{t}^{2} v_{ex,cont}$$
Each of these right-hand sides are of the following form (recall Section \ref{second_order_matching_pt2_section})
$$\frac{1}{r}\sum_{k=0}^{2}f_{k}(t)\log^{k}(r)$$
In particular, the right-hand sides of the equations for $u_{w,2,sub,0}$ and $v_{ex,sub,0}$ are technically singular at $r=0$, even though the right-hand sides of the equations for $u_{w,2}-u_{w,2,ell}$ and $v_{ex,sub}$ are not. This will not cause any problems for us, and is done so that the principal parts of $u_{w,2,sub,0}$ and $v_{ex,sub,0}$ can be exactly calculated.\\
\\
We will first prove the following lemma which will lead to fairly explicit formulae for $u_{w,2,sub,0}$ and $v_{ex,sub,0}$. After the lemma we will provide some information about how the formulae arise.
\begin{lemma}\label{inhomlemma}
Let $f_{k}$ be the functions arising from expressing either $\partial_{t}^{2}v_{ex,cont}$ or $\partial_{t}^{2}u_{w,2,ell,0,cont}$ in the form $\frac{1}{r}\sum_{k=0}^{2}f_{k}(t)\log^{k}(r)$. If $u$ is given by
\begin{equation}\label{solntologroverr}\begin{split} u(t,r) = -\int_{t}^{\infty} f_{1}(s) K_{0}\left(\frac{r}{s-t},r\right)ds \end{split}\end{equation}
where
\begin{equation}\begin{split} K_{0}(x,r) = \frac{1}{x} \begin{cases}  \begin{aligned}[t]&\left(\left(1-\sqrt{1-x^2}\right) \left(\log(r)-\log(x)\right)+\sqrt{1-x^2}-1+\sqrt{1-x^2} \log \left(\sqrt{1-x^2}+1\right)\right.\\
&\left.+\log \left(\sqrt{1-x^2}+1\right)-2 \sqrt{1-x^2} \log \left(2 \sqrt{1-x^2}\right)\right), \quad 0<x<1\end{aligned}\\
\left(\sqrt{x^{2}-1} \sin ^{-1}\left(\frac{1}{x}\right)+\log (r)-1\right), \quad x>1\end{cases} \end{split}\end{equation} 
then, $u$ solves
$$-\partial_{tt}u+\partial_{rr}u+\frac{1}{r}\partial_{r}u-\frac{u}{r^{2}} = \frac{f_{1}(t) \log(r)}{r}, \quad r>0$$
with 0 Cauchy data at infinity. Similarly, if $u$ is defined by 
\begin{equation}\label{solntologsquaredroverr}u(t,r) = \int_{t}^{\infty}  f_{2}(s) K_{2}(\frac{s-t}{r},r)ds\end{equation}
with
$$K_{2}(a,r)=\sum_{k=0}^{2} F_{k}(a) \log^{k}(r)$$
where, for $a>1$, we have
\begin{equation}\begin{split}&F_{0}(a)=a \left(\text{Li}_2\left(\frac{\sqrt{a^2-1}}{a}\right)+\text{Li}_2\left(-\frac{a}{\sqrt{a^2-1}}\right)+\text{Li}_2\left(1-2 a \left(a+\sqrt{a^2-1}\right)\right)\right.\\
&\left.\left.-\text{Li}_2\left(2-2 a \left(a+\sqrt{a^2-1}\right)\right)+2\right)+\frac{1}{2}\left(4 \sqrt{a^2-1} \text{Li}_2\left(2 a \left(a-\sqrt{a^2-1}\right)-1\right)+\frac{1}{4} a \log ^2\left(1-\frac{1}{a^2}\right)\right.\right.\\
&\left.+2 \left(\sqrt{a^2-1}+a\right) \log ^2\left(\sqrt{a^2-1}+a\right)+4 \left(-a \log \left(2 \left(a^2-1\right)\right)-\sqrt{a^2-1} \left(\log \left(a^2-1\right)-1+\log (4)\right)\right.\right.\\
&\left.\left.+a+a \log (a)\right) \log \left(\sqrt{a^2-1}+a\right)+\frac{1}{6} \sqrt{a^2-1} \left(12 \log \left(a^2-1\right) \left(\log \left(a^2-1\right)-2+\log (16)\right)\right.\right.\\
&\left.\left.-\pi ^2+24+24 (\log (2)-1) \log (4)\right)-8 a\right), \quad a>1\end{split}\end{equation}
\begin{equation}\begin{split}&F_{1}(a) = 2 \left(-\sqrt{a^2-1}-\left(\sqrt{a^2-1}+a\right) \log \left(\sqrt{a^2-1}+a\right)+\sqrt{a^2-1} \log \left(4 \left(a^2-1\right)\right)+a\right), \quad a>1\end{split}\end{equation}
\begin{equation}F_{2}(a) = \sqrt{a^2-1}-a, \quad a>1\end{equation}
and, for $0<a<1$, we have
\begin{equation}\begin{split} F_{0}(a) &= -2 \sqrt{1-a^2} \text{Im}\left(\text{Li}_2\left(\frac{1}{2}-\frac{i a}{2 \sqrt{1-a^2}}\right)\right)+\pi  \sqrt{1-a^2}\\
&+\left(\frac{3 \pi  a}{4}-2 \sqrt{1-a^2}\right) \cos ^{-1}(a)+\frac{1}{4} a \left(\tan ^{-1}\left(\frac{1-2 a^2}{2 a \sqrt{1-a^2}}\right)\right)^2-\sqrt{1-a^2} \log \left(4-4 a^2\right) \sin ^{-1}(a)\\
&+\frac{1}{2} a \sin ^{-1}(a) \tan ^{-1}\left(\frac{1-2 a^2}{2 a \sqrt{1-a^2}}\right)-2 a-\frac{3 \pi ^2 a}{16}-a \left(\text{arccos}(a)\right)^2, \quad 0<a<1\end{split}\end{equation}
\begin{equation}F_{1}(a) = -2 \sqrt{1-a^2} \sin ^{-1}(a) +2 a, \quad 0<a<1 \end{equation}
\begin{equation}F_{2}(a) = -a, \quad 0<a<1\end{equation} 
then, $u$ solves
$$-\partial_{tt}u+\partial_{rr}u+\frac{1}{r}\partial_{r}u-\frac{u}{r^{2}} = \frac{f_{2}(t) \log^{2}(r)}{r}, \quad r>0$$
with 0 Cauchy data at infinity.
\end{lemma}
\begin{proof} To verify that \eqref{solntologroverr} is the claimed solution to the Cauchy problem, we start with
$$u(t,r) = -\int_{0}^{1} K_{0}(\frac{1}{y},r) f_{1}(t+r y) r dy -\int_{1}^{\infty} K_{0}(\frac{1}{y},r) f_{1}(t+r y) r dy$$
and note that the explicit formulae for $K_{0}(x,r)$ in the regions $x<1$ and $x>1$, combined with the definitions of $f_{k}$ and the Dominated convergence theorem allow us to differentiate up to two times in either $t$ or $r$ under the integral signs. An integration by parts then establishes the lemma. (The same procedure is used to verify the stated property of \eqref{solntologsquaredroverr}).  \end{proof}
\emph{Remark}. To see how the formulae from the lemma can be formally derived, we  could first use the same procedure used for $w_{1}$ to obtain, for example, for $u_{w,2,sub,0}$, that
\begin{equation}\label{duhameluw2sub0}u_{w,2,sub,0}(t,r) = \int_{t}^{\infty}  u_{w,2,sub,0,s}(t,r)ds\end{equation}
where
\begin{equation}\label{uw2sub0ssphericalmeans}\begin{split} &u_{w,2,sub,0,s}(t,r) \\
&= \frac{-1}{2\pi} \int_{0}^{s-t}\frac{\rho d\rho}{\sqrt{(s-t)^{2}-\rho^{2}}} \int_{0}^{2\pi} d\theta \frac{\left(r+\rho\cos(\theta)\right)}{r^{2}+\rho^{2}+2 r \rho \cos(\theta)} \sum_{k=0}^{2} f_{k}(s) \log^{k}\left(\sqrt{r^{2}+\rho^{2}+2 r \rho \cos(\theta)}\right)\end{split}\end{equation}
The integral of the $k=0$ term in the sum above has been evaluated when computing $w_{1}$, via Cauchy's residue theorem. This same procedure can not directly be applied to the integrals of the $k \neq 0$ terms. We can still explicitly compute these integrals, with a procedure involving introducing a parameter into the integrals, differentiating in this parameter, and then using Cauchy's residue theorem. This gives the following. For $r \neq \rho$, we have
\begin{equation}\label{logint}\begin{split}&\int_{0}^{2\pi} 4 \log(r^{2}+\rho^{2}+2 r \rho \cos(\theta))\frac{\left(r+\rho \cos(\theta)\right)}{r^{2}+\rho^{2}+2 r \rho \cos(\theta)} d\theta=\begin{cases} \frac{8 \pi}{r} \left(\log(r^{2})+\log(1-\frac{\rho^{2}}{r^{2}})\right), \quad r>\rho\\
\frac{-8 \pi}{r} \log(1-\frac{r^{2}}{\rho^{2}}), \quad \rho>r\end{cases}\end{split}\end{equation}
\begin{equation}\label{logsquaredint}\begin{split}&\int_{0}^{2\pi} \frac{\log^{2}(r^{2}+\rho^{2}+2 r \rho \cos(\theta))}{r^{2}+\rho^{2}+2 r \rho \cos(\theta)}\left(r+\rho\cos(\theta)\right) d\theta \\
&= \begin{cases} \frac{4 \pi}{r} \left(\log^{2}(1-\frac{\rho^{2}}{r^{2}})+2 \log(r)\log(r-\frac{\rho^{2}}{r})+\text{Li}_{2}(\frac{\rho^{2}}{r^{2}})\right), \quad r > \rho\\
\frac{-4\pi}{r} \log(\rho^{2}-r^{2})\log(1-\frac{r^{2}}{\rho^{2}}), \quad \rho>r\end{cases}\end{split}\end{equation}
Carrying out the $\rho$ integral in \eqref{uw2sub0ssphericalmeans} then leads to the formulae in the Lemma statement.\\
\\
Using the information from Lemma \ref{inhomlemma}, we can calculate the principal parts of $u_{w,2,sub,0}(t,r)$ and $v_{ex,sub,0}(t,r)$ in the matching region, $r \sim g(t)\lambda(t)$. We describe in detail how this is done for $u_{w,2,sub,0}$, since $v_{ex,sub,0}$ can be treated in the same way. First, we introduce some notation. Let
\begin{equation}\label{k20def}\begin{split}K_{2,0}(\frac{s-t}{r},r) &=\frac{r \left(-12 \log (s-t) \log (4 (s-t))+\pi ^2-12 \log ^2(2)\right)}{24 (s-t)}\end{split}\end{equation}
We define $F_{k,0}$ by 
\begin{equation}\label{fk0def}K_{2,0}(a,r) = \sum_{k=0}^{2} F_{k,0}(a) \log^{k}(r)\end{equation}
We also recall the definition of $u_{w,2,ell,0,cont}$,  \eqref{uw2ell0contdef}, and write
\begin{equation}\label{uw2ell0contg}\partial_{t}^{2}u_{w,2,ell,0,cont}(t,r) = \frac{g_{0}''(t)}{r}+\frac{g_{1}''(t)\log(r)}{r}+\frac{g_{2}''(t) \log^{2}(r)}{r}\end{equation}
for 
\begin{equation}\begin{split}g_{0}(t) = -2 &\left(\frac{1}{2} \lambda(t)^2 \lambda''(t)-\frac{1}{12} \pi ^2 \lambda(t)^2 \lambda''(t)+\lambda(t)^2 \lambda''(t) \log ^2(\lambda(t))+\lambda(t)^2 \lambda''(t) \log (\lambda(t))+\lambda(t) \lambda'(t)^2 \log (\lambda(t))\right)\end{split}\end{equation}
$$g_{1}(t) = -2 \left(-\lambda(t)^2 \lambda''(t)-2 \lambda(t)^2 \lambda''(t) \log (\lambda(t))-\lambda(t) \lambda'(t)^2\right)$$
$$g_{2}(t) = -2 \lambda(t)^2 \lambda''(t)$$
Finally, we let
$$c_{0}=\int_{1}^{\infty} \left(F_{0}(a)-F_{0,0}(a)\right) da $$
Note that the exact value of $c_{0}$ is not needed for our purposes, though the precise value of many other constants appearing in the following lemma are needed. 
\begin{lemma}\label{uw2sub0leading}[Leading part of $u_{w,2,sub,0}$] Let $u_{w,2,sub,0,cont}$ be defined by
\begin{equation}\begin{split}&u_{w,2,sub,0,cont}(t,r)\\
&=-\frac{1}{2}g_{0}''(t) r -\frac{r}{2} g_{0}''(t) \log(\frac{t}{r}) - \frac{r}{2} \int_{t}^{2t}\frac{(g_{0}''(s)-g_{0}''(t))}{(s-t)} ds - \frac{r}{2} \int_{2t}^{\infty} \frac{g_{0}''(s)}{s-t} ds - g_{0}''(t) \frac{r}{4} \left(\log(4)-1\right)\\
&-\frac{r}{16} g_{1}''(t) \left(-12+\pi^{2}+8 \log(r)\right)+-g_{1}''(t)r \left(\frac{\log(\frac{t}{r})\log(2)}{2}+\frac{1}{4}\left(\log^{2}(t)-\log^{2}(r)\right)\right) \\
&-\frac{r}{2} \int_{2t}^{\infty} \frac{g_{1}''(s)}{(s-t)} \log(2(s-t)) ds-\frac{r}{2} \int_{t}^{2t}\left(g_{1}''(s)-g_{1}''(t)\right) \frac{\log(2(s-t))}{(s-t)} ds\\
&+ \frac{-r g_{1}''(t)}{24}\left(6 (\log (4)-1) \log (r)-2 \pi ^2+15+6 \log ^2(2)\right)\\
&-\frac{r g_{2}''(t)}{16} \left(2 \log (r) \left(4 \log (r)+\pi ^2-12\right)-28 \zeta (3)+\pi ^2+28\right)\\
&+g_{2}''(t)  \frac{r}{24} \log \left(\frac{r}{t}\right) \left(4 \log (t) \log (8 r t)+4 \log (r) \log (8 r)-\pi ^2+12 \log ^2(2)\right)\\
&+\int_{t}^{2t} \left(g_{2}''(s)-g_{2}''(t)\right) K_{2,0}(\frac{s-t}{r},r) ds+\int_{2t}^{\infty} g_{2}''(s)K_{2,0}(\frac{s-t}{r},r) ds\\
&+ g_{2}''(t) r \left(c_{0}+\frac{1}{12} \left(2 \pi ^2-15-6 \log ^2(2)\right) \log(r) +\frac{1}{4}\left(1-\log(4)\right) \log^{2}(r)\right)\end{split}\end{equation}
Then, for $0 \leq j+ k \leq 2$ 
\begin{equation}\label{uw2sub0subleadingest}|r^{k}t^{j}\partial_{r}^{k}\partial_{t}^{j}\left(u_{w,2,sub,0}(t,r)-u_{w,2,sub,0,cont}(t,r)\right)| \leq \frac{C r^{3} \lambda(t)^{3} \log^{5}(t)}{t^{6} }, \quad \lambda(t)g(t) \leq r \leq 2 g(t)\lambda(t)\end{equation}
\end{lemma}
\begin{proof}

Using \eqref{uw2ell0contg} and our calculations for $w_{1}$, along with Lemma \ref{inhomlemma}, we have
\begin{equation}\label{uw2sub0decomp}\begin{split}u_{w,2,sub,0}(t,r) &= \frac{-1}{r} \int_{t}^{t+r} g_{0}''(s)(s-t) ds -\frac{1}{r} \int_{t+r}^{\infty} g_{0}''(s)\left((s-t)-\sqrt{(s-t)^{2}-r^{2}}\right) ds\\
&-\int_{t}^{\infty} g_{1}''(s) K_{0}\left(\frac{r}{s-t},r\right) ds + \int_{t}^{\infty} g_{2}''(s)K_{2}\left(\frac{s-t}{r},r\right) ds\end{split}\end{equation}
The first two terms on the right-hand side of the above expression are treated with the identical procedure used to study $w_{1}$, simply replacing $\lambda$ in the $w_{1}$ expressions with $\frac{-g_{0}}{2}$. The leading part of the sum of these two terms in the matching region $r \sim g(t)\lambda(t)$ is (note that we will prove concrete estimates on the terms we claim are subleading) 
$$-\frac{1}{2}g_{0}''(t) r -\frac{r}{2} g_{0}''(t) \log(\frac{t}{r}) - \frac{r}{2} \int_{t}^{2t}\frac{(g_{0}''(s)-g_{0}''(t))}{(s-t)} ds - \frac{r}{2} \int_{2t}^{\infty} \frac{g_{0}''(s)}{s-t} ds - g_{0}''(t) \frac{r}{4} \left(\log(4)-1\right)$$ 
In fact, by inspecting \eqref{w1expr}, we have
\begin{equation}\begin{split}|&\frac{-1}{r} \int_{t}^{t+r} g_{0}''(s)(s-t) ds -\frac{1}{r} \int_{t+r}^{\infty} g_{0}''(s)\left((s-t)-\sqrt{(s-t)^{2}-r^{2}}\right) ds \\
&- \left(-\frac{1}{2}g_{0}''(t) r -\frac{r}{2} g_{0}''(t) \log(\frac{t}{r}) - \frac{r}{2} \int_{t}^{2t}\frac{(g_{0}''(s)-g_{0}''(t))}{(s-t)} ds - \frac{r}{2} \int_{2t}^{\infty} \frac{g_{0}''(s)}{s-t} ds - g_{0}''(t) \frac{r}{4} \left(\log(4)-1\right)\right)| \\
&\leq \frac{C r^{3} \lambda(t)^{3} \log^{2}(t) \left(1+\log(\frac{t}{r})\right)}{t^{6} }, \quad r \leq t\end{split}\end{equation}
Estimates on the $t$ and $r$ derivatives are done similarly, as in the proof of Lemma \ref{w1strlemma}. The third and fourth terms of \eqref{uw2sub0decomp} are treated with the same argument, which is essentially the same argument we used for $w_{1}$, but with a few differences in the details. For clarity, we start with the third term. We have
\begin{equation}\label{k0g1int}-\int_{t}^{\infty} g_{1}''(s) K_{0}\left(\frac{r}{s-t},r\right) ds = -\int_{t}^{t+r} g_{1}''(s)K_{0}\left(\frac{r}{s-t},r\right) ds - \int_{t+r}^{\infty} g_{1}''(s) K_{0}\left(\frac{r}{s-t},r\right) ds\end{equation}
With $y=\frac{s-t}{r}$, we get
\begin{equation}\begin{split}-\int_{t}^{t+r} g_{1}''(s)K_{0}\left(\frac{r}{s-t},r\right) ds &= -r\int_{0}^{1} g_{1}''(t+r y) K_{0}\left(\frac{1}{y},r\right) dy \\
&= -r\int_{0}^{1} g_{1}''(t+ry) \left(K_{0,0}(\frac{1}{y})+K_{0,1}(\frac{1}{y})\log(r)\right) dy\end{split}\end{equation}
where we note that, for $x >1$,
$$K_{0}(x,r) = K_{0,0}(x) + K_{0,1}(x) \log(r), \quad K_{0,0}(x) = \frac{\sqrt{x^{2}-1}  \sin ^{-1}\left(\frac{1}{x}\right)-1}{x}, \quad K_{0,1}(x) = \frac{1}{x}$$
Since $g(t)\lambda(t) \ll t$, the first term in the following decomposition is lower order relative to the second term, for $r \sim g(t)\lambda(t)$ 
\begin{equation}\begin{split}&-r\int_{0}^{1} g_{1}''(t+ry) \left(K_{0,0}(\frac{1}{y})+K_{0,1}(\frac{1}{y})\log(r)\right) dy\\
&=-rg_{1}''(t)\int_{0}^{1}  \left(K_{0,0}(\frac{1}{y})+K_{0,1}(\frac{1}{y})\log(r)\right) dy -r\int_{0}^{1} \left(g_{1}''(t+ry)-g_{1}''(t)\right) \left(K_{0,0}(\frac{1}{y})+K_{0,1}(\frac{1}{y})\log(r)\right) dy\end{split}\end{equation}
We thus get that the leading contribution from $-r\int_{0}^{1} g_{1}''(t+ry) \left(K_{0,0}(\frac{1}{y})+K_{0,1}(\frac{1}{y})\log(r)\right) dy$ is
\begin{equation}\label{contribution1logr}-\frac{r}{16} g_{1}''(t) \left(-12+\pi^{2}+8 \log(r)\right)\end{equation}
We treat the next term in \eqref{k0g1int}, namely $-\int_{t+r}^{\infty} g_{1}''(s) K_{0}\left(\frac{r}{s-t},r\right)ds.$ For this term, we start by decomposing $K_{0}(x,r)$ into the leading piece for small $x$ plus the remainder. 
$$K_{0}(x,r) = K_{0,sm}(x,r) + O\left(x^{3}\left(1+|\log(\frac{r}{x})|\right)\right), \quad x \rightarrow 0, \quad K_{0,sm}(x,r) =  \frac{1}{2} x (\log (r)-\log (x)+\log (2))$$
We therefore get
\begin{equation}\label{lgsminustdecomp}\begin{split} -\int_{t+r}^{\infty} g_{1}''(s) K_{0}\left(\frac{r}{s-t},r\right)ds &= -g_{1}''(t) \int_{t+r}^{2t} K_{0,sm}(\frac{r}{s-t},r) ds - \int_{t}^{2t} \left(g_{1}''(s)-g_{1}''(t)\right) K_{0,sm}(\frac{r}{s-t},r) ds\\
&+\int_{t}^{t+r}\left(g_{1}''(s)-g_{1}''(t)\right) K_{0,sm}(\frac{r}{s-t},r) ds  - \int_{2t}^{\infty} g_{1}''(s) K_{0,sm}(\frac{r}{s-t},r)ds\\
&-\int_{t+r}^{\infty} g_{1}''(s) \left(K_{0}-K_{0,sm}\right)(\frac{r}{s-t},r) ds\end{split}\end{equation}
This decomposition also appears in the $w_{1}$ expression \eqref{w1expr}. The leading contribution from the first four terms of the right-hand side of \eqref{lgsminustdecomp} is
\begin{equation}\label{contribution2alogr}\begin{split}&-g_{1}''(t)r \left(\frac{\log(\frac{t}{r})\log(2)}{2}+\frac{1}{4}\left(\log^{2}(t)-\log^{2}(r)\right)\right) \\
&-\frac{r}{2} \int_{2t}^{\infty} \frac{g_{1}''(s)}{(s-t)} \log(2(s-t)) ds-\frac{r}{2} \int_{t}^{2t}\left(g_{1}''(s)-g_{1}''(t)\right) \frac{\log(2(s-t))}{(s-t)} ds\end{split}\end{equation}
Moreover, the difference between $-\int_{t}^{\infty} g_{1}''(s) K_{0}\left(\frac{r}{s-t},r\right)ds$ and the leading parts of the terms considered thus far is equal to
\begin{equation} \begin{split} &-\int_{t}^{t+r} g_{1}''(s) K_{0}(\frac{r}{s-t},r) ds - \left(\frac{-r g_{1}''(t)}{16} \left(-12+\pi^{2} + 8 \log(r)\right)\right)\\
&-\int_{t+r}^{\infty} g_{1}''(s) K_{0}\left(\frac{r}{s-t},r\right)ds-\begin{aligned}[t]&\left(-g_{1}''(t)r \left(\frac{\log(\frac{t}{r})\log(2)}{2}+\frac{1}{4}\left(\log^{2}(t)-\log^{2}(r)\right)\right)\right. \\
&\left.-\frac{r}{2} \int_{2t}^{\infty} \frac{g_{1}''(s)}{(s-t)} \log(2(s-t)) ds-\frac{r}{2} \int_{t}^{2t}\left(g_{1}''(s)-g_{1}''(t)\right) \frac{\log(2(s-t))}{(s-t)} ds\right)\end{aligned}\\
&= \int_{t}^{t+r} \left(g_{1}''(s)-g_{1}''(t)\right) K_{0,sm}(\frac{r}{s-t},r) ds - \int_{t+r}^{\infty} g_{1}''(s) \left(K_{0}-K_{0,sm}\right)(\frac{r}{s-t},r) ds \\
&-r \int_{0}^{1} \left(g_{1}''(t+r y)-g_{1}''(t)\right) K_{0}(\frac{1}{y},r) dy\end{split}\end{equation}
We have
\begin{equation}\begin{split}&-r \int_{0}^{1} \left(g_{1}''(t+r y)-g_{1}''(t)\right) K_{0}(\frac{1}{y},r) dy + \int_{t}^{t+r} \left(g_{1}''(s)-g_{1}''(t)\right) K_{0,sm}(\frac{r}{s-t},r) ds\\
&= -r \int_{0}^{1} \left(g_{1}''(t+r y)-g_{1}''(t)-r y g_{1}'''(t)\right) K_{0}(\frac{1}{y},r) dy + \int_{t}^{t+r} \left(g_{1}''(s)-g_{1}''(t)-(s-t) g_{1}'''(t)\right) K_{0,sm}(\frac{r}{s-t},r) ds\\
&-r^{2} \int_{0}^{1} y g_{1}'''(t) K_{0}(\frac{1}{y},r) dy + \int_{t}^{t+r} g_{1}'''(t) (s-t) K_{0,sm}(\frac{r}{s-t},r) ds\end{split}\end{equation}
which gives
\begin{equation}\begin{split}&|-r \int_{0}^{1} \left(g_{1}''(t+r y)-g_{1}''(t)\right) K_{0}(\frac{1}{y},r) dy + \int_{t}^{t+r} \left(g_{1}''(s)-g_{1}''(t)\right) K_{0,sm}(\frac{r}{s-t},r) ds\\
&-\left(\frac{r^{2} g_{1}'''(t)}{18}\left(-7+\log(512)+3 \log(r)\right)\right)|\\
&\leq \frac{C r^{3} \left(1+|\log(r)|\right) \lambda(t)^{3} \log(t)}{t^{6} }\end{split}\end{equation}
The higher derivatives are estimated similarly, for example, by writing
$$\int_{t}^{t+r} \left(g_{1}''(s)-g_{1}''(t)-(s-t) g_{1}'''(t)\right) K_{0,sm}(\frac{r}{s-t},r) ds = \frac{r}{2}\int_{0}^{1}\left(g_{1}''(t+r y)-g_{1}''(t)-r y g_{1}'''(t)\right)\log(2 r y) \frac{dy}{y}$$
and using the symbol-type estimates on $\lambda$ to justify the differentiation under the integral, exactly as in the proof of Lemma \ref{w1strlemma}. We now turn attention to the fifth term on the right-hand side of \eqref{lgsminustdecomp}. The point here is that $(K_{0}-K_{0,sm})(\frac{r}{s-t},r)$ decays faster in $s-t$ than does $K_{0,sm}(\frac{r}{s-t},r)$. To exploit this, we integrate by parts in $s$ in the following integral, and the symbol type estimates on $g_{1}''(s)$ show that the non-boundary terms obtained are subleading relative to the boundary terms (this will be proven once we estimate the terms that we claim are subleading) 
\begin{equation}\label{k0subdecomp}\begin{split}&-\int_{t+r}^{\infty} g_{1}''(s) \left(K_{0}-K_{0,sm}\right)(\frac{r}{s-t},r) ds\\
&=-g_{1}''(t+r) \int_{t+r}^{\infty} \left(K_{0}-K_{0,sm}\right)(\frac{r}{y-t},r) dy - \int_{t+r}^{\infty} \left(\int_{s}^{\infty} \left(K_{0}-K_{0,sm}\right)(\frac{r}{y-t},r) dy\right) g_{1}'''(s) ds\end{split}\end{equation}
Recalling that $g(t)\lambda(t) \ll t$, the leading behavior of $-\int_{t+r}^{\infty} g_{1}''(s) \left(K_{0}-K_{0,sm}\right)(\frac{r}{s-t},r) ds$ in the matching region is
\begin{equation}\label{contribution2blogr}-g_{1}''(t) \int_{1}^{\infty} \left(K_{0}-K_{0,sm}\right)(\frac{1}{q},r) r dq = \frac{-r g_{1}''(t)}{24}\left(6 (\log (4)-1) \log (r)-2 \pi ^2+15+6 \log ^2(2)\right)\end{equation}
The subleading part of $-\int_{t+r}^{\infty} g_{1}''(s) \left(K_{0}-K_{0,sm}\right)(\frac{r}{s-t},r) ds$ is therefore given by
\begin{equation}\begin{split} &-\int_{t+r}^{\infty} g_{1}''(s) \left(K_{0}-K_{0,sm}\right)(\frac{r}{s-t},r) ds-\left(\frac{-r g_{1}''(t)}{24}\left(6 (\log (4)-1) \log (r)-2 \pi ^2+15+6 \log ^2(2)\right)\right)\\
&= -\left(g_{1}''(t+r)-g_{1}''(t)\right) \int_{t+r}^{\infty} \left(K_{0}-K_{0,sm}\right)(\frac{r}{y-t},r) dy - \int_{t+r}^{\infty} \int_{s}^{\infty} \left(K_{0}-K_{0,sm}\right)(\frac{r}{y-t},r) dy g_{1}'''(s) ds\end{split}\end{equation}
We have 
\begin{equation}\begin{split} &-\left(g_{1}''(t+r)-g_{1}''(t)\right) \int_{t+r}^{\infty} \left(K_{0}-K_{0,sm}\right)(\frac{r}{y-t},r) dy\\
&=-r \left(g_{1}''(t+r)-g_{1}''(t)-r g_{1}'''(t)\right) \int_{1}^{\infty} \left(K_{0}-K_{0,sm}\right)(\frac{1}{z},r) dz-r^{2} g_{1}'''(t) \int_{1}^{\infty} \left(K_{0}-K_{0,sm}\right)(\frac{1}{z},r) dz\end{split}\end{equation}
which gives
\begin{equation} \begin{split} &|-\left(g_{1}''(t+r)-g_{1}''(t)\right) \int_{t+r}^{\infty} \left(K_{0}-K_{0,sm}\right)(\frac{r}{y-t},r) dy-\left(-r^{2} g_{1}'''(t) \int_{1}^{\infty} \left(K_{0}-K_{0,sm}\right)(\frac{1}{z},r) dz\right)|\\
&\leq \frac{C r^{3} \log(t) \lambda(t)^{3} \left(1+|\log(r)|\right)}{t^{6} }\end{split}\end{equation}
and the higher derivatives are estimated similarly. On the other hand, we have
\begin{equation}\begin{split} &-\int_{t+r}^{\infty} \int_{s}^{\infty} \left(K_{0}-K_{0,sm}\right)(\frac{r}{y-t},r) dy g_{1}'''(s) ds\\
&= - g_{1}'''(t+r) \int_{t+r}^{\infty} \int_{q}^{\infty} \left(K_{0}-K_{0,sm}\right)(\frac{r}{y-t},r) dy dq-\int_{t+r}^{\infty} \int_{s}^{\infty} \int_{q}^{\infty} \left(K_{0}-K_{0,sm}\right)(\frac{r}{y-t},r) g_{1}''''(s)dy dq  ds\end{split}\end{equation}
Therefore, for $r \sim g(t)\lambda(t)$,
\begin{equation}\begin{split} &|-\int_{t+r}^{\infty} \int_{s}^{\infty} \left(K_{0}-K_{0,sm}\right)(\frac{r}{y-t},r) dy g_{1}'''(s) ds-\left(-r^{2} g_{1}'''(t) \int_{1}^{\infty} (x-1) \left(K_{0}-K_{0,sm}\right)(\frac{1}{x},r) dx\right)|\\
&\leq \frac{C r^{3} \log(t) \lambda(t)^{3}}{t^{6} } \left(1+|\log(r)|\right)\left(\log(1+\frac{t}{r})+1\right)\end{split}\end{equation}
As with the previous terms, the higher derivatives of $-\int_{t+r}^{\infty} \int_{s}^{\infty} \int_{q}^{\infty} \left(K_{0}-K_{0,sm}\right)(\frac{r}{y-t},r) g_{1}''''(s)dy dq  ds$ can be estimated by letting $x=\frac{y-t}{r}$, then, $w=\frac{q-t}{r}$, and $z=\frac{s-t}{r}$, and differentiating under the integral. Finally, we get
\begin{equation}\begin{split} &|-\left(g_{1}''(t+r)-g_{1}''(t)\right) \int_{t+r}^{\infty} \left(K_{0}-K_{0,sm}\right)(\frac{r}{y-t},r) dy-\int_{t+r}^{\infty} \int_{s}^{\infty} \left(K_{0}-K_{0,sm}\right)(\frac{r}{y-t},r) dy g_{1}'''(s) ds\\
&-\left(-r^{2} g_{1}'''(t) \int_{1}^{\infty} x \left(K_{0}-K_{0,sm}\right)(\frac{1}{x},r) dx\right)|\\
&\leq \frac{C r^{3} \log(t) \lambda(t)^{3} (1+|\log(r)|)(\log(1+\frac{t}{r})+1)}{t^{6} }\end{split}\end{equation}
Using 
$$-r^{2} g_{1}'''(t) \int_{1}^{\infty} x \left(K_{0}-K_{0,sm}\right)(\frac{1}{x},r) dx = \frac{-r^{2} g_{1}'''(t)}{18} \left(-7+\log(512)+3 \log(r)\right)$$
we finally get the following expression which identifies the leading part of $-\int_{t}^{\infty} g_{1}''(s) K_{0}(\frac{r}{s-t},r) ds$. 
\begin{equation} \begin{split} &|-\int_{t}^{\infty} g_{1}''(s) K_{0}(\frac{r}{s-t},r) ds - \left(\frac{-r g_{1}''(t)}{16} \left(-12+\pi^{2} + 8 \log(r)\right)\right)\\
&-\begin{aligned}[t]&\left(-g_{1}''(t)r \left(\frac{\log(\frac{t}{r})\log(2)}{2}+\frac{1}{4}\left(\log^{2}(t)-\log^{2}(r)\right)\right)\right. \\
&\left.-\frac{r}{2} \int_{2t}^{\infty} \frac{g_{1}''(s)}{(s-t)} \log(2(s-t)) ds-\frac{r}{2} \int_{t}^{2t}\left(g_{1}''(s)-g_{1}''(t)\right) \frac{\log(2(s-t))}{(s-t)} ds\right)\end{aligned}\\
&-\left(\frac{-r g_{1}''(t)}{24} \left(6 \left(\log(4)-1\right)\log(r) -2 \pi^{2} + 15 + 6 \log^{2}(2)\right)\right)|\\
&\leq \frac{C r^{3} \log(t) \lambda(t)^{3} \left(1+|\log(r)|\right)\left(\log(1+\frac{t}{r})+1\right)}{t^{6} }, \quad \lambda(t)g(t) \leq r \leq 2 g(t)\lambda(t) \end{split}\end{equation}
Notice that the $r^{2}$ terms from 
$$-\left(g_{1}''(t+r)-g_{1}''(t)\right) \int_{t+r}^{\infty} \left(K_{0}-K_{0,sm}\right)(\frac{r}{y-t},r) dy-\int_{t+r}^{\infty} \int_{s}^{\infty} \left(K_{0}-K_{0,sm}\right)(\frac{r}{y-t},r) dy g_{1}'''(s) ds$$ 
and 
$$-r \int_{0}^{1} \left(g_{1}''(t+r y)-g_{1}''(t)\right) K_{0}(\frac{1}{y},r) dy + \int_{t}^{t+r} \left(g_{1}''(s)-g_{1}''(t)\right) K_{0,sm}(\frac{r}{s-t},r) ds$$ 
canceled to give the above expression.\\
\\
The same procedure is carried out for the $K_{2}$ term in \eqref{uw2sub0decomp}. The details are as follows. The analog of \eqref{contribution1logr} is
\begin{equation}\begin{split}&r g_{2}''(t) \int_{0}^{1} \left(F_{0}(a)+F_{1}(a)\log(r)+F_{2}(a)\log^{2}(r)\right)da\\
&=-\frac{r g_{2}''(t)}{16} \left(2 \log (r) \left(4 \log (r)+\pi ^2-12\right)-28 \zeta (3)+\pi ^2+28\right)\end{split}\end{equation}
where $\zeta$ denotes the Riemann zeta function. We remark that one part of this computation involves the following.
$$\text{Im}\left(\text{Li}_{2}\left(\frac{1}{2}-\frac{i a}{2 \sqrt{1-a^{2}}}\right)\right) = \int_{0}^{a} \left(\frac{x \sin^{-1}(x)}{x^{2}-1}-\frac{\log(4(1-x^{2}))}{2\sqrt{1-x^{2}}}\right) dx$$
and
$$\int_0^1 \frac{1}{2} \left(\cos ^{-1}(x)-x \sqrt{1-x^2}\right) \left(\frac{x \sin ^{-1}(x)}{x^2-1}-\frac{\log \left(4 \left(1-x^2\right)\right)}{2 \sqrt{1-x^2}}\right) \, dx = \frac{1}{32} \left(-21 \zeta (3)+\pi ^2+\log (256)\right)$$
The analog of \eqref{contribution2alogr} is 
\begin{equation}\begin{split}&g_{2}''(t)  \frac{r}{24} \log \left(\frac{r}{t}\right) \left(4 \log (t) \log (8 r t)+4 \log (r) \log (8 r)-\pi ^2+12 \log ^2(2)\right)\\
&+\int_{t}^{2t} \left(g_{2}''(s)-g_{2}''(t)\right) K_{2,0}(\frac{s-t}{r},r) ds+\int_{2t}^{\infty} g_{2}''(s)K_{2,0}(\frac{s-t}{r},r) ds\end{split}\end{equation}
where we recall that $K_{2,0}$ is defined in \eqref{k20def}. The analog of \eqref{contribution2blogr} is 
$$g_{2}''(t) r \int_{1}^{\infty} \left(K_{2}-K_{2,0}\right)(a,r) da$$
We have
$$\int_{1}^{\infty} \left(K_{2}-K_{2,0}\right)(a,r) da = c_{0}+d_{1}\log(r)+d_{2}\log^{2}(r)$$
The exact value of $c_{0}$ is not important for our purposes, but we compute $d_{1}$ and $d_{2}$ explicitly. From direct integration, 
$$d_{2} = \int_1^{\infty } \left(\sqrt{a^2-1}-a+\frac{1}{2 a}\right) \, da = \frac{1}{4}\left(1-\log(4)\right)$$
To compute $d_{1}$, we recall that $K_{2}(a,r)$ satisfies the wave equation:
$$-K_{2}+r \partial_{r}K_{2} + r^{2}\partial_{r}^{2} K_{2}+a \partial_{a}K_{2}-2 a r \partial_{ar}K_{2}+(a^{2}-1)\partial_{a}^{2}K_{2}=0, \quad a >1$$
and $K_{2} = \sum_{k=0}^{2} F_{k}(a) \log^{k}(r).$ Integrating the equation solved by $K_{2}$ in the $a$ variable, one relation we get is (for $M>1$)
$$\int_{1}^{M} F_{1}(a) da = -\int_{1}^{M} F_{2}(a) da + \left(a F_{1}(a) + \frac{(1-a^{2})}{2} F_{0}'(a) + \frac{a F_{0}(a)}{2}\right)\Bigr|_{a=1}^{M}$$
Recalling the definitions of $F_{k,0}$ in \eqref{fk0def}, we have 
\begin{equation}\label{delicateint}d_{1}=\int_{1}^{\infty}\left(F_{1}(a)-F_{1,0}(a)\right) da = \lim_{M \rightarrow \infty} \int_{1}^{M}\left(F_{1}(a)-F_{1,0}(a)\right) da = \frac{1}{12} \left(2 \pi ^2-15-6 \log ^2(2)\right)\end{equation}
Therefore, 
$$g_{2}''(t) r \int_{1}^{\infty} \left(K_{2}-K_{2,0}\right)(a,r) da = g_{2}''(t) r \left(c_{0}+\frac{1}{12} \left(2 \pi ^2-15-6 \log ^2(2)\right) \log(r) +\frac{1}{4}\left(1-\log(4)\right) \log^{2}(r)\right)$$ 
For estimating the subleading terms, we use the same procedure used for the $K_{0}$ integral term previously. In particular, we have
\begin{equation}\label{k2subleading} \begin{split} &|\int_{t}^{\infty} g_{2}''(s) K_{2}(\frac{s-t}{r},r) ds - r g_{2}''(t) \int_{0}^{1} K_{2}(y,r) dy - r g_{2}''(t) \int_{1}^{\infty} (K_{2}-K_{2,0})(a,r) da\\
&-g_{2}''(t) \int_{t+r}^{2t} K_{2,0}(\frac{s-t}{r},r) ds - \int_{t}^{2t}\left(g_{2}''(s)-g_{2}''(t)\right) K_{2,0}(\frac{s-t}{r},r) ds-\int_{2t}^{\infty} g_{2}''(s) K_{2,0}(\frac{s-t}{r},r) ds\\
&-r^{2} g_{2}'''(t) \int_{0}^{1} y K_{2}(y,r) dy + r^{2} g_{2}'''(t) \int_{0}^{1} y K_{2,0}(y,r) dy - r^{2} g_{2}'''(t) \int_{1}^{\infty} a \left(K_{2}-K_{2,0}\right)(a,r) da|\\
&\leq \frac{C r^{3} \lambda(t)^{3} \left(1+\log^{2}(r)\right)\left(1+|\log(\frac{r}{t})|^{3}\right)}{t^{6} }, \quad r \leq t\end{split}\end{equation}
The higher derivatives are estimated similarly, with the same procedure used for the $g_{1}''(t)$ terms just studied. Finally, we will show that the last three terms on the left-hand side of the above inequality exactly cancel, just as was the case for the analogous terms arising from the $K_{0}$ integral previously. By direct computation, we have
$$\int_{0}^{1} y K_{2,0}(y,r) dy =-\frac{\log ^2(r)}{2}-\log (2) \log (r)+\log (2 r)+\frac{\pi ^2}{24}-1-\frac{1}{2} \log ^2(2)$$
$$\int_{0}^{1} y K_{2}(y,r) dy = -\frac{1}{3} \log ^2(r)+\frac{2 \log (r)}{9}+\frac{1}{18} (16 \log (2)-4)+\frac{1}{108} \left(-9 \pi ^2+88-96 \log (2)\right)$$
Finally, using a similar procedure as in \eqref{delicateint}, we get
\begin{equation} \begin{split} \int_{1}^{\infty} \left(F_{0}(a)-F_{0,0}(a)\right) a da &= \lim_{M\rightarrow \infty}\left(-\frac{2}{3} \int_{1}^{M} a F_{2}(a) da -\frac{a^{2} F_{0}(a)}{3} \Bigr|_{a=1}^{M} + \frac{2}{3} \int_{1}^{M} a^{2} F_{1}'(a) da\right.\\
&\left.-\frac{1}{3} \left((a^{3}-a) F_{0}'(a) -(3 a^{2}-1) F_{0}(a)\right)\Bigr|_{a=1}^{M} - \int_{1}^{M} F_{0,0}(a) a da\right)\\
&=\frac{\pi ^2}{8}-\frac{43}{27}-\frac{1}{2} \log ^2(2)+\log (2)\end{split}\end{equation}
A straightforward computation then gives
$$\int_{1}^{\infty} a \left(K_{2}-K_{2,0}\right)(a,r) da = -\frac{\log ^2(r)}{6}+\left(\frac{7}{9}-\log (2)\right) \log (r)+\frac{\pi ^2}{8}-\frac{43}{27}-\frac{1}{2} \log ^2(2)+\log (2)$$
In total, we then note that
\begin{equation}\begin{split}&-\frac{\log ^2(r)}{2}-\log (2) \log (r)+\log (2 r)+\frac{\pi ^2}{24}-1-\frac{1}{2} \log ^2(2)\\
&-\left(-\frac{1}{3} \log ^2(r)+\frac{2 \log (r)}{9}+\frac{1}{18} (16 \log (2)-4)+\frac{1}{108} \left(-9 \pi ^2+88-96 \log (2)\right)\right)\\
&-\left(-\frac{\log ^2(r)}{6}+\left(\frac{7}{9}-\log (2)\right) \log (r)+\frac{\pi ^2}{8}-\frac{43}{27}-\frac{1}{2} \log ^2(2)+\log (2)\right)\\
&=0\end{split}\end{equation}
which verifies that the last three terms on the left-hand side of \eqref{k2subleading} exactly cancel. Combining our computations and estimates above finishes the proof of the lemma. \end{proof}
\noindent Next, we note that
$$\partial_{t}^{2}v_{ex,cont}(t,r) = \frac{1}{r}\left(h_{0}''(t)+\log(r)h_{1}''(t)\right)$$
where
$$h_{0}(t) = -2 \log(\lambda(t))f_{1}(t)+\lambda(t)^{2}f_{2}(t), \quad h_{1}(t) = 2 f_{1}(t)$$
and
$$f_{1}(t) = \frac{-1}{2}\left(2\lambda(t)\lambda'(t)^{2}+\lambda(t)^{2}\lambda''(t)\right), \quad f_{2}(t) = \frac{-\lambda''(t)}{2}$$
Using the same procedure as for $u_{w,2,sub,0}$ (in fact, the same computations, except with $g_{j}$ replaced with $h_{j}$, for $j=0,1$) we get the following lemma.
\begin{lemma}\label{vexsub0leading}[Leading part of $v_{ex,sub,0}$] Let 
\begin{equation}\begin{split}&v_{ex,sub,0,cont}(t,r)\\
&=\frac{-r}{2} h_{0}''(t) -\frac{r}{2}h_{0}''(t) \log(\frac{t}{r}) - \frac{r}{2} \int_{t}^{2t} \frac{\left(h_{0}''(s)-h_{0}''(t)\right)}{(s-t)} ds \\
&- \frac{r}{2}\int_{2t}^{\infty} \frac{h_{0}''(s) ds}{s-t}-r h_{0}''(t)\frac{\left(\log(4)-1\right)}{4}-\frac{r}{8}f_{1}''(t)\left(-12+\pi^{2}+8 \log(r)\right)\\
&-2 f_{1}''(t)\left(\frac{r \log(2) \log(\frac{t}{r})}{2}+\frac{r}{4}\left(\log^{2}(t)-\log^{2}(r)\right)\right)\\
&-r \int_{t}^{2t} \frac{\left(f_{1}''(s)-f_{1}''(t)\right)}{(s-t)} \log(2(s-t)) ds\\
&-r \int_{2t}^{\infty} f_{1}''(s) \left(\frac{ \log(2(s-t))}{(s-t)}\right) ds-\frac{r f_{1}''(t)}{12}\left(6 \left(\log(4)-1\right)\log(r)-2 \pi^{2}+15+6 \log^{2}(2)\right) \end{split}\end{equation}
Then, for $0 \leq j+k \leq 2$, 
\begin{equation}|r^{k}t^{j}\partial_{r}^{k}\partial_{t}^{j}\left(v_{ex,sub,0}-v_{ex,sub,0,cont}\right)|(t,r) \leq \frac{C r^{3} \lambda(t)^{3} \log^{3}(t)}{t^{6} }, \quad g(t)\lambda(t) \leq r \leq 2g(t)\lambda(t)\end{equation}
\end{lemma}
Next, we consider 
\begin{equation}\label{uw2minusuw2ell0expr}\begin{split}&u_{w,2,ell}(t,r)-u_{w,2,ell,0}(t,r)\\
&=-\frac{1}{2}\left(\frac{1}{r}\int_{0}^{r} s^{2}\left(RHS_{2}(t,s)-RHS_{2,0}(t,s)\right) ds + r \int_{r}^{\infty} \left(RHS_{2}(t,s)-RHS_{2,0}(t,s)\right) ds\right)\\
&=\frac{-1}{2}\left(\frac{1}{r}\int_{0}^{r} s^{2}\left(RHS_{2}(t,s)-RHS_{2,0}(t,s)\right)ds + r\int_{0}^{\infty} \left(RHS_{2}(t,s)-RHS_{2,0}(t,s)\right)ds\right.\\
&\left. - r \int_{0}^{r} \left(RHS_{2}(t,s)-RHS_{2,0}(t,s)\right) ds\right)\end{split}\end{equation}
We proceed to study each integral term on the right-hand side. The integral over $(0,\infty)$ requires the longest argument:
\begin{lemma} We have 
\begin{equation}\label{rhs2minusrhs20infint} \begin{split}&\int_{0}^{\infty}\left(RHS_{2}(t,s)-RHS_{2,0}(t,s)\right) ds\\
&= \lambda(t) \int_{t}^{\infty} W_{3}(\frac{s-t}{\lambda(t)}) \lambda''''(s) ds+ \int_{0}^{\infty} d\xi \left(\frac{-8}{\xi \lambda(t)^{2}}+4 \xi K_{2}(\xi \lambda(t)) + 2\xi\right) \sin(t\xi) \widehat{v_{2,0}}(\xi)\end{split}\end{equation} 
where 
\begin{equation}\begin{split} W_{3}(x) &= \frac{2}{3}  \left(x \log \left(64 x^6\right)+4 x^3 \log (2 x)+2 \left(x^2+1\right)^{3/2} \log \left(2 x \left(x-\sqrt{x^2+1}\right)+1\right)+x\right)\\
&=O\left(\frac{\log(x)}{x}\right), \quad x \rightarrow \infty\end{split}\end{equation}
\end{lemma}
\begin{proof}
We first note
\begin{equation} r\int_{0}^{\infty} \left(RHS_{2}(t,s)-RHS_{2,0}(t,s)\right)ds = r\int_{0}^{\infty} \left(\left(\frac{\cos(2	Q_{1}(\frac{s}{\lambda(t)}))-1}{s^{2}}\right) \left(v_{2}+w_{1}\right) - RHS_{2,0}(t,s)\right)ds\end{equation}
Using the representation formula \eqref{v2} for $v_{2}$, and Fubini's theorem, we have
\begin{equation}\label{intwithv2} \int_{0}^{\infty}\left(\frac{\cos(2Q_{1}(\frac{s}{\lambda(t)}))-1}{s^{2}}\right) v_{2}(t,s) ds = F(t) + \int_{0}^{\infty} d\xi \left(\frac{-8}{\xi \lambda(t)^{2}}+4 \xi K_{2}(\xi \lambda(t)) + 2\xi\right) \sin(t\xi) \widehat{v_{2,0}}(\xi)\end{equation}
where we used
\begin{equation}\label{lowfreqsmall}\int_{0}^{\infty} dr \left(\frac{\cos(2Q_{1}(\frac{r}{\lambda(t)}))-1}{r^{2}}\right) J_{1}(r\xi) = -\frac{8}{\xi\lambda(t)^{2}} + 4 \xi K_{2}(\xi \lambda(t)) = -2\xi + O\left(\xi^{3}\log(\xi)\right), \quad \xi \rightarrow 0\end{equation}
and we recall that $F$ is defined in \eqref{Fdef}. Notice that the first term on the right-hand side of \eqref{intwithv2} is the leading part, given the smallness at low frequencies of the $O$ term in  \eqref{lowfreqsmall}. Next, using
$$w_{1}(t,r) =\frac{2}{r} \int_{t}^{t+r} \lambda''(s)(s-t)ds + 2 r \int_{1}^{\infty} \lambda''(t+r y) \left(y-\sqrt{y^{2}-1}\right) dy$$
we get
\begin{equation}\label{rhs21infint}\begin{split} &\int_{0}^{\infty} \left(\frac{\cos(2Q_{1}(\frac{r}{\lambda(t)}))-1}{r^{2}}\right) w_{1}(t,r) dr\\
&=\int_{0}^{\infty} dw \lambda''(t+w) w \frac{8}{\lambda(t)^{2}} \left(\frac{\lambda(t)^{2}}{\lambda(t)^{2}+w^{2}}-\log\left(1+\frac{\lambda(t)^{2}}{w^{2}}\right)\right)\\
&+\int_{0}^{\infty} dw \lambda''(t+w)\frac{4}{\lambda(t)^{2}} \begin{aligned}[t]&\left(\frac{-2 \lambda(t)^{2} w}{\lambda(t)^{2}+w^{2}} + w \log(16) + 2 w \log(1+\frac{w^{2}}{\lambda(t)^{2}}) \right.\\
&\left.+ \frac{\left(\lambda(t)^{2}+2w^{2}\right) \log(1+\frac{2 w (w-\sqrt{\lambda(t)^{2}+w^{2}})}{\lambda(t)^{2}})}{\sqrt{\lambda(t)^{2}+w^{2}}}\right)\end{aligned}\end{split}\end{equation}
Using the definitions of $v_{2,sub}$, $w_{1,sub}$, and the first order matching (see, e.g. \eqref{vellfirstorder}), we get
$$RHS_{2,0}(t,r) =\frac{r}{\lambda(t)} \left(\frac{\cos(2Q_{1}(\frac{r}{\lambda(t)}))-1}{r^{2}}\right) \left(\frac{1}{2}\lambda'(t)^{2}+\lambda(t)\lambda''(t) -\lambda(t)\lambda''(t) \log(\frac{r}{\lambda(t)})\right)$$
which gives
\begin{equation}\label{rhs20infint} -\int_{0}^{\infty} RHS_{2,0}(t,s) ds = \frac{4 }{\lambda(t)} \left(\frac{\lambda'(t)^{2}}{2}+\lambda(t)\lambda''(t)\right)\end{equation}
Given that $RHS_{2,0}$ was the leading part of $RHS_{2}$ in the matching region, the integrals \eqref{intwithv2}, \eqref{rhs20infint}, and \eqref{rhs21infint} have cancellation when added together. We show this in detail now. Recalling the definition of $F(t)$, we have (for all $t \geq T_{0}$)
\begin{equation}\label{fdef2}\begin{split}& F(t)\\
 &= 4 \left(\left(\log(2)-\frac{1}{2}\right)\lambda''(t)+\int_{t}^{2t} \left(\frac{\lambda''(s)-\lambda''(t)}{s-t}\right) ds + \lambda''(t) \log(\frac{t}{\lambda(t)})+\int_{2t}^{\infty} \frac{\lambda''(s) ds}{s-t}-\frac{(\lambda'(t))^{2}}{2\lambda(t)}\right)\end{split}\end{equation}
Therefore, the $\lambda'(t)^{2}$ terms from $F(t)$ and \eqref{rhs20infint} cancel. Next, we determine the leading terms of the first term on the right-hand side of \eqref{rhs21infint}. The integral is 
$$\frac{8}{\lambda(t)^{2}}\int_{0}^{\infty} \lambda''(t+w) w \left(\frac{\lambda(t)^{2}}{\lambda(t)^{2}+w^{2}}-\log(1+\frac{\lambda(t)^{2}}{w^{2}})\right) dw$$
Given the decay (for large $w$) of the part of the integrand multiplying $\lambda''(t+w)$, we integrate by parts in $w$, differentiating the symbol $\lambda''(t+w)$, and integrating (backwards from infinity) the rest of the integrand. This gives
\begin{equation}\label{intbypartsasymp}\begin{split} &\frac{8}{\lambda(t)^{2}}\int_{0}^{\infty} \lambda''(t+w) w \left(\frac{\lambda(t)^{2}}{\lambda(t)^{2}+w^{2}}-\log(1+\frac{\lambda(t)^{2}}{w^{2}})\right) dw \\
&= -4 \lambda''(t)-\frac{4 \pi}{3}\lambda(t)\lambda'''(t)+\frac{4}{3\lambda(t)^{2}} \int_{0}^{\infty} \left(w \lambda(t)^{2}-2 \tan^{-1}\left(\frac{\lambda(t)}{w}\right)\lambda(t)^{3} - w^{3} \log(1+\frac{\lambda(t)^{2}}{w^{2}})\right) \lambda''''(t+w) dw\end{split}\end{equation}
 Notice the cancellation between the $\lambda''(t)$ terms in the expression above and \eqref{rhs20infint}. Next, we treat the second term on the right-hand side of \eqref{rhs21infint}. We start with
\begin{equation}\begin{split}&\frac{4}{\lambda(t)^{2}} \int_{0}^{t} dw \lambda''(t+w) G_{2}(w,\lambda(t))\end{split}\end{equation}
where
\begin{equation}\begin{split}G_{2}(w,\lambda(t)) &= \left(\frac{-2 \lambda(t)^{2}w}{\lambda(t)^{2}+w^{2}}+w \log(16) + 2 w \log(1+\frac{w^{2}}{\lambda(t)^{2}})\right.\\
&\left.+\frac{\left(\lambda(t)^{2}+2w^{2}\right)}{\sqrt{\lambda(t)^{2}+w^{2}}}\log(1+\frac{2 w (w-\sqrt{\lambda(t)^{2}+w^{2}})}{\lambda(t)^{2}})\right)=\lambda(t) G_{3}(\frac{w}{\lambda(t)})\end{split}\end{equation}
Since there will be some cancellation between the integral under consideration and the second term in \eqref{fdef2}, we start by writing $\lambda''(t+w)=\lambda''(t)+\lambda''(t+w)-\lambda''(t)$, we have
\begin{equation}\label{g2intstep}\begin{split}&\frac{4}{\lambda(t)^{2}} \int_{0}^{t} dw \lambda''(t+w) G_{2}(w,\lambda(t))\\
&=-2 \lambda''(t) \left(-1+2 \log(2)-2 \log(\frac{\lambda(t)}{t})\right) + \frac{4 \lambda''(t)}{\lambda(t)^{2}} \lambda(t)^{2} G_{4}(\frac{t}{\lambda(t)}) \\
&+\frac{4}{\lambda(t)^{2}} \int_{0}^{t} dw \left(\lambda''(t+w)-\lambda''(t)\right) G_{2}(w,\lambda(t))\end{split}\end{equation}
where 
\begin{equation}\label{g4def}\begin{split}G_{4}(x) &= \frac{1}{2} \left(2 x^2 \log \left(4 \left(x^2+1\right)\right)+2x^{2} \sqrt{\frac{1}{x^2}+1}  \log \left(1-2 \left(\sqrt{\frac{1}{x^2}+1}-1\right) x^2\right)+2 \log (2 x)-1\right)\\
&=O\left(\frac{\log(x)}{x^{2}}\right), \quad x \rightarrow \infty\end{split}\end{equation}
Note that the first term on the right-hand side of \eqref{g2intstep} cancels with all of the $\lambda''(t)$ terms outside the integral operators in \eqref{fdef2}. Next, we note that 
\begin{equation}\label{g3asymp}G_{3}(x) = \frac{-1}{x} + O\left(\frac{\log(x)}{x^{3}}\right), \quad x \rightarrow \infty\end{equation}
Recalling that $G_{2}(w,\lambda(t))=\lambda(t)G_{3}(\frac{w}{\lambda(t)})$, we have
\begin{equation}\begin{split}&\frac{4}{\lambda(t)^{2}} \int_{0}^{t} dw \left(\lambda''(t+w)-\lambda''(t)\right) G_{2}(w,\lambda(t))\\
&=\frac{-4}{\lambda(t)^{2}} \int_{0}^{t} dw \left(\lambda''(t+w)-\lambda''(t)\right) \frac{\lambda(t)^{2}}{w} +\frac{4}{\lambda(t)^{2}}\int_{0}^{t} dw \left(\lambda''(t+w)-\lambda''(t)\right) \left(G_{2}(w,\lambda(t))+\frac{\lambda(t)^{2}}{w}\right)\end{split}\end{equation}
The first term on the right-hand side of the above expression cancels with the second term on the right-hand side of \eqref{fdef2}, and $G_{2}(w,\lambda(t))+\frac{\lambda(t)^{2}}{w}$ decays much more quickly for large $w$ than $G_{2}(w,\lambda(t))$ does, given \eqref{g3asymp}. Using the same procedure as in \eqref{intbypartsasymp}, we get
\begin{equation}\label{g2minusprincstep}\begin{split}&\frac{4}{\lambda(t)^{2}}\int_{0}^{t} dw \left(\lambda''(t+w)-\lambda''(t)\right) \left(G_{2}(w,\lambda(t))+\frac{\lambda(t)^{2}}{w}\right)\\
&=\frac{4}{\lambda(t)^{2}}\begin{aligned}[t]&\left(\left(\lambda''(2t)-\lambda''(t)\right)\lambda(t)^2 G_{4}(\frac{t}{\lambda(t)})\right.\\
&\left.+\lambda(t)^{3}\left( - \lambda'''(2t) W_{2}(\frac{t}{\lambda(t)}) +\frac{\lambda'''(t)}{3}\pi  + \int_{t}^{2t} W_{2}(\frac{s-t}{\lambda(t)}) \lambda''''(s) ds\right)\right)\end{aligned}\end{split}\end{equation}
where we recall that $G_{4}$ was defined in \eqref{g4def}, and
\begin{equation}\begin{split} W_{2}(x) = \frac{1}{6} &\left(x^3 \log (16)+2 \left(x^2+1\right)^{3/2} \log \left(2 x^2-2 x \sqrt{x^2+1} +1\right)+2 x^3 \log \left(x^2+1\right)\right.\\
&\left.-x+x \log (64)+6 x \log (x)+4 \tan ^{-1}(\frac{1}{x})\right)\end{split}\end{equation}
Note that the second term on the right-hand side of our previous computation \eqref{g2intstep} cancels with the term $\frac{4}{\lambda(t)^{2}}\left(-\lambda''(t)\lambda(t)^{2}G_{4}(\frac{t}{\lambda(t)})\right)$ from \eqref{g2minusprincstep}.\\
\\
Lastly, we consider
$$\frac{4}{\lambda(t)^{2}}\int_{t}^{\infty} dw \lambda''(t+w) G_{2}(w,\lambda(t)) = -4 \int_{2t}^{\infty} ds \frac{\lambda''(s)}{(s-t)} + \frac{4}{\lambda(t)^{2}} \int_{t}^{\infty} dw \lambda''(t+w)\left(G_{2}(w,\lambda(t))+\frac{\lambda(t)^{2}}{w}\right)$$
The first term on the right-hand side of the above expression cancels with the fourth term on the right-hand side of \eqref{fdef2}. Then, we treat the following integral with the same procedure used in \eqref{intbypartsasymp}
\begin{equation}\begin{split} &\frac{4}{\lambda(t)^{2}} \int_{t}^{\infty} dw \lambda''(t+w) \left(G_{2}(w,\lambda(t))+\frac{\lambda(t)^{2}}{w}\right)\\
&=\frac{4}{\lambda(t)^{2}} \left(-\lambda''(2t)\lambda(t)^{2} G_{4}(\frac{t}{\lambda(t)}) + \lambda'''(2t) \lambda(t)^{3} W_{2}(\frac{t}{\lambda(t)}) +\lambda(t)^{3} \int_{t}^{\infty} dw \lambda''''(t+w) W_{2}(\frac{w}{\lambda(t)}) dw\right)\end{split}\end{equation} 
After combining all of our computations, we end up with \eqref{rhs2minusrhs20infint}, completing the proof of the lemma.
 \end{proof}
Finally, to compute the principal part (in the matching region $r \sim \lambda(t)g(t)$) of the other two integrals in \eqref{uw2minusuw2ell0expr}, we start with
$$RHS_{2}(t,r) - RHS_{2,0}(t,r) = \left(\frac{\cos(2Q_{1}(\frac{r}{\lambda(t)}))-1}{r^{2}}\right) \left(v_{2,sub}(t,r) + w_{1,sub}(t,r)\right)$$
Next, we replace $v_{2,sub}$ and $w_{1,sub}$ by their principal parts, which are $v_{2,cubic,main}$ and $w_{1,cubic,main}$, respectively. Then, we use part 1 of second order matching, which says that
$$v_{2,cubic,main}(t,r) + w_{1,cubic,main}(t,r) = v_{ell,2,0,main}(t,\frac{r}{\lambda(t)})$$ 
Whence, we get
\begin{equation}\begin{split}&\text{leading part of } \frac{1}{r}\int_{0}^{r} s^{2}\left(\frac{\cos(2Q_{1}(\frac{s}{\lambda(t)}))-1}{s^{2}}\right) v_{ell,2,0,main}(t,\frac{s}{\lambda(t)}) ds\\
&:= (u_{w,2,ell}-u_{w,2,ell,0})_{princ,1}(t,r)\\
&=-\left(2 r \left(2 j_{1}(t) \lambda(t)^2+2 j_{2}(t) \lambda(t)^2 \log (r)-j_{2}(t) \lambda(t)^2\right)\right)\end{split}\end{equation}
\begin{equation}\begin{split}&\text{leading part of }-r \int_{0}^{r} \left(\frac{\cos(2Q_{1}(\frac{s}{\lambda(t)}))-1}{s^{2}}\right) v_{ell,2,0,main}(t,\frac{s}{\lambda(t)})ds\\
&:=(u_{w,2,ell}-u_{w,2,ell,0})_{princ,2}(t,r)\\
&=\frac{1}{3} r \lambda(t)^2 \begin{aligned}[t]&\left(-24 j_{1}(t) \log (\lambda(t))+24 j_{1}(t) \log (r)-12 j_{1}(t)-12 j_{2}(t) \log ^2(\lambda(t))\right.\\
&\left.-12 j_{2}(t) \log (\lambda(t))+12 j_{2}(t) \log ^2(r)-\pi ^2 j_{2}(t)\right)\end{aligned}\end{split}\end{equation}
where
$$j_{1}(t) = \frac{3}{32} \lambda''''(t)+\frac{1}{8} \partial_{t}^{2}\left(\lambda''(t)+\lambda''(t) \log (\lambda(t))+\frac{\lambda'(t)^2}{2 \lambda(t)}\right)$$
$$j_{2}(t) = \frac{-1}{8}\lambda''''(t)$$
and we used
$$v_{ell,2,0,main}(t,\frac{r}{\lambda(t)}) = r^{3} j_{1}(t)+r^{3}\log(r) j_{2}(t)$$
We now estimate the difference between $u_{w,2,ell}-u_{w,2,ell,0}$ and its principal part in the matching region, which will help us study other contributions to the third order matching.
\begin{lemma} \label{ellminusell0}
Let \begin{equation}\begin{split} &(u_{w,2,ell}(t,r)-u_{w,2,ell,0}(t,r))_{princ}\\
&:= -\frac{1}{2}\left((u_{w,2,ell}-u_{w,2,ell,0})_{princ,1}(t,r)+(u_{w,2,ell}-u_{w,2,ell,0})_{princ,2}(t,r)\right)\\
&-\frac{r}{2}\lambda(t)\int_{t}^{\infty} W_{3}(\frac{s-t}{\lambda(t)}) \lambda''''(s) ds - \frac{r}{2} \int_{0}^{\infty}  \left(\frac{-8}{\xi \lambda(t)^{2}}+4 \xi K_{2}(\xi \lambda(t))+2 \xi\right) \sin(t\xi) \widehat{v_{2,0}}(\xi)d\xi\end{split}\end{equation}
Then, for $0 \leq k \leq 2$, and $0 \leq j \leq 8$  we have
\begin{equation}\begin{split}&|\partial_{t}^{j}\partial_{r}^{k}\left(u_{w,2,ell}(t,r)-u_{w,2,ell,0}(t,r)-(u_{w,2,ell}(t,r)-u_{w,2,ell,0}(t,r))_{princ}\right)| \\
&\leq \frac{C}{r^{1+k}} \frac{\lambda(t)^{3} \log^{3}(t)}{t^{4+j}}\left(\frac{r^{4}}{t^{2}}+\lambda(t)^{2}\right), \quad \lambda(t) \leq r \leq \frac{t}{2}\end{split}\end{equation}
Moreover, for $0 \leq m+k \leq 1$, $m,k\geq 0$,
\begin{equation}\label{uw2ellminusell0forthirdorder}|\partial_{t}^{2+k}\partial_{r}^{m}\left(u_{w,2,ell}(t,r)-u_{w,2,ell,0}(t,r)\right)| \leq \frac{C \lambda(t)^{2}\log(t) \sup_{y \in [100,t]}\left(\lambda(y)\log(y)\right)}{t^{9/2}\langle t-r\rangle^{1/2+k+m}}, \quad \frac{t}{2} \leq r \leq 2t\end{equation}
In addition, for $s \geq t$, and $0 \leq k \leq 2$,
\begin{equation}\label{uw2ellminusell0forthirdorder2} |\partial_{s}^{4+k} \left(u_{w,2,ell}-u_{w,2,ell,0}\right)(s,y)| \leq C \lambda(s)^{2} \log^{3}(s) \sup_{x \in [100,s]}\left(\lambda(x) \log(x)\right) \begin{cases} \frac{y}{s^{8+k}}, \quad y \leq \frac{s}{2}\\
\frac{1}{s^{9/2}t^{5/2+k}}, \quad \frac{s}{2} \leq y \leq s-t+2g(t)\lambda(t)\end{cases}\end{equation}
 \end{lemma}
 \begin{proof}
 We first note that
 \begin{equation}\label{uw2ellminusell0minusprinc}\begin{split}&u_{w,2,ell}(t,r)-u_{w,2,ell,0}(t,r)-(u_{w,2,ell}(t,r)-u_{w,2,ell,0}(t,r))_{princ}\\
 &=-\frac{1}{2r}\left(\int_{0}^{r} s^{2} \left(\frac{\cos(2Q_{1}(\frac{s}{\lambda(t)}))-1}{s^{2}}\right) \left(w_{1,sub}(t,s)+v_{2,sub}(t,s)-\left(v_{2,cubic,main}(t,s)+w_{1,cubic,main}(t,s)\right)\right) ds\right.\\
 &\left.+\int_{0}^{r} s^{2} \left(\frac{\cos(2Q_{1}(\frac{s}{\lambda(t)}))-1}{s^{2}}\right) v_{ell,2,0,main}(t,\frac{s}{\lambda(t)})ds-r (u_{w,2,ell}-u_{w,2,ell,0})_{princ,1}(t,r)\right)\\
 &+\frac{r}{2}\left(\int_{0}^{r} \left(\frac{\cos(2Q_{1}(\frac{s}{\lambda(t)}))-1}{s^{2}}\right) \left(w_{1,sub}(t,s)+v_{2,sub}(t,s)-\left(v_{2,cubic,main}(t,s)+w_{1,cubic,main}(t,s)\right)\right) ds\right.\\
 &+\left. \int_{0}^{r} \left(\frac{\cos(2Q_{1}(\frac{s}{\lambda(t)}))-1}{s^{2}}\right) v_{ell,2,0,main}(t,\frac{s}{\lambda(t)}) ds +\frac{1}{r}(u_{w,2,ell}-u_{w,2,ell,0})_{princ,2}(t,r) \right)\end{split}\end{equation}
 where we used 
 $$v_{2,cubic,main}(t,s)+w_{1,cubic,main}(t,s) = v_{ell,2,0,main}(t,\frac{s}{\lambda(t)})$$
 which follows from the second order matching. Next, we use Lemmas \ref{w1strlemma} and \ref{v2strlemma}, noting that $ s \leq r \leq \frac{t}{2}$ in the integrals which compose the first and third terms of \eqref{uw2ellminusell0minusprinc}.
  
The estimate \eqref{uw2ellminusell0forthirdorder} follows directly from Lemma \ref{uw2minuselllemma} and \eqref{uw2ell0exp}, and \eqref{uw2ellminusell0forthirdorder2} is proven with the same procedure as in Lemma \ref{uw2minuselllemma}.  \end{proof}
Next, we need to understand  $w_{3}(t,r):=u_{w,2}(t,r)-u_{w,2,ell}(t,r)-u_{w,2,sub,0}$, since $w_{3}$ will turn out to contribute terms which are logarithmically smaller than the largest contributions of $u_{w,2,sub,0}(t,r)$ in the matching region, but not quite perturbative. Recalling the equations that these functions solve, \eqref{uw2minusuw2elleqn} and \eqref{uw2sub0eqn}, we see that $w_{3}$ solves the following equation with $0$ Cauchy data at infinity.
$$-\partial_{t}^{2}w_{3}+\partial_{r}^{2}w_{3}+\frac{1}{r}\partial_{r}w_{3}-\frac{w_{3}}{r^{2}} = \partial_{t}^{2}\left(u_{w,2,ell}-u_{w,2,ell,0,cont}\right)$$
Therefore,
$$w_{3}(t,r) = \int_{t}^{\infty}w_{3,s}(t,r) ds, \quad w_{3,s}\text{ solves } \begin{cases} -\partial_{t}^{2}w_{3,s}+\partial_{r}^{2}w_{3,s}+\frac{1}{r}\partial_{r}w_{3,s}-\frac{w_{3,s}}{r^{2}} =0\\
w_{3,s}(s,r)=0\\
\partial_{1}w_{3,s}(s,r) = \partial_{1}^{2}(u_{w,2,ell}-u_{w,2,ell,0,cont})(s,r)\end{cases}$$
By the finite speed of propagation, in the region $r \leq t$, we have 
$$w_{3,s}(t,r) = v_{3,s}(t,r), \quad v_{3,s} \text{ solves } \begin{cases} -\partial_{t}^{2}v_{3,s}+\partial_{r}^{2}v_{3,s}+\frac{1}{r}\partial_{r}v_{3,s}-\frac{v_{3,s}}{r^{2}} =0\\
v_{3,s}(s,r)=0\\
\partial_{1}v_{3,s}(s,r) = \psi_{\leq 1}(r-s) \partial_{1}^{2}(u_{w,2,ell}-u_{w,2,ell,0,cont})(s,r)\end{cases} $$
where
\begin{equation}\label{psileq1def}\psi_{\leq 1}\in C^{\infty}(\mathbb{R}), \text{ and }\psi_{\leq 1}(x) = \begin{cases} 1, \quad x \leq 0\\
0, \quad x \geq 1\end{cases}\end{equation}
Since we will only be interested in estimating $w_{3}(t,r)$ in the matching region $g(t)\lambda(t) \leq r \leq 2 g(t)\lambda(t)$, it will suffice to estimate $v_{3}$ given below
$$v_{3}(t,r) :=\int_{t}^{\infty} v_{3,s}(t,r) ds$$ 
I.e., $v_{3}$  solves  \begin{equation}\label{v3def}-\partial_{t}^{2}v_{3}+\partial_{r}^{2}v_{3}+\frac{1}{r}\partial_{r}v_{3}-\frac{v_{3}}{r^{2}} = \psi_{\leq 1}(r-t) \partial_{t}^{2}(u_{w,2,ell}-u_{w,2,ell,0,cont})(t,r):=RHS_{3}(t,r)\end{equation}
with $0$ Cauchy data at infinity.
Let 
$$ellsoln=\frac{-1}{2}\left(\frac{1}{r}\int_{0}^{r} s^{2} RHS_{3}(t,s)ds+r \int_{r}^{\infty} RHS_{3}(t,s) ds\right)$$
and 
$$Q(t) = \frac{-1}{2} \int_{0}^{\infty} \psi_{\leq 1}(s-t) \partial_{t}^{2}\left(u_{w,2,ell}-u_{w,2,ell,0}\right)(t,s) ds$$
Note that $ellsoln$ satisfies
$$\partial_{r}^{2}ellsoln+\frac{1}{r}\partial_{r}ellsoln-\frac{ellsoln}{r^{2}} = RHS_{3}(t,r)$$
\begin{lemma}\label{ellsolnlemma} We have the following estimates. For $0 \leq j,k \leq 1$,
$$|r^{k}\partial_{r}^{k}\partial_{t}^{j}\left(ellsoln-rQ(t)\right)| \leq \begin{cases} \frac{C r \lambda(t)^{3} \left(|\log^{3}(r)|+\log^{3}(t)\right)}{t^{4+j}}, \quad r \leq \lambda(t)\\
\frac{C r^{3} \lambda(t)^{2} \log^{3}(t) \sup_{x \in [100,t]} \left(\lambda(x) \log(x)\right)}{t^{6+j}}+\frac{C}{r} \frac{\lambda(t)^{5} \log^{2}(t)}{t^{4+j}}, \quad \lambda(t) < r < \frac{t}{2}\\
\frac{C \lambda(t)^{2} \log^{3}(t) \sup_{x \in [100,t]} \left(\lambda(x) \log(x)\right)}{t^{3+\frac{j}{2}}}, \quad \frac{t}{2} \leq r \leq 2t\end{cases}$$
\begin{equation}\begin{split}|\partial_{r}^{2}\left(ellsoln-r Q(t)\right)| &\leq  \frac{C \lambda(t)^{5} \log^{2}(t)}{r^{3} t^{4} } + \frac{C r \lambda(t)^{2} \sup_{x \in [100,t]}\left(\lambda(x) \log(x)\right) \log^{3}(t)}{t^{6}} , \quad g(t)\lambda(t)\leq r \leq 2g(t)\lambda(t)\end{split}\end{equation}
\end{lemma}
\begin{proof}
We have
\begin{equation}\begin{split} &ellsoln-r Q(t) \\
&=\frac{-1}{2r} \int_{0}^{r} s^{2}\psi_{\leq 1}(s-t) \partial_{t}^{2}\left(u_{w,2,ell}-u_{w,2,ell,0,cont}\right)(t,s) ds + \frac{r}{2}\int_{0}^{r} \psi_{\leq 1}(s-t) \partial_{t}^{2}\left(u_{w,2,ell}-u_{w,2,ell,0}\right)(t,s) ds\\
&-\frac{r}{2}\int_{r}^{\infty} \psi_{\leq 1}(s-t) \partial_{t}^{2}\left(u_{w,2,ell,0}-u_{w,2,ell,0,cont}\right)(t,s) ds\end{split}\end{equation}
We then estimate directly, using Lemmas \ref{ellminusell0} and \ref{uw2minuselllemma} and the definitions of $u_{w,2,ell,0}$ and $u_{w,2,ell,0,cont}$ from \eqref{uw2ell0def}, \eqref{uw2ell0exp}, and \eqref{uw2ell0contdef}. We use the same procedure to estimate the higher derivatives, except for the last estimate in the lemma statement, which is obtained by using
$$\left(\partial_{r}^{2}+\frac{1}{r}\partial_{r}-\frac{1}{r^{2}}\right)\left(ellsoln-rQ(t)\right) = \psi_{\leq 1}(r-t) \partial_{t}^{2}\left(u_{w,2,ell}-u_{w,2,ell,0,cont}\right)$$ and Lemmas \ref{ellminusell0} and \ref{uw2minuselllemma}.
 \end{proof}
We return to $v_{3}$, which we recall solves \eqref{v3def}, with $0$ Cauchy data at infinity. If we define
$$q_{3}(t,r) := v_{3}(t,r)-\left(ellsoln-rQ(t)\right)\psi_{\leq 1}(r-t)$$
then, $q_{3}$ solves the following equation. Moreover, our estimates from Lemmas \ref{ellsolnlemma} and \ref{uw2minuselllemma} show that $q_{3}(t,r)$ also has $0$ Cauchy data at infinity.
\begin{equation}\begin{split}&-\partial_{t}^{2} q_{3}+\partial_{r}^{2}q_{3}+\frac{1}{r}\partial_{r}q_{3}-\frac{q_{3}}{r^{2}}\\
&= \psi_{\leq 1}(r-t) \left(\partial_{t}^{2}\left(ellsoln-rQ(t)\right)+\partial_{t}^{2}\left(u_{w,2,ell}-u_{w,2,ell,0,cont}\right)\cdot(1-\psi_{\leq 1}(r-t))\right)\\
&+\psi_{\leq 1}'(r-t)\left(-2 \left(\partial_{t}+\partial_{r}\right)\left(ellsoln-rQ(t)\right)-\frac{(ellsoln-rQ(t))}{r}\right)\\
&:=RHS_{4}(t,r)\end{split}\end{equation}
In other words, we have
$$q_{3}(t,r) = \int_{t}^{\infty}  q_{3,s}(t,r)ds, \quad q_{3,s} \text{ solves } \begin{cases}\left(-\partial_{t}^{2}+\partial_{r}^{2}+\frac{1}{r}\partial_{r}-\frac{1}{r^{2}}\right) q_{3,s}=0\\
q_{3,s}(s,r) =0\\
\partial_{1}q_{3,s}(s,r) =RHS_{4}(s,r)\end{cases}$$
We recall that we only need to estimate $q_{3}(t,r)$ for $r \leq 2 g(t)\lambda(t) < t-2$, and 
$$\psi_{\leq 1}(x) =1, \quad \psi_{\leq 1}'(x) =0, \quad x\leq 0$$ 
Therefore, by the finite speed of propagation, for $r \leq 2g(t)\lambda(t)$, we have
$$q_{3,s}(t,r) = q_{4,s}(t,r), \quad q_{4,s} \text{ solves } \begin{cases}\left(-\partial_{t}^{2}+\partial_{r}^{2}+\frac{1}{r}\partial_{r}-\frac{1}{r^{2}}\right) q_{4,s}=0\\
q_{4,s}(s,r) =0\\
\partial_{1}q_{4,s}(s,r) =\psi_{\leq 1}(r-s) \partial_{s}^{2}\left(ellsoln-rQ(s)\right)\end{cases}$$
So, it suffices to estimate 
$$q_{4}(t,r):=\int_{t}^{\infty} q_{4,s}(t,r) ds $$
We recall that
\begin{equation}\begin{split}&\psi_{\leq 1}(r-t) \partial_{t}^{2}\left(ellsoln(t,r)-r Q(t)\right)\\
&=\psi_{\leq 1}(r-t) \partial_{t}^{2}\begin{aligned}[t]&\left(\frac{-1}{2r} \int_{0}^{r} s^{2}\partial_{t}^{2}\left(u_{w,2,ell}-u_{w,2,ell,0}\right)(t,s) \psi_{\leq 1}(s-t) ds\right.\\
&\left.-\frac{1}{2r} \int_{0}^{r} s^{2} \partial_{t}^{2}\left(u_{w,2,ell,0}-u_{w,2,ell,0,cont}\right)(t,s) \psi_{\leq 1}(s-t) ds\right.\\
&+\frac{r}{2}\int_{0}^{r} \partial_{t}^{2}\left(u_{w,2,ell}-u_{w,2,ell,0}\right)(t,s) \psi_{\leq 1}(s-t) ds\\
&\left.-\frac{r}{2} \int_{r}^{\infty} \partial_{t}^{2}\left(u_{w,2,ell,0}-u_{w,2,ell,0,cont}\right)(t,s) \psi_{\leq 1}(s-t) ds\right)\end{aligned}\end{split}\end{equation}
We correspondingly decompose $q_{4}$ into 
$$q_{4}(t,r) :=q_{4,1}(t,r)+q_{4,2}(t,r)$$
where $q_{4,2}$ solves the following equation with $0$ Cauchy data at infinity (and $q_{4,1}=q_{4}-q_{4,2}$)
\begin{equation}\label{rhs42def}\begin{split}&\left(-\partial_{t}^{2}+\partial_{r}^{2}+\frac{1}{r}\partial_{r}-\frac{1}{r^{2}}\right) q_{4,2}(t,r) \\
&= \psi_{\leq 1}(r-t) \partial_{t}^{2}\left(\frac{-1}{2r} \int_{0}^{r} s^{2} \partial_{t}^{2}\left(u_{w,2,ell,0}-u_{w,2,ell,0,cont}\right)(t,s) \psi_{\leq 1}(s-t) ds\right)\\
&-\psi_{\leq 1}(r-t) \partial_{t}^{2}\left(\frac{r}{2} \int_{r}^{\infty} \partial_{t}^{2}\left(u_{w,2,ell,0}-u_{w,2,ell,0,cont}\right)(t,s) \psi_{\leq 1}(s-t) ds\right)\\
&:=RHS_{4,2}(t,r)\end{split}\end{equation}
To estimate $q_{4,2}$, it will suffice to use energy estimates.
\begin{lemma}\label{q42enest} We have the following estimates, for $k=0,1$, and $r>0$:
\begin{equation}\begin{split}&|\partial_{t}^{k}q_{4,2}(t,r)| + \sqrt{E(\partial_{t}^{k}q_{4,2},\partial_{t}^{k+1}q_{4,2})} +||\partial_{r}^{2} q_{4,2}||_{L^{2}((\lambda(t)g(t), 2 \lambda(t)g(t)),r dr)}\leq \frac{C \lambda(t)^{5} \log^{3}(t)}{t^{5} }\end{split}\end{equation}
\end{lemma}
\begin{proof} We first note that the right-hand side of \eqref{rhs42def} includes the term
$$\psi_{\leq 1}(r-t) \left(\frac{-1}{2r} \int_{0}^{r} s^{2}\psi_{\leq 1}''(s-t)\partial_{t}^{2}(u_{w,2,ell,0}-u_{w,2,ell,0,cont})ds - \frac{r}{2}\int_{r}^{\infty} \psi_{\leq 1}''(s-t)\partial_{t}^{2}(u_{w,2,ell,0}-u_{w,2,ell,0,cont}) ds\right)$$
We integrate by parts in each integral, integrating $\psi_{\leq 1}''(s-t)$, and note that the boundary contributions from each integral at $s=r$ cancel. Then, we directly insert the estimates from Lemma \ref{uw2minuselllemma} into the other terms in \eqref{rhs42def}, and use the same procedure used for \eqref{uw2enest}. Finally, the symbol type estimates on $u_{w,2,ell,0}-u_{w,2,ell,0,cont}$ from Lemma \ref{uw2minuselllemma} show that $\partial_{t}q_{4,2}$ solves
$$\left(-\partial_{t}^{2}+\partial_{r}^{2}+\frac{1}{r}\partial_{r}-\frac{1}{r^{2}}\right) \partial_{t}q_{4,2}(t,r) =\partial_{t}RHS_{4,2}(t,r)$$
also with 0 Cauchy data at infinity. Then, we use the same procedure used to estimate $q_{4,2}$. Finally, we estimate 
$$||\partial_{r}^{2} q_{4,2}||_{L^{2}((\lambda(t)g(t), 2 \lambda(t)g(t)),r dr)}$$
using the equation solved by $q_{4,2}$ and our earlier estimates from the proof of this lemma. \end{proof}
Now, we study $q_{4,1}$.  Let
\begin{equation}\label{rhs41def}\begin{split} RHS_{4,1}(t,r) &= \psi_{\leq 1}(r-t) \partial_{t}^{2}\begin{aligned}[t]&\left(\frac{-1}{2r} \int_{0}^{r} s^{2}\partial_{t}^{2}\left(u_{w,2,ell}-u_{w,2,ell,0}\right)(t,s) \psi_{\leq 1}(s-t) ds\right.\\
&\left.+\frac{r}{2}\int_{0}^{r} \partial_{t}^{2}\left(u_{w,2,ell}-u_{w,2,ell,0}\right)(t,s) \psi_{\leq 1}(s-t) ds\right)\end{aligned}\end{split}\end{equation}
so that we have
$$q_{4,1}(t,r) = \int_{t}^{\infty} q_{4,1,s}(t,r) ds, \quad q_{4,1,s} \text{ solves } \begin{cases}\left(-\partial_{t}^{2}+\partial_{r}^{2}+\frac{1}{r}\partial_{r}-\frac{1}{r^{2}}\right) q_{4,1,s}=0\\
q_{4,1,s}(s,r) =0\\
\partial_{1}q_{4,1,s}(s,r) =RHS_{4,1}(s,r)\end{cases}$$
Since we only need to estimate $q_{4,1}(t,r)$ (and hence $q_{4,1,s}(t,r)$) in the region $r \leq 2 g(t)\lambda(t) <\frac{t}{2}$, the finite speed of propagation shows that $q_{4,1,s}(t,r)$ only depends on $RHS_{4,1}(s,y)$ for $y \leq s-t+2 g(t)\lambda(t) < s-\frac{t}{2}$. Given the limits of the integrals in the terms defining $RHS_{4,1}$, this means that we may replace $\psi_{\leq 1}(s-t)$ and $\psi_{\leq 1}(r-t)$ by $1$ when considering $q_{4,1}(t,r)$ in the region $r \leq 2g(t)\lambda(t)$. 
Now, we estimate $RHS_{4,1}$. 
\begin{lemma}\label{rhs41lemma} For $0 \leq k,j \leq 2$, $j=3, 0 \leq k \leq 1$, and $s \geq t$, we have
$$y^{k} |\partial_{y}^{k}\partial_{s}^{j}RHS_{4,1}(s,y)| \leq C \lambda(s)^{2} \log^{3}(s) \sup_{x \in [100,s]}\left(\lambda(x) \log(x)\right) \begin{cases} \frac{y^{3}}{s^{8+j}}, \quad y \leq \frac{s}{2}\\
\frac{1}{s^{5/2} t^{5/2+j}}, \quad \frac{s}{2} < y < s-t+2 g(t)\lambda(t)\end{cases}$$
\end{lemma}
\begin{proof}
We re-write \eqref{rhs41def} as
\begin{equation}\label{rhs41rewrite} RHS_{4,1}(s,y) = \int_{0}^{y} F(w,y) \partial_{s}^{4}\left(RHS_{2}-RHS_{2,0}\right)(s,w) dw - \frac{y^{3}}{16} \int_{0}^{\infty} \partial_{s}^{4}\left(RHS_{2}-RHS_{2,0}\right)(s,w) dw\end{equation}
where
$$F(w,y) = \frac{-\left(w^{4}-y^{4}+4 w^{2}y^{2}\log(y/w)\right)}{16 y}$$
This is of the same form as \eqref{uw2ellminusell0initialstep}, and is treated in the same way, noting that $\partial_{1}^{j}F(y,y)=0, \quad j=0,1,2$. In addition, we use $s-t+2g(t)\lambda(t) < s-\frac{t}{2}$, which implies that
$$\frac{1}{\langle s-y\rangle} \leq \frac{C}{t}, \quad \frac{s}{2} < y < s-t+2 g(t)\lambda(t)$$
We then directly differentiate \eqref{rhs41rewrite} to treat higher derivatives of $RHS_{4,1}$.
 \end{proof}
Let
\begin{equation}\label{q410def} q_{4,1,0}(t,r):=\frac{-r}{2} \int_{t}^{\infty} ds \int_{0}^{s-t} \frac{\rho d\rho}{\sqrt{(s-t)^{2}-\rho^{2}}} \left(\partial_{2}RHS_{4,1}(s,\rho)+\frac{RHS_{4,1}(s,\rho)}{\rho}\right)\end{equation}
Now, we can estimate $q_{4,1}$. 
\begin{lemma} \label{q41lemma}For $0 \leq k \leq 1$, and $0 \leq j \leq 2$ or $k=2,j=0$, we have
\begin{equation}\label{q41est} |r^{k}\partial_{r}^{k}\partial_{t}^{j}\left(q_{4,1}(t,r)-q_{4,1,0}(t,r)\right)| \leq \frac{C r^{2} \sup_{x \in [100,t]}\left(\lambda(x) \log(x)\right) \lambda(t)^{2}\log^{3}(t)}{t^{5+j}}, \quad \lambda(t) g(t) \leq r \leq 2 \lambda(t)g(t)\end{equation}
Also, for $0 \leq j \leq 3$,
\begin{equation} |\partial_{t}^{j} q_{4,1,0}(t,r)| \leq \frac{C r \sup_{x \in [100,t]}\left(\lambda(x) \log(x)\right) \lambda(t)^{2} \log^{3}(t)}{t^{4+j}}\end{equation} 
\end{lemma}
\begin{proof}
The spherical means representation formula gives (see also (4.99), pg. 24 of  \cite{wm})
\begin{equation}\label{q41orig} \begin{split} q_{4,1}(t,r) &= -\frac{1}{2\pi} \int_{t}^{\infty} ds \int_{0}^{s-t}\frac{\rho d\rho}{\sqrt{(s-t)^{2}-\rho^{2}}}\int_{0}^{2\pi} \frac{RHS_{4,1}(s,\sqrt{r^{2}+\rho^{2}+2 r \rho \cos(\theta)})}{\sqrt{r^{2}+\rho^{2}+2 r \rho\cos(\theta)}} \left(r+\rho\cos(\theta)\right)d\theta\end{split}\end{equation}
If 
$$G(s,r,\rho) = \int_{0}^{2\pi} d\theta \frac{\left(r+\rho \cos(\theta)\right)}{\sqrt{r^{2}+\rho^{2}+2 r \rho \cos(\theta)}} RHS_{4,1}(s,\sqrt{r^{2}+\rho^{2}+2 r \rho \cos(\theta)}), \quad \rho >0$$
Then, by the dominated convergence theorem and Lemma \ref{rhs41lemma}, $\partial_{r}G(s,r,\rho)$ can be computed by differentiation under the integral sign, and
$$G(s,0,\rho)=0 \implies G(s,r,\rho) = r \int_{0}^{1} \partial_{2}G(s,\beta r,\rho)  d\beta$$
Therefore, we have
\begin{equation}\label{q41onederivform} q_{4,1}(t,r) =  \frac{-r}{2\pi} \int_{t}^{\infty} ds \int_{0}^{s-t} \frac{\rho d\rho}{\sqrt{(s-t)^{2}-\rho^{2}}} \int_{0}^{1} d\beta \int_{0}^{2\pi} I_{RHS_{4,1}}(s,r\beta,\rho,\theta) d\theta\end{equation}
where we recall the notation \eqref{Inotation}. Recalling \eqref{q410def}, we get
\begin{equation}\label{q41decomp} \begin{split} &q_{4,1}(t,r)-q_{4,1,0}(t,r) = \frac{-r}{2\pi} \int_{t}^{\infty} ds \int_{0}^{s-t} \frac{\rho d\rho}{\sqrt{(s-t)^{2}-\rho^{2}}} \int_{0}^{1}d\beta \int_{0}^{2\pi} d\theta \text{integrand}_{4,1,2}(s,r\beta,\rho,\theta)\end{split}\end{equation}
where
\begin{equation}\begin{split} \text{integrand}_{4,1,2}(s,y,\rho,\theta) &= I_{RHS_{4,1}}(s,y,\rho,\theta) -\left(\partial_{2}RHS_{4,1}(s,\rho) \cos^{2}(\theta) + \frac{RHS_{4,1}(s,\rho)}{\rho}\sin^{2}(\theta)\right)\end{split}\end{equation}
We note that
$$ \text{integrand}_{4,1,2}(s,y,\rho,\theta)= I_{RHS_{4,1}}(s,y,\rho,\theta)- I_{RHS_{4,1}}(s,0,\rho,\theta)$$ 
Also, for all $r\geq0, \rho>3r, s \geq T_{0}, \theta \in [0,2\pi],$ and $\beta \in [0,1]$, 
$$y \mapsto I_{RHS_{4,1}}(s,y,\rho,\theta) \in C^{1}([0,r\beta]), \quad \text{(note that, in this setting, }\sqrt{y^{2}+\rho^{2}+2 y \rho\cos(\theta)} \geq C \rho >0)$$
Therefore, when $s-t \geq 3r$, and $\rho>3r$ in the integral in \eqref{q41decomp}, we use
$$\text{integrand}_{4,1,2}(s,r\beta,\rho,\theta) = \int_{0}^{r\beta} \partial_{2}I_{RHS_{4,1}}(s,y,\rho,\theta) dy$$
and we directly substitute our estimates from Lemma \ref{rhs41lemma} into \eqref{q41decomp} for the other regions, to get \eqref{q41est} for $k=0, j=0$. Next, we use \eqref{q41decomp} to get
\begin{equation} \partial_{r} \left(q_{4,1}-q_{4,1,0}\right)(t,r) = -\frac{1}{2\pi} \int_{t}^{\infty} ds \int_{0}^{s-t} \frac{\rho d\rho}{\sqrt{(s-t)^{2}-\rho^{2}}} \int_{0}^{2\pi} d\theta \text{integrand}_{4,1,2}(s,r,\rho,\theta)\end{equation}
which is treated with the same argument used for \eqref{q41decomp}. This gives \eqref{q41est} for $k=1, j=0$. To estimate $\partial_{t}^{j} \left(q_{4,1}-q_{4,1,0}\right)(t,r)$, we return to \eqref{q41decomp}, and let $w=s-t$. Then, the dominated convergence theorem, along with the estimates of Lemma \ref{rhs41lemma} allow us to differentiate under the resulting integral signs, and we get, for $j=1,2$,
\begin{equation}\begin{split}&\partial_{t}^{j}\left(q_{4,1}-q_{4,1,0}\right)(t,r) = \frac{-r}{2\pi} \int_{t}^{\infty} ds \int_{0}^{s-t} \frac{\rho d\rho}{\sqrt{(s-t)^{2}-\rho^{2}}} \int_{0}^{1}d\beta \int_{0}^{2\pi} d\theta \partial_{s}^{j}\text{integrand}_{4,1,2}(s,r\beta,\rho,\theta)\end{split}\end{equation}
We then repeat the same procedure used for $q_{4,1}-q_{4,1,0}$. To estimate $q_{4,1,0}$, we use
$$q_{4,1,0}(t,r) = \frac{-r}{2} \int_{0}^{\infty} dw \int_{0}^{w} \frac{\rho d\rho}{\sqrt{w^{2}-\rho^{2}}} \left(\partial_{2}RHS_{4,1}(t+w,\rho)+\frac{RHS_{4,1}(t+w,\rho)}{\rho}\right)$$
differentiate under the integral sign, then substitute the estimates from Lemma \ref{rhs41lemma}. We prove \eqref{q41est} for $j=0,k=2$ by using the equation solved by $q_{4,1}$ and our previous estimates from this lemma.
 \end{proof}
Next, we study the analogous quantities related to $v_{ex,sub}$ which we recall is defined in \eqref{vexsubdef}. In particular, we start by studying 
$$w_{5}(t,r):= v_{ex,sub}(t,r)-v_{ex,sub,0}(t,r)$$ 
The function $w_{5}$ satisfies the following equation with $0$ Cauchy data at infinity
$$\left(-\partial_{t}^{2}+\partial_{r}^{2}+\frac{1}{r}\partial_{r}-\frac{1}{r^{2}}\right) w_{5}(t,r) = \partial_{t}^{2} \left(v_{ex,ell}-v_{ex,cont}\right) = \partial_{t}^{2}v_{ex,ell,1}(t,r)$$
where we define
 $$v_{ex,ell,1}(t,r) :=v_{ex}(t,r)-v_{ex,sub}(t,r)-v_{ex,cont}(t,r)=v_{ex,ell}(t,r)-v_{ex,cont}(t,r)$$
 and recall that $v_{ex,ell}$ and $v_{ex,cont}$ are explicitly given in \eqref{vexell} and \eqref{vexcontdef}, respectively. 
 We have the following lemma
 \begin{lemma}
For  $0 \leq k,j \leq 2$, we have 
$$|\partial_{t}^{j}\partial_{r}^{k} v_{ex,ell,1}(t,r)| \leq \frac{C \lambda(t)^{2-k}}{t^{2+j}g(t)^{3+k}}, \quad g(t)\lambda(t) \leq r \leq 2g(t)\lambda(t)$$\end{lemma}
\begin{proof} We directly estimate $v_{ex,ell,1}$ from the formulae \eqref{vexell} and \eqref{vexcontdef}. \end{proof}
Next, we define $ellsoln_{2}$ by
$$ellsoln_{2}(t,r):= \frac{-1}{2}\left(\frac{1}{r}\int_{0}^{r} s^{2}\partial_{t}^{2}v_{ex,ell,1}(t,s) ds + r \int_{r}^{\infty} \partial_{t}^{2}v_{ex,ell,1}(t,s) ds\right)$$
Then, we have the following lemma.
\begin{lemma}\label{ellsoln2lemma} For all $j,k \geq 0$, there exists $C_{j,k}>0$ such that
\begin{equation}\label{ellsoln2est}r^{k}t^{j}|\partial_{r}^{k}\partial_{t}^{j}ellsoln_{2}(t,r)| \leq C_{j,k}\cdot \begin{cases} \frac{1}{r} \frac{\lambda(t)^{5}(1+|\log(\frac{r}{\lambda(t)})|)}{t^{4} }, \quad r \geq \lambda(t)\\
 r \frac{\lambda(t)^{3}(1+\log^{2}(\frac{r}{\lambda(t)}))}{t^{4} }, \quad r \leq \lambda(t)\end{cases}\end{equation}
In addition, $\partial_{t}^{2}ellsoln_{2}(t,r) = F_{ellsoln_{2}}(t,r)+\log^{2}(r) r g_{1}''(t)+r \log(r) g_{2}''(t)$, where, for each $t \geq T_{0}$,
$$r \mapsto F_{ellsoln_{2}}(t,r) \text{ admits a }C^{2}\text{ extension to }[0,\infty)$$
For $0 \leq j+k \leq 2$,
$$|\partial_{r}^{k}\partial_{t}^{j} F_{ellsoln_{2}}(t,r)| \leq  C \begin{cases} \frac{r^{1-k}\lambda(t)^{3} \log^{2}(t)}{t^{6+j}}, \quad 0 \leq k \leq 1\\
\frac{r \lambda(t)(1+|\log(\frac{r}{\lambda(t)})|)}{t^{6+j}}, \quad k=2\end{cases}, \quad r \leq \lambda(t)$$
$$|g_{1}^{(k)}(t)| + |g_{2}^{(k)}(t)| \leq \frac{C_{k} \log(t)\lambda(t)^{3}}{t^{4+k}}, \quad k \geq 0$$
 \end{lemma}
\begin{proof}
We have
\begin{equation}\begin{split}ellsoln_{2}(t,r)&= \frac{-1}{2}\left(\frac{1}{r}\int_{0}^{r} s^{2}\partial_{t}^{2}v_{ex,ell,1}(t,s) ds + r \int_{r}^{\infty} \partial_{t}^{2}v_{ex,ell,1}(t,s) ds\right)=-\frac{1}{2}\left(\frac{1}{r}f_{7}(t,\frac{r}{\lambda(t)})+r f_{8}(t,\frac{r}{\lambda(t)})\right)\end{split}\end{equation}
where
\begin{equation}\begin{split}f_{7}(t,R)&= \frac{\lambda(t)}{8} \left(\frac{-8(-2 R^{2}+3(1+R^{2}) \log(1+R^{2})) \lambda'(t)^{4}}{1+R^{2}}\right.\\
&\left.-24 \lambda(t) \left(R^{2} \log(1+\frac{1}{R^{2}})+3 \log(R^{2}+1)\right) \lambda'(t)^{2}\lambda''(t)\right.\\
&\left.-4 \lambda(t)^{2} \left(R^{2}\log(1+\frac{1}{R^{2}}) + \log(1+R^{2})\right)\left(3\lambda''(t)^{2}+4 \lambda'(t)\lambda'''(t)\right) \right.\\
&\left.- \lambda(t)^{3} \left(-R^{2}+R^{2}(2+R^{2}) \log(1+\frac{1}{R^{2}}) + \log(1+R^{2})\right)\lambda''''(t)\right)\end{split}\end{equation}
and
\begin{equation}\begin{split}f_{8}(t,R)&= \frac{1}{4} \left(\frac{-4(2+(1+R^{2})\log(1+\frac{1}{R^{2}})) \lambda'(t)^{4}}{(1+R^{2})\lambda(t)}\right.\\
&+\left.12 \left(-2 \log(1+\frac{1}{R^{2}}) + \text{Li}_{2}(-\frac{1}{R^{2}})\right)\lambda'(t)^{2}\lambda''(t)+2 \lambda(t) \text{Li}_{2}(\frac{-1}{R^{2}}) \left(3\lambda''(t)^{2}+4 \lambda'(t)\lambda'''(t)\right)\right.\\
&+\left.\lambda(t)^{2}(-1+(1+R^{2})\log(1+\frac{1}{R^{2}}) + \text{Li}_{2}(\frac{-1}{R^{2}})) \lambda''''(t)\right)\end{split}\end{equation}
The lemma statement now follows from inspection.
 \end{proof}
We define
$$q_{5}(t,r):= w_{5}(t,r)-ellsoln_{2}(t,r)\psi_{\leq 1}(r-t)$$
so that $q_{5}$ solves the following equation with $0$ Cauchy data at infinity.
\begin{equation}\label{q5eqn}\begin{split}-\partial_{t}^{2}q_{5}+\partial_{r}^{2}q_{5}(t,r) + \frac{1}{r}\partial_{r}q_{5}-\frac{q_{5}}{r^{2}}&= \psi_{\leq 1}(r-t)\partial_{t}^{2} ellsoln_{2} + \left(1-\psi_{\leq 1}(r-t)\right) \partial_{t}^{2} v_{ex,ell,1}(t,r)\\
&-\psi_{\leq 1}'(r-t) \left(2\left(\partial_{t}+\partial_{r}\right) ellsoln_{2}+\frac{1}{r}ellsoln_{2}\right)\end{split}\end{equation}
We define $q_{5,0}$ by
\begin{equation}\label{q50defthirdorder}q_{5,0}(t,r)= \frac{-r}{2} \int_{t}^{\infty} ds \int_{0}^{s-t} \frac{\rho d\rho}{\sqrt{(s-t)^{2}-\rho^{2}}} \left(\partial_{112}ellsoln_{2}(s,\rho) + \frac{\partial_{11}ellsoln_{2}(s,\rho)}{\rho}\right)\end{equation}
Then, we have the following lemma.
\begin{lemma}\label{q5lemma} For $0 \leq k \leq 1$ and $0 \leq j \leq 2$,
\begin{equation} |t^{j}r^{k}\partial_{t}^{j}\partial_{r}^{k} \left(q_{5}-q_{5,0}\right)(t,r)| \leq \frac{C r^{2}\lambda(t)^{3} \left(\log^{3}(t)+|\log(r)|^{3}\right)}{t^{5} }, \quad r \leq 2 g(t)\lambda(t)\end{equation}
For all $j \geq 0$, there exists $C_{j}>0$ such that
\begin{equation}|\partial_{t}^{j} q_{5,0}(t,r)| \leq \frac{C_{j} r \log^{2}(t)\lambda(t)^{3}}{t^{4+j} }\end{equation}
\begin{equation}|\partial_{r}^{2}\left(q_{5}-q_{5,0}\right)(t,r)| \leq \frac{C \lambda(t)^{3} \log^{3}(t)}{t^{5}}, \quad g(t)\lambda(t) \leq r \leq 2 g(t)\lambda(t)\end{equation}
\end{lemma}
\begin{proof}  First, we claim that in the region $r \leq 2 g(t)\lambda(t)$,
\begin{equation}\label{q5claim}\begin{split}q_{5}(t,r)&=\frac{-1}{2\pi} \int_{t}^{\infty} ds \int_{0}^{s-t}\frac{\rho d\rho}{\sqrt{(s-t)^{2}-\rho^{2}}} \int_{0}^{2\pi}d\theta \partial_{1}^{2}ellsoln_{2}(s,\sqrt{r^{2}+\rho^{2}+2 r \rho\cos(\theta)}) \frac{(r+\rho\cos(\theta))}{\sqrt{r^{2}+\rho^{2}+2 r \rho\cos(\theta)}}\end{split}\end{equation}
To verify this, we recall \eqref{q5eqn}, and use the finite speed of propagation. The only item remaining is to show that the integral on the right -hand side of \eqref{q5claim} solves \eqref{q5eqn} with $\partial_{t}^{2}ellsoln_{2}$ on the right-hand side, and zero Cauchy data at infinity. For this purpose, we note that 
\begin{equation}\begin{split}&\frac{-1}{2\pi} \int_{t}^{\infty} ds \int_{0}^{s-t}\frac{\rho d\rho}{\sqrt{(s-t)^{2}-\rho^{2}}} \int_{0}^{2\pi}d\theta \partial_{1}^{2}ellsoln_{2}(s,\sqrt{r^{2}+\rho^{2}+2 r \rho\cos(\theta)}) \frac{(r+\rho\cos(\theta))}{\sqrt{r^{2}+\rho^{2}+2 r \rho\cos(\theta)}}\\
&= \frac{-1}{2\pi} \int_{t}^{\infty} ds \int_{0}^{s-t}\frac{\rho d\rho}{\sqrt{(s-t)^{2}-\rho^{2}}} \int_{0}^{2\pi}d\theta F_{ellsoln_{2}}(s,\sqrt{r^{2}+\rho^{2}+2 r \rho\cos(\theta)}) \frac{(r+\rho\cos(\theta))}{\sqrt{r^{2}+\rho^{2}+2 r \rho\cos(\theta)}}\\
&+\frac{-r^{3}}{2\pi} \int_{0}^{\infty} dy \int_{0}^{y}\frac{w dw}{\sqrt{y^{2}-w^{2}}} \int_{0}^{2\pi}d\theta  (1+w \cos(\theta))\begin{aligned}[t]&\left( \log^{2}(r) g_{1}''(t+r y) \right.\\
&\left.+  \log(r) \left(\log(1+w^{2}+2 w \cos(\theta))g_{1}''(t+r y)+g_{2}''(t+r y)\right)\right.\\
&\left. +\log^{2}(\sqrt{1+w^{2}+2 w \cos(\theta)})g_{1}''(t+r y)\right.\\
&\left.+\log(\sqrt{1+w^{2}+2 w \cos(\theta)})g_{2}''(t+r y))\right)\end{aligned}\end{split}\end{equation}
The point of this splitting is that, by Lemma \ref{ellsoln2lemma} and the dominated convergence theorem, we can differentiate up to two times in $(t,r)$ under the integral sign. Then, we can proceed as in any standard verification of the spherical means formula. Then, we return to \eqref{q5claim} and use \eqref{ellsoln2est} and the same procedure used to establish Lemma \ref{q41lemma}.

 \end{proof}
Now, we compute the leading behavior of $u_{ell,2}(t,r)-v_{ell,2,0,main}(t,\frac{r}{\lambda(t)})$ in the matching region. We first recall that $v_{ell,2,0}$ and $v_{ell,2,0,main}$ are explicitly given in \eqref{vell20def} and \eqref{vell20maindef}, respectively. So, it suffices to explain how to compute the leading behavior of $u_{ell,2}(t,r)-v_{ell,2,0}(t,\frac{r}{\lambda(t)})$. From \eqref{ell2minusell20}, we have
\begin{equation}\label{ell2minusell20forasymp}\begin{split}&u_{ell,2}(t,r)-v_{ell,2,0}(t,\frac{r}{\lambda(t)})\\
& = \frac{-\phi_{0}(\frac{r}{\lambda(t)})}{2}\int_{0}^{\frac{r}{\lambda(t)}} err_{1}(t,s\lambda(t)) s e_{2}(s) ds + e_{2}(\frac{r}{\lambda(t)}) \int_{0}^{\frac{r}{\lambda(t)}} err_{1}(t,s\lambda(t)) \frac{s \phi_{0}(s)}{2} ds\\
&= \frac{-\phi_{0}(\frac{r}{\lambda(t)})}{2}\int_{0}^{\frac{r}{\lambda(t)}} err_{1,0}(t,s\lambda(t)) s e_{2}(s) ds + \frac{-\phi_{0}(\frac{r}{\lambda(t)})}{2}\int_{0}^{\frac{r}{\lambda(t)}} \left(err_{1}-err_{1,0}\right)(t,s\lambda(t)) s e_{2}(s) ds\\
&+\frac{r}{2\lambda(t)} \int_{0}^{\frac{r}{\lambda(t)}} err_{1}(t,s\lambda(t)) \frac{s \phi_{0}(s)}{2}ds+ \left(e_{2}(\frac{r}{\lambda(t)})-\frac{r}{2\lambda(t)}\right) \int_{0}^{\frac{r}{\lambda(t)}} err_{1}(t,s\lambda(t)) \frac{s \phi_{0}(s)}{2}ds\end{split}\end{equation}
where we recall that $err_{1,0}(t,r)$ is defined in \eqref{err10def}. Using the definition of $v_{ell,2,0,main}$ given in \eqref{vell20maindef}, we  can write down the leading behavior of $u_{ell,2}(t,r) - v_{ell,2,0,main}(t,\frac{r}{\lambda(t)})$ in the region $r \sim g(t)\lambda(t)$, which we denote by $u_{e,3}(t,r)$. 
\begin{equation}\label{ue3def}\begin{split}u_{e,3}(t,r) &= \frac{-2 \lambda(t)}{2 r}\int_{0}^{\frac{r}{\lambda(t)}} err_{1,0}(t,s\lambda(t)) \frac{s}{2} s ds +\frac{r}{2\lambda(t)} \int_{0}^{\infty}\left(err_{1}(t,s\lambda(t))-err_{1,0}(t,s\lambda(t))\right) s \frac{\phi_{0}(s)}{2} ds\\
 &+\frac{r}{2\lambda(t)} \int_{0}^{\frac{r}{\lambda(t)}} err_{1,0}(t,s\lambda(t)) \frac{\phi_{0}(s) s ds}{2}+\frac{r \lambda(t)^{2}\lambda''''(t)}{96}\left(111-2\pi^{2}-108 \log(\frac{r}{\lambda(t)}) + 24 \log^{2}(\frac{r}{\lambda(t)})\right)\\
 & + \frac{r}{ 16} (5-4 \log(\frac{r}{\lambda(t)})) \begin{aligned}[t]&\left(\partial_{t}^{2}\left(\frac{\lambda'(t)^{2}}{2 \lambda(t)}+\lambda''(t) (1+\log(\lambda(t)))\right) 2 \lambda(t)^{2} -2\lambda(t)^{2}\lambda''''(t)(1+\log(\lambda(t)))\right)\end{aligned}\end{split}\end{equation}
 By computing the integrals, we get
\begin{equation}\label{ue3def1}\begin{split}&u_{e,3}(t,R \lambda(t))\\
&=\frac{\left(2\lambda'(t)^{4}-7 \lambda(t)\lambda'(t)^{2}\lambda''(t)+4 \lambda(t)^{2}\lambda''(t)^{2} + 6 \lambda(t)^{2}\lambda'(t)\lambda'''(t)\right)R\left(5-4 \log(R)\right)}{16}\\
&+\frac{1}{96}\left(111-2\pi^{2}-108 \log(R)+24 \log^{2}(R)\right)R \lambda(t)^{3}\lambda''''(t)\\
&+\frac{3}{4} R \lambda'(t)^{4}+\frac{3}{8}R \left(-2 \lambda'(t)^{4}+5 \lambda(t)\lambda'(t)^{2}\lambda''(t)\right)+\frac{3}{4} R \lambda(t)^{2}\partial_{t}\left(\lambda'(t)\lambda''(t)\right)\\
&+\frac{1}{24} R \lambda(t)^{2}\left(21+\pi^{2}-18\log(R)+12 \log^{2}(R)\right)\partial_{t}^{2}\left(\lambda(t)\lambda''(t)\right)\\
&+\frac{1}{24} R \lambda(t)^{2}\left(39+\pi^{2}-42 \log(R) + 12 \log^{2}(R)\right)\left(\partial_{t}^{2}\left(\lambda(t)\lambda''(t)\right)-\lambda(t)\lambda''''(t)\right)\\
&+\frac{1}{12}R \lambda(t)\left(60+\pi^{2}-54 \log(R)+12 \log^{2}(R)\right)\lambda'(t)^{2}\lambda''(t)\\
&-\frac{1}{4}R \lambda(t)^{2}c_{1}''(t)-\frac{R}{2}\left(\lambda(t)c_{1}'(t)\lambda'(t)+\frac{c_{1}(t)}{2} \left(-2 \lambda'(t)^{2}+\lambda(t)\lambda''(t)\right)\right)-\frac{R}{2}c_{1}(t)\lambda'(t)^{2}\\
&+\frac{R}{8} \lambda'(t)^{4}\left(17-12 \log(R)\right)-\frac{R}{4}\left(5 \lambda(t)\lambda'(t)^{2}\lambda''(t)-2\lambda'(t)^{4}\right)\left(-3+2\log(R)\right)\\
&+\frac{R}{4} \left(3-8 \log(R)\right)\lambda(t)^{2}\partial_{t}\left(\lambda''(t)\lambda'(t)\right)\\
&+\frac{R \lambda(t)^{2}\partial_{t}^{2}\left(\lambda(t)\lambda''(t)\right)}{24}\left(6+\pi^{2}+12 \zeta(3)-\left(36+2 \pi^{2}\right)\log(R)+12 \log^{2}(R) - 8 \log^{3}(R)\right)\\
&-\frac{R\lambda(t)^{2}}{24}\left(\partial_{t}^{2}\left(\lambda(t)\lambda''(t)\right)-\lambda(t)\lambda''''(t)\right) \begin{aligned}[t]&\left(-27 -2 \pi^{2}-12 \zeta(3)+(36+2\pi^{2})\log(R) \right.\\
&\left.-24 \log^{2}(R)+8 \log^{3}(R)\right)\end{aligned}\\
&+ R\frac{\lambda(t)\lambda''(t)\lambda'(t)^{2}}{72} \left(255+17 \pi^{2}+72 \zeta(3)-(360+12 \pi^{2})\log(R) + 252 \log^{2}(R)-48 \log^{3}(R)\right)\\
&+R \frac{c_{1}''(t)\lambda(t)^{2}}{4} \left(-1+2 \log(R)\right)+R\left(\lambda(t)\lambda'(t)c_{1}'(t)+\frac{c_{1}(t)}{2}\left(\lambda(t)\lambda''(t)-2\lambda'(t)^{2}\right)\right) (-1+\log(R))\\
&+ R\frac{c_{1}(t)}{12} \lambda'(t)^{2}\left(-17+12 \log(R)\right)+R\frac{\lambda(t)^{3}\lambda''''(t)}{48} \left(\pi^{2}-12 \log^{2}(R)\right) \\
&+ \frac{R}{2}\log(R) \lambda(t)^{3} \left(\partial_{t}^{2}\left(\frac{\lambda'(t)^{2}}{2 \lambda(t)}+ \lambda''(t)(1+\log(\lambda(t)))\right)-\log(\lambda(t))\lambda''''(t)\right)\end{split}\end{equation}
Now, we estimate the difference between $u_{ell,2}(t,r)-v_{ell,2,0,main}(t,\frac{r}{\lambda(t)})$ and $u_{e,3}(t,r)$.
\begin{lemma}\label{uell2minusvell20mainminusue3lemma}For $u_{ell,2}$ defined in \eqref{uell2def}, $v_{ell,2,0,main}$ defined in \eqref{vell20maindef}, and $u_{e,3}$ defined in  \eqref{ue3def} we have, for $0 \leq j+k \leq 2$,
 $$|\partial_{t}^{j}\partial_{r}^{k}\left(u_{ell,2}(t,r)-v_{ell,2,0,main}(t,\frac{r}{\lambda(t)})-u_{e,3}(t,r)\right)| \leq \frac{C \lambda(t)^{5} (1+\log^{4}(\frac{r}{\lambda(t)}))}{r^{1+k}t^{4+j}}, \quad r > \lambda(t)$$\end{lemma}
 \begin{proof}We estimate each term in the difference between \eqref{soln1def} plus \eqref{ell2minusell20} and \eqref{ue3def} directly, using the symbol type estimates on $\lambda(t)$, and the fact that $err_{1}(t,r)$ and $err_{1,0}(t,r)$ are symbols in $(t,r)$. \end{proof}
Combining our work, we get
\begin{equation}\label{uw3finalexpr}\begin{split}&\text{leading part of }\left(u_{w,2}-u_{w,2,ell,0}+v_{ex,sub}\right)(t,r):=u_{w,3}(t,r)\\
&=u_{w,2,sub,0,cont}(t,r) + v_{ex,sub,0,cont}(t,r)\\
&-\frac{1}{2}\left((u_{w,2,ell}-u_{w,2,ell,0})_{princ,1}(t,r)+(u_{w,2,ell}-u_{w,2,ell,0})_{princ,2}(t,r)\right)\\
&-\frac{r}{2}  \lambda(t) \int_{t}^{\infty} W_{3}(\frac{s-t}{\lambda(t)}) \lambda''''(s) ds-\frac{r}{2} \int_{0}^{\infty} d\xi \left(\frac{-8}{\xi \lambda(t)^{2}}+4 \xi K_{2}(\xi \lambda(t)) + 2\xi\right) \sin(t\xi) \widehat{v_{2,0}}(\xi)\\
&-\frac{r}{2} \int_{t}^{\infty} ds \int_{0}^{s-t} \frac{\rho d\rho}{\sqrt{(s-t)^{2}-\rho^{2}}} \left(\partial_{112}ellsoln_{2}(s,\rho) + \frac{\partial_{11}ellsoln_{2}(s,\rho)}{\rho}\right)\\
&-\frac{r}{2} \int_{0}^{\infty} dw \int_{0}^{w} \frac{\rho d\rho}{\sqrt{w^{2}-\rho^{2}}} \left(\partial_{2}RHS_{4,1}(t+w,\rho)+\frac{RHS_{4,1}(t+w,\rho)}{\rho}\right)\end{split}\end{equation}
A careful inspection of all of the terms shows that the $r \log^{k}(r)$ terms from $u_{e,3}$ and $u_{w,3}$ all exactly match, for $k=1,2,3$. In other words, we have
\begin{equation}\label{f3def}\begin{split} &u_{e,3}(t,r)-u_{w,3}(t,r) \\
&= r F_{2}(t)+\frac{r\lambda(t)}{2}\int_{t}^{\infty} W_{3}(\frac{s-t}{\lambda(t)}) \lambda''''(s) ds\\
& +\frac{r}{2}\int_{t}^{2t} \begin{aligned}[t]&\left(\frac{(g_{0}+h_{0})''(s)-(g_{0}+h_{0})''(t)}{(s-t)} + \frac{\left((g_{1}+2 f_{1})''(s)-(g_{1}+2f_{1})''(t)\right)\log(2(s-t))}{(s-t)}\right.\\
&\left.-\frac{\left(g_{2}''(s)-g_{2}''(t)\right)\left(-12 \log(s-t)\log(4(s-t))+\pi^{2}-12 \log^{2}(2)\right)}{12(s-t)}\right) ds\end{aligned}\\
&+\frac{r}{2} \int_{2t}^{\infty} \begin{aligned}[t]&\left((g_{0}+h_{0})''(s)+(g_{1}+2f_{1})''(s)\log(2(s-t))\right.\\
&\left.-\frac{g_{2}''(s) \left(-12 \log(s-t)\log(4(s-t))+\pi^{2}-12 \log^{2}(2)\right)}{12}\right) \frac{ds}{(s-t)}\end{aligned}\\
&+ \frac{r}{2} \int_{0}^{\infty} d\xi \left(\frac{-8}{\xi \lambda(t)^{2}}+4 \xi K_{2}(\xi \lambda(t))+2 \xi\right) \sin(t\xi) \widehat{v_{2,0}}(\xi)\\
&+\frac{r}{2} \int_{t}^{\infty} ds \int_{0}^{s-t} \frac{\rho d\rho}{\sqrt{(s-t)^{2}-\rho^{2}}} \left(\partial_{112}ellsoln_{2}(s,\rho) + \frac{\partial_{11}ellsoln_{2}(s,\rho)}{\rho}\right)\\
&+\frac{r}{2} \int_{0}^{\infty} dw \int_{0}^{w} \frac{\rho d\rho}{\sqrt{w^{2}-\rho^{2}}} \left(\partial_{2}RHS_{4,1}(t+w,\rho)+\frac{RHS_{4,1}(t+w,\rho)}{\rho}\right)\\
&:=r F_{3}(t)\end{split}\end{equation}
where $F_{2}(t) = F_{2,0}(t)+F_{2,1}(t)$, and
\begin{equation}\begin{split}F_{2,0}(t) &= \frac{3  \lambda'(t)^4}{4 \lambda(t)}+\frac{ (12 \log (\lambda(t))+17) \lambda'(t)^4}{8 \lambda(t)}-\frac{ c_{1}(t) \lambda'(t)^2}{2 \lambda(t)}+\frac{c_{1}(t)  (-12 \log (\lambda(t))-17) \lambda'(t)^2}{12 \lambda(t)}\\
&+\frac{1}{12}  \left(12 \log ^2(\lambda(t))+54 \log (\lambda(t))+\pi ^2+60\right) \lambda''(t) \lambda'(t)^2+\frac{3}{4}  \partial_{t} \left(\lambda'(t) \lambda''(t)\right) \lambda(t)\\
&+\frac{1}{4}  \partial_{t} \left(\lambda''(t) \lambda'(t)\right) \lambda(t) (8 \log (\lambda(t))+3)-\frac{ \left(5 \lambda(t) \lambda'(t)^2 \lambda''(t)-2 \lambda'(t)^4\right) (-2 \log (\lambda(t))-3)}{4 \lambda(t)}\\
&+\frac{1}{24}   \partial_{t}^{2} \left(\lambda(t) \lambda''(t)\right) \lambda(t) \left(12 \log ^2(\lambda(t))+18 \log (\lambda(t))+\pi ^2+21\right)\\
&+\frac{1}{24} \partial_{t}^{2}  \left(\lambda(t) \lambda''(t)\right)  \lambda(t) \left(8 \log ^3(\lambda(t))+12 \log ^2(\lambda(t))+\left(36+2 \pi ^2\right) \log (\lambda(t))+\pi ^2+12 \zeta (3)+6\right)\\
&-\frac{1}{24} \lambda(t) \left(\partial_{t}^{2}  \left(\lambda(t) \lambda''(t)\right)-\lambda(t) \lambda''''(t)\right) \begin{aligned}[t]&\left(-8 \log ^3(\lambda(t))-24 \log ^2(\lambda(t))-\left(36+2 \pi ^2\right) \log (\lambda(t))\right.\\
&\left.-2 \pi ^2-12 \zeta (3)-27\right)\end{aligned}\\
&+\frac{ (-\log (\lambda(t))-1) \left(\lambda(t) c_{1}'(t) \lambda'(t)+\frac{1}{2} c_{1}(t) \left(\lambda(t) \lambda''(t)-2 \lambda'(t)^2\right)\right)}{\lambda(t)}-\frac{1}{4}  \lambda(t) c_{1}''(t)\\
&+\frac{1}{4} \lambda(t)  (-2 \log (\lambda(t))-1) c_{1}''(t)\\
&+\frac{\lambda''(t)}{72} \lambda'(t)^2  \left(48 \log ^3(\lambda(t))+252 \log ^2(\lambda(t))+\left(360+12 \pi ^2\right) \log (\lambda(t))+17 \pi ^2+72 \zeta (3)+255\right) \\
&+\frac{3  \left(5 \lambda(t) \lambda'(t)^2 \lambda''(t)-2 \lambda'(t)^4\right)}{8 \lambda(t)}-\frac{ \left(\lambda(t) c_{1}'(t) \lambda'(t)+\frac{1}{2} c_{1}(t) \left(\lambda(t) \lambda''(t)-2 \lambda'(t)^2\right)\right)}{2 \lambda(t)}\\
&+\frac{(4 \log (\lambda(t))+5)  \left(2 \lambda'(t)^4-7 \lambda(t) \lambda''(t) \lambda'(t)^2+6 \lambda(t)^2 \lambda'''(t) \lambda'(t)+4 \lambda(t)^2 \lambda''(t)^2\right)}{16 \lambda(t)}\\
&+\frac{1}{96}  \lambda(t)^2 \left(24 \log ^2(\lambda(t))+108 \log (\lambda(t))-2 \pi ^2+111\right) \lambda''''(t)+\frac{1}{48}  \lambda(t)^2 \left(\pi ^2-12 \log ^2(\lambda(t))\right) \lambda''''(t)\\
&+\frac{1}{24}  \lambda(t) \left(12 \log ^2(\lambda(t))+42 \log (\lambda(t))+\pi ^2+39\right) \left(\partial_{t}^{2} \left(\lambda(t) \lambda''(t)\right)-\lambda(t) \lambda''''(t)\right)
\end{split}\end{equation}
\begin{equation}\begin{split}F_{2,1}(t)&=-\frac{1}{2}  \log (\lambda(t)) \lambda(t)^2 \left(\partial_{t}^{2} \left(\frac{\lambda'(t)^2}{2 \lambda(t)}+(\log (\lambda(t))+1) \lambda''(t)\right)-\log (\lambda(t)) \lambda''''(t)\right)\\
&-\left(c_{0}  g_{2}''(t)-\frac{1}{2}  g_{0}''(t)-\frac{1}{4}  (\log (4)-1) g_{0}''(t)-\frac{1}{2}  \log (t) g_{0}''(t)-\frac{1}{16} \left(\pi ^2-12\right)  g_{1}''(t)\right)\\
&-\left(-\left(\frac{\log ^2(t)}{4}+\frac{1}{2} \log (2) \log (t)\right) g_{1}''(t)-\frac{1}{24}  \left(-2 \pi ^2+15+6 \log ^2(2)\right) g_{1}''(t)\right)\\
&-\left(-\frac{1}{16}  \left(-28 \zeta (3)+\pi ^2+28\right) g_{2}''(t)-\frac{1}{24}  \log (t) \left(4 \log (t) (\log (t)+\log (8))-\pi ^2+12 \log ^2(2)\right) g_{2}''(t)\right)\\
&-\left(-\frac{1}{8} \left(\pi ^2-12\right)  f_{1}''(t)-2 f_{1}''(t) \left(\frac{1}{4}  \log ^2(t)+\frac{1}{2}  \log (2) \log (t)\right)-\frac{1}{2}  h_{0}''(t)\right)\\
&-\left(-\frac{1}{4}  (\log (4)-1) h_{0}''(t)-\frac{1}{2}  \log (t) h_{0}''(t)-\frac{1}{12} \left(-2 \pi ^2+15+6 \log ^2(2)\right) f_{1}''(t)\right)\\
&+\frac{1}{2}\left(\frac{1}{3}  \lambda(t)^2 \left(-24 j_{1} \log (\lambda(t))-12 j_{1}-12 j_{2} \log ^2(\lambda(t))-12 j_{2} \log (\lambda(t))-\pi ^2 j_{2}\right)\right)\\
&+\frac{1}{2}\left(-2  \left(2 j_{1} \lambda(t)^2-j_{2} \lambda(t)^2\right)\right)\end{split}\end{equation}
We quickly record some symbol-type estimates on $F_{3}(t)$. 
\begin{lemma} \label{f3lemma}There exists $C>0$ such that, for $0 \leq k \leq 3 $,
\begin{equation}\label{f3symb}t^{k}|F_{3}^{(k)}(t)| \leq \frac{C \lambda(t)^{2}}{t^{4}} \sup_{x \in [100,t]}\left(\lambda(x) \log(x)\right) \log^{3}(t)\end{equation}\end{lemma}
\begin{proof}
Using the symbol type estimates on $\lambda$, Lemmas \ref{q41lemma} and \ref{q5lemma}, and integration by parts in the region $\xi t \geq 1$ to treat the fifth term in \eqref{f3def},  we get
$$|F_{3}(t)| \leq \frac{C \lambda(t)^{2}}{t^{4}} \sup_{x \in [100,t]}\left(\lambda(x) \log(x)\right) \log^{3}(t)$$
We prove \eqref{f3symb} for $k>0$ as follows. Using the symbol-type estimates on $\lambda$, we see that $c_{1},f_{n},h_{n}, g_{n}$, and $j_{n}$ are all symbols, with estimates implying that their contributions to $F_{3}^{(k)}(t)$ are bounded above by the right-hand side of \eqref{f3symb}. Moreover, 
$$ \int_{t}^{\infty} W_{3}\left(\frac{s-t}{\lambda(t)}\right) \lambda''''(s) ds = \int_{0}^{\infty} W_{3}(w) \lambda''''(\lambda(t)w+t) \lambda(t) dw$$
is a symbol in $t$ because $\lambda(t)$ is. We similarly note that the sum of the third and fourth terms on the right-hand side of \eqref{f3def} is a symbol in $t$. Finally, we study the following integral
$$ I_{7}(t)=\int_{0}^{\infty} d\xi \left(\frac{-8}{\xi \lambda(t)^{2}}+4 \xi K_{2}(\xi \lambda(t))+2 \xi\right) \sin(t\xi) \widehat{v_{2,0}}(\xi)$$
Letting 
$$H(x) = \frac{-8}{x} + 4 x K_{2}(x) + 2x, \quad x>0$$
we get, for $0 \leq k \leq 7$,
$$|H^{(k)}(x)| \leq C \begin{cases} x^{3-k} (1+|\log(x)|), \quad x \leq \frac{1}{2}\\
\begin{cases} x^{1-k}, \quad 0 \leq k \leq 1\\
\frac{1}{x^{1+k}}, \quad 2 \leq k \end{cases}, \quad x> \frac{1}{2}\end{cases}$$
Then, by Lemma \eqref{v20ests} and inspection of the following formula, $I_{7}(t)$ is a symbol in $t$ with $|t^{k}I_{7}^{(k)}(t)|$ bounded above by the right-hand side of \eqref{f3symb} (for $0 \leq k \leq 3$).
$$I_{7}(t) = \frac{1}{\lambda(t)} \int_{0}^{\infty} \frac{d \omega}{t} H\left(\frac{\omega \lambda(t)}{t}\right) \sin(\omega) \widehat{v_{2,0}}(\frac{\omega}{t})$$
 \end{proof}
Now, we can choose the free wave, $v_{2,2}$ from \eqref{v22def},  so that the leading part of $v_{2,2}(t,r)$ in the region $r \sim g(t)\lambda(t)$ is exactly equal to $u_{e,3}-u_{w,3}$. As in \eqref{v2}, the leading part of $v_{2,2}(t,r)$ in the matching region is given by
\begin{equation}\label{g3def}\frac{-r}{4}\left(-2 \int_{0}^{\infty} \xi \sin(t\xi) \widehat{v_{2,3}}(\xi) d\xi\right):=\frac{-r}{4} G_{3}(t)\end{equation}
and we choose the initial velocity, $v_{2,3}$ by the following, which is the analog of \eqref{v20eqn}.
\begin{equation}\label{v23def}\widehat{v_{2,3}}(\xi) = \frac{-1}{\pi \xi} \int_{0}^{\infty} (-4F_{3}(t)) \psi_{2}(t) \sin(t\xi) dt, \quad \psi_{2}(x) = \begin{cases} 0, \quad x \leq T_{2}\\
1, \quad x \geq 2 T_{2}\end{cases}, \quad 0 \leq \psi_{2}(x) \leq 1\end{equation}
(We recall that $T_{2}$ is defined in \eqref{t2constraint}). Note that the sine transform inversion formula implies that
\begin{equation}\label{g3f3}G_{3}(t) = -4 F_{3}(t)\psi_{2}(t),\quad  t>0\end{equation}
Up to this  point, all of our computations and estimates were valid for all $t \geq T_{0}$ for any $T_{0} \geq T_{2}$. In particular, they are valid for all $t \geq T_{2}$. At this stage, we restrict $T_{0}$ so that $T_{0} \geq 2T_{2}$ but is otherwise arbitrary. Thus, we have $G_{3}(t) = -4 F_{3}(t)$ for $t \geq T_{0}$. We start with some estimates on $\widehat{v_{2,3}}(\xi)$:
\begin{lemma} For $0 \leq k \leq 3$, 
 \begin{equation}\label{v23ests}|\xi^{k}\partial_{\xi}^{k}\widehat{v_{2,3}}(\xi)| \leq \begin{cases} C , \quad \xi < \frac{1}{100}\\
\frac{C \xi^{k}}{\xi^{4}}, \quad \xi \geq \frac{1}{100}\end{cases}\end{equation}\end{lemma}
\begin{proof} We use the same procedure as in Lemma \ref{v20ests}. For instance, if $\xi < \frac{1}{100}$, then,
$$\widehat{v_{2,3}}(\xi) = \frac{-1}{\pi \xi} \int_{0}^{\frac{1}{\xi}} (-4F_{3}(t)) \psi_{2}(t) \sin(t\xi) dt - \frac{1}{\pi \xi} \int_{\frac{1}{\xi}}^{\infty} (-4F_{3}(t)) \psi_{2}(t) \sin(t\xi) dt$$
So,
\begin{equation}\begin{split}&|\widehat{v_{2,3}}(\xi)|\\
&\leq C \int_{0}^{\frac{1}{\xi}} \frac{\lambda(t)^{2}\sup_{x \in [100,t]}\left(\lambda(x) \log(x)\right) \log^{3}(t) \mathbbm{1}_{\{t \geq T_{2}\}} dt}{t^{3}} \\
&+ \frac{C}{\xi} \int_{\frac{1}{\xi}}^{\infty} \frac{\lambda(t)^{2} \sup_{x \in [100,t]}\left(\lambda(x) \log(x)\right) \log^{3}(t) \mathbbm{1}_{\{t \geq T_{2}\}} dt}{t^{4}}\\
&\leq C\end{split}\end{equation}
where we used \eqref{lambdacomparg}. Also, the estimates established in this lemma are not rapidly decaying because we only estimated up to 3 derivatives of $F_{3}$ in Lemma \ref{f3lemma}. This was due to the fact that the seventh term of \eqref{f3def} was estimated using Lemma \ref{q41lemma}.  \end{proof} 
Now, we can estimate $v_{2,2}$
\begin{lemma} \label{v22lemma} We have the following estimates. For $0 \leq j \leq 2, \quad 0 \leq k \leq 1$,
\begin{equation}\label{v22ests}|\partial_{t}^{j}\partial_{r}^{k}v_{2,2}(t,r)| \leq \begin{cases}C r^{1-k} \frac{\lambda(t)^{2}}{t^{4+j}} \sup_{x \in [100,t]}\left(\lambda(x) \log(x)\right) \log^{3}(t), \quad r \leq \frac{t}{2}\\
\frac{C}{\langle t-r\rangle^{5/2+j+k} \sqrt{t}} \sup_{x \in [100,t]}\left(\lambda(x) \log(x)\right) \sup_{x \in [100,t]}\left(\lambda(x)^{2}\right)\log^{3}(t), \quad t>r>\frac{t}{2}\\
\frac{C}{\sqrt{r}\langle t-r \rangle^{\frac{1}{2}+j+k}}, \quad r>t, \quad j+k \leq 2\end{cases}
\end{equation}

$$|\partial_{r}^{2}v_{2,2}(t,r)| \leq \begin{cases}\frac{C r \lambda(t)^{2}}{t^{6}} \sup_{x \in [100,t]}\left(\lambda(x) \log(x)\right) \log^{3}(t), \quad r \leq \frac{t}{2}\\
\frac{C}{\langle t-r\rangle^{9/2} \sqrt{t}} \sup_{x \in [100,t]}\left(\lambda(x) \log(x)\right) \sup_{x \in [100,t]}\left(\lambda(x)^{2}\right)\log^{3}(t), \quad t>r>\frac{t}{2}\\
\frac{C}{\sqrt{r}\langle t-r \rangle^{\frac{5}{2}}} , \quad r>t\end{cases}$$
For all $0 \leq j+k \leq 3$,
\begin{equation}\label{v22sqrtest}|\partial_{t}^{j}\partial_{r}^{k}v_{2,2}(t,r)| \leq \frac{C}{\sqrt{r}}, \quad r \geq \frac{t}{2}\end{equation}
Finally, for $0 \leq j,k \leq 1$ or $k=2,j=0$,
\begin{equation}\label{v22minusrf3}|\partial_{t}^{j}\partial_{r}^{k}\left(v_{2,2}(t,r)-r F_{3}(t)\right)| \leq \frac{C r^{3-k} \lambda(t)^{2} \sup_{x \in [100,t]}\left(\lambda(x) \log(x)\right) \log^{3}(t)}{t^{6+j}}, \quad r \leq \frac{t}{2}\end{equation}
\end{lemma}
\begin{proof} In the region $r \leq \frac{t}{2}$, we use \eqref{lambdacomparg}, Lemma \ref{f3lemma}, and the same procedure used to estimate $v_{2}$. Next, we treat the region $t > r >\frac{t}{2}$. Here, we use
\begin{equation}\label{v22g3formula} v_{2,2}(t,r) = \frac{-r}{4\pi} \int_{0}^{\pi} \sin^{2}(\theta)\left(G_{3}(t+r\cos(\theta))+G_{3}(t-r\cos(\theta))\right)d\theta\end{equation}
Note that $t\pm r\cos(\theta) >0$ since $t>r$, so $G_{3}(t\pm r\cos(\theta)) = -4 \left(F_{3}\cdot \psi_{2}\right)(t\pm r \cos(\theta))$, by \eqref{g3f3}. It suffices to treat the following integral
\begin{equation}\begin{split}&|\frac{-r}{4\pi} \int_{0}^{\pi}\sin^{2}(\theta) G_{3}(t+r\cos(\theta)) d\theta| \leq C r \sup_{x \in [100,t]}\left(\lambda(x) \log(x)\right) \sup_{x \in [100,t]}\left(\lambda(x)^{2}\right) \log^{3}(t) \int_{0}^{\pi} \frac{\sin^{2}(\theta) d\theta}{(t+r\cos(\theta))^{4}}\end{split}\end{equation}
where we use the fact that $\psi_{2}(x) =0, \quad x < 100$. We then use Cauchy's residue theorem, recalling that $t>r$, to get
\begin{equation}\label{tplusrcosthetaint}\int_{0}^{\pi} \frac{\sin^{2}(\theta) d\theta}{(t+r\cos(\theta))^{4}} = \frac{1}{2}\int_{0}^{2\pi} \frac{\sin^{2}(\theta) d\theta}{(t+r\cos(\theta))^{4}} = \frac{\pi  t}{2 \left(t^2-r^2\right)^{5/2}}\end{equation}
and this gives
$$|v_{2,2}(t,r)| \leq \frac{C}{(t-r)^{5/2} \sqrt{t}} \sup_{x \in [100,t]}\left(\lambda(x) \log(x)\right)\sup_{x \in [100,t]}\left(\lambda(x)^{2}\right) \log^{3}(t), \quad t>r>\frac{t}{2}$$
The higher derivatives are treated similarly. Note that $G_{3}(t) = -2 \int_{0}^{\infty} \xi \sin(t\xi) \widehat{v_{2,3}}(\xi) d\xi$ is an odd function of $t$, and 
$$G_{3}(t) = -4 F_{3}(t)\psi_{2}(t),\quad t>0$$
Also, if $r >t$, then, the argument of $G_{3}$ is negative in a region of the $\theta$ integral in \eqref{v22g3formula}. So, the procedure of estimating $v_{2,2}(t,r)$ in the region $r\geq t$ using \eqref{v22g3formula} is more involved than in the region $r <t$. Instead, we simply use
$$v_{2,2}(t,r) = \int_{0}^{\infty} \widehat{v_{2,3}}(\xi) J_{1}(r\xi) \sin(t\xi) d\xi$$
and proceed exactly as in the proof of Lemma \ref{v2estlemma}. We also do the same procedure as in \eqref{v2sqrtest} to finish the proof of all of the estimates in the lemma statement except for \eqref{v22minusrf3}. The estimate \eqref{v22minusrf3} follows from
\begin{equation}\begin{split}& v_{2,2}(t,r) -r F_{3}(t) = \frac{-r}{4 \pi} \int_{0}^{\pi} \sin^{2}(\theta)\left(G_{3}(t_{+})-G_{3}(t)-r \cos(\theta) G_{3}'(t)+G_{3}(t_{-})-G_{3}(t)+r \cos(\theta) G_{3}'(t)\right) d\theta\end{split}\end{equation}
where
$$t_{\pm} = t \pm r \cos(\theta)$$
 \end{proof}
Finally, we are ready to obtain the main result of this section, Proposition \ref{thirdorderprop}. We start with the decomposition
\begin{equation}\begin{split} &u_{ell,2}(t,r)-v_{ell,2,0,main}(t,\frac{r}{\lambda(t)})-\left(v_{ex}(t,r)-v_{ex,cont}(t,r)+u_{w,2}(t,r)-u_{w,2,ell,0,cont}(t,r)+v_{2,2}(t,r)\right)\\
&= u_{ell,2}(t,r)-v_{ell,2,0,main}(t,\frac{r}{\lambda(t)})-u_{e,3}(t,r)\\
&-\left(v_{ex,sub,0}-v_{ex,sub,0,cont}+q_{5}-q_{5,0}+v_{ex,ell}-v_{ex,cont}+u_{w,2,sub,0}-u_{w,2,sub,0,cont}\right)\\
&-\left(q_{4,2}+(ellsoln-r Q(t))\psi_{\leq 1}(r-t)+u_{w,2,ell}-u_{w,2,ell,0}-(u_{w,2,ell}(t,r)-u_{w,2,ell,0}(t,r))_{princ}\right)\\
&-\left(q_{4,1}-q_{4,1,0}+u_{w,2,ell,0}-u_{w,2,ell,0,cont}+v_{2,2}-r F_{3}(t)+ellsoln_{2}(t,r)\psi_{\leq 1}(r-t)\right)\end{split}\end{equation}
and use Lemmas \ref{uw2minuselllemma}, \ref{uw2sub0leading}, \ref{vexsub0leading}, \ref{ellminusell0} - \ref{q42enest}, \ref{q41lemma} - \ref{uell2minusvell20mainminusue3lemma}, \ref{v22lemma}.   \end{proof}
\subsection{Joining the small $r$ and large $r$ solutions}
Define
\begin{equation}\label{uewdef}u_{e}(t,r) = u_{ell}(t,r)+u_{ell,2}(t,r), \quad u_{wave}(t,r) = u_{w}(t,r)+u_{w,2}(t,r)+v_{2,2}(t,r)\end{equation}
where we recall that $u_{w}$ is defined in \eqref{uwdef}. The point of the matching done in the previous sections is that we can transition between $u_{e}(t,r)$, which is accurate for small $r$ and $u_{wave}(t,r)$, which is accurate for large $r$ with an expression of the form
$$ \chi_{\leq 1}\left(\frac{r}{g(t) \lambda(t)}\right) u_{e}(t,r) + \left(1-\chi_{\leq 1}\left(\frac{r}{g(t)\lambda(t)}\right)\right)u_{wave}(t,r)$$
for an appropriate choice of $\chi_{\leq 1}$ (which will be defined later), and not incur large error terms when derivatives act on $\chi_{\leq 1}$, upon substitution of the above expression into the left-hand side of \eqref{u1eqn}. As mentioned previously, the basic idea of this procedure is inspired by matched asymptotic expansions. (The books \cite{bo}, \cite{n} have more information about matched asymptotic expansions for ODEs). In addition, the idea is motivated by the fact that the correction denoted by $v_{1}$, defined in (4.12), pg. 11 of \cite{wm}, has a leading order cancellation with $v_{2}$, defined in (4.63), pg. 18 of \cite{wm} near the origin, and simultaneously a leading order cancellation with $Q_{1}(\frac{r}{\lambda(t)})-\pi$ for large $r$, reminiscent of procedures used to match asymptotic expansions in various regions.\\
\\
We now define a cutoff $\chi_{\geq 1}\in C^{\infty}([0,\infty))$, with the following properties.
\begin{lemma}\label{chilemma}[Properties of $\chi_{\geq 1}$] There exists a function $\chi_{\geq 1}\in C^{\infty}([0,\infty))$ satisfying 
\begin{equation}\label{chiconds}\chi_{\geq 1}(x) = \begin{cases} 0, \quad x \leq 1\\
1, \quad x \geq 2\end{cases}, \text{ } \int_{0}^{\infty} \frac{\chi_{\geq 1}(x)}{x^{3}}dx= \int_{0}^{\infty} x^{3} \left(1-\chi_{\geq 1}(x)\right) dx = \int_{0}^{\infty} x^{3} \log(x) \left(1-\chi_{\geq 1}(x)\right) dx=0 \end{equation}
In particular, this implies that, for all $k \geq 0$, there exists $C_{k}>0$ such that
$$|\widehat{\frac{\chi_{\geq 1}(\cdot)}{\left(\cdot\right)^{5}}}(\eta)| \leq \begin{cases} C \eta^{3} \langle \log(\eta)\rangle, \quad \eta < 1\\
\frac{C_{k}}{\eta^{k}}, \quad \eta \geq 1\end{cases}$$
\end{lemma} 
\begin{proof}
A direct computation shows that
$$\chi_{\geq 1}(x) = \left(\int_{-\infty}^{\infty} \alpha(y) dy\right)^{-1} \int_{-\infty}^{x}\alpha(s) ds$$
where
$$\alpha(s) = \frac{1}{s^{4}} \frac{d}{ds}\left(s \frac{d}{ds}\left(s^{7}\phi'(s)\right)\right), \quad \phi(x) = \begin{cases} e^{-\frac{1}{1-(2x-3)^{2}}}, \quad |x-\frac{3}{2}| < \frac{1}{2}\\
0, \quad |x-\frac{3}{2}| \geq \frac{1}{2}\end{cases}$$
satisfies \eqref{chiconds}, given that
$$\alpha \in C^{\infty}_{c}([1,2]), \quad \int_{0}^{\infty} \alpha(s)\frac{ds}{s^{2}}= \int_{0}^{\infty} \alpha(s) s^{4} ds = \int_{0}^{\infty} \alpha(s) s^{4} \log(s) ds =0, \quad \int_{-\infty}^{\infty}\alpha(s)ds\neq0$$ 
To verify the stated estimates on $\widehat{\frac{\chi_{\geq 1}(\cdot)}{\left(\cdot\right)^{5}}}(\eta)$, we start with the case of $\eta >1$. We have
$$\widehat{\frac{\chi_{\geq 1}(\cdot)}{\left(\cdot\right)^{5}}}(\eta) = \int_{0}^{\infty} \frac{\chi_{\geq 1}(x)}{x^{5}} J_{1}(\eta x) x dx = \int_{0}^{\infty} \frac{\chi_{\geq 1}(x)}{x^{5}} \frac{H_{1}\left(J_{1}(\eta x)\right)}{\eta^{2}} x dx$$
where
$$H_{1}(f)(x) = -f''(x) - \frac{1}{x}f'(x) + \frac{f(x)}{x^{2}}$$
Then, 
$$|\widehat{\frac{\chi_{\geq 1}(\cdot)}{\left(\cdot\right)^{5}}}(\eta)| \leq \frac{C_{k}}{\eta^{k}}, \quad \eta > 1$$
follows from repeated integration by parts, noting that there are no non-zero boundary terms obtained in the process. For $\eta <1$, we use the integral condition on $\chi_{\geq 1}$. In particular, we have
$$\widehat{\frac{\chi_{\geq 1}(\cdot)}{\left(\cdot\right)^{5}}}(\eta) = \frac{\eta}{2} \int_{0}^{\frac{1}{\eta}} \frac{\chi_{\geq 1}(x)}{x^{3}} dx + O\left(\eta^{3} \int_{0}^{\frac{1}{\eta}} \frac{|\chi_{\geq 1}(x)|}{x} dx\right) + O\left(\int_{\frac{1}{\eta}}^{\infty} \frac{|\chi_{\geq 1}(x)|}{x^{4}} dx\right)$$
which implies the estimate in the lemma statement.  \end{proof} 
We recall that $u_{e}$ and $u_{wave}$ are defined in \eqref{uewdef}, and define $\chi_{\leq 1}(x) = 1-\chi_{\geq 1}(x)$, and
\begin{equation}\label{ucdef}u_{c}(t,r) = \chi_{\leq 1}(\frac{r}{\lambda(t) g(t)}) u_{e}(t,r) + \left(1-\chi_{\leq 1}(\frac{r}{\lambda(t) g(t)})\right)u_{wave}(t,r)\end{equation}
Let $h(t) = \lambda(t) g(t)$. The error term of $u_{c}$ in solving the linear PDE \eqref{u1eqn} is
\begin{equation}\label{uceqn}\begin{split} &-\left(\left(-\partial_{tt}+\partial_{rr}+\frac{1}{r}\partial_{r}-\frac{\cos(2Q_{1}(\frac{r}{\lambda(t)}))}{r^{2}}\right) u_{c} -\partial_{t}^{2}Q_{1}(\frac{r}{\lambda(t)})\right) \\
&= \chi_{\leq 1}(\frac{r}{h(t)}) \partial_{tt}u_{ell,2}(t,r) + \left(1-\chi_{\leq 1}(\frac{r}{h(t)})\right)\left(\frac{\cos(2Q_{1}(\frac{r}{\lambda(t)}))-1}{r^{2}}\right)\left(v_{ex}(t,r)+u_{w,2}(t,r)+v_{2,2}(t,r)\right)\\
&-\left(u_{e}-u_{wave}\right)\left(-\chi_{\leq 1}''(\frac{r}{h(t)}) \frac{r^{2} h'(t)^{2}}{h(t)^{4}} - \chi_{\leq 1}'(\frac{r}{h(t)})\left(-\frac{r h''(t)}{(h(t))^{2}} + \frac{2 r h'(t)^{2}}{h(t)^{3}}\right) + \frac{1}{r} \frac{\chi_{\leq 1}'(\frac{r}{h(t)})}{h(t)} + \frac{\chi_{\leq 1}''(\frac{r}{h(t)})}{h(t)^{2}}\right)\\
&- 2 \chi_{\leq 1}'(\frac{r}{h(t)}) \frac{r h'(t)}{h(t)^{2}} \partial_{t}\left(u_{e}-u_{wave}\right) - \frac{2}{h(t)} \chi_{\leq 1}'(\frac{r}{h(t)}) \partial_{r}(u_{e}-u_{wave}) \end{split}\end{equation}
Some of the error terms in \eqref{uceqn} are already perturbative, while some will need additional corrections. We first record estimates on those terms which are already perturbative. We recall the definitions of $v_{ex,ell}$ in \eqref{vexell}, and $u_{w,2}$ in \eqref{uw2def}, and start with the following lemma.
\begin{lemma}\label{eexellew2estlemma} If
$$e_{ex,ell}(t,r) = \left(1-\chi_{\leq 1}(\frac{r}{\lambda(t)g(t)})\right)\left(\frac{\cos(2Q_{1}(\frac{r}{\lambda(t)}))-1}{r^{2}}\right) v_{ex,ell}(t,r)$$
and
$$e_{w,2}(t,r) = \left(1-\chi_{\leq 1}(\frac{r}{\lambda(t)g(t)})\right)\left(\frac{\cos(2Q_{1}(\frac{r}{\lambda(t)}))-1}{r^{2}}\right) u_{w,2}(t,r)$$
then, for $k=0,1$,
\begin{equation}\label{eexellest}||L_{\frac{1}{\lambda(t)}}^{k}(e_{ex,ell})(t,r)||_{L^{2}(r dr)} \leq \frac{C \lambda(t)^{1-k} \log(t)}{t^{2} g(t)^{4+k}}\end{equation}
\begin{equation}\label{ew2est} ||L_{\frac{1}{\lambda(t)}}^{k}(e_{w,2})(t,r)||_{L^{2}(r dr)} \leq \frac{C \lambda(t)^{2-k}}{t^{4}} \frac{\left(1+\lambda(t)\right) \log^{5}(t) \sup_{x \in [100,t]}\left(\lambda(x) \log(x)\right)}{g(t)^{2}} + \frac{C \lambda(t)^{1-k} \log^{2}(t)}{g(t)^{4+k} t^{2}}\end{equation}
\end{lemma}
\begin{proof} The estimate in \eqref{eexellest} follows from straightforward estimation of the expression \eqref{vexell} for $v_{ex,ell}$. To establish \eqref{ew2est}, we write $u_{w,2} = u_{w,2}-u_{w,2,ell}+u_{w,2,ell}-u_{w,2,ell,0}+u_{w,2,ell,0}$, and use Lemma \ref{uw2minuselllemma} as well as \eqref{uw2ell0exp}.  \end{proof}
We recall the definitions of $v_{ell,sub}$ from \eqref{vellsubdef} and $v_{ell,sub,cont}$ from \eqref{vellsubcontdef}, and define 
 \begin{equation}\label{vellsub1def}v_{ell,sub,1}(t,R) = v_{ell,sub}(t,R)-v_{ell,sub,cont}(t,R\lambda(t))\end{equation}
We will now describe the difference $u_{e}-u_{wave}$. We have
$$u_{e}(t,r)=u_{ell}(t,r)+u_{ell,2}(t,r), \quad u_{wave}(t,r)=w_{1}(t,r)+v_{ex}(t,r)+v_{2}(t,r)+u_{w,2}(t,r)+v_{2,2}(t,r)$$
We recall the first order matching \eqref{firstmatch}, which says that
$$u_{ell}(t,r) - v_{ell,sub}(t,\frac{r}{\lambda(t)}) = v_{2}(t,r)-v_{2,sub}(t,r)+w_{1}(t,r)-w_{1,sub}(t,r)$$
From the second order matching, we have
$$v_{ell,2,0,main}(t,\frac{r}{\lambda(t)})=w_{1,cubic,main}(t,r)+v_{2,cubic,main}(t,r)$$
and 
$$v_{ell,sub,cont}(t,r)=v_{ex,cont}(t,r)+u_{w,2,ell,0,cont}(t,r)$$
From the third order matching, we have the estimates from Proposition \ref{thirdorderprop}. These give
\begin{equation}\begin{split}&u_{e}(t,r)-u_{wave}(t,r)\\
&=u_{ell,2}(t,r)-v_{ell,2,0,main}(t,\frac{r}{\lambda(t)})-\left(u_{w,2}(t,r)-u_{w,2,ell,0,cont}(t,r)+v_{2,2}(t,r)+v_{ex}-v_{ex,cont}-q_{4,2}\right)\\
&-q_{4,2}(t,r)+v_{ell,sub,1}(t,\frac{r}{\lambda(t)})\\
&-\left(w_{1,sub}(t,r)-w_{1,cubic,main}(t,r)-\frac{r^{5} \lambda^{(6)}(t)}{576}\left(5+\log(8)\right)\right.\\
&\left.-\frac{r^{5}}{192} \left(\int_{t}^{2t} \frac{\left(\lambda^{(6)}(s)-\lambda^{(6)}(t)\right)ds}{s-t} + \lambda^{(6)}(t) \log(\frac{t}{r})+\int_{2t}^{\infty} \frac{\lambda^{(6)}(s) ds}{s-t}\right)\right)\\
&- \frac{r^{5}}{192} \left(\int_{t}^{2t} \frac{\left(\lambda^{(6)}(s)-\lambda^{(6)}(t)\right)ds}{s-t} + \lambda^{(6)}(t) \log(\frac{t}{r})+\int_{2t}^{\infty} \frac{\lambda^{(6)}(s) ds}{s-t}\right)\\
&-\frac{r^{5} \lambda^{(6)}(t)}{576}\left(5+\log(8)\right)-\left(v_{2,sub}(t,r)-v_{2,cubic,main}(t,r)+\frac{r^{5}}{768} F^{(4)}(t)\right)+\frac{r^{5}}{768} F^{(4)}(t)\end{split}\end{equation}
where any function appearing without arguments is evaluated at the point $(t,r)$. We also recall that $F$ is defined in \eqref{Fdef} (see also \eqref{v20eqn}). Using
\begin{equation}\begin{split} F^{(4)}(t) &= 4\left(\left(\log(2)-\frac{1}{2}\right) \lambda^{(6)}(t) + \int_{t}^{2t} \frac{\left(\lambda^{(6)}(s)-\lambda^{(6)}(t)\right)}{s-t} ds + \lambda^{(6)}(t) \log(t)\right.\\
&\left. - \partial_{t}^{4}\left(\lambda''(t) \log(\lambda(t))\right) -\partial_{t}^{4}\left(\frac{\lambda'(t)^{2}}{2\lambda(t)}\right) + \int_{2t}^{\infty} \frac{\lambda^{(6)}(s)}{s-t} ds\right)\end{split}\end{equation}
we get
\begin{equation}\begin{split}&u_{e}(t,r)-u_{wave}(t,r)\\
&=u_{ell,2}(t,r)-v_{ell,2,0,main}(t,\frac{r}{\lambda(t)})-\left(u_{w,2}(t,r)-u_{w,2,ell,0,cont}(t,r)+v_{2,2}(t,r)+v_{ex}-v_{ex,cont}-q_{4,2}\right)\\
&-q_{4,2}(t,r)+v_{ell,sub,1}(t,\frac{r}{\lambda(t)})\\
&-\left(w_{1,sub}(t,r)-w_{1,cubic,main}(t,r)-\frac{r^{5} \lambda^{(6)}(t)}{576}\left(5+\log(8)\right)\right.\\
&\left.-\frac{r^{5}}{192} \left(\int_{t}^{2t} \frac{\left(\lambda^{(6)}(s)-\lambda^{(6)}(t)\right)ds}{s-t} + \lambda^{(6)}(t) \log(\frac{t}{r})+\int_{2t}^{\infty} \frac{\lambda^{(6)}(s) ds}{s-t}\right)\right)\\
&-\left(v_{2,sub}(t,r)-v_{2,cubic,main}(t,r)+\frac{r^{5}}{768} F^{(4)}(t)\right)\\
&-\frac{r^{5}}{1152}\left(6\left(\partial_{t}^{4}\left(\lambda''(t)\log(\lambda(t))\right)+\partial_{t}^{4}\left(\frac{\lambda'(t)^{2}}{2\lambda(t)}\right)\right)+(13-6 \log(r))\lambda^{(6)}(t)\right)\end{split}\end{equation}
All of the terms (and sufficiently many of their derivatives) appearing in this expression have been estimated already, except for the last term (which will be eliminated with another correction, utilizing the properties of $\chi_{\geq 1}$) and $v_{ell,sub,1}(t,\frac{r}{\lambda(t)})$. We now record estimates on $v_{ell,sub,1}(t,\frac{r}{\lambda(t)})$, and another estimate on $v_{ex,sub}$ which will be useful later on.
\begin{lemma}\label{estlemma}For $j=0,1$ and $k=0,1,2$, the following estimate is true.
$$|\partial_{t}^{j}\partial_{r}^{k}\left(v_{ell,sub,1}(t,\frac{r}{\lambda(t)})\right)| \leq C \frac{\lambda(t)^{2-k} \log^{2}(t)}{t^{2+j}  g(t)^{3+k}}, \quad g(t)\lambda(t) \leq r \leq 2g(t)\lambda(t)$$
In addition, we  have
\begin{equation}\label{vexalonelgr}|v_{ex}(t,r)| \leq \frac{C \log^{2}(t+r) \lambda(t)^{3}}{\sqrt{r}t^{5/2}}, \quad r \geq \frac{t}{2}\end{equation}
\begin{equation}\label{vexsubsymb}|\partial_{t}^{k}v_{ex,sub}(t,r)| \leq \frac{C r \lambda(t)^{3} \log^{3}(t+r)}{t^{4+k} }, \quad g(t) \lambda(t) \leq r, \quad k=0,1,2\end{equation}
$$|\partial_{r}\partial_{t}^{k}v_{ex,sub}(t,r)| \leq \frac{C \lambda(t)^{3} \log^{3}(t+r)}{t^{4+k} }, \quad g(t) \lambda(t) \leq r, \quad k=0,1$$
$$|\partial_{r}^{2} v_{ex,sub}(t,r)| \leq \frac{C \lambda(t)^{3} \log^{3}(t)}{r t^{4} }, \quad g(t) \lambda(t) \leq r \leq t$$
Finally, we also have the following (non-sharp, but sufficient) estimates
\begin{equation}\label{drvexinftyest}|\partial_{r}v_{ex}(t,r)| \leq \frac{C \lambda(t)^{\frac{5}{2}}}{r^{\frac{3}{2}}t^{2}}, \quad |\partial_{t}v_{ex}(t,r)| \leq \frac{C \lambda(t)^{3/2}}{t r^{3/2}}, \quad r \geq g(t)\lambda(t)\end{equation}
\end{lemma}
\begin{proof}
We directly estimate $v_{ell,sub,1}$ using its definition, \eqref{vellsub1def}, and the explicit formulae for $v_{ell,sub}$ and $v_{ell,sub,cont}$ from \eqref{vellsubdef} and \eqref{vellsubcontdef}, respectively. To obtain \eqref{vexsubsymb}, we start with the definition \eqref{vexsubdef}, use Fubini's theorem to switch the order of $s$ and $\xi$ integrals, and (for $r \geq g(t) \lambda(t)$) divide the $s$ integral into three regions. 
\begin{equation}\label{vexestdecomp}\textbf{a}: t \leq s \leq t+\lambda(t), \quad \textbf{b}: t+\lambda(t) \leq s \leq t+r, \quad \textbf{c}: t+r \leq s < \infty\end{equation}
In each above region of $s$ integration, we then divide the $\xi$ region of integration into four subintervals, based on the scales $\frac{1}{r}, \frac{1}{\lambda(t)}, \frac{1}{s-t}$. For example, in region \textbf{a} of the $s$ integration, we have the four regions of the $\xi$ variable
$$0 < \xi \leq \frac{1}{r}, \quad \frac{1}{r} < \xi \leq \frac{1}{\lambda(t)}, \quad \frac{1}{\lambda(t)} < \xi \leq \frac{1}{s-t}, \quad \frac{1}{s-t} < \xi <\infty$$
In the regions where $\xi \leq \frac{1}{s-t}$, we use 
$$|\sin((t-s)\xi)| \leq C (s-t) \xi$$ 
On the other hand, when $\xi >\frac{1}{s-t}$, we integrate by parts in the $\xi$ variable. To be clear, we show how the $s$ integral over the region \textbf{c} is treated.
\begin{equation}\begin{split}&-\int_{t+r}^{\infty} ds \int_{0}^{\infty} d\xi \frac{\sin((t-s)\xi)}{\xi^{2}} J_{1}(r\xi) \partial_{s}^{2} \widehat{RHS}(s,\xi)\\
&= - \int_{t+r}^{\infty} ds \int_{0}^{\frac{1}{s-t}} d\xi \frac{\sin((t-s)\xi)}{\xi^{2}} J_{1}(r\xi) \partial_{s}^{2} \widehat{RHS}(s,\xi)-\int_{t+r}^{\infty} ds \int_{\frac{1}{s-t}}^{\infty} d\xi \frac{\sin((t-s)\xi)}{\xi^{2}} J_{1}(r\xi) \partial_{s}^{2} \widehat{RHS}(s,\xi)\end{split}\end{equation}
The integral over the region $\xi \geq \frac{1}{s-t}$ is then further treated as
\begin{equation}\begin{split}&-\int_{t+r}^{\infty} ds \int_{\frac{1}{s-t}}^{\infty} d\xi \frac{\sin((t-s)\xi)}{\xi^{2}} J_{1}(r\xi) \partial_{s}^{2} \widehat{RHS}(s,\xi)\\
&=-\int_{t+r}^{\infty} ds J_{1}(\frac{r}{s-t}) \partial_{1}^{2}\widehat{RHS}(s,\frac{1}{s-t}) (s-t)^{2} \frac{\cos(1)}{(t-s)}-\int_{t+r}^{\infty} ds \int_{\frac{1}{s-t}}^{\frac{1}{r}} \frac{\cos((t-s)\xi)}{(t-s)} \partial_{\xi} \left(\frac{J_{1}(r\xi) \partial_{s}^{2}\widehat{RHS}(s,\xi)}{\xi^{2}}\right)d\xi\\
&-\int_{t+r}^{\infty} ds \int_{\frac{1}{r}}^{\frac{1}{\lambda(t)}} \frac{\cos((t-s)\xi)}{(t-s)} \partial_{\xi} \left(\frac{J_{1}(r\xi) \partial_{s}^{2}\widehat{RHS}(s,\xi)}{\xi^{2}}\right)d\xi\\
&-\int_{t+r}^{\infty} ds \int_{\frac{1}{\lambda(t)}}^{\infty} \frac{\cos((t-s)\xi)}{(t-s)} \partial_{\xi} \left(\frac{J_{1}(r\xi) \partial_{s}^{2}\widehat{RHS}(s,\xi)}{\xi^{2}}\right)d\xi\end{split}\end{equation}
Finally, we use the formula for $\widehat{RHS}(s,\xi)$, namely \eqref{rhshat}, to get, for $0 \leq k \leq 5$
\begin{equation}\label{ds2rhsest}\xi s^{k}|\partial_{\xi}\partial_{s}^{k}\widehat{RHS}(s,\xi)|+s^{k} |\partial_{s}^{k} \widehat{RHS}(s,\xi)| \leq  \begin{cases}\frac{C\xi \lambda(s)^{3} (1+|\log(\xi \lambda(s))|)}{s^{2} }, \quad \xi \lambda(s) < \frac{1}{2}\\
\frac{C \lambda(s)^{2}}{s^{2}  \xi \lambda(s)}, \quad \xi \lambda(s) > \frac{1}{2}\end{cases}\end{equation}
Then, we use
$$|J_{1}(x)| \leq C \begin{cases} x, \quad x \leq 1\\
\frac{1}{\sqrt{x}}, \quad x > 1\end{cases}$$
as appropriate to estimate each integral above. We get \eqref{vexalonelgr} by using the same procedure as above, except that we don't integrate by parts when $\xi \geq \frac{1}{s-t}$ and $s \leq t+r$, combined with \eqref{vexell}. For \eqref{vexsubsymb}, we use 
$$\partial_{t}^{k}v_{ex,sub}(t,r) =-\int_{0}^{\infty} d\xi J_{1}(r\xi) \int_{t}^{\infty} ds \frac{\sin((t-s)\xi)}{\xi^{2}} \partial_{s}^{2+k}\widehat{RHS}(s,\xi), \quad k=1,2$$ 
and the same procedure used for $k=0$. Finally, to estimate $\partial_{r}v_{ex,sub}(t,r)$, we start with
\begin{equation}\label{drvexsubform}\partial_{r}v_{ex,sub}(t,r) = -\int_{0}^{\infty} d\xi J_{1}'(r\xi) \xi \int_{t}^{\infty} ds \frac{\sin((t-s)\xi)}{\xi^{2}} \partial_{s}^{2}\widehat{RHS}(s,\xi)\end{equation}
We make the same decomposition as described in \eqref{vexestdecomp}. This time, we integrate by parts in $\xi$ (integrating $J_{1}'(r\xi)$) when $t \leq s \leq t+\lambda(t)$, and $\frac{1}{r} \leq \xi \leq \frac{1}{s-t}$. We also integrate by parts in $s$ (integrating $\sin((t-s)\xi)$) when $\xi \geq \frac{1}{s-t}$ and $t+\lambda(t) \leq s $. Next, using
$$\partial_{tr}v_{ex,sub}(t,r) = -\int_{0}^{\infty} d\xi J_{1}'(r\xi) \xi \int_{t}^{\infty} ds \frac{\sin((t-s)\xi)}{\xi^{2}} \partial_{s}^{3}\widehat{RHS}(s,\xi)$$
and the same argument above, we obtain the estimates on $\partial_{tr}v_{ex,sub}(t,r)$ in the lemma statement. Finally, we use the fact that $v_{ex,sub}$ satisfies
$$-\partial_{tt}v_{ex,sub}+\partial_{rr}v_{ex,sub}+\frac{1}{r}\partial_{r}v_{ex,sub}-\frac{v_{ex,sub}}{r^{2}} = \partial_{t}^{2} v_{ex,ell}(t,r)$$
to estimate $\partial_{r}^{2}v_{ex,sub}(t,r)$. To obtain the estimate \eqref{drvexinftyest}, we start with
\begin{equation}\label{drvexsubsplit}\begin{split} \partial_{r}v_{ex,sub}(t,r)&=-\int_{0}^{\infty} d\xi m_{\leq 1}(r\xi) J_{1}'(r\xi) \int_{t}^{\infty} ds \frac{\sin((t-s)\xi)}{\xi} \partial_{s}^{2} \widehat{RHS}(s,\xi)\\
&-\int_{0}^{\infty} d\xi (1-m_{\leq 1}(r\xi)) J_{1}'(r\xi) \int_{t}^{\infty} ds \frac{\sin((t-s)\xi)}{\xi} \partial_{s}^{2} \widehat{RHS}(s,\xi)\end{split}\end{equation}
where $m_{\leq 1}$ is defined in \eqref{mleq1def}. For the first term on the right-hand side of \eqref{drvexsubsplit}, we simply directly insert \eqref{ds2rhsest} into the integral, and estimate. For the second term, we integrate by parts, to get
\begin{equation}\begin{split}&-\int_{0}^{\infty} d\xi (1-m_{\leq 1}(r\xi)) J_{1}'(r\xi) \int_{t}^{\infty} ds \frac{\sin((t-s)\xi)}{\xi} \partial_{s}^{2} \widehat{RHS}(s,\xi)\\
&= \frac{1}{r} \int_{t}^{\infty} ds \int_{0}^{\infty} d\xi J_{1}(r\xi) \partial_{\xi}\left(\left(1-m_{\leq 1}(r\xi)\right) \frac{\sin((t-s)\xi)}{\xi} \partial_{s}^{2} \widehat{RHS}(s,\xi)\right)d\xi\end{split}\end{equation}
and then, we directly estimate using \eqref{ds2rhsest}. Next, using \eqref{vexell} and $v_{ex}(t,r)=v_{ex,ell}(t,r)+v_{ex,sub}(t,r)$, we get 
$$|\partial_{r}v_{ex}(t,r)| \leq \frac{C \lambda(t)^{\frac{5}{2}}}{r^{\frac{3}{2}}t^{2}}, \quad r \geq g(t)\lambda(t)$$
Finally, the dominated convergence theorem and the formula for $v_{ex}$, namely \eqref{vexdef}, give
$$\partial_{t}v_{ex}(t,r) = -\int_{0}^{\infty} dw \int_{0}^{\infty} d\xi J_{1}(r\xi) \sin(w\xi) \partial_{1}\widehat{RHS}(w+t,\xi)$$
Then, we use the same procedure used in \eqref{drvexsubsplit} to finish the proof of \eqref{drvexinftyest}.
 \end{proof}
Finally, we let
\begin{equation}\begin{split}&(u_{e}-u_{wave})_{0}(t,r) =u_{e}-u_{wave}+\frac{r^{5}}{1152}\left(6\left(\partial_{t}^{4}\left(\lambda''(t)\log(\lambda(t))\right)+\partial_{t}^{4}\left(\frac{\lambda'(t)^{2}}{2\lambda(t)}\right)\right)+(13-6 \log(r))\lambda^{(6)}(t)\right)\end{split}\end{equation}
and we note the following lemma.
\begin{lemma}\label{ueminusuwave0est} We have the following estimate, for $0 \leq j,k \leq 1$ or $k=2,j=0$, and $g(t) \lambda(t) \leq r \leq 2 g(t)\lambda(t)$
\begin{equation}\begin{split}&|\partial_{t}^{j}\partial_{r}^{k}\begin{aligned}[t]&\left(\left(u_{e}-u_{wave}\right)_{0}+q_{4,2}\right)|\end{aligned}\\
&\leq \frac{C \lambda(t)^{2-k} \log^{2}(t)}{t^{2+j} g(t)^{3+k}}+\frac{C g(t)^{2-k} \lambda(t)^{4-k} \log^{4}(t) \sup_{x \in [100,t]}\left(\lambda(x) \log(x)\right)}{t^{5+j}}+\frac{C \lambda(t)^{8-k} g(t)^{7-k} \log(t)}{t^{8+j} }\end{split}\end{equation}
\end{lemma}
\begin{proof} We directly combine the results of Proposition \ref{thirdorderprop}, Lemma \ref{w1strlemma}, Lemma \ref{v2strlemma}, and Lemma \ref{estlemma} \end{proof}
We now can estimate the $(u_{e}-u_{wave})_{0}$ contribution to the error terms of our ansatz \eqref{ucdef} which involve at least one derivative of $\chi_{\leq 1}$. Let
\begin{equation}\begin{split}&e_{match,0}(t,r)\\
&=-\left(u_{e}-u_{wave}\right)_{0}\left(-\chi_{\leq 1}''(\frac{r}{h(t)}) \frac{r^{2} h'(t)^{2}}{h(t)^{4}} - \chi_{\leq 1}'(\frac{r}{h(t)})\left(-\frac{r h''(t)}{(h(t))^{2}} + \frac{2 r h'(t)^{2}}{h(t)^{3}}\right) + \frac{1}{r} \frac{\chi_{\leq 1}'(\frac{r}{h(t)})}{h(t)} + \frac{\chi_{\leq 1}''(\frac{r}{h(t)})}{h(t)^{2}}\right)\\
&- 2 \chi_{\leq 1}'(\frac{r}{h(t)}) \frac{r h'(t)}{h(t)^{2}} \partial_{t}\left(u_{e}-u_{wave}\right)_{0} - \frac{2}{h(t)} \chi_{\leq 1}'(\frac{r}{h(t)}) \partial_{r}(u_{e}-u_{wave})_{0}\end{split}\end{equation}
Then, we have the following estimates.
\begin{lemma}\label{ematch0estlemma}[Estimates on the matching-induced error terms] For $0 \leq k \leq 1$, \begin{equation}\begin{split}||L_{\frac{1}{\lambda(t)}}^{k}(e_{match,0})(t,r)||_{L^{2}(r dr)} &\leq  \frac{C \lambda(t)^{5} \log^{2}(t)}{t^{2}  h(t)^{4+k}}+\frac{C h(t)^{1-k} \lambda(t)^{2} \log^{4}(t) \sup_{x \in [100,t]}\left(\lambda(x) \log(x)\right)}{t^{5}}\\
&+\frac{C \lambda(t) h(t)^{6-k} \log(t)}{t^{8} } \end{split}\end{equation}
\end{lemma}
\begin{proof} We directly use Lemma \ref{ueminusuwave0est} and the estimates on $q_{4,2}$ from Proposition \ref{thirdorderprop}. \end{proof}
Next, we consider the error terms of $u_{c}$ which involve derivatives of $\chi_{\leq 1}$, and which result from replacing $u_{e}-u_{wave}$ with $(u_{e}-u_{wave})_{0}$. In particular, let
\begin{equation}\begin{split}e_{5}(t,r)&=-f_{5}(t,r)\left(-\chi_{\leq 1}''(\frac{r}{h(t)}) \frac{r^{2} h'(t)^{2}}{h(t)^{4}} - \chi_{\leq 1}'(\frac{r}{h(t)})\left(-\frac{r h''(t)}{(h(t))^{2}} + \frac{2 r h'(t)^{2}}{h(t)^{3}}\right) + \frac{1}{r} \frac{\chi_{\leq 1}'(\frac{r}{h(t)})}{h(t)} + \frac{\chi_{\leq 1}''(\frac{r}{h(t)})}{h(t)^{2}}\right)\\
&- 2 \chi_{\leq 1}'(\frac{r}{h(t)}) \frac{r h'(t)}{h(t)^{2}} \partial_{t}f_{5}(t,r) - \frac{2}{h(t)} \chi_{\leq 1}'(\frac{r}{h(t)}) \partial_{r}f_{5}(t,r)\end{split}\end{equation}
where
$$f_{5}(t,r) = -\frac{r^{5}}{1152}\left(6\left(\partial_{t}^{4}\left(\lambda''(t)\log(\lambda(t))\right)+\partial_{t}^{4}\left(\frac{\lambda'(t)^{2}}{2\lambda(t)}\right)\right)+(13-6 \log(r))\lambda^{(6)}(t)\right)$$
The following piece of $e_{5}$ turns out to be perturbative. Let
\begin{equation}\begin{split}e_{5,1}(t,r)&=-f_{5}(t,r)\left(-\chi_{\leq 1}''(\frac{r}{h(t)}) \frac{r^{2} h'(t)^{2}}{h(t)^{4}} - \chi_{\leq 1}'(\frac{r}{h(t)})\left(-\frac{r h''(t)}{(h(t))^{2}} + \frac{2 r h'(t)^{2}}{h(t)^{3}}\right) \right)\\
&- 2 \chi_{\leq 1}'(\frac{r}{h(t)}) \frac{r h'(t)}{h(t)^{2}} \partial_{t}f_{5}(t,r) \end{split}\end{equation}
\begin{lemma}\label{e51estlemma}We have the following estimates for $k=0,1$.
$$||L_{\frac{1}{\lambda(t)}}^{k}(e_{5,1})(t,r)||_{L^{2}(r dr)} \leq \frac{C h(t)^{6-k} \lambda(t) \log(t)}{t^{8}}$$
\end{lemma}
\begin{proof} This follows from a straightforward and direct computation \end{proof}
Next, we need to consider $e_{5,0}$, which is given by
\begin{equation}\label{e50def}\begin{split} &e_{5,0}(t,r):=e_{5}(t,r)-e_{5,1}(t,r)=-f_{5}(t,r)\left( \frac{1}{r} \frac{\chi_{\leq 1}'(\frac{r}{h(t)})}{h(t)} + \frac{\chi_{\leq 1}''(\frac{r}{h(t)})}{h(t)^{2}}\right)- \frac{2}{h(t)} \chi_{\leq 1}'(\frac{r}{h(t)}) \partial_{r}f_{5}(t,r)\end{split}\end{equation}
The point is that, although $e_{5,0}(t,r)$ does not decay fast enough (in $L^{2}$, for example) to be perturbative, it is orthogonal to $\phi_{0}(\frac{r}{\lambda(t)})$ to leading order. This is because $e_{5,0}(t,r)$ is supported in the region $\lambda(t) \ll h(t) \leq r \leq 2h(t)$. So, the leading order behavior of $(2 \lambda(t))^{-1}\langle \phi_{0}(\frac{\cdot}{\lambda(t)}),e_{5,0}(t,\cdot)\rangle_{L^{2}(r dr)}$ is 
\begin{equation}\label{e50int}\begin{split} \int_{0}^{\infty} e_{5,0}(t,r) dr &= \int_{0}^{\infty} \chi_{\leq 1}(x) \left(\left(24 p_{0}(t)+10 p_{1}(t)\right)h(t)^{3} x^{3} + 24 p_{1}(t) h(t)^{3} x^{3} \left(\log(h(t))+\log(x)\right) \right) h(t) dx=0\end{split}\end{equation}
where we integrated by parts, defined $p_{j}(t)$ by
$$f_{5}(t,r) = (p_{0}(t)+p_{1}(t) \log(r))r^{5}$$
and used Lemma \ref{chilemma}. Therefore, we will add a term to $u_{c}$ which will be an appropriate truncation of a solution to the ODE
\begin{equation}\label{ellode}\partial_{RR}w(t,R) + \frac{1}{R} \partial_{R}w(t,R)-\frac{\cos(2Q_{1}(R))}{R^{2}} w(t,R) = F(t,R)\end{equation}
for an appropriate choice of $F$. (We leave $F$ general here, since we will use a correction of this form to eliminate error terms of $u_{c}$ other than just $e_{5,0}$).  For the class of $F$ which we will need to consider, we will use the following particular solution to \eqref{ellode} (recall the notation \eqref{e2def})
\begin{equation}\label{ellsoln}\begin{split} w(t,R) &=  - \frac{R}{1+R^{2}} \int_{0}^{R} F(t,s) s e_{2}(s) ds + e_{2}(R) \int_{0}^{R} \frac{F(t,s) s^{2}}{1+s^{2}} ds\end{split}\end{equation}
Returning to the correction associated to $e_{5,0}$, we establish the following lemma
\begin{lemma}\label{w50lemma} Let $w_{5,0}$ denote the function defined in \eqref{ellsoln} where
$$F(t,R) =  \lambda(t)^{2}e_{5,0}(t,R\lambda(t))$$
Then, we have the following estimate, for $0 \leq j+k \leq 2$, $j=2, k=1$, $j=0,k=3$, and $j=1,k=2$.
\begin{equation}|R^{k}t^{j}\partial_{t}^{j}\partial_{R}^{k} w_{5,0}(t,R)| \leq \begin{cases} 0, \quad R \leq \frac{h(t)}{\lambda(t)}\\
C \frac{\lambda(t) h(t)^{5} \log(t)}{t^{6} } , \quad \frac{h(t)}{\lambda(t)} \leq R \leq \frac{2 h(t)}{\lambda(t)}\\
\frac{C}{R} \frac{h(t)^{6} \log(t)}{t^{6} } + \frac{C R h(t)^{2} \lambda(t)^{4} \log(t)}{t^{6} }, \quad R \geq \frac{2 h(t)}{\lambda(t)} \end{cases}\end{equation}
\end{lemma}
\begin{proof}
By a straightforward insertion of the definition of $e_{5,0}$, \eqref{e50def}, into the following integral, we have
$$|\int_{0}^{R} \lambda(t)^{2} e_{5,0}(t,s\lambda(t)) s e_{2}(s) ds| \leq \begin{cases} 0, \quad R \leq \frac{h(t)}{\lambda(t)}\\
\frac{C h(t)^{6} \log(t)}{t^{6} }, \quad R \geq \frac{h(t)}{\lambda(t)}\end{cases}$$
On the other hand, we use \eqref{e50int} to get
\begin{equation} \lambda(t)^{2} \int_{0}^{R} e_{5,0}(t,s\lambda(t)) \frac{s^{2} ds}{1+s^{2}} = -\lambda(t)^{2} \int_{R}^{\infty} e_{5,0}(t,s\lambda(t)) ds + \lambda(t)^{2} \int_{0}^{R} e_{5,0}(t,s\lambda(t))\left(\frac{-1}{1+s^{2}}\right) ds\end{equation}
(Note that the support properties of $\chi_{\leq 1}$ imply that the left-hand side of the above equation vanishes when $R \leq \frac{h(t)}{\lambda(t)}$). Then, we get
\begin{equation}|\lambda(t)^{2} \int_{0}^{R} e_{5,0}(t,s\lambda(t)) \frac{s^{2} ds}{1+s^{2}}| \leq \begin{cases} 0, \quad R \leq \frac{h(t)}{\lambda(t)}\\
\frac{C h(t)^{4} \lambda(t)^{2} \log(t)}{t^{6} }, \quad \frac{h(t)}{\lambda(t)} \leq R \leq \frac{2 h(t)}{\lambda(t)}\\
\frac{C h(t)^{2} \lambda(t)^{4} \log(t)}{t^{6} }, \quad R \geq \frac{2h(t)}{\lambda(t)} \end{cases}\end{equation}
This gives the estimate of the lemma statement for $j=k=0$. Next, we have
\begin{equation}\begin{split} &\partial_{R}w_{5,0}(t,R) = - \partial_{R}\left(\frac{R}{1+R^{2}}\right) \int_{0}^{R} \lambda(t)^{2} e_{5,0}(t,s\lambda(t)) s e_{2}(s) ds + e_{2}'(R) \int_{0}^{R} \frac{\lambda(t)^{2} e_{5,0}(t,s\lambda(t)) s^{2}}{1+s^{2}} ds\end{split}\end{equation}
So, $\partial_{R}w_{5,0}(t,R)$ is the same expression as $w_{5,0}(t,R)$, except with an extra derivative on the coefficients of each integral term. Our proof of the lemma for $j=0,k=0$ therefore immediately implies the lemma statement is true for $j=0,k=1$. We prove the $j=0,k=2$ case of the lemma statement by noting that
$$\partial_{R}^{2} w_{5,0}(t,R) = \lambda(t)^{2} e_{5,0}(t,R\lambda(t)) -\frac{1}{R}\partial_{R}w_{5,0}(t,R)+\frac{\cos(2Q_{1}(R))}{R^{2}} w_{5,0}(t,R)$$
Finally, the symbol-type estimates on $\lambda(t)$, definition of $e_{5,0}$, and the fact that
$$\int_{0}^{\infty} \lambda(t)^{2} e_{5,0}(t,s\lambda(t)) ds =0 \text{ for all t implies  } \partial_{t}^{j} \left(\int_{0}^{\infty} \lambda(t)^{2} e_{5,0}(t,s\lambda(t)) ds\right)=0$$
finishes the proof of the lemma.
 \end{proof}
We define $f_{5,0}(t,r)$, the function to be added to our ansatz, by
\begin{equation}\label{f50def}f_{5,0}(t,r) = m_{\leq1}(\frac{r}{t}) w_{5,0}(t,\frac{r}{\lambda(t)})\end{equation}
where we recall the definition of $m_{\leq 1}$ in \eqref{mleq1def}. Then, we define the error term of $f_{5,0}$ by
\begin{equation}\begin{split} e_{f_{5,0}}(t,r)&:= -\left(-\partial_{tt}f_{5,0} + \partial_{rr}f_{5,0}+\frac{1}{r}\partial_{r}f_{5,0}-\frac{\cos(2Q_{1}(\frac{r}{\lambda(t)}))}{r^{2}} f_{5,0}\right)+e_{5,0}(t,r)\\
&=-m_{\leq 1}''(\frac{r}{t}) \left(\frac{-r^{2}}{t^{4}} + \frac{1}{t^{2}}\right) w_{5,0}(t,\frac{r}{\lambda(t)})\\
&-m_{\leq 1}'(\frac{r}{t}) \left(\left(\frac{-2r}{t^{3}} + \frac{1}{r t}\right) w_{5,0}(t,\frac{r}{\lambda(t)})+\frac{2r}{t^{2}} \partial_{t}\left(w_{5,0}(t,\frac{r}{\lambda(t)})\right)+\frac{2}{t}\partial_{r}\left(w_{5,0}(t,\frac{r}{\lambda(t)})\right)\right)\\
&+m_{\leq 1}(\frac{r}{t})\partial_{t}^{2}\left(w_{5,0}(t,\frac{r}{\lambda(t)})\right)+e_{5,0}(t,r) \left(1-m_{\leq 1}(\frac{r}{t})\right)\end{split}\end{equation}
Then, an insertion of our estimates from Lemma \ref{w50lemma} into the above expression for $e_{f_{5,0}}$ gives the following lemma. 
\begin{lemma}\label{ef50estlemma} For $k=0,1$,
$$||L_{\frac{1}{\lambda(t)}}^{k}(e_{f_{5,0}})(t,r)||_{L^{2}(r dr)} \leq \frac{C \lambda(t)}{t^{2}} \left(\frac{h(t)^{6-k} \log^{2}(t)}{t^{6}} + \frac{h(t)^{2-k}\lambda(t)^{2} \log^{2}(t)}{t^{4}}\right)$$
\end{lemma}
Now, we will consider the error term in \eqref{uceqn} involving $\left(\frac{\cos(2Q_{1}(\frac{r}{\lambda(t)}))-1}{r^{2}}\right) v_{ex,sub}(t,r)$.
\begin{lemma}\label{wexsublemma}
Let $w_{ex,sub}(t,R)$ be defined by the expression \eqref{ellsoln} for
$$F(t,s) = \left(\frac{\cos(2Q_{1}(s))-1}{s^{2}}\right) v_{ex,sub}(t,s\lambda(t)) \chi_{\geq 1}(\frac{s}{g(t)})$$
Then, we have
$$w_{ex,sub}(t,\frac{r}{\lambda(t)})=0, \quad r \leq \lambda(t) g(t)$$
For $0 \leq k+j \leq 2$, $j=1,k=2$, and $j=2,k=1$,
$$|\partial_{t}^{k}\partial_{r}^{j}\left(w_{ex,sub}(t,\frac{r}{\lambda(t)})\right)| \leq \frac{C \lambda(t)^{5} \log^{3}(t)}{r^{1+j} t^{4+k} } + \frac{C r^{1-j} \lambda(t)^{3} \log^{2}(t)}{t^{4+k}  g(t)^{2}}, \quad g(t) \lambda(t) \leq r \leq t$$
\end{lemma}
\begin{proof}
We define $int_{ex,sub}$ by the integral 
$$int_{ex,sub}(t,R) = \int_{0}^{R} \frac{F(t,s) s^{2}}{1+s^{2}} ds = -\int_{0}^{\infty} d\xi \int_{t}^{\infty} dx \frac{\sin((t-x)\xi)}{\xi^{2}} \partial_{x}^{2}\widehat{RHS}(x,\xi) I_{1}(R,\xi,t)$$
where
$$I_{1}(R,\xi,t) = \int_{0}^{R} \left(\frac{\cos(2Q_{1}(s))-1}{s^{2}}\right) J_{1}(s\lambda(t)\xi) \frac{\chi_{\geq 1}(\frac{s}{g(t)}) s^{2}}{1+s^{2}} ds$$
and we used Fubini's theorem. Note that $I_{1}(R,\xi,t) = 0$ for $R \leq g(t)$. For $R \geq g(t)$,  we decompose $I_{1}$ as
$$I_{1}(R,\xi,t) = I_{100}(\xi,t)+I_{101}(\xi,t) + I_{11}(R,\xi,t)$$
where
$$I_{100}(\xi,t)= \frac{-8}{g(t)^{3}} \widehat{\frac{\chi_{\geq 1}(\cdot)}{\left(\cdot\right)^{5}}}(\xi g(t) \lambda(t))$$
$$I_{101}(\xi,t) = \int_{0}^{\infty} \left(\frac{(\cos(2Q_{1}(s))-1)}{s^{2}}\frac{s^{2}}{1+s^{2}}+\frac{8}{s^{4}}\right)J_{1}(s\lambda(t)\xi) \chi_{\geq 1}(\frac{s}{g(t)}) ds$$
and
$$I_{11}(R,\xi,t)= - \int_{R}^{\infty} \left(\frac{\cos(2Q_{1}(s))-1}{s^{2}}\right) J_{1}(s\lambda(t)\xi) \chi_{\geq 1}(\frac{s}{g(t)}) \frac{s^{2} ds}{1+s^{2}}$$
We claim that, for $0 \leq j \leq 2$,
$$|\partial_{t}^{j}I_{100}(\xi,t)| \leq \frac{C}{g(t)^{3}t^{j}} \begin{cases} \xi^{3} g(t)^{3}\lambda(t)^{3} \langle \log(\xi g(t) \lambda(t))\rangle, \quad \xi \leq \frac{1}{g(t)\lambda(t)}\\
\frac{C_{k}}{(\xi g(t)\lambda(t))^{k}}, \quad \xi \geq \frac{1}{g(t)\lambda(t)}\end{cases}$$
$$|\partial_{t}^{j}I_{101}(\xi,t)| \leq \frac{C}{t^{j}} \begin{cases} \frac{ \xi \lambda(t)}{g(t)^{4}}, \quad \xi \lambda(t) \leq \frac{1}{g(t)}\\
\frac{C_{k}}{g(t)^{5} (\xi \lambda(t) g(t))^{2k}}, \quad k \geq 5, \quad \xi \lambda(t) \geq \frac{1}{g(t)}\end{cases}$$
$$|\partial_{t}^{j}\left(I_{11}(\frac{r}{\lambda(t)},\xi,t)\right)| \leq \frac{C}{t^{j}}\begin{cases} \frac{ \lambda(t)^{3}}{ \xi^{3/2} r^{9/2}}, \quad \xi \geq \frac{1}{\lambda(t) g(t)}\\
\frac{ \xi \lambda(t)^{3}}{r^{2} }, \quad \xi \leq \frac{1}{r}\\
\frac{ \lambda(t)^{3}}{r^{7/2} \sqrt{\xi}}, \quad \frac{1}{r} \leq \xi \leq\frac{1}{g(t)\lambda(t)}\end{cases}$$
For $j=0$, $I_{100}$ is estimated by using Lemma \ref{chilemma}, while, for all $0 \leq j \leq 2$, $I_{101}$ and $I_{11}$ are estimated by integrating by parts when $\xi \lambda(t) \geq \frac{1}{g(t)}$, and directly estimating otherwise (as in Lemma \ref{chilemma}). For $j >0$, the most delicate estimate is on $\partial_{t}^{j}I_{100}(\xi,t)$ in the region $\xi \leq \frac{1}{\lambda(t) g(t)}$. We write
$$I_{100}(\xi,t) = -8 \int_{0}^{\infty} \frac{J_{1}(g(t)\lambda(t) y \xi)}{g(t)^{3} y^{4}} \chi_{\geq 1}(y) dy$$
which gives
$$\partial_{t}I_{100}(\xi,t) = \frac{-3 g'(t)}{g(t)} I_{100}(\xi,t) - 8 \int_{0}^{\infty} \frac{J_{1}'(g(t)\lambda(t) y \xi)}{g(t)^{3} y^{4}} (g(t)\lambda(t))' y \xi \chi_{\geq 1}(y) dy$$
Then,
\begin{equation}\begin{split} &- 8 \int_{0}^{\infty} \frac{J_{1}'(g(t)\lambda(t) y \xi)}{g(t)^{3} y^{4}} (g(t)\lambda(t))' y \xi \chi_{\geq 1}(y) dy\\
&= -8 \int_{0}^{\frac{1}{g(t)\lambda(t) \xi}} \frac{1}{2} \frac{(g(t)\lambda(t))'}{g(t)^{3} y^{4}} y \xi \chi_{\geq 1}(y) dy-8\int_{0}^{\frac{1}{g(t)\lambda(t)\xi}} \frac{\left(J_{1}'(g(t)\lambda(t)y\xi)-\frac{1}{2}\right)\left(g(t)\lambda(t)\right)'}{g(t)^{3}y^{4}} y \xi \chi_{\geq 1}(y) dy\\
&-8 \int_{\frac{1}{g(t)\lambda(t)\xi}}^{\infty} \frac{J_{1}'(g(t)\lambda(t) y \xi)}{g(t)^{3} y^{4}} (g(t)\lambda(t))' y \xi \chi_{\geq 1}(y) dy\end{split}\end{equation}
and to estimate the first term on the right hand side of the above expression, we use Lemma \ref{chilemma}. We use a similar procedure for $\partial_{t}^{2} I_{100}(\xi,t)$.
\begin{equation}\begin{split}&\partial_{t}^{2}\left(-int_{ex,sub}(t,\frac{r}{\lambda(t)})\right) = \partial_{t}^{2}\left(\int_{0}^{\infty} d\xi \int_{t}^{\infty} dx \frac{\sin((t-x)\xi)}{\xi^{2}} \partial_{x}^{2}\widehat{RHS}(x,\xi) I_{1}(\frac{r}{\lambda(t)},\xi,t)\right)\\
&=\partial_{t}^{2}\left(-\int_{0}^{\infty} d\xi \int_{0}^{\infty} dw \frac{\sin(w\xi)}{\xi^{2}} \partial_{1}^{2}\widehat{RHS}(t+w,\xi) \left(I_{100}(\xi,t)+I_{101}(\xi,t)+I_{11}(\frac{r}{\lambda(t)},\xi,t)\right)\right)\end{split}\end{equation}
After differentiating under the integral sign in the last integral, we let $x=w+t$, and  split the $x$ integration into two regions: $t \leq x \leq t+\frac{1}{\xi}$ and $x > t+\frac{1}{\xi}$. In the latter region, we integrate by parts in the $x$ variable, integrating $\sin((t-x)\xi)$. Then, we make a similar decomposition of the $\xi$ integral as was made while proving Lemma \ref{estlemma}.  This gives, for $k=0,1,2$,
\begin{equation}\begin{split}|\partial_{t}^{k}\left(int_{ex,sub}(t,\frac{r}{\lambda(t)})\right)| &\leq \frac{C \lambda(t)^{4} \log^{2}(t)}{t^{4+k}  g(t)^{2}}, \quad t \geq r \geq g(t) \lambda(t)\end{split}\end{equation}
We recall the definition of $e_{2}$ in \eqref{e2def}. The next integral to consider is
\begin{equation}\begin{split}int_{ex,sub,2}(t,R) &= \int_{0}^{R} \left(\frac{\cos(2Q_{1}(s))-1}{s^{2}}\right) v_{ex,sub}(t,s\lambda(t)) \chi_{\geq 1}(\frac{s}{g(t)}) s e_{2}(s) ds\\
&=-\int_{0}^{\infty} d\xi \int_{t}^{\infty} dx \frac{\sin((t-x)\xi)}{\xi^{2}} \partial_{x}^{2} \widehat{RHS}(x,\xi) I_{2}(R,\xi,t)\end{split}\end{equation}
where
$$I_{2}(R,\xi,t) = \int_{0}^{R} \left(\frac{\cos(2Q_{1}(s))-1}{s^{2}}\right) J_{1}(s\lambda(t) \xi) \chi_{\geq 1}(\frac{s}{g(t)}) s e_{2}(s) ds$$
As with $I_{1}$ above, $I_{2}(R,\xi,t) =0$ for $R \leq g(t)$. A direct estimation gives, for $0 \leq k \leq 2$,
$$|\partial_{t}^{k}\left(I_{2}(\frac{r}{\lambda(t)},\xi,t)\right)| \leq \frac{C}{t^{k}} \begin{cases} \log(\frac{r}{\lambda(t) g(t)}) \lambda(t) \xi, \quad \xi \leq \frac{1}{r}\\
 \xi \lambda(t) \langle \log(\xi \lambda(t) g(t))\rangle, \quad \frac{1}{r} \leq \xi \leq \frac{1}{\lambda(t) g(t)}\\
\frac{C}{g(t)^{3/2} \sqrt{\lambda(t) \xi}}, \quad \xi \geq \frac{1}{\lambda(t) g(t)}\end{cases}$$
Then, as above, we get, for $k=0,1,2$,
$$|\partial_{t}^{k}\left(int_{ex,sub,2}(t,\frac{r}{\lambda(t)})\right)| \leq \frac{C \lambda(t)^{4} \log^{3}(t)}{t^{4+k} }, \quad g(t) \leq \frac{r}{\lambda(t)} \leq \frac{t}{\lambda(t)}$$
Recalling \eqref{ellsoln}, we have
$$w_{ex,sub}(t,\frac{r}{\lambda(t)}) = -\frac{\phi_{0}(\frac{r}{\lambda(t)})}{2} int_{ex,sub,2}(t,\frac{r}{\lambda(t)}) + e_{2}(\frac{r}{\lambda(t)}) int_{ex,sub}(t,\frac{r}{\lambda(t)})$$
which gives, for $0 \leq k \leq 2$,
$$|\partial_{t}^{k}\left(w_{ex,sub}(t,\frac{r}{\lambda(t)})\right)| \leq \frac{C \lambda(t)^{5} \log^{3}(t)}{r t^{4+k} } + \frac{C r \lambda(t)^{3} \log^{2}(t)}{t^{4+k}  g(t)^{2}}, \quad g(t) \lambda(t) \leq r \leq t$$
Next, we note that
$$\partial_{r} \left(w_{ex,sub}(t,\frac{r}{\lambda(t)})\right) = \frac{-\phi_{0}'(\frac{r}{\lambda(t)})}{2\lambda(t)} int_{ex,sub,2}(t,\frac{r}{\lambda(t)}) + e_{2}'(\frac{r}{\lambda(t)}) \frac{1}{\lambda(t)} int_{ex,sub}(t,\frac{r}{\lambda(t)})$$
Finally, we estimate $\partial_{r}^{2}\left(w_{ex,sub}(t,\frac{r}{\lambda(t)})\right)$ using the equation solved by $w_{ex,sub}$, Lemma \ref{estlemma}, and our previous estimates from the proof of this lemma. 
 \end{proof}
The truncation of $w_{ex,sub}$ which we will add to our ansatz is
\begin{equation}\label{fexsub}f_{ex,sub}(t,r) = m_{\leq 1}(\frac{r}{t}) w_{ex,sub}(t,\frac{r}{\lambda(t)})\end{equation}
where we recall that $m_{\leq 1}$ was defined in \eqref{mleq1def}.
We define the error term of $f_{ex,sub}(t,r)$ as
\begin{equation}\begin{split}e_{ex,sub}(t,r)& = -\left(-\partial_{tt}f_{ex,sub} + \partial_{rr}f_{ex,sub}+\frac{1}{r}\partial_{r}f_{ex,sub}-\frac{\cos(2Q_{1}(\frac{r}{\lambda(t)}))}{r^{2}} f_{ex,sub}\right)\\
&+\left(\frac{\cos(2Q_{1}(\frac{r}{\lambda(t)}))-1}{r^{2}}\right) v_{ex,sub}(t,r) \chi_{\geq 1}(\frac{r}{\lambda(t)g(t)})\end{split}\end{equation}
We get
\begin{equation}\begin{split} e_{ex,sub}(t,r) &= -m_{\leq 1}''(\frac{r}{t}) \left(\frac{-r^{2}}{t^{4}} + \frac{1}{t^{2}}\right) w_{ex,sub}(t,\frac{r}{\lambda(t)})\\
&-m_{\leq 1}'(\frac{r}{t}) \left(\left(\frac{-2r}{t^{3}} + \frac{1}{r t}\right) w_{ex,sub}(t,\frac{r}{\lambda(t)})+\frac{2r}{t^{2}} \partial_{t}\left(w_{ex,sub}(t,\frac{r}{\lambda(t)})\right)+\frac{2}{t}\partial_{r}\left(w_{ex,sub}(t,\frac{r}{\lambda(t)})\right)\right)\\
&+m_{\leq 1}(\frac{r}{t})\partial_{t}^{2}\left(w_{ex,sub}(t,\frac{r}{\lambda(t)})\right)\\
&+\left(\frac{\cos(2Q_{1}(\frac{r}{\lambda(t)}))-1}{r^{2}}\right) v_{ex,sub}(t,r) \chi_{\geq 1}(\frac{r}{\lambda(t)g(t)}) \left(1-m_{\leq 1}(\frac{r}{t})\right)\end{split}\end{equation}
Lemma \ref{wexsublemma} directly gives
\begin{lemma}\label{eexsubestlemma} For $k=0,1$,
$$||L_{\frac{1}{\lambda(t)}}^{k}(e_{ex,sub})(t,r)||_{L^{2}(r dr)} \leq \frac{C \lambda(t)^{3-k} \log^{2}(t)}{t^{4} g(t)^{2+k}}$$
\end{lemma}
\begin{lemma}\label{fell2lemma} Let $g_{ell,2}(t,R)$ be defined by the expression \eqref{ellsoln} for the choice
$$F(t,s) = \lambda(t)^{2} \chi_{\leq 1}(\frac{s}{g(t)}) \partial_{1}^{2} u_{ell,2}(t,s\lambda(t))$$
and let
$$f_{ell,2}(t,r) = g_{ell,2}(t,\frac{r}{\lambda(t)}) m_{\leq 1}(\frac{r}{t})$$
Then, for $0 \leq  j \leq 2, \quad 0 \leq k \leq3$,
$$|R^{k}t^{j}\partial_{t}^{j}\partial_{R}^{k}g_{ell,2}(t,R)| \leq \frac{C \lambda(t)^{6}}{t^{6}} \begin{cases} R^{5}\left(1+\log^{2}(R)\right), \quad R \leq 1\\
R^{5} \log(t), \quad 1 \leq R \leq 2 g(t)\\
R g(t)^{2} \log^{3}(t) + \frac{g(t)^{6} \log(t)}{R}, \quad R > 2 g(t)\end{cases}$$
Also, letting $e_{ell,2}$ denote the error term of $f_{ell,2}$:
\begin{equation}\begin{split}e_{ell,2}(t,r)& = -\left(-\partial_{tt}f_{ell,2} + \partial_{rr}f_{ell,2}+\frac{1}{r}\partial_{r}f_{ell,2}-\frac{\cos(2Q_{1}(\frac{r}{\lambda(t)}))}{r^{2}}f_{ell,2}\right)+\chi_{\leq 1}(\frac{r}{g(t)\lambda(t)}) \partial_{1}^{2}u_{ell,2}(t,r)\end{split}\end{equation}
we have, for $k=0,1$,
$$||L_{\frac{1}{\lambda(t)}}^{k}(e_{ell,2})(t,r)||_{L^{2}(r dr)} \leq \frac{C \lambda(t)^{7-k} \log^{3}(t) g(t)^{6-k}}{t^{8}}$$
\end{lemma}
\begin{proof}
We will take advantage of two of the orthogonality conditions in Lemma \ref{chilemma} when studying $g_{ell,2}$. We recall that $u_{ell,2}$ is defined in \eqref{uell2def}, and we use \eqref{soln2def} to get
\begin{equation}\begin{split} g_{ell,2}(t,R)&=\frac{e_{2}(R)}{2} \lambda(t)^{2} \int_{0}^{R} \chi_{\leq 1}(\frac{s}{g(t)}) \left(\partial_{t}^{2}\left(v_{ell,2,0,main}(t,\frac{r}{\lambda(t)})\right)\Bigr|_{r=s\lambda(t)} + \partial_{t}^{2}\left(soln_{1}(t,\frac{r}{\lambda(t)})\right)\Bigr|_{r=s\lambda(t)}\right. \\
&+\left. \partial_{t}^{2}\left(soln_{2}(t,\frac{r}{\lambda(t)})\right)\Bigr|_{r=s\lambda(t)} \right) s \phi_{0}(s) ds\\
&- \frac{\phi_{0}(R)}{2} \lambda(t)^{2} \int_{0}^{R} \chi_{\leq 1}(\frac{s}{g(t)}) \partial_{1}^{2} u_{ell,2}(t,s\lambda(t)) s e_{2}(s) ds\end{split}\end{equation}
where, for the reader's convenience, we recall that
$$v_{ell,2,0,main}(t,s) = \frac{s^{3} \lambda(t)^{3}}{8}\left(\partial_{t}^{2}\left(\frac{\lambda'(t)^{2}}{2\lambda(t)}+\lambda''(t)+\lambda''(t)\log(\lambda(t))\right)-\lambda''''(t)\log(s\lambda(t)) + \frac{3}{4} \lambda''''(t)\right)$$
We then estimate $g_{ell,2}$ by using the orthogonality conditions of Lemma \ref{chilemma} to treat the term in the above expression for $g_{ell,2}$ which involves $v_{ell,2,0,main}$, and estimate the rest of $g_{ell,2}$ directly. We have the following two estimates for $0 \leq k \leq 1$ and $0 \leq j \leq 2$,
\begin{equation}\label{soln1est}|\partial_{t}^{j}\partial_{s}^{k}\left(\partial_{t}^{2}\left(soln_{1}(t,\frac{r}{\lambda(t)})\right)\Bigr|_{r=s\lambda(t)}\right)| \leq \frac{C \lambda(t)^{4}}{t^{6+j}} \begin{cases} s^{5-k}(1+|\log(s)|), \quad s \leq 1\\
s^{1-k} (1+\log^{2}(s)), \quad s \geq 1\end{cases}\end{equation}
\begin{equation}\label{soln2est}|\partial_{t}^{j}\partial_{s}^{k}\left(\partial_{t}^{2}\left(soln_{2}(t,\frac{r}{\lambda(t)})\right)\Bigr|_{r=s\lambda(t)}\right)| \leq \frac{C \lambda(t)^{4}}{t^{6+j}} \begin{cases} s^{3-k} (1+\log^{2}(s)), \quad s \leq 1\\
s^{1-k} (1+\log^{3}(s)), \quad s \geq 1\end{cases}\end{equation}
This results in the estimates on $g_{ell,2}$ and $e_{ell,2}$ in the lemma statement.  \end{proof}
Finally, we treat the linear error term associated to $v_{2,2}$ which we recall is a free wave added to $u_{w,2}$, and is chosen so as to allow for the third order matching. This error term is 
$$e_{2,2}(t,r) = \left(\frac{\cos(2Q_{1}(\frac{r}{\lambda(t)}))-1}{r^{2}}\right) v_{2,2}(t,r) \left(1-\chi_{\leq 1}(\frac{r}{h(t)})\right)$$
We define $u_{2,2}$ to be the solution to the following equation, with $0$ Cauchy data at infinity.
\begin{equation}\label{u22eqn}-\partial_{t}^{2}u_{2,2}+\partial_{r}^{2}u_{2,2}+\frac{1}{r}\partial_{r}u_{2,2}-\frac{u_{2,2}}{r^{2}}=e_{2,2}(t,r)\end{equation}
Then, we have the following lemma.
\begin{lemma}\label{u22lemma} For $\frac{h(t)}{2} \leq r \leq \frac{t}{2}$, and $0 \leq k \leq 1$, $0 \leq j \leq 2$,
\begin{equation}\label{u22ptwse}\begin{split}t^{j}r^{k}|\partial_{t}^{j}\partial_{r}^{k}u_{2,2}(t,r)| &\leq \frac{C r \lambda(t)^{4} \sup_{x \in [100,t]}\left(\lambda(x) \log(x)\right) \log^{5}(t)}{t^{4} h(t)^{2}} \\
&+ \frac{C r\lambda(t)^{2} \sup_{x \in [100,t]}\left(\lambda(x) \log(x)\right) \sup_{x \in [100,t]}\left(\lambda(x)^{2}\right) \log^{4}(t)}{t^{6}}\end{split}\end{equation}
Also, for $\frac{h(t)}{2} \leq r \leq \frac{t}{2}$,
\begin{equation}\begin{split}|\partial_{r}^{2}u_{2,2}(t,r)| &\leq \frac{C  \lambda(t)^{4} \sup_{x \in [100,t]}\left(\lambda(x) \log(x)\right) \log^{5}(t)}{t^{4} r h(t)^{2}} \\
&+ \frac{C \lambda(t)^{2} \sup_{x \in [100,t]}\left(\lambda(x) \log(x)\right) \sup_{x \in [100,t]}\left(\lambda(x)^{2}\right) \log^{4}(t)}{r t^{6}}\end{split}\end{equation}
In addition, for all $r>0$,
\begin{equation}\label{u22enest}\begin{split}&\sqrt{E(u_{2,2},\partial_{t}u_{2,2})} + |u_{2,2}(t,r)| \leq \frac{C \lambda(t)^{2} \sup_{x \in [100,t]}\left(\lambda(x) \log(x)\right) \sup_{x \in [100,t]}\left(\lambda(x)^{2}\right) \log^{3}(t)}{t^{3}}\end{split}\end{equation}
Finally, for $0 \leq j \leq 2$,
\begin{equation}\label{dru22ptwseuptocone}\begin{split}|\partial_{t}^{j}\partial_{r}u_{2,2}(t,r)| &\leq \frac{C \lambda(t)^{4} \sup_{x \in [100,t]}\left(\lambda(x) \log(x)\right) \log^{3}(t)}{\langle t-r \rangle^{1+j} h(t)^{2} t^{3}} \\
&+ \frac{C \lambda(t)^{2} \sup_{x \in [100,t]}\left(\lambda(x) \log(x)\right) \sup_{x \in [100,t]}\left(\lambda(x)^{2}\right) \log^{3}(t)}{ t^{5/2} \langle t-r \rangle^{7/2+j}}, \quad \frac{t}{2} < r < t\end{split}\end{equation}
\begin{equation}\label{dru22inftyest} |\partial_{r}u_{2,2}(t,r)| \leq \frac{C \log^{4}(t)}{t^{\frac{5}{2}-5C_{u}}}, \quad r>0\end{equation}
\end{lemma}
\begin{proof}
As in \eqref{q41onederivform}, we have
\begin{equation} u_{2,2}(t,r) = \frac{-r}{2\pi} \int_{t}^{\infty} ds \int_{0}^{s-t} \frac{\rho d\rho}{\sqrt{(s-t)^{2}-\rho^{2}}} \int_{0}^{1}d\beta \int_{0}^{2\pi} d\theta I_{e_{2,2}}(s,r\beta,\rho,\theta)\end{equation}
where we recall the notation \eqref{Inotation}. We start with the estimate of the lemma statement in the region $h(t) \leq r \leq 2h(t) \leq \frac{t}{2}$. For all $\rho \leq s-t$, $r \leq \frac{t}{2}$, and $0 \leq \beta \leq 1$, we have
$$\sqrt{\beta^{2}r^{2}+\rho^{2}+2\beta r \rho\cos(\theta)} \leq \beta r + \rho \leq s-t + \frac{t}{2}<s$$
From Lemma \ref{v22ests}, we get
\begin{equation}\begin{split}|I_{e_{2,2}}(s,r \beta,\rho,\theta)| \leq C &\left( \frac{\lambda(s)^{2}}{s^{4}} \sup_{x \in [100,s]}\left(\lambda(x) \log(x)\right) \frac{\log^{3}(s) \lambda(s)^{2}}{(h(s)^{2}+\beta^{2}r^{2}+\rho^{2}+2\beta r \rho \cos(\theta))^{2}}\right.\\
&\left. +\frac{\lambda(s)^{2} \sup_{x \in [100,s]}\left(\lambda(x) \log(x)\right) \sup_{x \in [100,s]}\left(\lambda(x)^{2}\right) \log^{3}(s)}{\langle s-\sqrt{r^{2}\beta^{2}+\rho^{2}+2 \beta r \rho \cos(\theta)}\rangle^{7/2} s^{9/2}}\right)\end{split}\end{equation}
We then make the analogous decomposition as in \eqref{q41decomp}, and use 
$$s- \sqrt{\beta^{2}r^{2}+\rho^{2}+2 \beta r \rho \cos(\theta)} \geq s-(r+\rho) \geq t-r \geq C t$$
which is true for $\rho \leq s-t$ and all $r  \leq \frac{t}{2}$. This gives \eqref{u22ptwse} for $k =0,j=0$, and $\frac{h(t)}{2} \leq r \leq 2h(t)$. For $k=1,j=0$, we note that
\begin{equation}\partial_{r}u_{2,2}(t,r) = \frac{-1}{2\pi} \int_{t}^{\infty} ds \int_{0}^{s-t} \frac{\rho d\rho}{\sqrt{(s-t)^{2}-\rho^{2}}} \int_{0}^{2\pi} d\theta I_{e_{2,2}}(s,r,\rho,\theta)\end{equation}
which, when combined with the same procedure used for $u_{2,2}$, completes the proof of \eqref{u22ptwse} for $j=0$ and $0 \leq k \leq 1$, in the region $\frac{h(t)}{2} \leq r \leq 2h(t)$. To treat higher $j$, we simply note that
\begin{equation} \partial_{t}^{j}u_{2,2}(t,r) = \frac{-r}{2\pi} \int_{t}^{\infty} ds \int_{0}^{s-t} \frac{\rho d\rho}{\sqrt{(s-t)^{2}-\rho^{2}}} \int_{0}^{1}d\beta \int_{0}^{2\pi} d\theta \partial_{1}^{j}I_{e_{2,2}}(s,r\beta,\rho,\theta)\end{equation}
and use the same procedure used for $j=0$.\\
\\
In the region $2h(t) \leq r \leq \frac{t}{2}$, a slightly more complicated argument is needed because factors of $\frac{r}{h(t)}$ are no longer controlled by a constant in this region. We have
\begin{equation}\label{u22intermediatestep}\begin{split} &|u_{2,2}(t,r)| \\
&\leq C r \int_{t}^{\infty} ds \int_{0}^{s-t}\frac{\rho d\rho}{\sqrt{(s-t)^{2}-\rho^{2}}}\int_{0}^{1}d\beta \int_{0}^{2\pi} d\theta \begin{aligned}[t]&\left(\frac{\lambda(s)^{4} \sup_{x \in [100,s]}\left(\lambda(x) \log(x)\right) \log^{3}(s)}{s^{4}(r^{2}\beta^{2}+\rho^{2}+2 r \beta \rho \cos(\theta)+h(s)^{2})^{2}}\right.\\
&\left.+\frac{\lambda(s)^{2} \sup_{x \in [100,s]}\left(\lambda(x) \log(x)\right) \sup_{x \in [100,s]}\left(\lambda(x)^{2}\right) \log^{3}(s)}{s^{9/2} t^{7/2}} \right)\end{aligned}\end{split}\end{equation}
The second term of the integrand of the expression \eqref{u22intermediatestep} is estimated with the following simple procedure.
\begin{equation}\begin{split}&|r \int_{t}^{\infty} ds \int_{0}^{s-t}\frac{\rho d\rho}{\sqrt{(s-t)^{2}-\rho^{2}}}\int_{0}^{1}d\beta \int_{0}^{2\pi} d\theta\frac{\lambda(s)^{2} \sup_{x \in [100,s]}\left(\lambda(x) \log(x)\right) \sup_{x \in [100,s]}\left(\lambda(x)^{2}\right) \log^{3}(s)}{s^{9/2} t^{7/2}}|\\
&\leq C r \int_{t}^{\infty} ds \frac{\lambda(s)^{2}  \sup_{x \in [100,s]}\left(\lambda(x) \log(x)\right) \sup_{x \in [100,s]}\left(\lambda(x)^{2}\right) \log^{3}(s)}{s^{7/2} t^{7/2}}\\
&\leq \frac{C r \lambda(t)^{2} \sup_{x \in [100,t]}\left(\lambda(x) \log(x)\right) \sup_{x \in [100,t]}\left(\lambda(x)^{2}\right) \log^{3}(t)}{t^{6}}, \quad r \leq \frac{t}{2}\end{split}\end{equation}
Using Cauchy's residue theorem appropriately, we get the following, for $a,\rho,y>0$.
\begin{equation}\label{thetaint}\int_{0}^{2\pi} \frac{d\theta}{(y^{2}+\rho^{2}+2 \rho y \cos(\theta)+a^{2})^{2}} = \frac{2\pi (a^{2}+\rho^{2}+y^{2})}{((a^{2}+(\rho-y)^{2})(a^{2}+(\rho+y)^{2}))^{3/2}}\end{equation}
So, for the first term of the integrand in \eqref{u22intermediatestep}, we get
\begin{equation}\begin{split}&|r \int_{t}^{\infty} ds \int_{0}^{s-t}\frac{\rho d\rho}{\sqrt{(s-t)^{2}-\rho^{2}}}\int_{0}^{1}d\beta \int_{0}^{2\pi} d\theta \frac{\lambda(s)^{4} \sup_{x \in [100,s]}\left(\lambda(x) \log(x)\right) \log^{3}(s)}{s^{4}(r^{2}\beta^{2}+\rho^{2}+2 r \beta \rho \cos(\theta)+h(s)^{2})^{2}}|\\
&\leq C r \int_{t}^{\infty} ds \int_{0}^{s-t} \frac{\rho d\rho}{\sqrt{(s-t)^{2}-\rho^{2}}} \int_{0}^{1}d\beta \frac{\lambda(s)^{4} \sup_{x \in [100,s]}\left(\lambda(x) \log(x)\right) \log^{3}(s)}{s^{4} \sqrt{h(s)^{2}+\rho^{2}+r^{2}\beta^{2}} (h(s)^{2}+(\rho-r\beta)^{2})^{3/2}}\\
&\leq C r \int_{t}^{t+3h(t)} ds \int_{0}^{s-t} \frac{\rho d\rho}{\sqrt{(s-t)^{2}-\rho^{2}}} \frac{\lambda(s)^{4} \sup_{x \in [100,s]}\left(\lambda(x) \log(x)\right) \log^{3}(s)}{s^{4} h(s)^{4}}\\
&+C r \int_{t+3h(t)}^{t+3r} ds \int_{0}^{3h(t)} \frac{\rho d\rho}{\sqrt{(s-t)^{2}-\rho^{2}}} \frac{\lambda(s)^{4} \sup_{x \in [100,s]}\left(\lambda(x) \log(x)\right) \log^{3}(s)}{s^{4}h(s)^{4}}\\
&+C r \int_{t+3h(t)}^{t+3r} ds \int_{3h(t)}^{s-t} \frac{d\rho}{\sqrt{(s-t)^{2}-\rho^{2}}} \frac{\lambda(s)^{4} \sup_{x \in [100,s]}\left(\lambda(x) \log(x)\right) \log^{3}(s)}{s^{4}} \int_{0}^{1}\frac{d\beta}{(h(s)^{2}+(\rho-r\beta)^{2})^{3/2}}\\
&+C r \int_{t+3r}^{\infty} ds \int_{0}^{3h(t)} \frac{\rho d\rho}{\sqrt{(s-t)^{2}-\rho^{2}}} \frac{\lambda(s)^{4} \sup_{x \in [100,s]}\left(\lambda(x) \log(x)\right) \log^{3}(s)}{s^{4} h(s)^{4}}\\
&+C r \int_{t+3r}^{\infty} ds \int_{3h(t)}^{3r} \frac{d\rho}{\sqrt{(s-t)^{2}-\rho^{2}}} \frac{\lambda(s)^{4} \sup_{x \in [100,s]}\left(\lambda(x) \log(x)\right) \log^{3}(s)}{s^{4}} \int_{0}^{1}\frac{d\beta}{(h(s)^{2}+(\rho-r \beta)^{2})^{3/2}}\\
&+C r \int_{t+3r}^{\infty} ds \int_{3r}^{s-t} \frac{d\rho}{\sqrt{(s-t)^{2}-\rho^{2}}} \frac{\lambda(s)^{4} \sup_{x \in [100,s]}\left(\lambda(x) \log(x)\right) \log^{3}(s)}{s^{4}(r+\rho)^{3}}\\
&\leq \frac{C r \lambda(t)^{4} \sup_{x \in [100,t]}\left(\lambda(x) \log(x)\right) \log^{5}(t)}{t^{4}h(t)^{2}}, \quad 2h(t) \leq r \leq \frac{t}{2}\end{split}\end{equation}
The derivatives of $u_{2,2}$ in the region $2h(t) \leq r \leq \frac{t}{2}$ are treated in a similar way. A similar procedure establishes \eqref{dru22ptwseuptocone}. Finally, the estimate on $\partial_{r}^{2} u_{2,2}$ is obtained by noting that
$$\partial_{r}^{2} u_{2,2}(t,r) = e_{2,2}(t,r) + \frac{u_{2,2}(t,r)}{r^{2}} - \frac{1}{r}\partial_{r}u_{2,2}(t,r) + \partial_{t}^{2}u_{2,2}(t,r)$$
For \eqref{u22enest}, we again use Lemma \ref{v22ests} to get that
\begin{equation}\label{e22globalest}|e_{2,2}(t,r)| \leq \mathbbm{1}_{\{r \geq h(t)\}} \begin{cases} \frac{C \lambda(t)^{2}}{r^{4}} \frac{r \lambda(t)^{2}}{t^{4}} \sup_{x \in [100,t]}\left(\lambda(x) \log(x)\right) \log^{3}(t), \quad h(t) \leq r \leq \frac{t}{2}\\
\frac{C \lambda(t)^{2}}{r^{4}} \frac{\sup_{x \in [100,t]}\left(\lambda(x) \log(x)\right) \sup_{x \in [100,t]}\left(\lambda(x)^{2}\right) \log^{3}(t)}{\langle t-r \rangle^{5/2}\sqrt{t}}, \quad \frac{t}{2} \leq r \leq t\\
\frac{C \lambda(t)^{2}}{r^{9/2} \sqrt{\langle t-r \rangle}}, \quad r \geq t\end{cases}\end{equation}
which gives
$$||e_{2,2}(t,r)||_{L^{2}(r dr)} \leq \frac{C \lambda(t)^{2} \sup_{x \in [100,t]}\left(\lambda(x) \log(x)\right) \sup_{x \in [100,t]}\left(\lambda(x)^{2}\right) \log^{3}(t)}{t^{4}}$$
Then, the same procedure used in \eqref{uw2enest} (energy estimate) finishes the proof of \eqref{u22enest}. Finally, we prove \eqref{dru22inftyest} by using Lemma \ref{v22ests} to get 
$$|I_{e_{2,2}}(s,r,\rho,\theta)| \leq \frac{C \log^{4}(s)}{s^{\frac{9}{2}-5C_{u}}}$$
which implies
$$|\partial_{r}u_{2,2}(t,r)| \leq C \int_{t}^{\infty} ds \frac{(s-t) \log^{4}(s)}{s^{\frac{9}{2}-5C_{u}}} \leq \frac{C \log^{4}(t)}{t^{\frac{5}{2}-5C_{u}}}$$
 \end{proof}
Note that $$(1-\chi_{\leq 1}(\frac{2r}{h(t)}))(1-\chi_{\leq 1}(\frac{r}{h(t)}))=(1-\chi_{\leq 1}(\frac{r}{h(t)}))$$
So, we will add the following truncation of $u_{2,2}$ into our ansatz:
\begin{equation}\label{f22def} f_{2,2}(t,r) := u_{2,2}(t,r) \left(1-\chi_{\leq 1}(\frac{2r}{h(t)})\right)\end{equation}
We define the error term associated to $f_{2,2}$ by
\begin{equation}\begin{split} e_{f_{2,2}}(t,r):&= \left(\frac{\cos(2Q_{1}(\frac{r}{\lambda(t)}))-1}{r^{2}}\right) v_{2,2}(t,r) \left(1-\chi_{\leq 1}(\frac{r}{h(t)})\right) -\left(-\partial_{t}^{2}+\partial_{r}^{2}+\frac{1}{r}\partial_{r}-\frac{1}{r^{2}}\right)f_{2,2}(t,r) \\
&+ \left(\frac{\cos(2Q_{1}(\frac{r}{\lambda(t)}))-1}{r^{2}}\right) f_{2,2}\\
&=u_{2,2}(t,r) \left(\frac{2}{rh(t)} \chi_{\leq 1}'(\frac{2r}{h(t)}) - \frac{4r}{h(t)^{3}} \chi_{\leq 1}'(\frac{2r}{h(t)}) h'(t)^{2} - \frac{4r^{2} h'(t)^{2}}{h(t)^{4}} \chi_{\leq 1}''(\frac{2r}{h(t)})\right.\\
&\left. + \frac{4 \chi_{\leq 1}''(\frac{2r}{h(t)})}{h(t)^{2}}+\frac{2r h''(t)}{h(t)^{2}} \chi_{\leq 1}'(\frac{2r}{h(t)})\right)\\
&+\frac{4 \chi_{\leq 1}'(\frac{2 r}{h(t)})}{h(t)^{2}} \left(h(t)\partial_{r}u_{2,2}+r h'(t) \partial_{t}u_{2,2}\right) + \left(\frac{\cos(2Q_{1}(\frac{r}{\lambda(t)}))-1}{r^{2}}\right) u_{2,2}(t,r) \left(1-\chi_{\leq 1}(\frac{2r}{h(t)})\right)\end{split}\end{equation}
Then, a direct estimation gives the following lemma.
\begin{lemma}\label{ef22estlemma} For $k=0,1$, \begin{equation} ||L_{\frac{1}{\lambda(t)}}^{k}(e_{f_{2,2}})(t,r)||_{L^{2}(r dr)} \leq \frac{C \lambda(t)^{4} \sup_{x \in [100,t]}\left(\lambda(x) \log(x)\right)\log^{5}(t)}{t^{4} h(t)^{2+k}} \end{equation}
\end{lemma}
Now, we let
\begin{equation}\label{uadef} u_{a}(t,r) = u_{c}(t,r)+f_{5,0}(t,r)+f_{ex,sub}(t,r)+f_{ell,2}(t,r)+f_{2,2}(t,r)\end{equation}
\begin{equation} u_{corr}(t,r) = u_{a}(t,r)+u_{1}(t,r)\end{equation}
where we recall that $u_{c}$ is defined by \eqref{ucdef}. The equation for $u_{1}$ which results from substituting $Q_{\frac{1}{\lambda(t)}}(r)+u_{corr}$ into \eqref{wm} is
\begin{equation}\begin{split} &-\partial_{t}^{2} u_{1}+\partial_{r}^{2}u_{1}+\frac{1}{r}\partial_{r}u_{1}-\frac{\cos(2Q_{1}(\frac{r}{\lambda(t)}))}{r^{2}} u_{1}\\
&=e_{a}(t,r)+\left(\frac{\cos(2u_{corr})-1}{2r^{2}}\right) \sin(2Q_{1}(\frac{r}{\lambda(t)})) + \frac{\cos(2Q_{1}(\frac{r}{\lambda(t)}))}{2r^{2}}\left(\sin(2u_{corr})-2u_{corr}\right)\end{split}\end{equation}
where 
$$e_{a}(t,r) = \left(e_{ex,ell}+e_{w,2}+e_{match,0}+e_{5,1}+e_{f_{5,0}}+e_{ex,sub}+e_{ell,2}+e_{f_{2,2}}\right)(t,r)$$
Combining Lemmas \ref{ef22estlemma}, \ref{fell2lemma}, \ref{eexsubestlemma}, \ref{ef50estlemma}, \ref{e51estlemma}, \ref{ematch0estlemma}, \ref{eexellew2estlemma} and using \eqref{lambdaonlyconstr} and \eqref{alphaconstr} gives the following lemma.
\begin{lemma}\label{eaestlemma}
\begin{equation}\begin{split}\frac{||e_{a}||_{L^{2}(r dr)}}{\lambda(t)^{2}} \leq \frac{C \log^{6}(t)}{t^{4+\text{min}\{4\alpha-C_{l}-2,1-\alpha-2C_{u},4-5C_{u}-6\alpha\}}}\end{split}\end{equation}
\begin{equation}\begin{split}&||L_{\frac{1}{\lambda(t)}}\left(e_{a}\right)||_{L^{2}(r dr)} \leq \frac{C \log^{6}(t)}{t^{4+2\alpha-3C_{u}}}+\frac{C \log^{2}(t)}{t^{2+5\alpha}}+\frac{C \log^{5}(t)}{t^{5-3C_{u}}}+\frac{C \log^{3}(t)}{t^{8-6C_{u}-5\alpha}}\end{split}\end{equation}
\end{lemma}
\subsection{First set of nonlinear interactions}
It only remains to treat the terms which are nonlinear in $u_{corr}=u_{a}+u_{1}$. For this, we define
\begin{equation}\begin{split}NL(t,r) &= \left(\frac{\cos(2u_{corr})-1}{2r^{2}}\right) \sin(2Q_{1}(\frac{r}{\lambda(t)})) + \frac{\cos(2Q_{1}(\frac{r}{\lambda(t)}))}{2r^{2}}\left(\sin(2u_{corr})-2u_{corr}\right)\\
&=\left(\frac{\cos(2u_{1})-1}{2r^{2}}\right) \sin(2Q_{1}(\frac{r}{\lambda(t)})+2u_{a})\\
&+\left(\frac{\sin(2u_{1})-2u_{1}}{2r^{2}}\right) \cos(2Q_{1}(\frac{r}{\lambda(t)})+2u_{a}) + \frac{\cos(2Q_{1}(\frac{r}{\lambda(t)}))}{2r^{2}} \left(\sin(2u_{a})-2u_{a}\right)\\
&+u_{1}\left(\frac{\cos(2Q_{1}(\frac{r}{\lambda(t)})+2u_{a})-\cos(2Q_{1}(\frac{r}{\lambda(t)}))}{r^{2}}\right) +\frac{\sin(2Q_{1}(\frac{r}{\lambda(t)}))}{2r^{2}} \left(\cos(2u_{a})-1\right)\end{split}\end{equation}
We focus on the $u_{1}$-independent terms of the expression above. In particular, we define
$$N_{a}(t,r) = \frac{\cos(2Q_{1}(\frac{r}{\lambda(t)}))}{2r^{2}} \left(\sin(2u_{a})-2u_{a}\right) + \frac{\sin(2Q_{1}(\frac{r}{\lambda(t)}))}{2r^{2}} \left(\cos(2u_{a})-1\right)$$
We recall that
$$u_{a} = \chi_{\leq 1}(\frac{r}{h(t)}) u_{e}+\left(1-\chi_{\leq 1}(\frac{r}{h(t)})\right) u_{wave} + f_{5,0}+f_{ex,sub}+f_{ell,2}+f_{2,2}  $$
We write
\begin{equation}\label{uasplit}u_{a} = u_{a,0}+u_{a,1}\end{equation}
with
$$u_{a,0}(t,r) = \chi_{\leq 1}(\frac{r}{h(t)}) u_{ell}(t,r) + \left(1-\chi_{\leq 1}(\frac{r}{h(t)})\right)\left(w_{1}(t,r)+v_{2}(t,r)+v_{2,2}(t,r)\right)$$
$$u_{a,1} = \chi_{\leq 1}(\frac{r}{h(t)}) \left(u_{e}-u_{ell}\right) + \left(1-\chi_{\leq 1}(\frac{r}{h(t)})\right)\left(v_{ex}+u_{w,2}\right)+ f_{5,0}+f_{ex,sub}+f_{ell,2}+f_{2,2}$$
Then, we get
\begin{equation}\begin{split} &N_{a}(t,r)= \frac{\cos(2Q_{1}(\frac{r}{\lambda(t)}))}{2r^{2}} \left(\sin(2u_{a,0})-2u_{a,0}\right)+\frac{\sin(2Q_{1}(\frac{r}{\lambda(t)}))}{2r^{2}} \left(\cos(2u_{a,0})-1\right)\\
&+\frac{\cos(2Q_{1}(\frac{r}{\lambda(t)}))}{2r^{2}} \left(\sin(2u_{a,0})\left(\cos(2u_{a,1})-1\right) + \left(\sin(2u_{a,1})-2u_{a,1}\right)\cos(2u_{a,0}) + \left(\cos(2u_{a,0})-1\right) 2u_{a,1}\right)\\
&+\frac{\sin(2Q_{1}(\frac{r}{\lambda(t)}))}{2r^{2}} \left(\left(\cos(2u_{a,1})-1\right) \cos(2u_{a,0}) - 2 \sin(2u_{a,0})  u_{a,1} -\sin(2u_{a,0})\left(\sin(2u_{a,1})-2u_{a,1}\right)\right)\\
&:=N_{0}+N_{1}\end{split}\end{equation}
where
\begin{equation}\label{N1def}\begin{split}&N_{1}(t,r)\\
&=\frac{\cos(2Q_{1}(\frac{r}{\lambda(t)}))}{2r^{2}} \left(\sin(2u_{a,0})\left(\cos(2u_{a,1})-1\right) + \left(\sin(2u_{a,1})-2u_{a,1}\right)\cos(2u_{a,0}) + \left(\cos(2u_{a,0})-1\right) 2u_{a,1}\right)\\
&+\frac{\sin(2Q_{1}(\frac{r}{\lambda(t)}))}{2r^{2}} \left(\left(\cos(2u_{a,1})-1\right) \cos(2u_{a,0}) - 2 \sin(2u_{a,0})  u_{a,1} -\sin(2u_{a,0})\left(\sin(2u_{a,1})-2u_{a,1}\right)\right)\end{split}\end{equation}
We proceed to estimate the term $N_{1}$, which turns out to be perturbative. 
\begin{lemma}\label{n1lemma}
We have the following estimates on $N_{1}$.
\begin{equation} \frac{||N_{1}(t,r)||_{L^{2}(r dr)}}{\lambda(t)^{2}} \leq \frac{C \log^{3}(t)}{t^{6-3C_{u}-2\alpha}} + \frac{C \log^{13}(t)}{t^{5-\frac{3}{2}C_{l}-9C_{u}}}, \quad ||L_{\frac{1}{\lambda(t)}}(N_{1})(t,r)||_{L^{2}(r dr)} \leq \frac{C \log^{3}(t)}{t^{6-5C_{u}-2\alpha-C_{l}}} + \frac{C \log^{13}(t)}{t^{5-\frac{19}{2}C_{u}-C_{l}}}\end{equation}
\end{lemma}
\begin{proof}
We estimate $N_{1}$ by combining the following estimates on various terms of $u_{a,0}$ and $u_{a,1}$. We use the explicit formulae following \eqref{velldef} and \eqref{c1def} to estimate $u_{ell}$. We use the decomposition \eqref{soln2def}, along with \eqref{vell20maindef} and the analogous estimates to \eqref{soln1est} and \eqref{soln2est} to estimate $u_{ell,2}$. We use \eqref{vexalonelgr} in the region $r \geq \frac{t}{2}$, and \eqref{vexsubsymb} and the expression for $v_{ex,ell}$, namely \eqref{vexell} in the region $g(t)\lambda(t) \leq r \leq \frac{t}{2}$, to estimate $v_{ex}$. We use Lemma \ref{uw2minuselllemma} to estimate $u_{w,2}$. (In particular, we use \eqref{uw2minusellest} and \eqref{uw2ellsymb} in the region $g(t)\lambda(t) \leq r \leq \frac{t}{2}$, and \eqref{uw2largerest} in the region $r \geq \frac{t}{2}$). Next, we use Lemma \eqref{w1estlemma}, Lemma \eqref{v2estlemma}, and Lemma \ref{v22ests} to estimate $w_{1}$, $v_{2}$, and $v_{2,2}$, respectively. Finally, we use Lemma \ref{u22lemma}, and \eqref{f22def} (for $f_{2,2}$), Lemma \ref{fell2lemma} (for $f_{ell,2}$), Lemma \ref{w50lemma} and \eqref{f50def} (for $f_{5,0}$), and Lemma \ref{wexsublemma} and \eqref{fexsub} (for $f_{ex,sub}$).  This gives rise to
\begin{equation}\label{ua0est}\begin{split}&|u_{a,0}(t,r)| \\
&\leq C  \begin{cases} \frac{r \lambda(t) \log(t)}{t^{2}}, \quad r \leq \frac{t}{2}\\
\frac{\sup_{x \in [100,r]}\left(\lambda(x) \log(x)\right) \log(r)}{\sqrt{r} \sqrt{\langle t-r \rangle}} + \frac{ \sup_{x \in [100,t]}\left(\lambda(x) \log(x)\right) \sup_{x \in [100,t]}\left(\lambda(x)^{2}\right) \log^{3}(t) \mathbbm{1}_{\{t > r > \frac{t}{2}\}}}{\langle t-r \rangle^{5/2} \sqrt{t}}, \quad r > \frac{t}{2}\end{cases}\end{split}\end{equation}
and
\begin{equation}\label{ua1est}|u_{a,1}(t,r)| \leq C \begin{cases} \frac{r^{3}(\log^{2}(r)+\log^{2}(t))}{t^{4-C_{u}}}, \quad r \leq 2 h(t)\\
\frac{\lambda(t) \log^{3}(t)}{t^{2-2C_{u}}r} + \frac{C r \lambda(t) \log^{6}(t)}{t^{4-3C_{u}}}, \quad 2 h(t) \leq r \leq \frac{t}{2}\\
\frac{\log^{5}(t)}{t^{\frac{5}{2}-\frac{7}{2}C_{u}}} \cdot \text{min}\{1,\frac{\lambda(t)}{t^{\frac{C_{u}-C_{l}}{2}}}\}, \quad r \geq \frac{t}{2}\end{cases}\end{equation}
Then, a straightforward, but slightly long computation gives the estimate on $||N_{1}(t,r)||_{L^{2}(r dr)}$ from the lemma statement. Recall that
$$L_{\frac{1}{\lambda(t)}}N_{1}(t,r) = \partial_{r}N_{1}(t,r)-\frac{\cos(Q_{1}(\frac{r}{\lambda(t)}))}{r} N_{1}(t,r)$$
To estimate $||L_{\frac{1}{\lambda(t)}}N_{1}(t,r)||_{L^{2}(r dr)}$, we use the procedure outlined at the beginning of the proof to estimate $\partial_{r}u_{a,0}$ and $\partial_{r}u_{a,1}$. The only difference here is the following. After differentiating $N_{1}(t,r)$ with respect to $r$, we use the pointwise estimates on $\partial_{r}u_{2,2}(t,r)$ given in \eqref{u22ptwse} for the region $r \leq \frac{t}{2}$. We use Holder's inequality and \eqref{u22enest} to estimate the terms involving $\partial_{r}u_{2,2}(t,r)$ in the region $r \geq \frac{t}{2}$.
 \end{proof}
It remains to treat $N_{0}$, which we recall is defined by
\begin{equation}\label{N0def}N_{0}(t,r) = \frac{\cos(2Q_{1}(\frac{r}{\lambda(t)}))}{2r^{2}} \left(\sin(2u_{a,0})-2u_{a,0}\right) + \frac{\sin(2Q_{1}(\frac{r}{\lambda(t)}))}{2r^{2}} \left(\cos(2u_{a,0})-1\right)\end{equation}
where
$$u_{a,0}(t,r) = \chi_{\leq 1}(\frac{r}{h(t)}) u_{ell}(t,r) + \left(1-\chi_{\leq 1}(\frac{r}{h(t)})\right)\left(w_{1}(t,r)+v_{2}(t,r)+v_{2,2}(t,r)\right)$$
We define $u_{N_{0}}$ to be the solution to the following equation with $0$ Cauchy data at infinity.
\begin{equation}\label{un0eqn}-\partial_{t}^{2}u_{N_{0}}+\partial_{r}^{2}u_{N_{0}}+\frac{1}{r}\partial_{r}u_{N_{0}}-\frac{u_{N_{0}}}{r^{2}} = N_{0}(t,r)\end{equation}
Then, we have
\begin{lemma}\label{un0lemma} We have the following estimates. For $0 \leq k \leq 2$ and $0 \leq j \leq 1$,
\begin{equation}\label{un0ptwse}|\partial_{t}^{k}\partial_{r}^{j}u_{N_{0}}(t,r)| \leq  \frac{C r^{1-j} \left(\sup_{x\in[100,t]}\left(\lambda(x) \log(x)\right)\right)^{3} \log^{3}(t)}{t^{4+k}}, \quad r \leq \frac{t}{2}\end{equation}

\begin{equation}\label{un0enest}\begin{split}&|u_{N_{0}}(t,r)| + \sqrt{E(u_{N_{0}},\partial_{t}u_{N_{0}})}\leq \frac{C \left(\sup_{x\in[100,t]}\left(\lambda(x) \log(x)\right)\right)^{3} \left(\sup_{x \in [100,t]}\left(\lambda(x)^{2}\right)\right)^{3} \log^{9}(t)}{t^{2}}\end{split}\end{equation}
Finally,
\begin{equation}\label{drun0ptwseuptocone} \begin{split} &|\partial_{r}u_{N_{0}}(t,r)|\\
&\leq \frac{C \left(\sup_{x \in [100,t]}\left(\lambda(x) \log(x)\right)\right)^{3} \log^{3}(t)}{t^{3/2} \langle t-r \rangle^{5/2}} + \frac{C \left(\sup_{x \in [100,t]}\left(\lambda(x) \log(x)\right)\right)^{3} \left(\sup_{x \in [100,t]}\left(\lambda(x)^{2}\right)\right)^{3} \log^{9}(t)}{t^{3/2} \langle t-r \rangle^{17/2}} \\
&+ \frac{C \lambda(t) \left(\sup_{x \in [100,t]}\left(\lambda(x) \log(x)\right)\right)^{2} \left(\sup_{x \in [100,t]}\left(\lambda(x)^{2}\right)\right)^{2} \log^{6}(t)}{t^{2} \langle t-r \rangle^{6}}, \quad \frac{t}{2} < r < t\end{split}\end{equation}
\begin{equation}\label{drun0inftyest} ||\partial_{r}u_{N_{0}}(t,r)||_{L^{\infty}_{r}} \leq \frac{C \log^{12}(t)}{t^{\frac{3}{2}-9C_{u}}}\end{equation}
\end{lemma}
\begin{proof}
As in \eqref{q41orig}, we have 
\begin{equation}\label{un01deriv}u_{N_{0}}(t,r) = -\frac{r}{2\pi} \int_{t}^{\infty} ds \int_{0}^{s-t} \frac{\rho d\rho}{\sqrt{(s-t)^{2}-\rho^{2}}} \int_{0}^{1}d\beta \int_{0}^{2\pi} d\theta I_{N_{0}}(s,r\beta,\rho,\theta)\end{equation}
where we use the notation defined in \eqref{Inotation}. We then estimate $u_{a,0}$ and its derivatives using the procedure described in the proof of Lemma \ref{n1lemma}. This results in the following estimate for $j=0,1,2$ and $k=0,1$.
\begin{equation}\label{n0estsforun0}\begin{split}&|\partial_{t}^{j}\partial_{r}^{k} N_{0}(t,r)| \\
&\leq C\left(\frac{1}{r}+\frac{1}{\langle t-r \rangle}\right)^{k} \left(\frac{1}{t}+\frac{1}{\langle t-r \rangle}\right)^{j} \begin{cases} \frac{r \lambda(t)^{3} \log^{3}(t)}{t^{4} (r^{2}+\lambda(t)^{2})}, \quad r \leq \frac{t}{2}\\
\frac{\lambda(t)}{t^{4}} \frac{\left(\sup_{x \in [100,t]}\left(\lambda(x) \log(x)\right)\right)^{2}}{\langle t-r \rangle^{5}} \left(\sup_{x \in [100,t]}\left(\lambda(x)^{2}\right)\right)^{2} \log^{6}(t) \\
+ \frac{\left(\sup_{x \in [100,t]}\left(\lambda(x) \log(x)\right)\right)^{3} \log^{3}(t)}{r^{7/2} \langle t-r \rangle^{3/2}}\\
+\frac{\left(\sup_{x \in [100,t]}\left(\lambda(x) \log(x)\right)\right)^{3} \left(\sup_{x \in [100,t]}\left(\lambda(x)^{2}\right)\right)^{3} \log^{9}(t)}{\langle t-r \rangle^{15/2} t^{7/2}}, \quad \frac{t}{2} < r < t\end{cases}\end{split}\end{equation}
 We start with the region $r \leq \frac{t}{2}$. Using an analog of \eqref{thetaint}, we get
\begin{equation}\label{un0formulaforest}\begin{split} &|u_{N_{0}}(t,r)| \\
& \leq C r \int_{t}^{\infty} ds \int_{0}^{s-t} \frac{\rho d\rho}{\sqrt{(s-t)^{2}-\rho^{2}}} \int_{0}^{1}d\beta \begin{aligned}[t]&\left(\frac{\lambda(s)^{3} \log^{3}(s)}{s^{4}  \sqrt{(\lambda(s)^{2}+(\rho+r\beta)^{2})(\lambda(s)^{2} +(\rho-r \beta)^{2})}}\right.\\
&\left.+\frac{\lambda(s) \left(\sup_{x\in[100,s]}\left(\lambda(x) \log(x)\right)\right)^{2} \left(\sup_{x \in [100,s]}\left(\lambda(x)^{2}\right)\right)^{2} \log^{6}(s)}{s^{4} (s-(\rho+r))^{6}}\right.\\
&+\left. \frac{\left(\sup_{x\in[100,s]}\left(\lambda(x) \log(x)\right)\right)^{3} \log^{3}(s)}{s^{7/2} (s-(\rho+r))^{5/2}}\right.\\
&\left.+\frac{\left(\sup_{x\in[100,s]}\left(\lambda(x) \log(x)\right)\right)^{3} \left(\sup_{x \in [100,s]}\left(\lambda(x)^{2}\right)\right)^{3} \log^{9}(s)}{s^{7/2} (s-(\rho+r))^{17/2}}\right)\end{aligned}\end{split}\end{equation}
 where we used the fact that, if $0\leq \beta \leq 1$, then,
$$\langle s-\sqrt{\rho^{2}+r^{2}\beta^{2}+2 r \beta \rho \cos(\theta)}\rangle \geq C\left(s-(\rho+r)\right). $$
By a direct estimation, we get
$$|\int_{0}^{1}\frac{d\beta}{\sqrt{\lambda(s)^{2}+(r\beta-\rho)^{2}}} \frac{1}{(\lambda(s)+r\beta+\rho)}| \leq \frac{C \log(s)}{(r+\rho)(\lambda(s)+\rho)}$$
and therefore,
\begin{equation}\label{firstintegralforun0}r \int_{t}^{\infty} ds  \int_{0}^{s-t} \frac{\rho d\rho}{\sqrt{(s-t)^{2}-\rho^{2}}} \int_{0}^{1}d\beta \frac{\lambda(s)^{3} \log^{3}(s)}{s^{4} \log^{2b}(s) \sqrt{(\lambda(s)^{2}+(\rho+r \beta)^{2})(\lambda(s)^{2}+(\rho-r\beta)^{2})}} \leq \frac{C r \lambda(t)^{3} \log^{6}(t)}{t^{4}}\end{equation}
The other terms of \eqref{un0formulaforest} are treated with a similar argument, using $s-(r+\rho) \geq s-r-(s-t) \geq C t$ (since $r \leq \frac{t}{2}$). This gives \eqref{un0ptwse}, for $k=0,j=0$. To obtain \eqref{un0ptwse} for $k=1,2$ and $j=0$, we first note that
$$\partial_{t}^{k} u_{N_{0}}(t,r) = -\frac{r}{2\pi} \int_{t}^{\infty} ds \int_{0}^{s-t} \frac{\rho d\rho}{\sqrt{(s-t)^{2}-\rho^{2}}} \int_{0}^{1}d\beta \int_{0}^{2\pi} d\theta \partial_{1}^{k}I_{N_{0}}(s,r\beta,\rho,\theta)$$
Then, after applying \eqref{n0estsforun0}, we obtain an extra factor of 
$$\left(\frac{1}{s}+\frac{1}{\langle s-\sqrt{r^{2}\beta^{2}+\rho^{2}+2 r \beta \rho \cos(\theta)}\rangle}\right)^{k}\leq \left(\frac{C}{s} + \frac{C}{s-(r+s-t)}\right)^{k} \leq \frac{C}{t^{k}}$$
in the integrand of the expression for $\partial_{t}^{k} u_{N_{0}}(t,r)$, relative to that for $u_{N_{0}}$, which we just estimated. This immediately gives \eqref{un0ptwse} for $k=1,2$. Next, we use the same procedure as in \eqref{ptwseenest}, to get \eqref{un0enest}. Finally, we use
\begin{equation}\label{drun0intstep} \partial_{r}u_{N_{0}}(t,r) = \frac{-1}{2\pi} \int_{t}^{\infty} ds \int_{0}^{s-t} \frac{\rho d\rho}{\sqrt{(s-t)^{2}-\rho^{2}}} \int_{0}^{2\pi} d\theta I_{N_{0}}(s,r,\rho,\theta)\end{equation} 
and a similar argument used to estimate $u_{N_{0}}$, to get \eqref{un0ptwse} for $j=1,k=0$. For completeness, we show how to estimate the most delicate integral, which is
$$\int_{t}^{\infty} ds \int_{0}^{s-t} \frac{\rho d\rho}{\sqrt{(s-t)^{2}-\rho^{2}}} \frac{\lambda(s)^{3} \log^{3}(s)}{s^{4}  \sqrt{\lambda(s)^{2}+(\rho+r)^{2}}\sqrt{\lambda(s)^{2}+(\rho-r)^{2}}}$$
Here, we use
$$\frac{d}{d\rho}\left(\text{arctanh}(\frac{\rho-r}{\sqrt{a^{2}+(\rho-r)^{2}}})\right) = \frac{1}{\sqrt{a^{2}+(r-\rho)^{2}}}$$
and estimate the integral directly. A similar argument establishes \eqref{drun0ptwseuptocone}. Finally, to obtain \eqref{drun0inftyest}, we again estimate $u_{a,0}$ and $\partial_{r}u_{a,0}$ using the procedure described in the proof of Lemma \ref{n1lemma} to get, for all $r>0, \theta \in [0,2\pi], \rho \leq s-t$,
$$|I_{N_{0}}(s,r,\rho,\theta)| \leq \frac{C \log^{12}(s)}{s^{\frac{7}{2}-9C_{u}}}$$
This gives
$$|\partial_{r}u_{N_{0}}(t,r)| \leq C \int_{t}^{\infty} ds (s-t) \frac{C \log^{12}(s)}{s^{\frac{7}{2}-9C_{u}}} \leq \frac{C \log^{12}(t)}{t^{\frac{3}{2}-9C_{u}}}$$
 \end{proof}
The linear error term of $u_{N_{0}}$ is $e_{N_{0}}$ defined by
\begin{equation}\label{en0def}e_{N_{0}}(t,r):= \left(\frac{\cos(2Q_{1}(\frac{r}{\lambda(t)}))-1}{r^{2}}\right) u_{N_{0}}(t,r)\end{equation}
Then, a straightforward estimation, using Lemma \ref{un0lemma} gives the following (recall \eqref{mleq1def}). 
\begin{lemma} We have the following estimates for $k=0,1$.
\begin{equation}\label{en0estlemma}\begin{split}&||L_{\frac{1}{\lambda(t)}}^{k}(e_{N_{0}}(t,r) (1-m_{\leq 1}(\frac{r}{2h(t)})))||_{L^{2}(r dr)} \\
&\leq \frac{C \lambda(t)^{2} \left(\sup_{x \in [100,t]}\left(\lambda(x) \log(x)\right)\right)^{3} \log^{3}(t)}{t^{4}}\left(\frac{1}{h(t)^{2+k}}+\frac{\left(\sup_{x \in [100,t]}\left(\lambda(x)^{2}\right)\right)^{3} \log^{6}(t)}{t^{1+k}}\right)\end{split}\end{equation}
\end{lemma}
It remains to treat $m_{\leq 1}(\frac{r}{2h(t)}) e_{N_{0}}(t,r)$. For the reader's convenience, we repeat the outline (that was given in Section 3, just below \eqref{un0summary}) of the procedure to be used. We start with solving the equation
\begin{equation}\label{un0elleqn}\partial_{rr}u_{N_{0},ell}+\frac{1}{r}\partial_{r}u_{N_{0},ell}-\frac{\cos(2Q_{\frac{1}{\lambda(t)}}(r))}{r^{2}}u_{N_{0},ell} = m_{\leq 1}(\frac{r}{2h(t)}) e_{N_{0}}(t,r)\end{equation}
and then inserting the following function into the ansatz
\begin{equation}\label{un0corrdef}u_{N_{0},corr}(t,r):=\left(u_{N_{0},ell}(t,r) - \frac{r \lambda(t)}{4} \langle m_{\leq 1}(\frac{R \lambda(t)}{2h(t)}) e_{N_{0}}(t,R \lambda(t)),\phi_{0}(R)\rangle_{L^{2}(R dR)}\right) m_{\leq 1}(\frac{2 r}{t}) + v_{2,4}(t,r)\end{equation}
where $v_{2,4}$ solves
\begin{equation}\label{v24eqn}-\partial_{tt}v_{2,4}+\partial_{rr}v_{2,4}+\frac{1}{r}\partial_{r}v_{2,4}-\frac{v_{2,4}}{r^{2}}=0\end{equation}
and $v_{2,4}(t,r)$ matches $\frac{r \lambda(t)}{4} \langle m_{\leq 1}(\frac{R \lambda(t)}{2h(t)}) e_{N_{0}}(t,R \lambda(t)),\phi_{0}(R)\rangle_{L^{2}(R dR)}$ for small $r$. (Recall that $m_{\leq 1}$ was defined in \eqref{mleq1def}). In particular, we choose the initial velocity, say $v_{2,5}$ of $v_{2,4}$ by requiring, for all $t$ sufficiently large,
$$\frac{-r}{4}\cdot -2 \int_{0}^{\infty} \xi \sin(t\xi) \widehat{v_{2,5}}(\xi) d\xi =\frac{r \lambda(t)}{4} \langle m_{\leq 1}(\frac{R \lambda(t)}{2h(t)}) e_{N_{0}}(t,R \lambda(t)),\phi_{0}(R)\rangle_{L^{2}(R dR)}$$ 
We refer the reader to the discussion in Section 3, just below \eqref{un0summary}, for some intuition regarding this procedure.\\
\\
We start by estimating $u_{N_{0},ell}(t,r)- \frac{r \lambda(t)}{4} \langle m_{\leq 1}(\frac{R \lambda(t)}{2h(t)}) e_{N_{0}}(t,R \lambda(t)),\phi_{0}(R)\rangle$.
\begin{lemma} \label{un0ellminusipestlemma} We have the following estimates. For $0 \leq k \leq 1$ and $0 \leq j \leq 2$, 
\begin{equation}\label{un0minusipest}\begin{split} &t^{j}r^{k} |\partial_{t}^{j}\partial_{r}^{k}\left(u_{N_{0},ell}(t,r)-\frac{r \lambda(t)}{4} \langle m_{\leq 1}(\frac{R \lambda(t)}{2h(t)}) e_{N_{0}}(t,R \lambda(t)),\phi_{0}(R)\rangle\right)|\\
&\leq \begin{cases} \frac{C r \log^{3}(t) \left(\sup_{x \in [100,t]}\left(\lambda(x) \log(x)\right)\right)^{3}}{t^{4}}  , \quad r \leq \lambda(t)\\
\frac{C (\log(t)+\log(\frac{r}{\lambda(t)})) \lambda(t)^{2}\left(\sup_{x \in [100,t]}\left(\lambda(x) \log(x)\right)\right)^{3} \log^{3}(t)}{r t^{4}} , \quad \lambda(t) < r\end{cases}\end{split}\end{equation}
In addition, we have
\begin{equation}\label{dr2un0minusipest}\begin{split}&|\partial_{r}^{2} \left(u_{N_{0},ell}(t,r)-\frac{r \lambda(t)}{4} \langle m_{\leq 1}(\frac{R \lambda(t)}{2h(t)}) e_{N_{0}}(t,R \lambda(t)),\phi_{0}(R)\rangle\right)| \\
&\leq C \begin{cases}  \frac{r}{\lambda(t)^{2}} \frac{\left(\sup_{x \in [100,t]}\left(\lambda(x) \log(x)\right)\right)^{3} \log^{3}(t)}{t^{4}}, \quad r \leq \lambda(t)\\
\frac{ \lambda(t)^{2}(\log(t)+\log(\frac{r}{\lambda(t)}))\left(\sup_{x \in [100,t]}\left(\lambda(x) \log(x)\right)\right)^{3} \log^{3}(t)}{r^{3}t^{4}}  , \quad r \geq \lambda(t)\end{cases}\end{split}\end{equation}
Finally, for $0 \leq k \leq 2$,
\begin{equation}\label{en0ipest}\begin{split}&|\partial_{t}^{k}\langle e_{N_{0}}(t,R\lambda(t)) m_{\leq 1}(\frac{R \lambda(t)}{2 h(t)}),\phi_{0}(R)\rangle| \leq \frac{C \log^{3}(t) \left(\sup_{x \in [100,t]}\left(\lambda(x) \log(x)\right)\right)^{3}}{t^{4+k}\lambda(t)}\end{split}\end{equation} 
\end{lemma}
\begin{proof}
We consider the particular solution given by 
$$u_{N_{0},ell}(t,r) = v_{N_{0},ell}(t,\frac{r}{\lambda(t)})$$
where
\begin{equation}\begin{split}v_{N_{0},ell}(t,R)&= e_{2}(R) \int_{0}^{R} \lambda(t)^{2} m_{\leq 1}(\frac{s}{2g(t)}) e_{N_{0}}(t,s\lambda(t)) \phi_{0}(s) \frac{s ds}{2}\\
&-\phi_{0}(R) \int_{0}^{R} \lambda(t)^{2} m_{\leq 1}(\frac{s}{2g(t)}) e_{N_{0}}(t,s\lambda(t)) e_{2}(s) \frac{s ds}{2}\end{split}\end{equation}
We remind the reader that $e_{N_{0}}$ is defined in \eqref{en0def}.  Let $int_{N_{0},i}(t,r)$ denote the $i$th term on the right-hand side of the above equation, evaluated at $R=\frac{r}{\lambda(t)}$. In the region $r \leq \lambda(t)$, we separately estimate $int_{N_{0},1}$, $int_{N_{0},2}$, and $\frac{r \lambda(t)}{4} \langle m_{\leq 1}(\frac{R \lambda(t)}{2h(t)}) e_{N_{0}}(t,R \lambda(t)),\phi_{0}(R)\rangle$, using Lemma \ref{un0lemma}. This gives \eqref{un0minusipest}, in the region $r \leq \lambda(t)$, and for $k=j=0$. When $r \geq \lambda(t)$, we take advantage of the fact that
$$|e_{2}(R)-\frac{R}{2}| \leq C\frac{(1+\log(R))}{R}, \quad R \geq 1$$
In particular, we estimate $int_{N_{0},2}(t,r)$ directly, using Lemma \ref{un0lemma} and write
\begin{equation}\begin{split} &int_{N_{0},1}(t,r) - \frac{r \lambda(t)}{4} \langle m_{\leq 1}(\frac{R \lambda(t)}{2h(t)}) e_{N_{0}}(t,R \lambda(t)),\phi_{0}(R)\rangle\\
&=\left(\frac{e_{2}(\frac{r}{\lambda(t)})}{2}-\frac{r}{4 \lambda(t)}\right) \int_{0}^{\frac{r}{\lambda(t)}} \lambda(t)^{2} m_{\leq 1}(\frac{s}{2g(t)}) e_{N_{0}}(t,s\lambda(t)) \phi_{0}(s) s ds\\
&-\frac{r \lambda(t)}{4} \int_{\frac{r}{\lambda(t)}}^{\infty} m_{\leq 1}(\frac{s}{2 g(t)}) e_{N_{0}}(t,s\lambda(t)) \phi_{0}(s) s ds\end{split}\end{equation}
Directly estimating the integrals using Lemma \ref{un0lemma}, we obtain \eqref{un0minusipest}, in the region $r \geq \lambda(t)$, and $k=j=0$. To obtain \eqref{un0minusipest} for $k=1$, we use
\begin{equation}\begin{split}\partial_{r}u_{N_{0},ell}(t,r)&=\frac{-\phi_{0}'(\frac{r}{\lambda(t)})}{\lambda(t)} \int_{0}^{\frac{r}{\lambda(t)}} \lambda(t)^{2} m_{\leq 1}(\frac{s}{2g(t)}) e_{N_{0}}(t,s\lambda(t)) e_{2}(s) \frac{s ds}{2} \\
&+ \frac{e_{2}'(\frac{r}{\lambda(t)})}{\lambda(t)} \int_{0}^{\frac{r}{\lambda(t)}} \lambda(t)^{2} m_{\leq 1}(\frac{s}{2g(t)}) e_{N_{0}}(t,s\lambda(t)) \phi_{0}(s) \frac{s ds}{2}\end{split}\end{equation}
and then repeat the same procedure used to obtain \eqref{un0minusipest} for $k=0$. For the estimate \eqref{un0minusipest} for $j>0$, we first re-write $u_{N_{0},ell}$ as follows, and then differentiate in $t$ directly.
\begin{equation}\begin{split}u_{N_{0},ell}(t,r)&=-\phi_{0}(\frac{r}{\lambda(t)}) \int_{0}^{r} m_{\leq 1}(\frac{x}{2h(t)}) e_{N_{0}}(t,x) e_{2}(\frac{x}{\lambda(t)})\frac{x dx}{2}\\
&+e_{2}(\frac{r}{\lambda(t)}) \int_{0}^{r} m_{\leq 1}(\frac{x}{2h(t)}) e_{N_{0}}(t,x) \phi_{0}(\frac{x}{\lambda(t)}) \frac{x dx}{2}\end{split}\end{equation}
Using the same procedure as for \eqref{un0minusipest} with $j=0$, and noting the symbol-type nature of the estimates in Lemma \ref{un0lemma}, we finish the proof of \eqref{un0minusipest} for $j>0$. Finally, to get \eqref{dr2un0minusipest}, we use the equation solved by $u_{N_{0},ell}$. The estimates in \eqref{en0ipest} follow directly from Lemma \ref{un0lemma}.
 \end{proof}
Just after \eqref{g3f3}, we restricted $T_{0}$ to satisfy $T_{0} \geq 2T_{2}$, and all of our computations and estimates are valid for all $t \geq T_{0}$ for any $T_{0} \geq 2 T_{2}$. At this stage, we restrict $T_{0}$ so that $T_{0} \geq 4T_{2}$ but is otherwise arbitrary. Then, recalling the cutoff $\psi_{2}$ defined in \eqref{v23def}, we define a function $v_{2,5}$ by
$$\widehat{v_{2,5}}(\xi) = \frac{1}{\pi \xi} \int_{0}^{\infty}\lambda(t) \langle m_{\leq 1}(\frac{R}{2 g(t)}) e_{N_{0}}(t,R\lambda(t)),\phi_{0}(R)\rangle_{L^{2}(R dR)} \psi_{2}(\frac{t}{2}) \sin(t\xi) dt$$
Letting 
$$F_{2,4}(t) = -2 \int_{0}^{\infty} \xi \sin(t \xi) \widehat{v_{2,5}}(\xi) d\xi$$ 
the inversion of the sine transform gives 
$$F_{2,4}(t)= -\lambda(t) \psi_{2}(\frac{t}{2})\langle m_{\leq 1}(\frac{R}{2 g(t)}) e_{N_{0}}(t,R\lambda(t)),\phi_{0}(R)\rangle_{L^{2}(R dR)}, \quad t >0$$  
Therefore, \eqref{en0ipest} gives the following estimate for $0 \leq k \leq 2$ and $t >0$:
\begin{equation}\label{f24ests}|\partial_{t}^{k}F_{2,4}(t)| \leq C \mathbbm{1}_{\{t>2T_{2}\}} \frac{\left(\sup_{x \in [100,t]}\left(\lambda(x) \log(x)\right)\right)^{3}\log^{3}(t)}{t^{4+k}}\end{equation} 
We have the following estimates on $\widehat{v_{2,5}}$.
\begin{lemma} For $0 \leq k \leq 2$,
\begin{equation}\label{v25hatest}|\xi^{k}\partial_{\xi}^{k} \widehat{v_{2,5}}(\xi)|\leq \begin{cases} C, \quad \xi \leq \frac{1}{100}\\
\frac{C}{\xi^{3-k}}, \quad \xi > \frac{1}{100}\end{cases}\end{equation}
\end{lemma}
\begin{proof}
We recall 
$$\widehat{v_{2,5}}(\xi) = \frac{-1}{\pi \xi} \int_{0}^{\infty} F_{2,4}(t) \sin(t\xi) dt$$
In the region $\xi \leq \frac{1}{100}$, we simply directly estimate as follows
$$|\widehat{v_{2,5}}(\xi)| \leq \frac{C}{\xi} \int_{0}^{\frac{1}{\xi}} |F_{2,4}(t)| t \xi dt + \frac{C}{\xi} \int_{\frac{1}{\xi}}^{\infty} |F_{2,4}(t)| dt$$
If $\xi > \frac{1}{100}$, then, we have
$$\widehat{v_{2,5}}(\xi) = \frac{1}{\pi \xi^{3}} \int_{0}^{\infty} \sin(t\xi) F_{2,4}''(t) dt$$ 
where we note that we can integrate by parts only twice (simply because of the estimates we have on $u_{N_{0}}$, which gave rise to estimates on up to two derivatives of $F_{2,4}$) which is why the decay for large $\xi$ in the estimate \eqref{v25hatest} is not as strong as the analogous estimate for $\widehat{v_{2,0}}$ in Lemma \ref{v20ests}. To estimate the derivatives of $\widehat{v_{2,5}}(\xi)$, we first re-write its formula as
$$\widehat{v_{2,5}}(\xi) = \frac{-1}{\pi \xi^{2}} \int_{0}^{\infty}  F_{2,4}(\frac{\sigma}{\xi})\sin(\sigma) d\sigma$$
and then differentiate in $\xi$. Then, we estimate as above. \end{proof}
We define $v_{2,4}$ to be the solution to the following Cauchy problem.
\begin{equation}\begin{cases}-\partial_{t}^{2}v_{2,4}+\partial_{r}^{2}v_{2,4}+\frac{1}{r}\partial_{r}v_{2,4}-\frac{v_{2,4}}{r^{2}}=0\\
v_{2,4}(0,r)=0\\
\partial_{t}v_{2,4}(0,r) = v_{2,5}(r)\end{cases}\end{equation}
Now, we obtain estimates on $v_{2,4}(t,r)-\left(\frac{-r}{4} F_{2,4}(t)\right)$
\begin{lemma}\label{v24estlemma} For $0 \leq k \leq 2$,
\begin{equation}\label{v24minusrf24}|\partial_{r}^{k}\left(v_{2,4}(t,r)+\frac{r}{4}F_{2,4}(t)\right)| \leq \frac{C r^{3-k}\log^{3}(t) \left(\sup_{x \in [100,t]}\left(\lambda(x) \log(x)\right)\right)^{3}}{t^{6}}  , \quad r \leq \frac{t}{2}\end{equation}
For $0 \leq k \leq 2$, $0 \leq j \leq 1$ and $j+k \leq 2$.
\begin{equation}\label{v24ests}|\partial_{t}^{k}\partial_{r}^{j} v_{2,4}(t,r)| \leq \frac{C r^{1-j}  \log^{3}(t)\left(\sup_{x \in [100,t]}\left(\lambda(x) \log(x)\right)\right)^{3}}{t^{4+k}} , \quad r \leq \frac{t}{2}\end{equation}
For $0 \leq k+j \leq 2$, we have the following estimate.
\begin{equation}\label{v24insideconeest}|\partial_{r}^{k}\partial_{t}^{j}v_{2,4}(t,r)| \leq \frac{C \log^{3}(t) \left(\sup_{x \in [100,t]}\left(\lambda(x) \log(x)\right)\right)^{3}}{\sqrt{t} \langle t-r\rangle^{5/2+k+j}}, \quad \frac{t}{2}\leq r < t\end{equation}
Finally, for $0 \leq k \leq 1$,
\begin{equation}\label{v24outsideconeest}|\partial_{r}^{k}v_{2,4}(t,r)| \leq \frac{C}{\sqrt{r} \langle t-r \rangle^{\frac{1}{2}+k}}, \quad r>t\end{equation}
\end{lemma}
\begin{proof}
The proof of this lemma uses the same procedure as the proof of Lemma \ref{v22lemma}. In particular, we have
\begin{equation}\label{v24intermsofF24} v_{2,4}(t,r) = \frac{-r}{2\pi} \int_{0}^{\pi} \sin^{2}(\theta) F_{2,4}(t+r\cos(\theta)) d\theta\end{equation}
In the region $r \leq \frac{t}{2}$, we have
\begin{equation} v_{2,4}(t,r) +\frac{r}{4} F_{2,4}(t) = \frac{-r}{2\pi} \int_{0}^{\pi} \sin^{2}(\theta) \left(F_{2,4}(t+r\cos(\theta))-F_{2,4}(t)\right) d\theta\end{equation}
and this immediately gives rise to \eqref{v24minusrf24}.\\
\\
 Next, we note that
\begin{equation} v_{2,4}(t,r)= \int_{0}^{\infty} d\xi J_{1}(r\xi) \sin(t\xi) \widehat{v_{2,5}}(\xi)\end{equation}
which leads to the following estimate, for all $r \geq \frac{t}{2}$.
\begin{equation}\label{v24sqrtrest}\begin{split}|v_{2,4}(t,r)| &\leq C \int_{0}^{\frac{2}{r}} d\xi r \xi + C \int_{\frac{1}{r}}^{\frac{1}{100}} \frac{d\xi}{\sqrt{r\xi}} + C \int_{\frac{1}{100}}^{\infty} \frac{d\xi}{\sqrt{r\xi}} \frac{1}{\xi^{3}}\leq \frac{C}{\sqrt{r}}\end{split}\end{equation}
The same computation leads to 
$$|\partial_{r}^{k}\partial_{t}^{j}v_{2,4}(t,r)| \leq \frac{C}{\sqrt{r}}, \quad r \geq\frac{t}{2}, \quad 0 \leq j+k \leq 2$$
On the other hand, if $t > r > \frac{t}{2}$, we have (using \eqref{v24intermsofF24})
$$|v_{2,4}(t,r)| \leq C r \log^{3}(t) \left(\sup_{x \in [100,t]}\left(\lambda(x) \log(x)\right)\right)^{3}\int_{0}^{\pi} \frac{\sin^{2}(\theta)d\theta}{(t+r\cos(\theta))^{4}}  $$
Recalling \eqref{tplusrcosthetaint}, the above estimate, combined with \eqref{v24sqrtrest}, gives \eqref{v24insideconeest}. (The same procedure is used to estimate derivatives of $v_{2,4}(t,r)$ in the region $\frac{t}{2} < r < t$). Finally, to establish \eqref{v24outsideconeest}, we use the identical procedure used to establish the analogous estimates in Lemma \ref{v2estlemma}. In particular, we do the same procedure which starts with \eqref{v2largerintstep12}. (The amount of high frequency decay in the estimate \eqref{v25hatest} is sufficient for this).
 \end{proof}
We can now define and estimate the linear error term associated to $u_{N_{0},corr}$. For ease of notation, let
$$h_{2}(t):=\lambda(t)\langle m_{\leq 1}(\frac{R \lambda(t)}{2h(t)}) e_{N_{0}}(t,R \lambda(t)),\phi_{0}(R)\rangle_{L^{2}(R dR)}$$ 
We first recall the definition of $u_{N_{0},corr}$ in \eqref{un0corrdef}:
$$u_{N_{0},corr}(t,r) = \left(u_{N_{0},ell}(t,r) - \frac{r h_{2}(t)}{4} \right) m_{\leq 1}(\frac{2 r}{t}) + v_{2,4}(t,r)$$
Then, we define the linear error term associated to $u_{N_{0},corr}$ by
\begin{equation}\begin{split} e_{N_{0},corr}(t,r)&= -\left(-\partial_{t}^{2}+\partial_{r}^{2}+\frac{1}{r}\partial_{r}-\frac{\cos(2Q_{1}(\frac{r}{\lambda(t)}))}{r^{2}}\right) u_{N_{0},corr}(t,r)+m_{\leq 1}(\frac{r}{2h(t)}) e_{N_{0}}(t,r)\end{split}\end{equation}
We therefore get
\begin{equation}\begin{split} e_{N_{0},corr}(t,r)&= \partial_{t}^{2}\left(u_{N_{0},ell}(t,r)-\frac{r}{4}h_{2}(t)\right) m_{\leq 1}(\frac{2r}{t}) - 2\partial_{t}\left(u_{N_{0,ell}}(t,r)-\frac{r}{4}h_{2}(t)\right) m_{\leq 1}'(\frac{2r}{t})\cdot \frac{2r}{t^{2}}\\
& +\left(u_{N_{0},ell}(t,r)-\frac{r}{4}h_{2}(t)\right) \left(m_{\leq 1}''(\frac{2r}{t})\frac{4r^{2}}{t^{4}}+\frac{4r}{t^{3}} m_{\leq 1}'(\frac{2r}{t})-\frac{4 m_{\leq 1}''(\frac{2r}{t})}{t^{2}}\right)\\
&-4\partial_{r}\left(u_{N_{0},ell}(t,r)-\frac{r}{4}h_{2}(t)\right) \frac{m_{\leq 1}'(\frac{2r}{t})}{t}-\frac{1}{r} m_{\leq 1}'(\frac{2r}{t}) \frac{2}{t} \left(u_{N_{0},ell}-\frac{r}{4}h_{2}(t)\right)\\
&+\left(\frac{\cos(2Q_{1}(\frac{r}{\lambda(t)}))-1}{r^{2}}\right) \left(v_{2,4}(t,r)-\frac{r}{4}h_{2}(t)\right) m_{\leq 1}(\frac{2r}{t})\\
&+\left(\frac{\cos(2Q_{1}(\frac{r}{\lambda(t)}))-1}{r^{2}}\right) v_{2,4}(t,r) \left(1-m_{\leq 1}(\frac{2r}{t})\right)\end{split}\end{equation}
where we used the fact that
$$(m_{\leq 1}(\frac{2r}{t})-1)\cdot m_{\leq 1}(\frac{r}{2h(t)})=0$$
We write
$$e_{N_{0},corr}(t,r) = e_{N_{0},corr,1}(t,r)+e_{N_{0},corr,2}(t,r)$$
where
\begin{equation}\label{en0corr2def}e_{N_{0},corr,2}(t,r) = \left(\frac{\cos(2Q_{1}(\frac{r}{\lambda(t)}))-1}{r^{2}}\right) v_{2,4}(t,r) \left(1-m_{\leq 1}(\frac{2r}{t})\right)\end{equation}
Then, it turns out that $e_{N_{0},corr,1}$ is perturbative, as per the following lemma.
\begin{lemma} \label{en0corr1estlemma}For $k=0,1$,
$$||L_{\frac{1}{\lambda(t)}}^{k}(e_{N_{0},corr,1})(t,r)||_{L^{2}(r dr)} \leq \frac{C \lambda(t)^{2-k} \log^{9/2}(t)}{t^{6}} \left(\sup_{x \in [100,t]}\left(\lambda(x) \log(x)\right)\right)^{3}$$
\end{lemma}
\begin{proof} This is a direct consequence of Lemma \ref{un0ellminusipestlemma} and \ref{v24estlemma} \end{proof}
We need to add one more term to our ansatz in order to eliminate $e_{N_{0},corr,2}$, given in \eqref{en0corr2def}. In particular, we define $u_{N_{0},corr,2}$ to be the solution to the following equation with 0 Cauchy data at infinity.
\begin{equation}\label{un0corr2eqn}-\partial_{t}^{2}u_{N_{0},corr,2}+\partial_{r}^{2}u_{N_{0},corr,2}+\frac{1}{r}\partial_{r}u_{N_{0},corr,2}-\frac{1}{r^{2}}u_{N_{0},corr,2}=e_{N_{0},corr,2}(t,r)\end{equation}
Then, we prove the following lemma.
\begin{lemma}\label{un0corrlemma} We have the following estimates on $u_{N_{0},corr,2}$. For $0 \leq k \leq 1$ and $0 \leq j \leq 1$,
\begin{equation}\label{un0corr2est}|\partial_{t}^{j}\partial_{r}^{k}u_{N_{0},corr,2}(t,r)| \leq \frac{C r^{1-k}\lambda(t)^{2} \log^{3}(t) \left(\sup_{x \in [100,t]}\left(\lambda(x) \log(x)\right)\right)^{3}}{ t^{5/2} \langle t-r \rangle^{7/2+j}}, \quad r \leq t\end{equation}
In addition,
\begin{equation}\label{dt2un0corr2est}|\partial_{t}^{2}u_{N_{0},corr,2}(t,r)|+|\partial_{r}^{2}u_{N_{0},corr,2}(t,r)| \leq \frac{C \lambda(t)^{2}\log^{3}(t) \left(\sup_{x \in [100,t]}\left(\lambda(x) \log(x)\right)\right)^{3}}{\langle t-r\rangle^{9/2}t^{5/2}}, \quad r \leq t\end{equation}
For all $r>0$,
\begin{equation}\label{un0corr2enest}|u_{N_{0},corr,2}(t,r)| + \sqrt{E(u_{N_{0},corr,2},\partial_{t}u_{N_{0},corr,2})} \leq \frac{C \lambda(t)^{2} \log^{3}(t) \left(\sup_{x \in [100,t]}\left(\lambda(x) \log(x)\right)\right)^{3}}{t^{3}}\end{equation}
Finally,
\begin{equation}\label{drun0corr2inftyest} ||\partial_{r}u_{N_{0},corr,2}(t,r)||_{L^{\infty}_{r}} \leq \frac{C \log^{6}(t)}{t^{\frac{5}{2}-5C_{u}}}\end{equation}
\end{lemma}
\begin{proof} As in Lemma \ref{un0lemma}, we have
\begin{equation}\label{un0corr21deriv}u_{N_{0},corr,2}(t,r) = \frac{-r}{2\pi} \int_{t}^{\infty} ds \int_{0}^{s-t}\frac{\rho d\rho}{\sqrt{(s-t)^{2}-\rho^{2}}} \int_{0}^{1}d\beta \int_{0}^{2\pi}  I_{e_{N_{0},corr,2}}(s,r\beta,\rho,\theta)d\theta\end{equation}
where we recall the notation \eqref{Inotation}. From Lemma \ref{v24estlemma}, we get the following estimate.
\begin{equation}|I_{e_{N_{0},corr,2}}(s,r\beta,\rho,\theta)| \leq C \frac{\lambda(s)^{2} \log^{3}(s) \left(\sup_{x \in [100,s]}\left(\lambda(x) \log(x)\right)\right)^{3}}{s^{9/2} \langle t-r \rangle^{7/2}}, \quad s \geq t \geq r,\quad \rho \leq s-t, \quad 0 \leq \beta \leq 1\end{equation}
Using \eqref{lambdacomparg}, we get \eqref{un0corr2est} for $k=j=0$. The formulae
\begin{equation}\label{drun0corr2form}\partial_{r}u_{N_{0},corr,2}(t,r) = \frac{-1}{2\pi} \int_{t}^{\infty} ds \int_{0}^{s-t} \frac{\rho d\rho}{\sqrt{(s-t)^{2}-\rho^{2}}} \int_{0}^{2\pi} d\theta I_{e_{N_{0},corr,2}}(s,r,\rho,\theta)\end{equation}
and
$$\partial_{t}u_{N_{0},corr,2}(t,r) = \frac{-r}{2\pi} \int_{t}^{\infty} ds \int_{0}^{s-t}\frac{\rho d\rho}{\sqrt{(s-t)^{2}-\rho^{2}}} \int_{0}^{1}d\beta \int_{0}^{2\pi}  \partial_{1}I_{e_{N_{0},corr,2}}(s,r\beta,\rho,\theta)d\theta,$$
along with the same procedure used to establish \eqref{un0corr2est} for $k=j=0$, implies \eqref{un0corr2est} for all larger $k,j$ in the lemma statement. Next, we note that
$$u_{N_{0},corr,2}(t,r) = \frac{-1}{2\pi} \int_{0}^{\infty} dw \int_{0}^{w} \frac{\rho d\rho}{\sqrt{w^{2}-\rho^{2}}} \int_{0}^{2\pi} \frac{e_{N_{0},corr,2}(t+w,\sqrt{r^{2}+\rho^{2}+2 r \rho \cos(\theta)})}{\sqrt{r^{2}+\rho^{2}+2 r \rho \cos(\theta)}} \left(r+\rho\cos(\theta)\right) d\theta.$$
Differentiation under the integral sign, combined with the procedure used to prove \eqref{un0corr2est} gives the estimate on $\partial_{t}^{2}u_{N_{0},corr,2}$ in \eqref{dt2un0corr2est}. We estimate $\partial_{r}^{2}u_{N_{0},corr,2}$ by directly differentiating \eqref{drun0corr2form}. Next,
$$||e_{N_{0},corr,2}(t,r)||_{L^{2}(rdr)} \leq \frac{C \lambda(t)^{2} \log^{3}(t) \left(\sup_{x \in [100,t]}\left(\lambda(x) \log(x)\right)\right)^{3}}{t^{4}}$$
and this implies \eqref{un0corr2enest}. Finally, using Lemma \ref{v24estlemma}, we get
$$|I_{e_{N_{0},corr,2}}(s,r,\rho,\theta)| \leq \frac{C \lambda(s)^{2} \log^{6}(s)}{s^{\frac{9}{2}-3C_{u}}}, \quad s \geq t, \quad  r >0, \rho \leq s-t, \theta \in [0,2\pi]$$
and this gives
$$|\partial_{r}u_{N_{0},corr,2}(t,r)| \leq C \int_{t}^{\infty} ds (s-t) \frac{\log^{6}(s)}{s^{\frac{9}{2}-5 C_{u}}} \leq \frac{C \log^{6}(t)}{t^{\frac{5}{2}-5C_{u}}}$$
 \end{proof}
Let $m_{\geq 1}:\mathbb{R}\rightarrow \mathbb{R}$ be any function satisfying
\begin{equation}\label{mgeq1def}m_{\geq 1}(x) = \begin{cases} 1, \quad x \geq \frac{1}{4}\\
0, \quad x \leq \frac{1}{8}\end{cases}, \quad m_{\geq 1} \in C^{\infty}(\mathbb{R})\end{equation}
Then, we will add $m_{\geq 1}(\frac{r}{t}) u_{N_{0},corr,2}(t,r)$ to our ansatz in order to eliminate the error term $e_{N_{0},corr,2}$. Accordingly, we define the error term, $err_{N_{0},corr,2}$, associated to $m_{\geq 1}(\frac{r}{t}) u_{N_{0},corr,2}(t,r)$, by
\begin{equation}\begin{split} err_{N_{0},corr,2}(t,r)&=-\left(-\partial_{t}^{2}+\partial_{r}^{2}+\frac{1}{r}\partial_{r}-\frac{\cos(2Q_{1}(\frac{r}{\lambda(t)}))}{r^{2}}\right) \left(m_{\geq 1}(\frac{r}{t}) u_{N_{0},corr,2}(t,r)\right)+e_{N_{0},corr,2}(t,r) \\
&=-\left(\frac{m_{\geq 1}''(\frac{r}{t})}{t^{2}}\left(1-\frac{r^{2}}{t^{2}}\right) + \frac{m_{\geq 1}'(\frac{r}{t})}{r t}\left(1-\frac{2r^{2}}{t^{2}}\right)\right) u_{N_{0},corr,2}(t,r)\\
&-\frac{2 m_{\geq 1}'(\frac{r}{t})}{t}\left(\partial_{r}u_{N_{0},corr,2}(t,r)+\frac{r}{t}\partial_{t}u_{N_{0},corr,2}(t,r)\right)\\
&-\left(\frac{1-\cos(2Q_{1}(\frac{r}{\lambda(t)}))}{r^{2}}\right) m_{\geq 1}(\frac{r}{t}) u_{N_{0},corr,2}(t,r)\end{split}\end{equation}
where we used the fact that
$$(1-m_{\geq 1}(\frac{r}{t})) e_{N_{0},corr,2}(t,r) =0$$
which follows from the definition of $m_{\geq 1}$, and $e_{N_{0},corr,2}$, which we recall is given in \eqref{en0corr2def}. The following lemma shows that $err_{N_{0},corr,2}$ is small enough to be treated with our final, perturbative argument.
\begin{lemma}\label{errn0corr2estlemma} We have the following estimates for $k=0,1$.
\begin{equation}||L_{\frac{1}{\lambda(t)}}^{k}(err_{N_{0},corr,2})(t,r)||_{L^{2}(r dr)} \leq \frac{C \lambda(t)^{2}(1+\lambda(t)^{2}) \log^{3}(t) \left(\sup_{x \in [100,t]}\left(\lambda(x) \log(x)\right)\right)^{3}}{t^{6+k}} \end{equation}
\end{lemma}
\begin{proof} We directly apply Lemma \ref{un0corrlemma}. The only point to note is that, to estimate\\
 $\left(\frac{1-\cos(2Q_{1}(\frac{r}{\lambda(t)}))}{r^{2}}\right) m_{\geq 1}(\frac{r}{t}) u_{N_{0},corr,2}(t,r)$, we use \eqref{un0corr2est} in the region $r \leq t-t^{3/7}$ and \eqref{un0corr2enest} for the region $r \geq t-t^{3/7}$. \end{proof}

\subsection{Second set of nonlinear interactions}
The next step is to treat the nonlinear interactions between the newest addition to our ansatz, namely 
\begin{equation}\label{undef}u_{n}(t,r):=u_{N_{0}}(t,r)+u_{N_{0},corr}(t,r)+m_{\geq 1}(\frac{r}{t}) u_{N_{0},corr,2}(t,r),\end{equation}
and $u_{a}$. The interactions between $v_{2,4}$ (part of $u_{N_{0},corr}$) and itself, along with the interactions between $v_{2,4}$ and $u_{a,0}$ are not quite perturbative, but the rest of the nonlinear interactions are. In order to show this, we note that
\begin{equation}\begin{split} &\sin(2Q_{1}(\frac{r}{\lambda(t)})+2u_{a}+2u_{n})= 2r^{2}\left(N_{0}+N_{1}+N_{2}+N_{3}\right) +\cos (2 Q_{1}(\frac{r}{\lambda(t)}))\cdot 2 (u_{a}+u_{n})+\sin (2 Q_{1}(\frac{r}{\lambda(t)}))\end{split}\end{equation}
where $N_{0}$ is defined in \eqref{N0def}, $N_{1}$ is defined in \eqref{N1def}, and we define 
\begin{equation}\label{N2def}\begin{split}N_{2}(t,r):&=\frac{\cos (2 Q_{1}(\frac{r}{\lambda(t)}))}{2r^{2}} (-2 (u_{a0}+v_{2,4})+\sin (2 (u_{a,0}+v_{2,4}))-(\sin (2 u_{a,0})-2 u_{a,0}))\\
&+\frac{\sin (2 Q_{1}(\frac{r}{\lambda(t)}))}{2r^{2}} (\cos (2 u_{a,0}+2 v_{2,4})-\cos (2 u_{a,0}))\\
\end{split}\end{equation}
We therefore have
\begin{equation}\label{N3def}\begin{split}N_{3}(t,r)&= \frac{\sin (2 Q_{1}(\frac{r}{\lambda(t)}))}{2r^{2}} \left(\cos (2 u_{a}+2 u_{n})-\cos (2 u_{a})-\left(\cos (2 u_{a,0}+2 v_{2,4})-\cos (2 u_{a,0})\right)\right)\\
&+\frac{\cos (2 Q_{1}(\frac{r}{\lambda(t)}))}{2r^{2}} \begin{aligned}[t]&\left(\sin (2 u_{a}+2 u_{n})-2 (u_{a}+u_{n})-(\sin (2 u_{a})-2 u_{a})\right.\\
&\left.-(\sin (2 u_{a,0}+2 v_{2,4})-2 (u_{a,0}+v_{2,4}))+\sin(2u_{a,0})-2 u_{a,0}\right)\end{aligned}\end{split}\end{equation}
where we recall the definition of $u_{a}$, from \eqref{uasplit}. We define $u_{new}$ by
\begin{equation}\begin{split}u_{new}(t,r):&=u_{n}(t,r)-v_{2,4}(t,r)\end{split}\end{equation}
We remark that $N_{2}$ contains the interactions between $v_{2,4}$ and $u_{a,0}$, except for the $u_{a,0}$ self-interactions, which were already contained in $N_{0}$. To give the reader an idea of how we estimate $N_{3}$, we re-write its expression as follows.
\begin{equation}\label{n3exp}\begin{split}&N_{3}(t,r)\\
&=\frac{\cos (2 Q_{1}(\frac{r}{\lambda(t)}))}{2r^{2}} \begin{aligned}[t]&\left((\cos (2 u_{a})-1) ((\cos (2 u_{new})-1) \sin (2 v_{2,4}))\right.\\
&\left.+(\cos (2 u_{new}+2 v_{2,4})-1) (\sin (2 u_{a,0}) (\cos (2 u_{a,1})-1))\right.\\
&\left.+\sin (2 v_{2,4}) (\cos (2 u_{a,0}) (\cos (2 u_{a,1})-1))+\sin (2 u_{a,0}) ((\cos (2 u_{new})-1) \cos (2 v_{2,4}))\right.\\
&\left.+(\cos (2 u_{new})-1) \sin (2 v_{2,4})-2 u_{new}+\sin (2 u_{new})\right)\end{aligned}\\
&+\frac{\sin (2 Q_{1}(\frac{r}{\lambda(t)}))}{2r^{2}} \left(\cos (2 (u_{a}+ u_{n}))-\cos (2 u_{a})-\cos (2 u_{a,0}+2 v_{2,4})+\cos (2 u_{a,0})\right)\\
&+\frac{\cos(2Q_{1}(\frac{r}{\lambda(t)}))}{2r^{2}} \begin{aligned}[t]&\left((\cos (2 u_{a})-1) \sin (2 u_{new}) \cos (2 v_{2,4})+(\cos (2 (u_{new}+ v_{2,4}))-1) \cos (2 u_{a,0}) \sin (2 u_{a,1})\right.\\
&\left.-\sin (2 v_{2,4}) \sin (2 u_{a,0}) \sin (2 u_{a,1})-\sin (2 u_{a,0}) \sin (2 u_{new}) \sin (2 v_{2,4})\right.\\
&\left.+\sin (2 u_{new}) (\cos (2 v_{2,4})-1)\right)\end{aligned}\end{split}\end{equation}
\begin{lemma}\label{n3estlemma} We have the following estimates.\begin{equation}\begin{split}||\frac{N_{3}(t,r)}{\lambda(t)^{2}}||_{L^{2}(r dr)} \leq \frac{C \log^{24}(t)}{t^{9/2-15C_{u}-2C_{l}}}, \quad ||L_{\frac{1}{\lambda(t)}}(N_{3}(t,r))||_{L^{2}(r dr)} \leq \frac{C \log^{24}(t)}{t^{9/2-15C_{u}-C_{l}}}\end{split}\end{equation}
\end{lemma}
\begin{proof}
 We note that, as per \eqref{mgeq1def},
$$1-m_{\geq 1}(\frac{r}{2t})=0, \quad r \geq \frac{t}{2}$$
Using Lemma \ref{un0lemma}, Lemma \ref{un0ellminusipestlemma}, and Lemma \ref{un0corrlemma}, we get
\begin{equation}\label{unewest}|u_{new}(t,r)| \leq \begin{cases} \frac{C r \log^{4}(t) \left(\sup_{x \in [100,t]}\left(\lambda(x) \log(x)\right)\right)^{3}}{t^{4}}, \quad r \leq \frac{t}{2}\\
\frac{\left(\sup_{x \in [100,t]}\left(\lambda(x) \log(x)\right)\right)^{3} \left(\sup_{x \in [100,t]}\left(\lambda(x)^{2}\right)\right)^{3} \log^{9}(t)}{t^{2}}, \quad \frac{t}{2} \leq r\end{cases}\end{equation}
We then straightforwardly estimate the expression  \eqref{n3exp}, for $N_{3}$ (as well as its $r$ derivative), using \eqref{ua1est}, \eqref{ua0est}, and Lemma \ref{v24estlemma} (to estimate $u_{a,1}$, $u_{a,0}$, and $v_{2,4}$, respectively).  \end{proof}
Next, we add a correction, $u_{N_{2}}$ to improve the error term $N_{2}$, defined in \eqref{N2def}. In particular, we define $u_{N_{2}}$ to be the solution to the following equation with $0$ Cauchy data at infinity.
\begin{equation}\label{un2eqn}-\partial_{t}^{2}u_{N_{2}}+\partial_{r}^{2}u_{N_{2}}+\frac{1}{r}\partial_{r}u_{N_{2}}-\frac{u_{N_{2}}}{r^{2}}=N_{2}(t,r)\end{equation}
Then, we have the following lemma.
\begin{lemma} \label{un2lemma}For $0 \leq k \leq 1$, and $r \leq t$
\begin{equation}\begin{split}&r^{k}|\partial_{r}^{k} u_{N_{2}}(t,r)|\\
&\leq C r (r+\lambda(t)) \left(\frac{\log^{7}(t)}{t^{5-3C_{u}}}\right)+ \frac{C r \lambda(t)}{t^{2}\langle t-r \rangle^{4}}\left(t^{4C_{u}}\log^{8}(t)+\frac{t^{6C_{u}} \log^{12}(t)}{\langle t-r \rangle^{2}}\right)+\frac{C r \log^{10}(t) t^{5C_{u}}}{\langle t-r \rangle^{9/2} t^{3/2}} \left(1+\frac{t^{4C_{u}} \log^{8}(t)}{\langle t-r \rangle^{4}}\right)\end{split}\end{equation}
In addition, for all $r>0$,
\begin{equation}\begin{split}&|u_{N_{2}}(t,r)| + \sqrt{E(u_{N_{2}},\partial_{t}u_{N_{2}})}\leq \frac{C \log^{18}(t)}{t^{2-9C_{u}}}\end{split}\end{equation}
Finally,
\begin{equation}\label{drun2inftyest} ||\partial_{r}u_{N_{2}}(t,r)||_{L^{\infty}_{r}} \leq \frac{C \log^{18}(t)}{t^{\frac{3}{2}-9C_{u}}}\end{equation}
\end{lemma}
\begin{proof}
The lemma is proven with the same procedure as in Lemma \ref{un0lemma}, using the estimates from Lemma \ref{v24estlemma} and \eqref{ua0est}. \end{proof}
It only remains to estimate the linear error term associated to $u_{N_{2}}(t,r)$ as well as its nonlinear interactions with all of the previous terms added into our ansatz. We start with the linear error term, which we denote by $e_{N_{2}}$:
$$e_{N_{2}}(t,r):= \left(\frac{\cos(2Q_{1}(\frac{r}{\lambda(t)}))-1}{r^{2}}\right)u_{N_{2}}(t,r)$$
By directly estimating $e_{N_{2}}$ using the estimates from Lemmas \ref{un2lemma}, we get the following lemma.
\begin{lemma}\label{en231estlemma} For $k=0,1$,
\begin{equation}\begin{split}&||L_{\frac{1}{\lambda(t)}}^{k}(e_{N_{2}})(t,r)||_{L^{2}(r dr)} \leq \frac{C \log^{18}(t)}{t^{5-11C_{u}-kC_{l}}}\end{split}\end{equation}
\end{lemma}
\subsection{Third set of nonlinear interactions}
Finally, it remains to treat the nonlinear interactions between $u_{N_{2}}$ and all of the previous terms of our ansatz (the sum of which is equal to $u_{a}+u_{n}$). I.e., it suffices to estimate $N_{4}$ given by
\begin{equation}\begin{split} &\sin(2Q_{1}(\frac{r}{\lambda(t)})+2u_{a}+2u_{n}+2u_{N_{2}})\\
&= 2r^{2}\left(N_{0}+N_{1}+N_{2}+N_{3}\right)+\cos(2Q_{1}(\frac{r}{\lambda(t)}))\cdot 2(u_{a}+u_{n}+u_{N_{2}}) + \sin(2Q_{1}(\frac{r}{\lambda(t)}))+2r^{2}N_{4}(t,r)\end{split}\end{equation}
So,
\begin{equation}\begin{split}N_{4}(t,r)&=\frac{1}{2r^{2}}\begin{aligned}[t]&\left(\sin(2Q_{1}(\frac{r}{\lambda(t)})+2u_{a}+2u_{n})(\cos(2u_{N_{2}})-1)+\cos(2Q_{1}(\frac{r}{\lambda(t)}))(\sin(2u_{N_{2}})-2u_{N_{2}})\right.\\
&+\left.\left(\cos(2Q_{1}(\frac{r}{\lambda(t)})+2u_{a}+2u_{n})-\cos(2Q_{1}(\frac{r}{\lambda(t)}))\right)\sin(2u_{N_{2}})\right)\end{aligned}\end{split}\end{equation}
\begin{lemma} \label{n4estlemma} 
\begin{equation}\begin{split}&\lambda(t)^{-2} ||N_{4}(t,r)||_{L^{2}(r dr)} \leq \frac{C \log^{30}(t)}{t^{\frac{9}{2}-2C_{l}-15C_{u}}}, \quad ||L_{\frac{1}{\lambda(t)}}(N_{4})(t,r)||_{L^{2}(r dr)} \leq \frac{C \log^{30}(t)}{t^{\frac{9}{2}-15C_{u}-C_{l}}}\end{split}\end{equation}\end{lemma}
\begin{proof}
We have
$$|N_{4}(t,r)| \leq \frac{C}{r^{2}} \left(\left(\frac{r \lambda(t)}{(r^{2}+\lambda(t)^{2})} |u_{N_{2}}|+u_{N_{2}}^{2}\right) (|u_{N_{2}}|+|u_{a}|+|u_{n}|)+(u_{a}^{2}+u_{n}^{2}) |u_{N_{2}}|\right)$$
where we recall the definitions of $u_{a}$ and $u_{n}$ in \eqref{uasplit}, and \eqref{undef}, respectively. We then use Lemmas \ref{un2lemma} and \ref{v24estlemma}, and estimates \eqref{ua0est}, \eqref{ua1est}, and \eqref{unewest}, to estimate $u_{N_{2}}$ and the various terms in $u_{a}$ and $u_{n}$, and then straightforwardly estimate $N_{4}$, and $L_{\frac{1}{\lambda(t)}}N_{4}$.
 \end{proof}
We recall that $u_{a}$ is defined in \eqref{uadef}, $u_{n}$ is defined in \eqref{undef}, and $u_{N_{2}}$ is defined in \eqref{un2eqn}, and define our final ansatz by the following.
\begin{equation}\label{uansatzdef} u_{ansatz}(t,r):=u_{a}(t,r)+u_{n}(t,r)+u_{N_{2}}(t,r)\end{equation}
The error term of $u_{ansatz}$, denoted by $F_{5}$, is equal to the following.
\begin{equation}\label{f5def}F_{5}(t,r) = e_{N_{0}}\left(1-m_{\leq 1}(\frac{r}{2h(t)})\right)+e_{N_{0},corr,1}+err_{N_{0},corr,2}+N_{1}+N_{3}+e_{N_{2}}+N_{4}+e_{a}\end{equation}
\begin{lemma}\label{uansatzerrorest}[Estimates on the error of $u_{ansatz}$] There exists $C_{5}>0$ such that, if 
$$\delta:= \frac{1}{2}\text{min}\{4\alpha-C_{l}-2,1-2C_{u}-\alpha,4-5C_{u}-6\alpha,\frac{1}{2}-2C_{l}-15C_{u}\}$$
then, $2\delta>C_{l}$, and
$$\frac{||F_{5}(t,r)||_{L^{2}(r dr)}}{\lambda(t)^{2}} \leq \frac{C_{5} \log^{30}(t)}{t^{4+2\delta}}, \quad ||L_{\frac{1}{\lambda(t)}}F_{5}(t,r)||_{L^{2}(r dr)} \leq \frac{C \log^{2}(t)}{t^{2+5\alpha}}+ \frac{C \log^{6}(t)}{t^{4+2\alpha-3C_{u}}}+\frac{C \log^{30}(t)}{t^{9/2-C_{l}-15C_{u}}}$$
\end{lemma}
\begin{proof}
The inequality $2\delta > C_{l}$ follows from \eqref{lambdaonlyconstr} and \eqref{alphaconstr}. The $i$th term in \eqref{f5def} is estimated in the $ith$ Lemma in the following list: Lemma \ref{en0estlemma}, \ref{en0corr1estlemma}, \ref{errn0corr2estlemma}, \ref{n1lemma}, \ref{n3estlemma}, \ref{en231estlemma}, \ref{n4estlemma}, \ref{eaestlemma}.  \end{proof}
We finish this section with some estimates on various quantities involving $u_{ansatz}$.
\begin{lemma}\label{uansatzlemma}
We have the following estimates.
\begin{equation}\label{uquot}\begin{split} ||\frac{8 \lambda(t) u_{ansatz}(t,r)}{r(r^{2}+\lambda(t)^{2})}||_{L^{\infty}_{r}} &\leq \frac{4}{t^{2}} \left(M^{2}+\frac{C_{2}}{3} \left(3+2\pi^{2}\right)\right)+ \frac{C \log(t)}{t^{2+2\alpha}} + \frac{C \log^{6}(t)}{t^{\frac{7}{2}-4C_{u}}}  \end{split}\end{equation}
\begin{equation}\label{uoverr} ||\frac{u_{ansatz}(t,r)}{r}||_{L^{\infty}_{r}} \leq \frac{C \log^{6}(t)}{t^{\frac{3}{2}-3C_{u}}}\end{equation} 
\begin{equation}\label{dru}||\partial_{r}u_{ansatz}(t,r)||_{L^{\infty}_{r}} \leq \frac{C \log^{6}(t)}{t^{\frac{1}{2}-3C_{u}}}\end{equation}
\begin{equation}\label{udruquot1}||\frac{u_{ansatz}(t,r) \partial_{r}u_{ansatz}(t,r)}{r}||_{L^{\infty}_{r}(\{r \leq \lambda(t)\})} \leq \frac{C \log^{7}(t)}{t^{\frac{5}{2}-4C_{u}}}\end{equation}
\begin{equation}\label{udruquot2}||\frac{u_{ansatz}(t,r)\partial_{r}u_{ansatz}(t,r)}{r^{2}}||_{L^{\infty}_{r}(\{r \geq \lambda(t)\})} \leq \frac{C \log^{7}(t)}{t^{\frac{5}{2}-3C_{u}}}\end{equation}
\begin{equation}\label{druquot}\lambda(t) ||\frac{\partial_{r}u_{ansatz}(t,r)}{r^{2}+\lambda(t)^{2}}||_{L^{\infty}_{r}} \leq \frac{1}{2t^{2}} \left(M^{2}+\frac{C_{2}}{3}(3+2\pi^{2})\right) +\frac{C \log^{6}(t)}{t^{\frac{5}{2}-4C_{u}}}\end{equation}
\end{lemma}
\begin{proof}We recall \eqref{velldef}, and note that
\begin{equation} \frac{8 \lambda(t)u_{ell}(t,R\lambda(t))}{R\lambda(t)^{3}(R^{2}+1)} = \frac{4}{(1+R^{2})} \left(\frac{\lambda'(t)}{\lambda(t)}\right)^{2} + \frac{\lambda''(t)}{\lambda(t)} m_{1}(R)\end{equation}
where
$$m_{1}(R) = \frac{8 \text{Li}_2\left(-R^2\right)+\frac{8}{3} \left(3 R^2+\pi ^2\right)-\frac{4 \left(R^4-1\right) \log \left(R^2+1\right)}{R^2}}{(1+R^{2})^{2}}:=\frac{num_{1}(R)}{(1+R^{2})^{2}}$$
We start with 
$$num_{1}'(R) = \frac{8(1+R^{2})(R^{2}-(1+R^{2})\log(1+R^{2}))}{R^{3}}$$
Then, since
$$\partial_{R}\left(R^{2}-(1+R^{2})\log(1+R^{2})\right) = -2 R \log(1+R^{2}) \leq 0$$
and
$$\left(R^{2}-(1+R^{2})\log(1+R^{2})\right)\Bigr|_{R=0}=0$$
we have
$$R^{2}-(1+R^{2})\log(1+R^{2}) \leq 0, \quad R \geq 0$$
Therefore,
$$num_{1}'(R) \leq 0, \quad R >0$$
So, 
$$num_{1}(R) \geq num_{1}(1) = 2\pi^{2}+8 > 0, \quad 0 < R \leq 1$$
Therefore,
$$R \mapsto \frac{num_{1}(R)}{(1+R^{2})^{2}}=m_{1}(R) \text{ is non-increasing on }(0,1]$$
Hence, for $0 \leq R \leq 1$,
$$|m_{1}(R)| \leq \lim_{R \rightarrow 0^{+}} m_{1}(R) = \frac{4(3+2\pi^{2})}{3}$$
If $R \geq 1$, then,
\begin{equation}\begin{split} |num_{1}(R)| &\leq |num_{1}(1)| + 8 \int_{1}^{R} \frac{(1+s^{2})(s^{2}+(1+s^{2})\log(1+s^{2}))}{s^{3}} ds\\
&\leq2 \left(2+\pi^{2}+2R^{2}+2\log(1+R^{2})+2(R^{2}+2 \log(1+R^{2}))\log(1+R^{2})\right):=m_{2}(R)\end{split}\end{equation}
If $1 \leq R \leq 2$, then,
$$\frac{d}{dR}\left(\frac{m_{2}(R)}{(1+R^{2})^{2}}\right)= \frac{-8R (\pi^{2}+\log(1+R^{2})(-3+R^{2}+4 \log(1+R^{2})))}{(1+R^{2})^{3}} \leq \frac{R}{(1+R^{2})^{3}} \left(-8\pi^{2}+24\log(5)\right) <0$$
Therefore, for $1 \leq R \leq 2$,
\begin{equation}|m_{1}(R)| \leq \frac{|num_{1}(R)|}{(1+R^{2})^{2}} \leq \frac{m_{2}(R)}{(1+R^{2})^{2}} \leq \frac{m_{2}(1)}{2^{2}} = 2+\frac{\pi^{2}}{2}+2 \log^{2}(2) +\log(4)\end{equation}
Finally, if $R \geq 2$, then,
\begin{equation}\label{m1bound}|m_{1}(R)| \leq \frac{m_{2}(R)}{(1+R^{2})^{2}} \leq \frac{4}{5}+\frac{2}{25} \left(2+\pi ^2\right)+\frac{8 \log ^2(5)}{25}+\frac{4 \log (5)}{5}+\frac{4 \log (5)}{25}\end{equation}
This gives \eqref{uquot}. To prove \eqref{dru}, we estimate $\partial_{r}u_{a}(t,r)$ using the procedure detailed in the proof of Lemma \ref{n1lemma}. We then estimate $\partial_{r}u_{n}(t,r)$ using \eqref{drun0inftyest}, \eqref{drun0corr2inftyest}, \eqref{unewest} and Lemma \ref{v24estlemma}. Finally, we estimate $\partial_{r}u_{N_{2}}$ using \eqref{drun2inftyest}. With the same procedure, we also get \eqref{udruquot1} and \eqref{udruquot2}. Next, 
$$\frac{\partial_{r}u_{ell}(t,r)}{r^{2}+\lambda(t)^{2}}\Bigr|_{r=R\lambda(t)} = \frac{\lambda'(t)^{2}}{2(1+R^{2})\lambda(t)^{3}}+\frac{\lambda''(t)}{\lambda(t)^{2}} m_{3}(R)$$
for
$$m_{3}(R) = \frac{1}{(1+R^{2})^{3}}\left(1+3R^{2}+\frac{\pi^{2}}{3}(1-R^{2})-\frac{(1+7R^{2}+7R^{4}+R^{6})\log(1+R^{2})}{2R^{2}} +(1-R^{2})\text{Li}_{2}(-R^{2})\right)$$
With a similar procedure as that used to obtain \eqref{m1bound}, and its analogs for other regions of $R$, we get
$$|m_{3}(R)| \leq \frac{3+2\pi^{2}}{6}, \quad R> 0$$
 \end{proof}
\section{Constructing the exact solution}
If we substitute 
$$u(t,r)=Q_{1}(\frac{r}{\lambda(t)})+u_{ansatz}(t,r)+v_{6}(t,r)$$
into \eqref{wm}, we get
\begin{equation}\label{v6eqn}\begin{split}&-\partial_{t}^{2}v_{6}+\partial_{r}^{2}v_{6}+\frac{1}{r}\partial_{r}v_{6}-\frac{\cos(2Q_{1}(\frac{r}{\lambda(t)}))}{r^{2}}v_{6}=F_{5}+F_{3}(v_{6})\end{split}\end{equation}
where (note that the following expressions are essentially the same as  $F_{3}$ on pg. 144 of \cite{wm}) 
\begin{equation}\label{finalf3def}F_{3}(f)=N(f)+L_{1}(f)\end{equation}

$$N(f) = \frac{ \sin \left(2Q_{1}(\frac{r}{\lambda(t)})+2 u_{ansatz}\right)}{2 r^2}(\cos (2 f)-1)+\frac{ \cos \left(2Q_{1}(\frac{r}{\lambda(t)})\right)}{2 r^2}(\sin (2 f)-2f)$$
\begin{equation}\label{L1def}L_{1}(f) = -\frac{ \sin \left(2Q_{1}(\frac{r}{\lambda(t)})\right)}{2 r^2}\sin (2 u_{ansatz}) \sin (2 f)+\frac{ \cos \left(2Q_{1}(\frac{r}{\lambda(t)})\right)}{2 r^2}(\cos (2 u_{ansatz})-1) \sin (2 f)\end{equation}
and $F_{5}$ is given in \eqref{f5def}. We will solve \eqref{v6eqn} by first formally deriving the equation for $y$ (namely \eqref{yeqn}) given by
\begin{equation}\label{yintermsofv6}y(t,\xi) = \mathcal{F}(\sqrt{\cdot} v_{6}(t,\cdot \lambda(t)))(\xi \lambda(t)^{2})\end{equation}
where $\mathcal{F}$ denotes the distorted Fourier transform of \cite{kst} (which is defined in section 5 of \cite{kst}). Then, we will prove that \eqref{yeqn} admits a solution, say $y_{0}$ (with 0 Cauchy data at infinity) which has enough regularity to rigorously justify the statement that if $v_{6}$ given by the following expression, with $y=y_{0}$
\begin{equation}\label{v6intermsofy}v_{6}(t,r) = \sqrt{\frac{\lambda(t)}{r}} \mathcal{F}^{-1}\left(y(t,\frac{\cdot}{\lambda(t)^{2}})\right)\left(\frac{r}{\lambda(t)}\right),\end{equation}
then, $v_{6}$ is a solution to \eqref{v6eqn}.
We have (see also (5.4), (5.5), pg. 145 of \cite{wm})
\begin{equation}\label{yeqn}\begin{split} \partial_{tt}y+\omega y = -\mathcal{F}(\sqrt{\cdot}F_{5}(t,\cdot \lambda(t)))(\omega \lambda(t)^{2})+F_{2}(y)(t,\omega) -\mathcal{F}(\sqrt{\cdot}F_{3}(v_{6}(y))(t,\cdot \lambda(t)))(\omega \lambda(t)^{2})\end{split}\end{equation}
where $v_{6}(y)$, which appears in the argument of $F_{3}$, is the expression given in \eqref{v6intermsofy}, and
\begin{equation}\label{F2def}\begin{split}F_{2}(y)(t,\omega) &=-\frac{\lambda'(t)}{\lambda(t)} \partial_{t}y(t,\omega) + \frac{2 \lambda'(t)}{\lambda(t)} \mathcal{K}\left(\partial_{1}y(t,\frac{\cdot}{\lambda(t)^{2}})\right)(\omega \lambda(t)^{2}) +\left(\frac{-\lambda''(t)}{2\lambda(t)}+\frac{\lambda'(t)^{2}}{4 \lambda(t)^{2}}\right)y(t,\omega) \\
&+ \frac{\lambda''(t)}{\lambda(t)} \mathcal{K}\left(y(t,\frac{\cdot}{\lambda(t)^{2}})\right)(\omega \lambda(t)^{2}) +2\frac{\lambda'(t)^{2}}{\lambda(t)^{2}}\left([\xi \partial_{\xi},\mathcal{K}](y(t,\frac{\cdot}{\lambda(t)^{2}}))\right)(\omega \lambda(t)^{2}) \\
&-\frac{\lambda'(t)^{2}}{\lambda(t)^{2}} \mathcal{K}\left(\mathcal{K}(y(t,\frac{\cdot}{\lambda(t)^{2}}))\right)(\omega \lambda(t)^{2})\end{split}\end{equation}
where $\mathcal{K}$ is the transference operator of \cite{kst} (which is defined in section 6 of \cite{kst}). Next, we note that, by Proposition 5.7 b of \cite{kst}, there exists $C_{\rho}>0$ such that
\begin{equation}\label{rhoscaling}\frac{\rho(\omega \lambda(t)^{2})}{\rho(\omega \lambda(x)^{2})} \leq C_{\rho}\left(\frac{\lambda(x)^{2}}{\lambda(t)^{2}}+\frac{\lambda(t)^{2}}{\lambda(x)^{2}}\right)\end{equation} 
Recall Lemma \ref{uansatzerrorest}, and let $C_{5}'=C_{5} \sqrt{C_{\rho}}$. Now, we define the space $Z$ in which we will solve \eqref{yeqn}. Let $Z$ be the set of (equivalence classes) of measurable functions $y: [T_{0},\infty)\times (0,\infty) \rightarrow \mathbb{R}$ such that
$$ y(t,\omega) t^{2+2\delta} \sqrt{\rho(\omega \lambda(t)^{2})} \langle \omega \lambda(t)^{2}\rangle \in C^{0}_{t}([T_{0},\infty),L^{2}(d\omega))$$
$$\partial_{t}y(t,\omega) t^{3+2\delta}  \left(1+ \sqrt{\omega} \lambda(t)\right)\sqrt{\rho(\omega \lambda(t)^{2})} \in C^{0}_{t}([T_{0},\infty), L^{2}(d\omega))$$
and $||y||_{Z} < \infty$
where
\begin{equation}\label{Zdef}||y||_{Z} := \text{sup}_{t  \geq T_{0}} \begin{aligned}[t]&\left(t^{2+2\delta}\left(C_{5}'^{-1} \beta^{-1} \log^{-30}(t) || y(t,\omega)||_{L^{2}(\rho(\omega \lambda(t)^{2}) d\omega)} + ||\sqrt{\omega \lambda(t)^{2}} \sqrt{\langle \omega \lambda(t)^{2}\rangle} y(t,\omega)||_{L^{2}(\rho(\omega \lambda(t)^{2}))}\right) \right.\\
&\left.+ t^{3+2\delta}\log^{-30}(t)\left(C_{5}'^{-1} \beta^{-1} || \partial_{t}y(t,\omega)||_{L^{2}(\rho(\omega \lambda(t)^{2})d\omega)}+C_{Z}^{-1} ||\sqrt{ \omega \lambda(t)^{2}} \partial_{t}y(t,\omega)||_{L^{2}(\rho(\omega \lambda(t)^{2})d\omega)}\right)\right)\end{aligned}\end{equation}
and $\beta$ and $C_{Z}$ are positive, but otherwise arbitrary, and will be further constrained later on. We start with the following estimates on $F_{2}$.
\begin{lemma}\label{f2lemma} For all $y \in Z$, and $k=0,1$, we have
\begin{equation}\label{f2l2}\begin{split}&||(\omega \lambda(t)^{2})^{\frac{k}{2}} F_{2}(t,\omega)||_{L^{2}(\rho(\omega \lambda(t)^{2})d\omega)}\\
&\leq ||\langle \omega \lambda(t)^{2} \rangle^{k/2} \partial_{1}y(t,\omega)||_{L^{2}(\rho(\omega \lambda(t)^{2})d\omega)} \frac{|\lambda'(t)|}{\lambda(t)} \left(1+2 ||\mathcal{K}||_{\mathcal{L}(L^{2,\frac{k}{2}}_{\rho})}\right)  \\
&+ ||\langle \omega \lambda(t)^{2} \rangle^{k/2} y(t,\omega)||_{L^{2}(\rho(\omega \lambda(t)^{2})d\omega)} \begin{aligned}[t]&\left(\left(\frac{\lambda'(t)}{\lambda(t)}\right)^{2} \left(\frac{1}{4}+2||[\xi\partial_{\xi},\mathcal{K}]||_{\mathcal{L}(L^{2,\frac{k}{2}}_{\rho})} + ||\mathcal{K}||^{2}_{\mathcal{L}(L^{2,\frac{k}{2}}_{\rho})}\right)\right.\\
&\left. + \frac{|\lambda''(t)|}{\lambda(t)} \left(\frac{1}{2}+||\mathcal{K}||_{\mathcal{L}(L^{2,\frac{k}{2}}_{\rho})}\right)\right)\end{aligned}\end{split}\end{equation}
\end{lemma}
\begin{proof} This follows directly from \eqref{F2def}, and Theorem 6.1, Proposition 6.2, and Proposition 5.7 of \cite{kst}. \end{proof}

For $i=1,2$, let $y_{i} \in Z$, and define $f_{i}$ by
$$f_{i}(t,r) = \sqrt{\frac{\lambda(t)}{r}} \mathcal{F}^{-1}\left(y_{i}(t,\frac{\cdot}{\lambda(t)^{2}})\right)\left(\frac{r}{\lambda(t)}\right).$$
We now estimate $F_{3}(f_{2})-F_{3}(f_{1})$. 
\begin{lemma}\label{l1nlemma} There exists $C>0$ such that, for all $y_{1},y_{2}\in Z$, 
\begin{fleqn}\begin{equation}\label{l1l2}\begin{split}&||\mathcal{F}\left(\sqrt{\cdot} \left(L_{1}(f_{2})-L_{1}(f_{1})\right)(t,\cdot \lambda(t))\right)(\omega \lambda(t)^{2})||_{L^{2}(\rho(\omega \lambda(t)^{2}) d\omega)} \\
&\leq ||\left(y_{2}-y_{1}\right)(t,\omega)||_{L^{2}(\rho(\omega \lambda(t)^{2})d\omega)} \left(8 \lambda(t) ||\frac{u_{ansatz}(t,r)}{r(r^{2}+\lambda(t)^{2})}||_{L^{\infty}_{r}} + 2 ||\frac{u_{ansatz}(t,r)}{r}||_{L^{\infty}_{r}}^{2}\right)\end{split}\end{equation}\end{fleqn}
\begin{fleqn}
\begin{equation}\label{ll1l2}\begin{split}&||\sqrt{\omega \lambda(t)^{2}} \mathcal{F}\left(\sqrt{\cdot}\left(L_{1}(f_{2})-L_{1}(f_{1})\right)(t,\cdot\lambda(t))\right)(\omega \lambda(t)^{2})||_{L^{2}(\rho(\omega \lambda(t)^{2}) d\omega)}\\
&\leq C  ||\sqrt{\langle \omega \lambda(t)^{2}\rangle} (y_{2}-y_{1})(t,\omega)||_{L^{2}(\rho(\omega \lambda(t)^{2})d\omega)}\\
&\cdot \begin{aligned}[t]&\left(\lambda(t) ||\frac{u_{ansatz}(t,r)}{r(r^{2}+\lambda(t)^{2})}||_{L^{\infty}_{r}} + ||\frac{u_{ansatz}(t,r)}{r}||_{L^{\infty}_{r}}^{2}+ \lambda(t) ||\frac{\partial_{r}u_{ansatz}(t,r)}{r^{2}+\lambda(t)^{2}}||_{L^{\infty}_{r}} \right.\\
&\left. +  ||\frac{u_{ansatz}(t,r) \partial_{r}u_{ansatz}(t,r)}{r}||_{L^{\infty}_{r}(r \leq \lambda(t))}+ \lambda(t) ||\frac{u_{ansatz}(t,r) \partial_{r}u_{ansatz}(t,r)}{r^{2}}||_{L^{\infty}_{r}(r \geq \lambda(t))} \right)\end{aligned}\end{split}\end{equation}
\end{fleqn}
\begin{fleqn}
\begin{equation}\label{nl2}\begin{split}&||\mathcal{F}\left(\sqrt{\cdot}\left(N(f_{1})-N(f_{2})\right)(t,\cdot\lambda(t))\right)(\omega \lambda(t)^{2})||_{L^{2}(\rho(\omega \lambda(t)^{2})d\omega)}\\
&\leq C ||\sqrt{\langle \omega \lambda(t)^{2}\rangle}(y_{1}-y_{2})(t,\omega) ||_{L^{2}(\rho(\omega\lambda(t)^{2})d\omega)}\\
&\cdot\left(||\langle \omega \lambda(t)^{2}\rangle y_{1}(t,\omega)||_{L^{2}(\rho(\omega\lambda(t)^{2})d\omega)}^{2}+||\langle \omega \lambda(t)^{2}\rangle y_{2}(t,\omega)||_{L^{2}(\rho(\omega\lambda(t)^{2})d\omega)}^{2}\right.\\
&+\left.\left(||\frac{u_{ansatz}(t,r)}{r}||_{L^{\infty}_{r}} + \frac{1}{\lambda(t)}\right)\left(||\sqrt{\langle \omega \lambda(t)^{2}\rangle}y_{1}(t,\omega) ||_{L^{2}(\rho(\omega\lambda(t)^{2})d\omega)}+||\sqrt{\langle \omega \lambda(t)^{2}\rangle}y_{2}(t,\omega) ||_{L^{2}(\rho(\omega\lambda(t)^{2})d\omega)}\right)\right)\end{split}\end{equation}
\end{fleqn}
\begin{fleqn}
\begin{equation}\label{lnl2}\begin{split}&||\sqrt{\omega \lambda(t)^{2}} \mathcal{F}\left(\sqrt{\cdot}\left(N(f_{2})-N(f_{1})\right)(t,\cdot\lambda(t))\right)(\omega\lambda(t)^{2})||_{L^{2}(\rho(\omega \lambda(t)^{2})d\omega)}\\
&\leq C ||\langle \omega \lambda(t)^{2}\rangle (y_{2}-y_{1})(t,\omega)||_{L^{2}(\rho(\omega \lambda(t)^{2})d\omega)}\left(||\langle \omega \lambda(t)^{2}\rangle y_{2}(t,\omega)||_{L^{2}(\rho(\omega \lambda(t)^{2})d\omega)}+||\langle \omega \lambda(t)^{2}\rangle y_{1}(t,\omega)||_{L^{2}(\rho(\omega \lambda(t)^{2})d\omega)}\right)\\
&\cdot \left(\frac{1}{\lambda(t)} + ||\partial_{r}u_{ansatz}(t,r)||_{L^{\infty}_{r}} + ||\frac{u_{ansatz}(t,r)}{r}||_{L^{\infty}_{r}}\right)\\
&+ C ||\langle \omega \lambda(t)^{2}\rangle (y_{2}-y_{1})(t,\omega)||_{L^{2}(\rho(\omega \lambda(t)^{2})d\omega)} \left(||\langle \omega \lambda(t)^{2}\rangle y_{2}(t,\omega)||_{L^{2}(\rho(\omega \lambda(t)^{2})d\omega)}^{2}+||\langle \omega \lambda(t)^{2}\rangle y_{1}(t,\omega)||_{L^{2}(\rho(\omega \lambda(t)^{2})d\omega)}^{2}\right)\end{split}\end{equation}
\end{fleqn}
\end{lemma}
\begin{proof}
We can read off most of the required estimates on $L_{1}$ and $N$ from (5.83), (5.84), pg. 155 of \cite{wm}, along with pgs. 180-187 of \cite{wm}. The only extra information we need is an estimate on $$||\mathcal{F}\left(\sqrt{\cdot} \left(L_{1}(f_{2})-L_{1}(f_{1})\right)(t,\cdot \lambda(t))\right)(\omega \lambda(t)^{2})||_{L^{2}(\rho(\omega \lambda(t)^{2}) d\omega)}$$ with quantitative constants, which is obtained by direct estimation of \eqref{L1def}. \end{proof}
Now, we are ready to solve \eqref{yeqn}. For $y \in Z$, define $T(y)$ by
\begin{equation}\label{Tdef}T(y)(t,\omega):=\int_{t}^{\infty}\frac{\sin((x-t)\sqrt{\omega})}{\sqrt{\omega}}\left(-\mathcal{F}(\sqrt{\cdot}\left(F_{5}+F_{3}(v_{6}(y))\right)(x,\cdot \lambda(x)))(\omega \lambda(x)^{2})+F_{2}(y)(x,\omega)\right) dx, \quad t \geq T_{0}\end{equation}
Then, we have
\begin{proposition}\label{contractionprop} There exists $\beta_{0}>0$ such that, for all $\beta>\beta_{0}$, there exists $C_{9}>0$ such that, for all $C_{Z}>C_{9}$, there exists $T_{1}>0$ such that, if $T_{0}>T_{1}$, then, $T$ is a strict contraction on $\overline{B}_{1}(0) \subset Z$.\end{proposition}
\begin{proof}
Let $y \in \overline{B}_{1}(0) \subset Z$. Any constant $C$ appearing in this proof is  \emph{independent} of $y$ and $T_{0}$. First, 
\begin{equation}\begin{split}&||\partial_{t}^{k}T(y)(t,\omega)||_{L^{2}(\rho(\omega \lambda(t)^{2})d\omega)}\\
&\leq \sqrt{C_{\rho}} \int_{t}^{\infty}(x-t)^{1-k} \left(\left(\frac{x}{t}\right)^{C_{u}}+\left(\frac{x}{t}\right)^{C_{l}}\right)\begin{aligned}[t]&\left(\frac{||F_{5}(x,r)||_{L^{2}(r dr)}}{\lambda(x)^{2}} +||\mathcal{F}(\sqrt{\cdot} F_{3}(x,\cdot\lambda(x)))(\omega\lambda(x)^{2})||_{L^{2}(\rho(\omega \lambda(x)^{2})d\omega)} \right.\\
&\left.+||F_{2}(x,\omega)||_{L^{2}(\rho(\omega \lambda(x)^{2})d\omega)}\right)dx, \quad k=0,1\end{aligned}\end{split}\end{equation}
where we used \eqref{rhoscaling} and
$$\frac{|\partial_{t}^{k}\left(\sin((x-t)\sqrt{\omega})\right)|}{\sqrt{\omega}} \leq (x-t)^{1-k}$$
Since the inequality in \eqref{cofconstr1p1} is strict, there exists a positive constant $\epsilon_{4}<\frac{1}{100}$ such that \eqref{cofconstr1p1} is true with $\frac{179}{267}$ replaced by $\frac{179}{267} + 4\epsilon_{4}$ on the left-hand side. Then, we directly apply  Lemmas \ref{uansatzerrorest}, \ref{l1nlemma} and \ref{f2lemma}, to get
\begin{equation}\label{dtkTest}\begin{split}&||\partial_{t}^{k}T(y)(t,\omega)||_{L^{2}(\rho(\omega \lambda(t)^{2})d\omega)}\\
&\leq \frac{C_{5}' \log^{30}(t)}{t^{2+2\delta+k}}\left(\frac{1}{\frac{5}{2}-C_{u}}+\frac{1}{\frac{5}{2}}\right)\\
&+ \frac{\beta C_{5}' \sqrt{C_{\rho}}\log^{30}(t)}{t^{2+k+2\delta}} \begin{aligned}[t]&\left(M(1+2||\mathcal{K}||_{\mathcal{L}(L^{2}_{\rho})})+M^{2}\left(\frac{1}{4}+2||[\xi\partial_{\xi},\mathcal{K}]||_{\mathcal{L}(L^{2}_{\rho})} + ||\mathcal{K}||_{\mathcal{L}(L^{2}_{\rho})}^{2}\right)\right.\\
&\left.+C_{2}\left(\frac{1}{2}+||\mathcal{K}||_{\mathcal{L}(L^{2}_{\rho})}\right)+4\left(M^{2}+\frac{C_{2}}{3}(3+2\pi^{2})\right)\right)\\
&\cdot \left(\frac{1}{(3-C_{u}+2\delta-4\epsilon_{4})}+\frac{1}{(3-C_{l}+2\delta-4\epsilon_{4})}\right)+\frac{C(1+\beta^{3})  \log^{60}(t)}{t^{2+k+2\delta+\delta_{3}}}\end{aligned}\\
&\leq \frac{\beta C_{5}' \log^{30}(t)}{t^{2+k+2\delta}}\left(\frac{1}{3}-\epsilon_{3}\right)+\frac{C(1+\beta^{3})  \log^{60}(t)}{t^{2+k+2\delta+\delta_{3}}}\end{split}\end{equation}
for some $\epsilon_{3}>0$, $\beta$ sufficiently large, depending on $\epsilon_{4}$, and if $T_{0}$ is further constrained to satisfy $T_{0}>e^{\frac{15}{2\epsilon_{4}}}$. Here, $\delta_{3}>0$ is given by
$$\delta_{3}=\text{min}\{2\alpha,\text{ }1-6C_{u},\text{ }2\delta-C_{l}\}$$
The $\epsilon_{4}$ terms in \eqref{dtkTest} arise from noting that $x \mapsto \frac{\log^{30}(x)}{x^{4\epsilon_{4}}}$ is decreasing on $(e^{\frac{15}{2\epsilon_{4}}},\infty)$. The last inequality in \eqref{dtkTest} is true by \eqref{cofconstr1p1} and $\frac{1}{(3-C_{u}+2\delta-4\epsilon_{4})}+\frac{1}{(3-C_{l}+2\delta-4\epsilon_{4})} < \frac{179}{267}+2\epsilon_{4}$.
Next,
\begin{equation}\label{weightT}\begin{split}&||\omega \lambda(t)^{2} T(y)(t,\omega)||_{L^{2}(\rho(\omega \lambda(t)^{2})d\omega)}\\
&\leq C \lambda(t) \int_{t}^{\infty}\left(\frac{x}{t}\right)^{M+C_{l}}\begin{aligned}[t]&\left(\frac{||L_{\frac{1}{\lambda(x)}}F_{5}(x,r)||_{L^{2}(r dr)}}{\lambda(x)} +||\sqrt{\omega\lambda(x)^{2}}\mathcal{F}(\sqrt{\cdot} F_{3}(x,\cdot\lambda(x)))(\omega\lambda(x)^{2})||_{L^{2}(\rho(\omega \lambda(x)^{2})d\omega)} \right.\\
&\left.+||\sqrt{\omega \lambda(x)^{2}}F_{2}(x,\omega)||_{L^{2}(\rho(\omega \lambda(x)^{2})d\omega)}\right)dx\end{aligned}\\
&\leq \frac{C}{t^{2+2\delta}}\left(\frac{\log^{2}(t)}{t^{-1-2\delta+5\alpha}}+\frac{\log^{6}(t)}{t^{1-2\delta +2\alpha -3C_{u}}} + \frac{\log^{30}(t)}{t^{\frac{3}{2}-2\delta-C_{l}-15C_{u}}}+\frac{(\beta^{3}+1+C_{Z})\log^{30}(t)}{t^{1-C_{u}}}\right)\end{split}\end{equation}
where we used $|\sin((x-t)\sqrt{\omega})| \leq 1$. Similarly, there exist $C_{7}>0$ and $0<\epsilon_{1}<\frac{1}{3}$, both \emph{independent} of $C_{Z}, y$ and $T_{0}$, such that
\begin{equation}\label{weightdtT}\begin{split}&||\sqrt{\omega \lambda(t)^{2}}\partial_{t}T(y)(t,\omega)||_{L^{2}(\rho(\omega \lambda(t)^{2})d\omega)}\\
&\leq C \int_{t}^{\infty} \left(\frac{x}{t}\right)^{M+C_{l}} ||\sqrt{\omega \lambda(x)^{2}} \mathcal{F}(\sqrt{\cdot}F_{5}(x,\cdot\lambda(x)))(\omega \lambda(x)^{2})||_{L^{2}(\rho(\omega \lambda(x)^{2}) d\omega)} dx\\
&+\sqrt{C_{\rho}} \int_{t}^{\infty} \left(\left(\frac{x}{t}\right)^{2C_{l}}+\left(\frac{x}{t}\right)^{C_{u}+C_{l}}\right) ||\sqrt{\omega \lambda(x)^{2}} F_{2}(y)(x,\omega)||_{L^{2}(\rho(\omega \lambda(x)^{2})d\omega)} dx\\
&+\sqrt{C_{\rho}} \int_{t}^{\infty}\left(\left(\frac{x}{t}\right)^{2C_{l}}+\left(\frac{x}{t}\right)^{C_{u}+C_{l}}\right) ||\sqrt{\omega \lambda(x)^{2}} \mathcal{F}(\sqrt{\cdot}L_{1}(v_{6}(y))(x,\cdot\lambda(x)))(\omega\lambda(x)^{2})||_{L^{2}(\rho(\omega\lambda(x)^{2})d\omega)} dx\\
&+C \int_{t}^{\infty} \left(\frac{x}{t}\right)^{M+C_{l}} ||\sqrt{\omega \lambda(x)^{2}} \mathcal{F}(\sqrt{\cdot}N(v_{6}(y))(x,\cdot\lambda(x)))(\omega\lambda(x)^{2})||_{L^{2}(\rho(\omega\lambda(x)^{2})d\omega)} dx\\
&\leq \frac{C_{7}(\beta^{3}+1) \log^{30}(t)}{t^{3+2\delta}}+ \frac{C_{Z} (\frac{1}{3}-\epsilon_{1})\log^{30}(t)}{t^{3+2\delta}}\leq \frac{C_{Z}\log^{30}(t)}{ t^{3+2\delta}}\left(\frac{1}{3}-\frac{\epsilon_{1}}{2}\right)\end{split}\end{equation}
where the second to last inequality in \eqref{weightdtT} follows from the strictness of the inequality in \eqref{cofconstr2} and is true for $T_{0}$ sufficiently large, depending on absolute constants (not $C_{Z}$), with a similar argument used in \eqref{dtkTest}. The last inequality in \eqref{weightdtT} is true, as long as $C_{Z}>\frac{2C_{7}(\beta^{3}+1)}{\epsilon_{1}}$, which we can enforce, recalling that $C_{7}$ and $\epsilon_{1}$ are independent of $C_{Z}$. With the same argument used in \eqref{dtkTest}, we get
$$||\sqrt{\omega \lambda(t)^{2}} T(y)(t,\omega)||_{L^{2}(\rho(\omega \lambda(t)^{2})d\omega)} \leq \frac{C(1+\beta^{3}) \log^{30}(t)}{t^{3+2\delta-C_{u}}}$$ 
Thus, by \eqref{dtkTest}, \eqref{weightT}, and \eqref{weightdtT}, there exists $\beta'>0$ such that, for all $\beta > \beta'$, there exists $C_{6}>0$ such that, for all $C_{Z}>C_{6}$, there exists $T_{2}>0$ such that if $T_{0}>T_{2}$ then, for all $y \in \overline{B}_{1}(0) \subset Z$, $||Ty||_{Z} <1$. So, $T$ maps $\overline{B}_{1}(0)$ into itself, and it remains to show that $T$ is a strict contraction. If $y_{2},y_{1} \in \overline{B}_{1}(0)$, then, since the expression for $F_{2}$ (\eqref{F2def}) depends linearly on $y$, we have
\begin{equation} T(y_{2})-T(y_{1}) = \int_{t}^{\infty}\frac{\sin((x-t)\sqrt{\omega})}{\sqrt{\omega}} \begin{aligned}[t]&\left(F_{2}(y_{2}-y_{1})(x,\omega)-\mathcal{F}(\sqrt{\cdot}(N(v_{6}(y_{2}))-N(v_{6}(y_{1})))(x,\cdot\lambda(x)))(\omega \lambda(x)^{2})\right.\\
&\left.-\mathcal{F}(\sqrt{\cdot}(L_{1}(v_{6}(y_{2}))-L_{1}(v_{6}(y_{1})))(x,\cdot\lambda(x)))(\omega \lambda(x)^{2})\right)dx\end{aligned}\end{equation}
Then, by the same procedure used in \eqref{dtkTest}, \eqref{weightT}, and \eqref{weightdtT}, we get, for $\beta$ and $C_{Z}$ sufficiently large, and some $\epsilon_{2}>0$,
\begin{equation}\label{dtkTlip} \begin{split}||\partial_{t}^{k}\left(T(y_{2})-T(y_{1})\right)||_{L^{2}(\rho(\omega\lambda(t)^{2})d\omega)}\leq||y_{1}-y_{2}||_{Z}\left(\frac{\beta C_{5}' \log^{30}(t)}{t^{2+k+2\delta}}\left(\frac{1}{3}-\epsilon_{2}\right)+\frac{C (1+\beta^{3}) \log^{60}(t)}{t^{2+k+2\delta+\delta_{3}}}\right), \quad k =0,1\end{split}\end{equation}
\begin{equation}\label{weightTlip}||\sqrt{\omega \lambda(t)^{2}} (T(y_{2})-T(y_{1})) ||_{L^{2}(\rho(\omega\lambda(t)^{2})d\omega)} \leq \frac{C (1+\beta^{3})||y_{2}-y_{1}||_{Z}\log^{30}(t)}{t^{3+2\delta-C_{u}}}\end{equation}
\begin{equation}\label{weightTlip2} ||\omega\lambda(t)^{2}(T(y_{2})-T(y_{1}))||_{L^{2}(\rho(\omega\lambda(t)^{2})d\omega)} \leq \frac{C ||y_{1}-y_{2}||_{Z} C_{Z} \log^{30}(t)}{t^{3+2\delta-C_{u}}}\end{equation}
\begin{equation}\label{weightdtTlip} ||\sqrt{\omega\lambda(t)^{2}}\partial_{t}\left(T(y_{2})-T(y_{1})\right)||_{L^{2}(\rho(\omega \lambda(t)^{2})d\omega)} \leq \frac{C_{Z}\log^{30}(t)||y_{1}-y_{2}||_{Z}}{t^{3+2\delta}} \left(\frac{1}{3}-\epsilon_{2}\right)\end{equation}
This completes the proof of the Proposition.
 \end{proof}
By Proposition \ref{contractionprop}, completeness of $(\overline{B}_{1}(0),||\cdot||_{Z})$, and the Banach fixed point theorem, there exists $\beta>0$ and $C_{Z}>0$ such that, for all $T_{0}$ sufficiently large, there exists $y_{0} \in \overline{B}_{1}(0)$ such that $T(y_{0})=y_{0}$. By inspection of \eqref{Tdef}, and \eqref{yeqn}, $y_{0}$ is a solution to \eqref{yeqn}. By the derivation of \eqref{yeqn} from \eqref{v6eqn}, the function $v_{6,0}$, defined by \eqref{v6intermsofy} with $y=y_{0}$, for $r>0$, is a solution to \eqref{v6eqn}. We also note that, $v_{6,0}(t,\cdot)$ admits a continuous extension to $[0,\infty)$, by defining $v_{6,0}(t,0):=0$, by Lemma 5.1 of \cite{wm}.
\section{Decomposition of the solution as in \ref{solndecomp}}
Let
$$u(t,r):=Q_{1}(\frac{r}{\lambda(t)})+u_{ansatz}(t,r)+v_{6,0}(t,r)$$
$$v_{rad}(t,r)=v_{2}(t,r)+v_{2,2}(t,r)+v_{2,4}(t,r)$$
In this section, we prove the following lemma, which finishes the proof of Theorem \ref{mainthm}.
\begin{lemma} 
\begin{equation}\label{enofremainder}E(u-Q_{\frac{1}{\lambda(t)}}-v_{rad},\partial_{t}\left(u-v_{rad}\right)) \leq \frac{C \log^{2}(t)}{t^{2-2C_{u}}}\end{equation}\end{lemma}
\begin{proof}
We first note that
\begin{equation} \begin{split}u(t,r)-Q_{\frac{1}{\lambda(t)}}(r) - v_{rad} &= v_{6,0}+\chi_{\leq 1}(\frac{r}{\lambda(t) g(t)})(u_{ell}+u_{ell,2}) + (1-\chi_{\leq 1}(\frac{r}{\lambda(t) g(t)}))(w_{1}+v_{ex}+u_{w,2})\\
&-\chi_{\leq 1}(\frac{r}{\lambda(t)g(t)}) (v_{2}+v_{2,2}) + f_{5,0}+f_{ex,sub}+f_{ell,2}+f_{2,2}+u_{N_{0}}\\
&+(u_{N_{0},ell}-\frac{r\lambda(t)}{4} \langle m_{\leq 1}(\frac{R \lambda(t)}{2h(t)}) e_{N_{0}}(t,R\lambda(t)),\phi_{0}(R)\rangle_{L^{2}(R dR)})m_{\leq 1}(\frac{2r}{t})\\
&+m_{\geq 1}(\frac{r}{t}) u_{N_{0},corr,2}+u_{N_{2}}\end{split}\end{equation}
Since, for non-constant $\lambda$, $\partial_{t}Q_{\frac{1}{\lambda(t)}}(r), \partial_{t}\left(\chi_{\geq 1}(\frac{r}{g(t)\lambda(t)})w_{1}(t,r)\right) \not \in L^{2}(r dr)$, the most delicate term to estimate will be the following, which \emph{is} finite.
\begin{equation}\label{delicatedt}||\partial_{t}\left(Q_{\frac{1}{\lambda(t)}}(r)+\chi_{\geq 1}(\frac{r}{g(t)\lambda(t)})w_{1}(t,r)\right)||_{L^{2}(r dr)}\end{equation}
Before explaining how to do this, we note that
\begin{equation}\begin{split} ||v_{6,0}(t,r)||_{\dot{H}^{1}_{e}} &= ||v_{6,0}(t,r\lambda(t))||_{\dot{H}^{1}_{e}} \leq C\left(||L(v_{6,0}(t,r\lambda(t)))||_{L^{2}(r dr)} + ||v_{6,0}(t,r\lambda(t))||_{L^{2}(r dr)}\right) \\
&\leq C \lambda(t) ||\langle \omega \lambda(t)^{2}\rangle y_{0}(t,\omega)||_{L^{2}(\rho(\omega\lambda(t)^{2})d\omega)} \leq \frac{C \log^{30}(t)}{t^{2+2\delta-C_{u}}}\end{split}\end{equation}
By the transference identity,
\begin{equation}\begin{split}||\partial_{t}v_{6,0}(t,r)||_{L^{2}(r dr)}&\leq C \frac{|\lambda'(t)|}{\lambda(t)} ||\sqrt{\frac{\lambda(t)}{r}} \mathcal{F}^{-1}(y_{0}(t,\frac{\cdot}{\lambda(t)^{2}}))(\frac{r}{\lambda(t)})||_{L^{2}(r dr)}\\
&+ C ||\sqrt{\frac{\lambda(t)}{r}} \mathcal{F}^{-1}(\partial_{1}y_{0}(t,\frac{\cdot}{\lambda(t)^{2}}))(\frac{r}{\lambda(t)})||_{L^{2}(r dr)}\\
&+C\frac{|\lambda'(t)|}{\lambda(t)} ||\sqrt{\frac{\lambda(t)}{r}} \mathcal{F}^{-1}(\mathcal{K}(y_{0}(t,\frac{\cdot}{\lambda(t)^{2}})))(\frac{r}{\lambda(t)})||_{L^{2}(r dr)}\\
&\leq \frac{C\log^{30}(t)}{t^{3+2\delta-2C_{u}}}\end{split}\end{equation}
The rest of the estimates needed to establish \eqref{enofremainder} are from \eqref{velldef} and \eqref{c1def} (for $u_{ell}$), \eqref{uell2def}, and the formulae following it, for $u_{ell,2}$,  Lemma \ref{w1estlemma} (for $w_{1}$), Lemma \ref{estlemma}, for $v_{ex}$, Lemma \ref{uw2minuselllemma}, for $u_{w,2}$, Lemma \ref{v2estlemma} for $v_{2}$, Lemma \ref{v22lemma} for $v_{2,2}$, Lemma  \ref{w50lemma}, for $f_{5,0}$, Lemma \ref{wexsublemma}, for $f_{ex,sub}$, Lemma \ref{fell2lemma}, for $f_{ell,2}$, Lemma \ref{u22lemma}, for $f_{2,2}$,  Lemma \ref{un0lemma}, for $u_{N_{0}}$, Lemma \ref{un0ellminusipestlemma}, Lemma \ref{un0corrlemma} for $u_{N_{0},corr,2}$, Lemma \ref{un2lemma} for $u_{N_{2}}$. Now, we return to \eqref{delicatedt}. We note that
$$\chi_{\geq 1}(\frac{r}{\lambda(t)g(t)}) =1-\chi_{\leq 1}(\frac{r}{\lambda(t)g(t)}) =1, \quad r \geq 2\lambda(t) g(t)$$
Therefore, we will show in detail how to estimate \eqref{delicatedt}, when the $L^{2}$ norm is taken over the region $r \geq 2g(t)\lambda(t)$, since the norm in the region $r \leq 2 g(t)\lambda(t)$ is a direct estimation, using Lemma \ref{w1estlemma}. 
From \eqref{w1simp},
\begin{equation}\label{w1fordelicateenergy}w_{1}(t,r) = -\frac{2}{r} \left(\lambda(t+r)-\lambda(t)-r \lambda'(t+r)\right)+2 r \int_{1}^{\infty} \lambda''(t+r y) (y-\sqrt{y^{2}-1}) dy\end{equation}
and
$$Q_{1}(\frac{r}{\lambda(t)}) = \pi-2\arctan(\frac{\lambda(t)}{r})$$
Therefore,
\begin{equation}\begin{split}Q_{1}(\frac{r}{\lambda(t)})+w_{1}(t,r) &= \pi-2\arctan(\frac{\lambda(t)}{r})+2\frac{\lambda(t)}{r}\\
&-\frac{2}{r}\left(\lambda(t+r)-r\lambda'(t+r)\right)+2 r \int_{1}^{\infty} \lambda''(t+ry) (y-\sqrt{y^{2}-1}) dy\end{split}\end{equation}
Therefore,
\begin{equation}\begin{split}\partial_{t}\left(Q_{1}(\frac{r}{\lambda(t)})+w_{1}(t,r)\right)&=\frac{2 \lambda'(t)}{r} \left(\frac{\lambda(t)^{2}}{r^{2}+\lambda(t)^{2}}\right) - \frac{2}{r}\left(\lambda'(t+r)-r\lambda''(t+r)\right)\\
&+2 r \int_{1}^{\infty} \lambda'''(t+r y) (y-\sqrt{y^{2}-1}) dy\end{split}\end{equation}
and this gives
$$||\partial_{t}\left(Q_{1}(\frac{r}{\lambda(t)})+\chi_{\geq 1}(\frac{r}{\lambda(t)g(t)}) w_{1}(t,r)\right)||_{L^{2}(r \geq 2 \lambda(t) g(t), r dr)}^{2} \leq \frac{C \log^{2}(t)}{t^{2-2C_{u}}}$$
which completes the proof of the lemma.
  \end{proof}

\printbibliography
Department of Mathematics, University of California, San Diego\\
\emph{E-mail address:} mkpillai@ucsd.edu
\end{document}